\documentclass[aop]{imsart}

\RequirePackage{amsthm,amsmath,amsfonts,amssymb}
\RequirePackage[numbers]{natbib}
\RequirePackage[colorlinks,citecolor=blue,urlcolor=blue]{hyperref}
\RequirePackage{graphicx}
\usepackage{mathrsfs}
\usepackage{bm}

\startlocaldefs
\theoremstyle{plain}

\newtheorem{theorem}{Theorem}[section]
\newtheorem{claim}[theorem]{Claim}
\newtheorem{lemma}[theorem]{Lemma}
\newtheorem{conjecture}[theorem]{Conjecture}
\newtheorem{corollary}[theorem]{Corollary}
\newtheorem{proposition}[theorem]{Proposition}
\theoremstyle{remark}
\newtheorem{definition}[theorem]{Definition}
\newtheorem{remark}[theorem]{Remark}
\newtheorem{assumption}[theorem]{Assumption}

\newcommand{\sfa}{{\mathsf a}}
\newcommand{\sfb}{{\mathsf b}}
\newcommand{\sfc}{{\mathsf c}}
\newcommand{\sfn}{{\mathsf n}}
\newcommand{\sfu}{{\mathsf u}}
\newcommand{\sfw}{{\mathsf w}}
\newcommand{\sft}{{\mathsf t}}
\newcommand{\sfs}{{\mathsf s}}
\newcommand{\sfx}{{\mathsf x}}

\newcommand{\sfh}{{\mathsf h}}

\newcommand{\sfq}{{\mathsf q}}

\newcommand{\sfL}{{\mathsf L}}
\newcommand{\sfT}{{\mathsf T}}
\newcommand{\sfN}{{\mathsf N}}

\newcommand{\sfP}{{\mathsf P}}

\renewcommand{\cal}{\mathcal}
\newcommand\cA{{\mathcal A}}
\newcommand\cB{{\mathcal B}}
\newcommand{\cC}{{\cal C}}

\newcommand{\cE}{{\cal E}}

\newcommand{\cG}{{\cal G}}

\newcommand{\cL}{{\cal L}}

\newcommand{\cM}{{\cal M}}

\newcommand{\cS}{{\mathcal S}}
\newcommand{\cT}{{\mathcal T}}

\newcommand{\fC}{{\frak C}}

\newcommand{\bme}{{\bm{e}}}

\newcommand{\bmv}{{\bm{v}}}

\newcommand{\bmx}{{\bm{x}}}
\newcommand{\bmy}{{\bm{y}}}

\newcommand{\bms}{\bm s}

\newcommand{\bmmu}{{\bm \mu}}
\newcommand{\bmla}{{\bm \la}}

\newcommand{\rd}{{\rm d}}

\newcommand{\ri}{\mathbf{i}}

\newcommand{\bC}{{\mathbb C}}
\newcommand{\bD}{{\mathbb D}}

\newcommand{\bE}{\mathbb{E}}
\newcommand{\bH}{\mathbb{H}}

\newcommand{\bP}{\mathbb{P}}
\newcommand{\bQ}{\mathbb{Q}}
\newcommand{\bR}{{\mathbb R}}

\newcommand{\bY}{\mathbb Y}
\newcommand{\bZ}{\mathbb{Z}}
\newcommand{\bW}{\mathbb{W}}
\newcommand{\bL}{\mathbb{L}}

\newcommand{\br}{\mathbf r}
\newcommand{\bl}{\mathbf l}

\newcommand{\al}{\alpha}

\newcommand{\la}{\lambda}

\newcommand{\eps}{\varepsilon}

\newcommand{\ii}{\mathbf i}
\newcommand{\uu}{\mathbf u}

\newcommand{\GFF}{\mathbf{GFF}}

\DeclareMathOperator{\supp}{supp}
\DeclareMathOperator{\dist}{dist}

\DeclareMathOperator{\OO}{O}
\DeclareMathOperator{\oo}{o}

\renewcommand{\Re}{\mathop{\mathrm{Re}}}
\renewcommand{\Im}{\mathop{\mathrm{Im}}}

\newcommand{\deq}{\mathrel{\mathop:}=} %define :=

\renewcommand{\leq}{\leqslant}
\renewcommand{\geq}{\geqslant}
\newcommand{\floor}[1] {\lfloor {#1} \rfloor}

\newcommand{\del}{\partial}

\newcommand{\beq}{\begin{equation}}
\newcommand{\eeq}{\end{equation}}
\newcommand{\GT}{\mathbb{GT}}

\def\La{{\scalebox{0.15}{\includegraphics{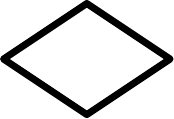}}}}
\def\Lb{{\scalebox{0.15}{\,\includegraphics{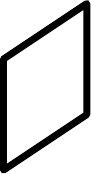}}}}
\def\Lc{{\scalebox{0.15}{\,\includegraphics{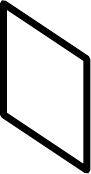}}}}

\newcommand{\cout}{{\omega_+}}
\newcommand{\cin}{{\omega_-}}

\endlocaldefs

\begin{document}

\begin{frontmatter}
\title{Dynamical Loop Equation}
%\title{A sample article title with some additional note\thanksref{t1}}
\runtitle{Dynamical Loop Equation}
%\thankstext{T1}{A sample additional note to the title.}

\begin{aug}
%%%%%%%%%%%%%%%%%%%%%%%%%%%%%%%%%%%%%%%%%%%%%%%
%% Only one address is permitted per author. %%
%% Only division, organization and e-mail is %%
%% included in the address.                  %%
%% Additional information can be included in %%
\author[A]{\fnms{Vadim}~\snm{Gorin}\ead[label=e1]{vadicgor@gmail.com}}
\author[B]{\fnms{Jiaoyang}~\snm{Huang}\ead[label=e2]{huangjy@wharton.upenn.edu}},

%%%%%%%%%%%%%%%%%%%%%%%%%%%%%%%%%%%%%%%%%%%%%%
%% Addresses                                %%
%%%%%%%%%%%%%%%%%%%%%%%%%%%%%%%%%%%%%%%%%%%%%%
\address[A]{University of California, Berkeley\printead[presep={,\ }]{e1}}
\address[B]{University of Pennsylvania \printead[presep={,\ }]{e2}}

\end{aug}

\begin{abstract}
We introduce dynamical versions of loop (or Dyson-Schwinger) equations for large families of two--dimensional interacting particle systems, including Dyson Brownian motion, Nonintersecting Bernoulli/Poisson random walks, $\beta$--corners processes,  uniform and Jack-deformed measures on Gelfand-Tsetlin patterns, Macdonald processes, and $(q,\kappa)$-distributions on lozenge tilings.
Under technical assumptions, we show that the dynamical loop equations lead to Gaussian field type fluctuations.

As an application, we compute the limit shape for $(q,\kappa)$--distributions on lozenge tilings and prove that their height fluctuations converge to the Gaussian Free Field in an appropriate complex structure.
\end{abstract}

\begin{keyword}[class=MSC]
\kwd[Primary ]{60F05}
\kwd{60K35}
\kwd[; secondary ]{82B44}
\end{keyword}

\begin{keyword}
\kwd{Interacting Particle Systems}
\kwd{Random Tilings}
\kwd{Dyson-Schwinger Equations}
\kwd{Gaussian Fluctuations}
\end{keyword}

\end{frontmatter}
%%%%%%%%%%%%%%%%%%%%%%%%%%%%%%%%%%%%%%%%%%%%%%
%% Please use \tableofcontents for articles %%
%% with 50 pages and more                   %%
%%%%%%%%%%%%%%%%%%%%%%%%%%%%%%%%%%%%%%%%%%%%%%
\setcounter{tocdepth}{2}
\tableofcontents

\section{Introduction}

\subsection{Overview}

Interacting particle systems in one time and one space dimension with non-local interactions appear in various probability models, such as random matrices, random tilings, and random growth models.

In the continuous setting, the most famous example is  \cite{MR0148397}, where Dyson observed that the eigenvalues of a matrix-valued Hermitian Brownian motion form an interacting particle system with a logarithmic Coulomb interaction and quadratic potential. In many important cases  one dimensional marginal densities of those $(1+1)$-dimensional interacting particle systems can be computed explicitly, and identified with log-gases.
For such one dimensional log-gas type systems, loop (or Dyson-Schwinger) equation is an important tool to study their macroscopic fluctuations.
They were first used in the theoretical physics literature (e.g., in the work of Migdal \cite{migdal1983}) and were later introduced to the mathematical community by Johansson \cite{MR1487983} to derive macroscopic central limit theorems for general $\beta$-ensembles of eigenvalues of random matrices, see also \cite{MR3010191, borot-guionnet2, KrSh}. One advantage of the asymptotic arguments based on the loop equations is that they tend to be very universal and extend to wide families of stochastic systems.

On the discrete side, marginals given by appropriate versions of log--gases were used in the analysis of random tilings starting from \cite{MR1641839,MR1900323,borodin2010q}  and in the study of random interface growth models starting from \cite{MR1737991}. Generalizing these distibutions, Borodin, Guionnet and Gorin \cite{MR3668648}  introduced discrete $\beta$-ensembles, and studied their Gaussian macroscopic fluctuations using discrete loop equations, which  originated in the work of Nekrasov and his collaborators \cite{Nekrasov, Nek_PS,Nek_Pes}; see also \cite{dimitrov2019log}.

Central limit theorems for macroscopic fluctuations of log-gases (in investigations of random matrices, random tilings, and random partitions) reveal that the limiting fields in these theorems can be viewed as one--dimensional slices of a unique two--dimensional limiting object, the \emph{Gaussian Free Field}, see \cite{borodin2015general, borodin2014clt,dimitrov2019log,duits2018global,bufetov2019fourier} for the results of such type. This fits into a general belief that the Gaussian Free Field should appear as a universal limit of the macroscopic fluctuations for random two-dimensional surfaces of various origins.

The aim of this paper is to lift the asymptotic analysis from one-dimensional sections to full two-dimensional fields by developing very general dynamical versions of loop equations, the \emph{dynamical loop equations}. They provide a new toolbox to study macroscopic fluctuations of two dimensional interacting particle systems, directly applicable to many well-known stochastic dynamics both in discrete and continuous settings, such as: Dyson Brownian motion \cite{MR0148397}, nonintersecting Bernoulli/Poisson random walks \cite{gorin2019universality,konig2002non, huang2017beta}, $\beta$--corner processes \cite{MR3418747, borodin2015general}, measures on Gelfand-Tsetlin patterns \cite{bufetov2018asymptotics,petrov2015asymptotics}, Macdonald processes \cite{borodin2014macdonald}, and $(q,\kappa)$-distributions on lozenge tilings \cite{borodin2010q,dimitrov2019log}. We prove that under mild technical assumptions, the dynamical loop equations lead to fluctuations described by Gaussian fields with accessible covariance.

There were hints in the literature, suggesting that the framework of the loop equations might be adaptable to multi-time two-dimensional setting,
 as in \cite{AGZ,huang2017beta,dimitrov2019asymptotics,huang2020height,dimitrov_knizel2021multi}. However, these results were mostly isolated and
relied on specifics of the stochastic systems they dealt with. In contrast, we work directly with transition probabilities of the stochastic evolutions of interest and allow for them a very general form, parameterized by three analytic functions and an arbitrary real parameter, playing the same role as $\beta$ in the random matrix theory. Our approach covers many of the old examples and simultaneously allows asymptotic analysis of new stochastic systems, which were not accessible by previous tools.

As our main application, we study the height fluctuations of $(q,\kappa)$-distributions on lozenge tilings introduced in \cite{borodin2010q}. In the special case  $\kappa=0$, the $(q,\kappa)$-distributions on lozenge tilings become the $q^{\text{volume}}$--weighted lozenge tilings (or plane partitions), whose particular cases were investigated in \cite{vershik1996statistical,MR1969205,cerf2001low,ahn2020global}. We compute the limit shapes for the $(q,\kappa)$--random tilings of a special class of polygonal domains (``trapezoids'') with arbitrary many sides  and discover that the shapes can be parameterized by algebraic curves in an appropriate coordinate system, thus giving the positive answer to the question raised in \cite[end of Section 2]{borodin2010q}. We further
show that the centered height fluctuations of $(q,\kappa)$-random tilings of these domains converge to the Gaussian free field. In particular, this verifies the  conjecture in \cite[Conjecture 8.4.1]{dimitrov2019log} for $(q,\kappa)$-distributions on lozenge tilings of hexagonal domains\footnote{For another proof of this conjecture see the forthcoming paper \cite{DuitsLiu}.},
and the prediction in \cite{KO_Burgers} for $q^{\text{volume}}$--weighted lozenge tilings.

In the next two subsections we, first, state the dynamical loop equation in its most general form, and then present in details our results on $(q,\kappa)$--random lozenge tilings.

\subsection{Dynamical loop equations}

In the most general discrete setting of our interest, we deal with discrete-time Markov chains with deterministic initial conditions. The state space of the chains, denoted $\GT_n$, consists of $n$--tuples of integers $\bmla: \lambda_1\geq \dots\geq\lambda_n$ which we call signatures of rank $n$ or (in the case $\lambda_n\geq 0$) Young diagrams with at most $n$ rows. \footnote{The notation $\GT$ comes from the names Gelfand and Tsetlin, who used combinatorics of signatures in the study of representation theory of classical Lie groups.} In this text we allow transitions of two types for the Markov chains:
\begin{enumerate}
  \item \emph{Ascending transition}: $n$ is fixed and at each step coordinates $\lambda_i$ either stay the same or grow by one. I.e., non-zero transition probabilities are from $\bm\lambda\in\GT_n$ to $\bm\mu\in \GT_n$ under restriction
      $$\mu_i-\lambda_i\in \{0,1\}\qquad \text{ for }i=1,2,\dots,n.$$
  \item \emph{Descending transition}: $n$ decreases by $1$ at each step and we transition to \emph{interlacing} signatures. I.e., non-zero transition probabilities are from $\bm\lambda\in \GT_{n+1}$ to $\bm\mu\in \GT_{n}$ under the restriction
      \begin{equation}
      \label{eq_interlacing}
        \lambda_1\geq \mu_1\geq \lambda_2\geq\dots\geq\mu_{n}\geq\lambda_{n+1}.
      \end{equation}
\end{enumerate}
Although requiring all transitions to be of one of these two types might seem restrictive, we present in Section \ref{s:examples} a large number of natural examples fitting into our formalism. In fact, the two types of transitions are closely related to each other: as we show in Section \ref{s:PHD} one can be obtained from another by the particle-hole duality. Hence, in this section it suffices to present our results only in the ascending setting.

\begin{figure}[t]
\center
\includegraphics[scale=0.5]{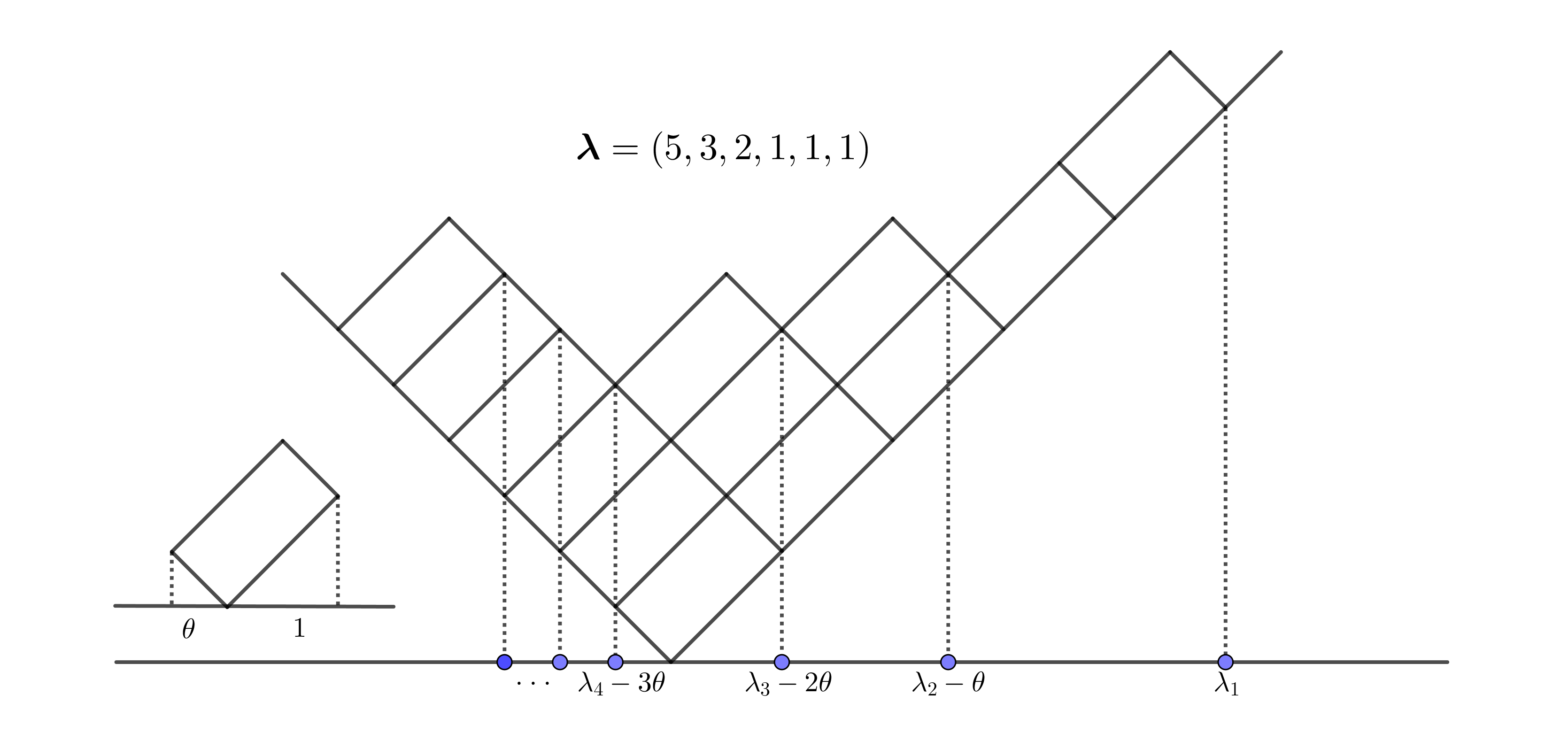}
\caption{We identify $(\la_1,\la_2,\dots,\la_n)\in \GT_n$ with $n$-particles $(\la_1,\la_2-\theta,\la_3-2\theta,\dots,\la_n-(n-1)\theta)$.}
\label{f:ytop}
\end{figure}

Fix $\theta>0$ and let $n\in \bZ_{>0}$. We identify each signature $\bmla=(\la_1\geq \la_2\geq \dots\geq \la_n)\in \GT_n$, with an $n$-particle system $(\la_1,\la_2-\theta,\la_3-2\theta,\dots,\la_n-(n-1)\theta)$. In other words, we use a bijection between $\GT_n$ and the $\theta$-shifted lattice
\begin{align}\label{e:defWtheta}
\bW_\theta^n\deq \bigl\{(x_1,\dots,x_n)\in\mathbb R^n\mid x_1\in \mathbb Z,\quad x_{i}-x_{i+1}\in \theta+\mathbb Z_{\geq 0}, \quad i=1,2,\dots,n-1\bigr\},
\end{align}
given by
\begin{align}\label{e:defyi}
\GT_n \ni \bmla\leftrightarrow \bmx=(x_1, x_2,\dots, x_n)\in \bW_\theta^n ,\qquad x_i=\la_i-(i-1)\theta,\quad i=1,2,\dots, n.
\end{align}
Graphically this procedure is illustrated in Figure \ref{f:ytop}. In the special case $\theta=1$, $\bW_1^n$ is the set of strictly ordered $n$-tuples of integers and we simply denote it by $\bW^n$.

Fix $\theta>0$, a particle configuration $\bmx=(x_1, x_2,\dots,x_n)\in \bW_\theta^n$, an interaction function $b(z)$, and weight functions $\phi^+(z),\phi^-(z)$. Our central object is the following transition probability
\begin{align}\label{e:m1}
\bP(\bmx+\bme|\bmx)=\frac{1}{Z(\bmx)} \prod_{1\leq i<j\leq n}\frac{b(x_i+\theta e_i)-b(x_j+\theta e_j)}{b(x_i)-b(x_j)}\prod_{i=1}^n \phi^+(x_i)^{e_i} \phi^-(x_i)^{1-e_i},
\end{align}
where $\bmx,\bmx+\bme\in \bW_\theta^n$, $\bme=(e_1,e_2,\dots, e_n)\in\{0,1\}^n$ and $Z(\bmx)$ is a normalization constant\footnote{Here and below we silently assume that such a constant exists.} guaranteeing that
\begin{equation}
\label{eq_transition_normalization}
 \sum_{\bme\in\{0,1\}^n} \bP(\bmx+\bme|\bmx)=1.
\end{equation}
We remark that, as in \cite{MR3668648}, $\theta$ can be linked to the parameter $\beta$ of the random matrix theory (taking values $\beta=1,2,4$ for real/complex/quaternionic matrices) through $\theta=\beta/2$. Thus, outside $\theta=1$ we do not expect Markov chains with transitions \eqref{e:m1} to be free fermionic or analysable through the techniques of determinantal point processes.

Some examples of Markov chains with transitions in the form \eqref{e:m1} and their limits are explained in Section \ref{s:examples}. They include nonintersecting Bernoulli/Poisson random walks, $\beta$--corners processes, Dyson Brownian motion, uniform and Jack-deformed measures on Gelfand-Tsetlin patterns, Macdonald processes, and $(q,\kappa)$-distributions on lozenge tilings.  The following dynamical version of the loop (or Schwinger-Dyson) equation provides a unified approach to study the macroscopic fluctuations of dynamics with transitions \eqref{e:m1}.

\begin{theorem}[Dynamical Loop Equation]\label{t:loopeq}
Choose an open set $U\subset \mathbb C$, a particle configuration ${\bmx=(x_1>x_2>\dots>x_n)\in \bW_\theta^n}$ such that the interval $[x_n,x_1]\subset U$, a parameter $\theta>0$, two holomorphic functions $\phi^+(z)$, $\phi^-(z)$ on $U$ and a conformal (i.e., holomorphic and injective) function $b(z)$ on $U$. Assume that the random $n$--tuple $\bme=(e_1,e_2,\dots, e_n)\in\{0,1\}^n$  is distributed according to the transition probability \eqref{e:m1}. Then the following observable is a holomorphic function of $z\in U$:
\begin{align}\label{e:sum1}
\bE\left[\phi^+(z)\prod_{j=1}^n\frac{b(z+\theta)-b(x_j+\theta e_j)}{b(z)-b(x_j)}+\phi^-(z)\prod_{j=1}^n\frac{b(z)-b(x_j+\theta e_j)}{b(z)-b(x_j)}\right].
\end{align}
\end{theorem}
\begin{remark}
 The condition $\bmx\in \bW_\theta^n$ is not crucial for the validity of Theorem \ref{t:loopeq} and one can replace $x_i$ by other complex numbers in $U$; yet, all the examples considered in our paper have this condition or its limiting forms.
\end{remark}

We give two proofs of Theorem \ref{t:loopeq} in Section \ref{Section_proof_of_le}. The first proof is a direct verification that each possible (simple) pole of \eqref{e:sum1} has a zero residue. The second proof identifies \eqref{e:sum1} with a ratio of two partition functions $Z(\bmx)$ of the ensemble \eqref{e:m1} --- with $n+1$ particles and with $n$ particles --- and notices that these partition functions are holomorphic functions of $x_i$. The second proof suggests possible further generalizations of Theorem \ref{t:loopeq}, see Remark \ref{Remark_extension}.

Special or limiting cases of the above dynamical loop equations were known before. For Dyson Brownian motion equivalent expressions were implicitly used in \cite{AGZ,MR2418256,MR1819483} to study the limiting empirical particle density and its fluctuations.
For $\beta$-nonintersecting Poisson random walks and uniformly random lozenge tiling, they were developed by the second named author in \cite{huang2017beta, huang2020height} to study the macroscopic fluctuations. Certain formulas, resembling our dynamical loop equations, can be also found in the work of Nekrasov \cite{Nekrasov4,Nekrasov5}\footnote{There is no Markovian dynamics in these papers, however, certain holomorphic observables on random ensembles of $n$ Young diagrams are being studied as one of the examples. The probabilities in these ensembles depend on a degree $2n$ polynomial and by choosing the polynomial appropriately, one can force each diagram to be either empty or one box --- resembling our $\bme\in\{0,1\}^n$. In order to closer match our weights \eqref{e:m1} with certain specific $b(\cdot)$ and $\phi^\pm(\cdot)$ functions, one needs to deal with (colored) theories with defects in the terminology of \cite{Nekrasov4,Nekrasov5}.}. The first main contribution of this paper is to introduce the general form \eqref{e:m1} (see also \eqref{e:m2}) of transition probability and analyze it by dynamical loop equations.

It is also instructive to compare Theorem \ref{t:loopeq} with two-level discrete loop equations of \cite{dimitrov2019asymptotics, dimitrov_knizel2021multi}. An important distinction is that in our setup $\bmx$ is deterministic, while in \cite{dimitrov2019asymptotics,dimitrov_knizel2021multi} it is random (of log-gas type). We are going to significantly rely on this feature in the applications of Theorem \ref{t:loopeq}: it allows us to iterate the equations and reach joint distributions at arbitrary times, rather than only at adjacent ones.

\smallskip

Proceeding to a broader context, we refer to \cite{guionnet2019asymptotics} for an introduction to the loop equations and historic overview. A general phenomena in the study of many particle systems with strong mean field interactions is that the moments of a large class of test functions admit an infinite system of equations that are called the loop equations, Dyson-Schwinger
equations,  Master equations, or Ward identities depending on the context. These equations are often asymptotically closed in the limit of growing dimensions and can be used to get delicate asymptotic information on the particle systems. This strategy has been developed at the formal level (i.e.\ assuming  that the asymptotic expansions should exist) in physics starting from \cite{ambjorn1990properties}, see  \cite{Borot:2013lpa, chekhov2006matrix, chekhov2011topological, eynard2005topological,eynard2016counting, Eynard:2007kz} for more recent and much more detailed studies. The approach was put on the rigorous grounds starting from \cite{MR1487983} and we refer to \cite{guionnet2019asymptotics} for an outline, as well as many references on the related developments in the mathematical literature. The theoretical physics literature made a curious observation that the asymptotic expansions in these particle systems (developed through the loop equations or by other means) satisfy the so-called topological recursion, which then connects them to enumerative geometry and generating functions for maps, cf.\  \cite{t1974magnetic}, \cite{brezin1978planar}, and \cite{tutte1963census}. It is yet to be seen, whether connections to enumerative geometry can be also found in the study of our dynamical loop equations.

\bigskip

In Section \ref{s:analysis} we prove that under technical assumptions the dynamical loop equations lead to Gaussian fluctuations with explicitly computable mean and covariance. More precisely, take a small scaling parameter $\varepsilon$ and assume that $n$ is inverse proportional to $\eps$, i.e.\ $n=\lfloor \sfn\eps^{-1}\rfloor$ for some $\sfn>0$. We encode the particle configurations $\bmx$ and $\bmx+\bme$ by a smoothed version of the empirical density:
\begin{align*}
\rho(s;\bmx)=\frac{1}{\theta}\sum_{i=1}^n \bm1(\varepsilon x_i \leq s\leq \varepsilon x_i+\varepsilon \theta ), \quad \rho(s;\bmx+\bme)=\frac{1}{\theta}\sum_{i=1}^n \bm1(\varepsilon (x_i+e_i) \leq s\leq \varepsilon (x_i+e_i)+\varepsilon \theta ).
\end{align*}
Then under the transition probability \eqref{e:m1}, $\rho(s;\bmx+\bme)$ is a random probability density. Its difference with $\rho(s;\bmx)$ has the following decomposition:
\begin{equation}
\label{eq_decomposition_field}
\rho(s;\bmx+\bme)-\rho(s;\bmx)=\text{``stochastic term"}+\text{``deterministic term"}.
\end{equation}
In Theorem \ref{t:loopstudy}, under technical assumptions, by analyzing the loop equation \eqref{e:sum1} we prove that as $\eps\to 0$ the stochastic term is asymptotically a Gaussian field and obtain asymptotic expansions for the deterministic term and the covariance of the stochastic term. As a consequence, for any two dimensional interacting particle system, if we can view it as a Markov process with transition probabilities \eqref{e:m1} (with potentially time-dependent functions $b$ and $\phi^\pm$), then its asymptotic limit shape can be found from deterministic terms in \eqref{eq_decomposition_field} and macroscopic fluctuations converge to a two-dimensional Gaussian field, which is found from the stochastic terms in \eqref{eq_decomposition_field}. Note that technical assumptions of Theorem \ref{t:loopstudy} still remain to be checked on case-by-case basis; we demonstrate in Section \ref{Section_assumptions} how this is done for $(q;\kappa)$--distributions on lozenge tilings.

\smallskip

Our analysis of the decomposition \eqref{eq_decomposition_field} based on Theorem \ref{t:loopeq} has one important difference from the numerous previous one-dimensional results based on the loop equations for log-gas type systems, as in \cite{MR3010191, borot-guionnet2, KrSh,MR1487983,MR3668648}. Those texts relied on a two-step procedure, in which one needs to establish a priori estimates on the concentration of the empirical density in the first step before using the loop equations in the second step. Such estimates rely on a different set of tools, such as variational problems or large deviations techniques. In contrast, our approach  does not require any a priori estimates and extracts all asymptotic information directly from Theorem \ref{t:loopeq}: since we deal with elementary one-step transitions, we can rely on the determinstic fact that $\bmx+\bme$ is close to $\bmx$.

\smallskip

A reader might be wondering why we chose to deal with the smoothed empirical density $\rho(s;\bmx)$ rather than with a more standard empirical measure $\eps\sum_{i=1}^n \delta_{\eps x_i}$. While we could have dealt with the latter as well (at the expense of modifications of some of the formulas), the former leads to ``better'' functions (polynomials, rational functions, etc) before taking limits, which is convenient when we aim at algebraic answers, as in the following Theorems \ref{t:arctic} and \ref{t:limitshape}. For instance, if $x_i=-i\theta$, $1\le i \le n$, then
$$
 \exp\left(\theta \int_{\mathbb R} \frac{\rho(s;\bmx)}{z-s} \rd s \right)= \exp\left( \int^{0}_{-\eps \theta n} \frac{\rd s}{z-s} \right)=
 \frac{z+\eps\theta n}{z}.
$$
 On the other hand, if we had used the standard empirical measure, then a similar expression would have become a much more singular function of $z$.

\subsection{Height fluctuations of $(q,\kappa)$-distributions on lozenge tilings}
\label{s:heightf}

We illustrate the strength of Theorem \ref{t:loopeq} by applying it to the study of the fluctuations of height functions of $(q,\kappa)$-distributions on lozenge tilings introduced in \cite{borodin2010q}, which were not accessible by previous methods.

We use the triangular grid with coordinate system $(t,x)$ based on two lattice directions inclined to each other under $120$ degrees, as in the left panel of Figure \ref{Fig_trapezoid}. $x$ increases in the vertical direction and $t$ increases in the down--right direction. The lozenges are defined as unions of adjacent elementary triangles of the grid and they have three types:
\begin{center}
\includegraphics[scale=0.5]{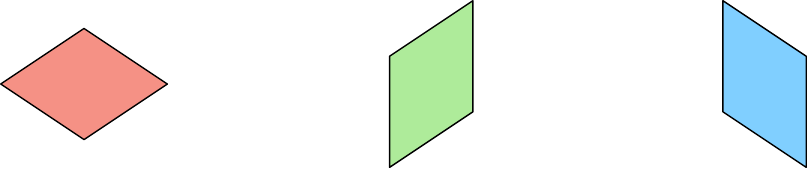}
\end{center}
We are going to refer to the first type as \emph{horizontal} lozenges; they are shown in red in our drawings.

\begin{figure}[t]
\begin{center}
 \includegraphics[scale=0.30]{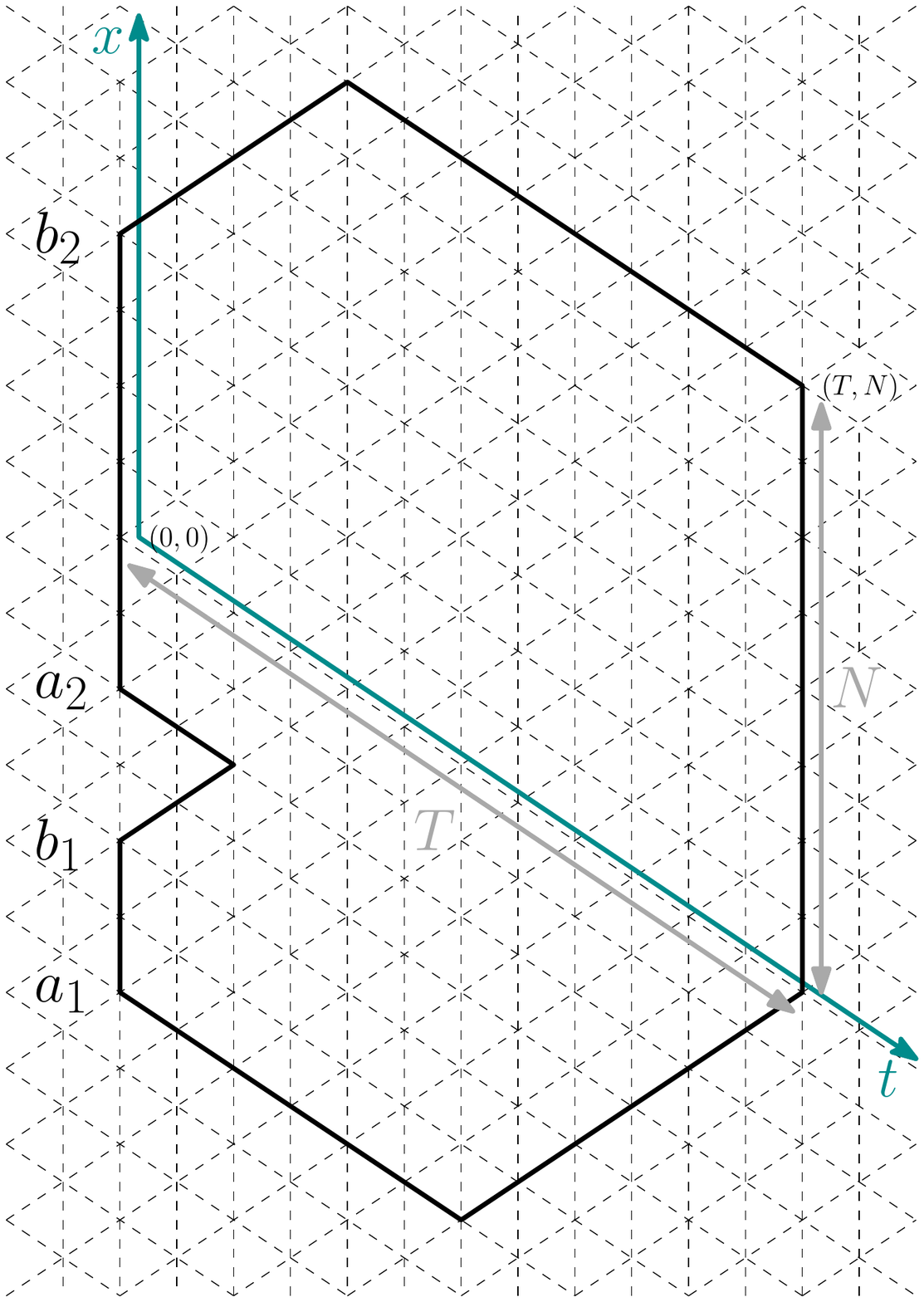}\qquad \,
 \includegraphics[scale=0.30]{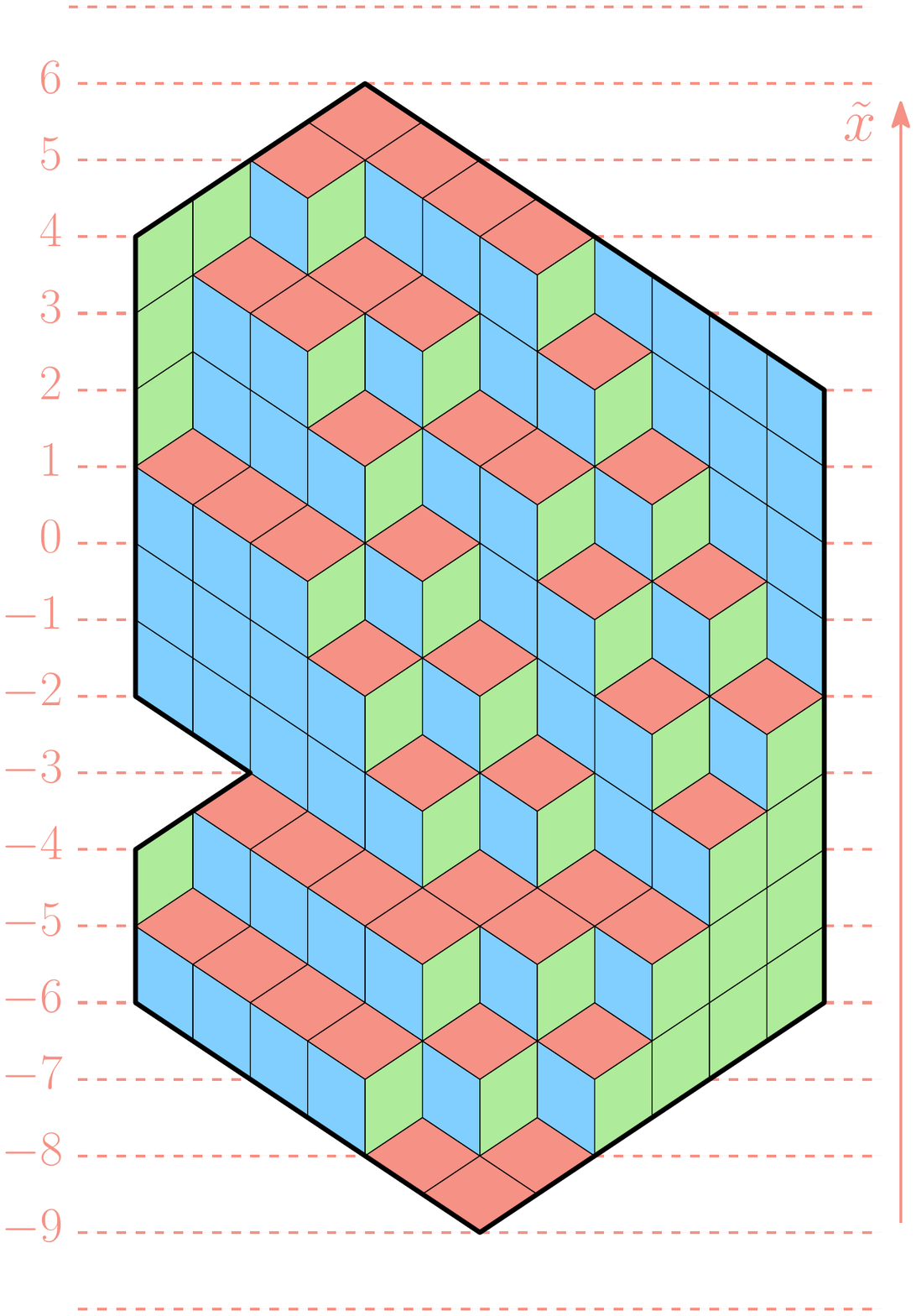} \qquad \,
 \includegraphics[scale=0.30]{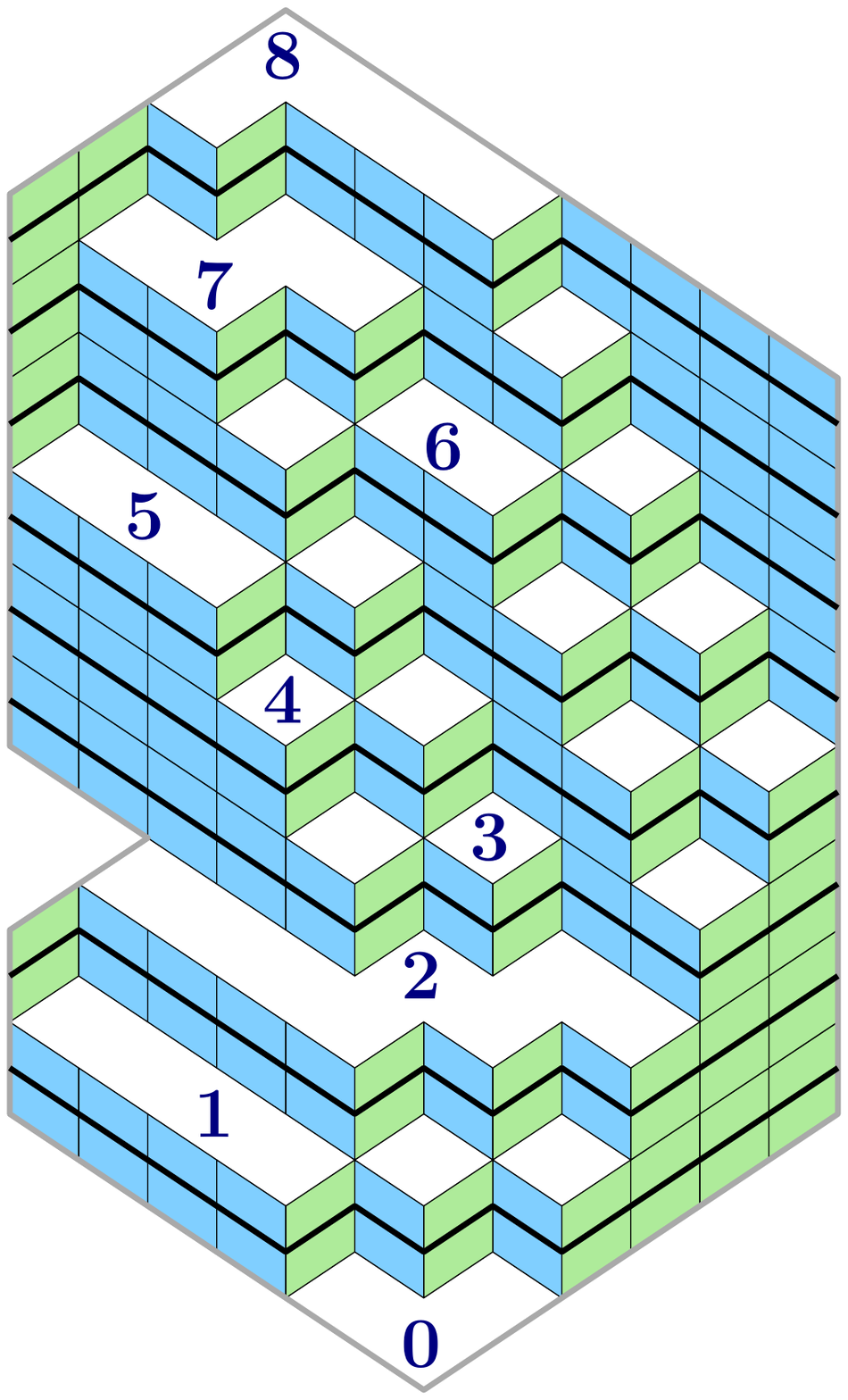}
 \caption{Left panel: trapezoid domain with $N=8$, $T=12$, $a_1=-6$, $b_1=-4$, $a_2=-2$, $b_2=4$; and $(t,x)$--coordinate system. Middle panel: one possible lozenge tiling and the values for $\tilde x$ in \eqref{eq_lozenge_weight}. Right panel: non-intersecting paths and values of the height function.  }
 \label{Fig_trapezoid}
 \end{center}
 \end{figure}

For any integers $N>0$, $T>0$, and $-T\leq a_1<b_1<a_2<b_2<\dots<a_r<b_r\leq N$, with $\sum_{i=1}^r (b_i-a_i)=N$, we consider a trapezoidal domain $P$ with right vertical side $[0,N]$ at horizontal coordinate $T$ and left vertical sides $[a_i,b_i]$ at horizontal coordinate $0$. Other sides of the domain are segments going in down-right direction from points $(0,a_i)$, $i=1,\dots,r$; segment in up-left direction from the point $(T,N)$; segments in up-right direction from points  $(0,b_i)$, $i=1,\dots,r$; and segment in down-left direction from $(T,0)$. As a result, $P$ is a polygonal domain with $3r+3$ sides. We refer to  the left panel of Figure  \ref{Fig_trapezoid} for an example with $r=2$.

We let $\Omega(P)$ denote  the set of all lozenge tilings of the domain $P$; one such tiling for a particular choice of the parameters is shown in the middle panel of Figure \ref{Fig_trapezoid}. Uniformly random lozenge tilings of $P$ are now well understood, see \cite{VG2020} for general overview and \cite{MR3278913, petrov2015asymptotics,GorinPanova,DuseMetcalfeI,DuseMetcalfeII, DuseJohanssonMetcalfe,bufetov2018fluctuations} for more specific results, and our task is to go beyond these results and consider more advanced measures on tilings. The $(q,\kappa)$-distribution on lozenge tilings of $P$ assigns to any $\cT\in \Omega(P)$ the probability
 \begin{align}\label{e:qtiling}
 \bP(\cT)=\frac{w(\cT)}{\sum_{\cT'\in \Omega(P)}w(\cT')},
 \end{align}
where the weights $w(\cT)$ of a tiling is the product of weights of its horizontal lozenges:
 \begin{align} \label{eq_lozenge_weight}
 w(\cT)=\prod_{\La\in \cT}w(\La),\qquad w(\La)=\kappa q^{x-t/2}- \kappa^{-1}q^{-x+t/2}=\kappa q^{\tilde x}-\kappa^{-1} q^{-\tilde x},
 \end{align}
and $(t,x)$ is the coordinate of the left corner of a horizontal lozenge $\La\in \cT$. The number $\tilde x=x-t/2$ appearing \eqref{eq_lozenge_weight} has a simple meaning: this is the vertical coordinate counted so that its fixed level lines are horizontal, i.e.\ by using the Cartesian coordinate system (rather than non-orthogonal $(t,x)$ coordinate system that we stick to), see the middle panel of Figure \ref{Fig_trapezoid}.

We need to make sure that the weights \eqref{e:qtiling} are positive. There are three possible restrictions on the parameters which guarantee positivity:
\begin{enumerate}
\item imaginary case: $q$ is a positive real number, $\kappa$ is an arbitrary pure imaginary complex number.
\item real case: $q$ is a positive real number, $\kappa$ is a real number with additional restrictions depending on the trapezoidal domain; $\kappa^2$ cannot lie inside the interval $[q^{-(N+b_r-1)}, q^{T+a_1-1}]$ or $[q^{T+a_1-1}, q^{-(N+b_r-1)}]$ depending on whether $q>1$ or $q<1$.
\item trigonometric  case: $q$ and $\kappa$ are complex numbers of modulus $1$, $q=e^{i\al}, \kappa=e^{\ri \beta}$, with additional restrictions depending on the trapezoidal domain;
$2\beta-\al(T+a_1-1)$ and $2\beta+\al(N+b_r-1)$ must lie in the same interval $(2k\pi, 2(k+1)\pi)$, for some $k\in \bZ$.
\end{enumerate}
Figure \ref{Figure_simulations} shows samples  from $(q,\kappa)$--distribution on lozenge tilings of hexagons for various values of the parameters. For special values of $q$ and $\kappa$ the formulas \eqref{e:qtiling},\eqref{eq_lozenge_weight} can be simplified: $\kappa=\infty$ leads to the weight $q^{\text{volume}}$ on tilings, where by volume we mean the total volume enclosed under the stepped surface corresponding to the tiling (the surface becomes visible by treating three types of lozenges as projections of three sides of a unit cube, cf.\ the middle panel of Figure \ref{Fig_trapezoid}); $\kappa=0$ leads to the weight $q^{-\text{volume}}$; $q=1$ leads to the uniform measure on tilings; in the limit $q\to 1$, $\kappa\to 1$ and after division by $q-1$, the weight \eqref{eq_lozenge_weight} turns into a linear function of $\tilde x$.

\begin{figure}
\begin{center}
 \includegraphics[width=0.3\linewidth]{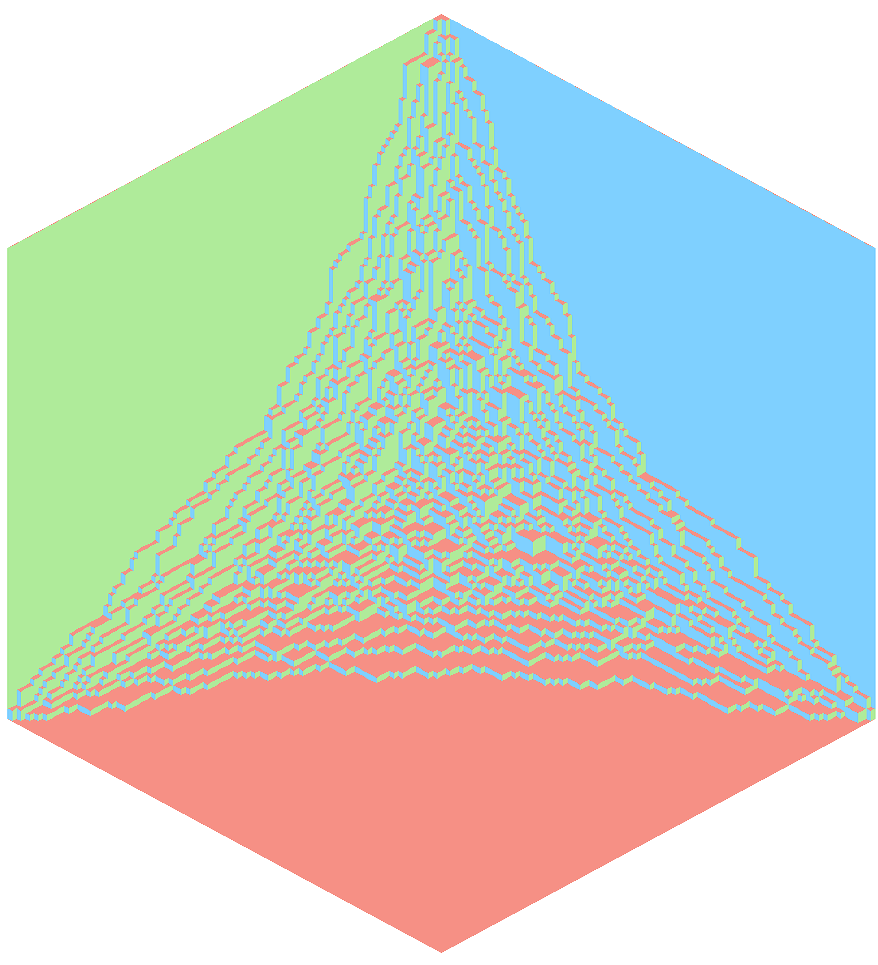}\hfill
 \includegraphics[width=0.3\linewidth]{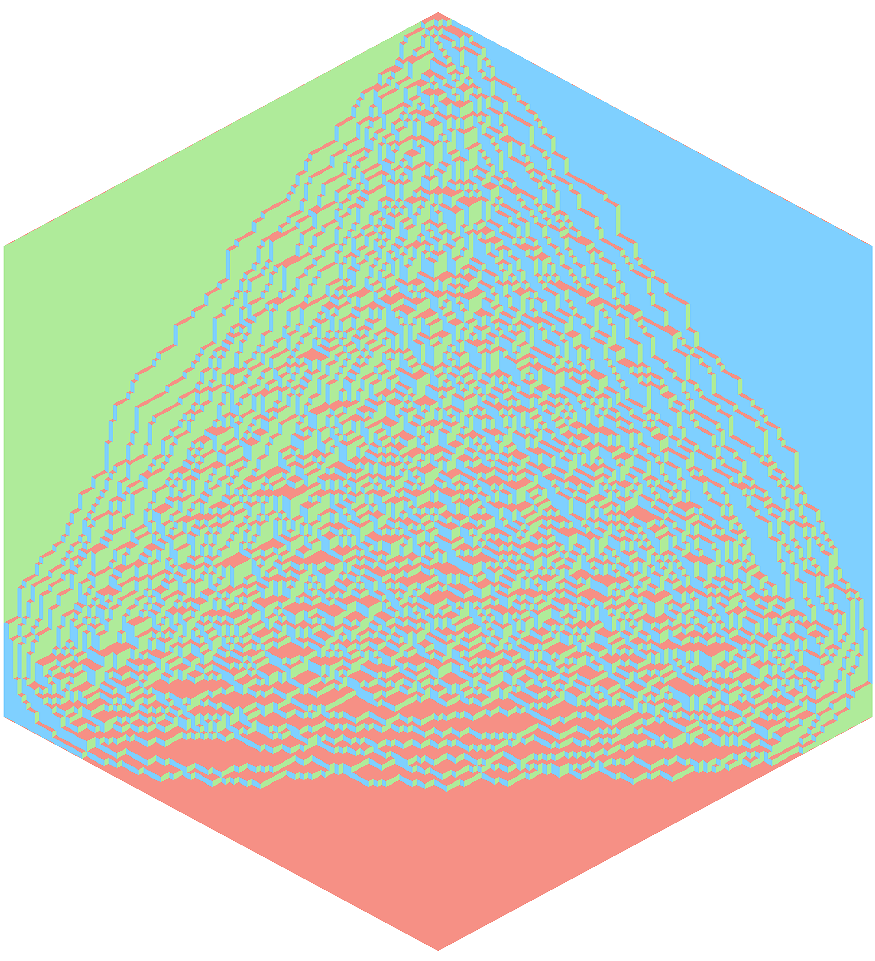}\hfill
 \includegraphics[width=0.3\linewidth]{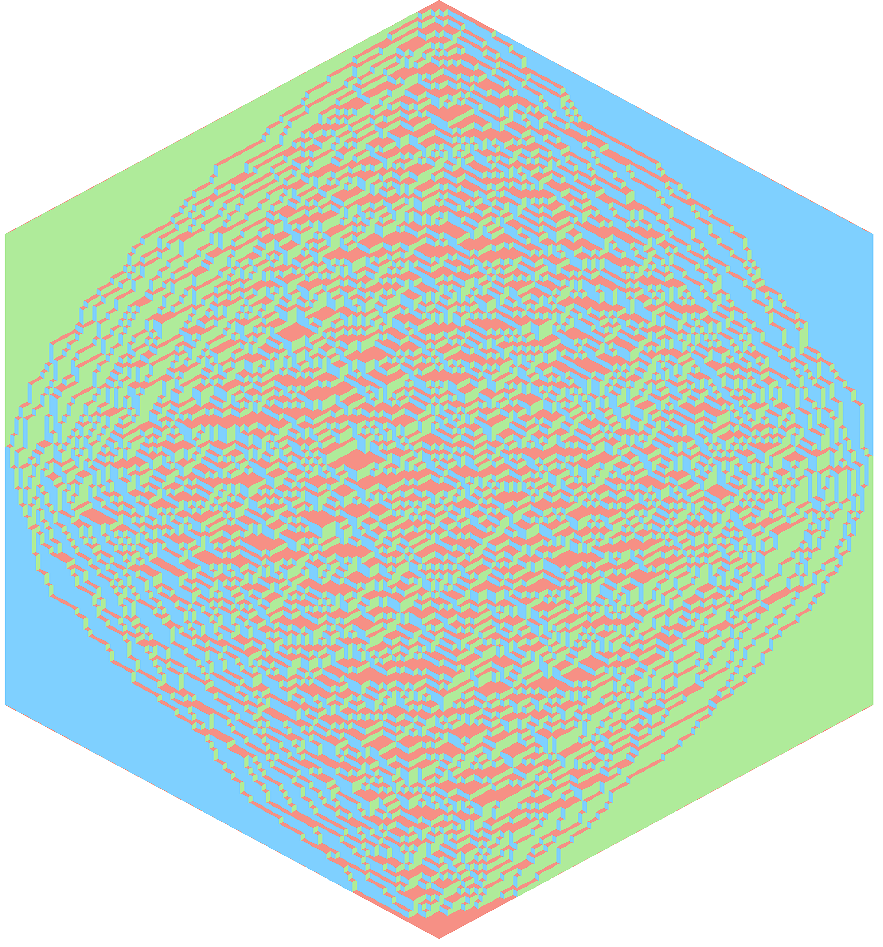}
 \caption{Samples of tilings of $100\times 100 \times 100$ hexagon. Left panel: imaginary case with $\kappa=\infty$, $q=0.97$. Center panel: real case with $q=0.99$, $\kappa=q^{-110}$. Right panel: trigonometric case with $\alpha=0.015$, $\beta=1.57$. }
 \label{Figure_simulations}
 \end{center}
 \end{figure}

The weights \eqref{eq_lozenge_weight} were first introduced in \cite{borodin2010q} for lozenge tilings of hexagons. The same paper showed that tilings with such weights lead to rich algebraic structures. In other words, in contrast to generic weights, \eqref{eq_lozenge_weight} should be treated as an exactly-solvable or integrable case: the analysis of \cite{borodin2010q} revealed connections to $q$--Racah orthogonal polynomials; in addition, the limit shape of the tilings was found to be encoded in an appropriate coordinate system by a quadratic polynomial in two variables. These results have led naturally to a question: is such integrability a specific feature of the hexagons, or does it extend to tilings with weights \eqref{eq_lozenge_weight} for more general domains? Our results show that the latter is true.\footnote{One conceptual explanation is as follows: the weights \eqref{e:qtiling},\eqref{eq_lozenge_weight} are not ad hoc, instead they are deeply related to the combinatorics of the Koornwinder symmetric polynomials, as we detail in Section \ref{Section_quasi_branching}.} We now present three theorems explaining how the macroscopic behavior of $(q,\kappa)$--random lozenge tilings of $P$ is encoded in certain algebraic equations.

We introduce a small real parameter $\varepsilon\ll 1$ and assume that $r$ and $\kappa$ stay fixed, while $q$, $N$, $T$, $a_i, b_i$, $i=1,\dots,r$, depend on $\eps$ in such a way that as $\eps$ goes to $0$ we have:
\begin{align} \label{eq_parameter_scaling}
\eps N \to \sfN, \quad \eps T\to \sfT, \quad \eps \ln(q)\to \ln(\sfq),\quad \eps a_i\to \sfa_i, \quad \eps b_i\to \sfb_i, \quad 1\leq i\leq r,
\end{align}
with asymptotic parameters $\sfq$, $\sfN$, $\sfT$, $\sfa_i, \sfb_i$, $i=1,\dots,r$ satisfying $-\sfT\leq \sfa_1<\sfb_1<\sfa_2<\sfb_2<\dots<\sfa_r<\sfb_r\leq \sfN$, and
 $\sum_{i=1}^r (\sfb_i-\sfa_i)=\sfN$. We also keep positivity constraints, so that in the real case $\sfq,\kappa\in\mathbb R$ and $\kappa^2$ cannot lie inside the interval $[\sfq^{-(\sfN+\sfb_r-1)}, \sfq^{\sfT+\sfa_1-1}]$ or $[\sfq^{\sfT+\sfa_1-1}, \sfq^{-(\sfN+\sfb_r-1)}]$ depending on whether $\sfq>1$ or $\sfq<1$.
  The conditions imply that the rescaled trapezoidal domain $\varepsilon P$ converges to a limiting domain $\sfP$.

It is convenient to state the asymptotic results on tilings in terms of their \emph{height functions.} We are going to use the following version of the height function $h$ defined on the vertices of the triangular lattice inside $P$. Let us ignore \La\, lozenges in the tiling and trace only the lozenges of the remaining two types \Lb\, and \Lc. Then these lozenges form $N$ \emph{non-intersecting paths} traversing $P$ from the left boundary to the right boundary, as shown in the right panel of Figure \ref{Fig_trapezoid}. We define $h(t,x)$ as the total number of paths below a given point $(t,x)$; this is originally defined for integer values of $t$ and $x$ and then extended to all real values by linear interpolation. Because tilings are random, so is the function $h(t,x)$; however, its values on the boundary of $P$ are deterministic: for instance, it equals $0$ on the two bottom-most segments of the boundary of $P$ and it equals $N$ on the two upper-most segments of the boundary of $P$.

A very general \emph{variational principle for tilings} of  \cite{MR1815214}  implies that in the regime \eqref{eq_parameter_scaling} the height functions $h(\varepsilon^{-1}\sft,\varepsilon^{-1}\sfx)$ for random lozenge tilings of $P$ converge in probability to a deterministic limit:
\begin{align}\label{e:hlimit0}
\varepsilon h(\varepsilon^{-1}\sft , \varepsilon^{-1}\sfx )\rightarrow  \sfh(\sft ,  \sfx ),\quad ( \sft, \sfx )\in {\sfP}.
\end{align}
The limiting height function $\sfh(\sft ,  \sfx )$ is the minimizer of a certain surface tension integral, as noticed in \cite[Section 2.4]{borodin2010q}. In general such minimization problems are hard to solve. However, in our situation $\sfh(\sft ,  \sfx )$ is explicit and we compute it for all trapezoidal domains $P$ in Theorems  \ref{t:arctic} and \ref{t:limitshape}.

For a point $ (\sft,\sfx)$ in the interior
of ${\sfP}$, we identify the gradient of the limiting height function $\nabla  \sfh$ with the local densities of three types of lozenges:
\begin{equation}\label{eq_proportions_through_derivatives}
p_{\La}(\sft,\sfx)=1-\del_{\sfx} \sfh(\sft,\sfx), \quad p_{\Lb}(\sft,\sfx)=-\del_\sft  \sfh(\sft,\sfx), \quad  p_{\Lc}(\sft,\sfx)=\del_{\sfx} \sfh(\sfx,\sft)+\del_\sft  \sfh(\sft,\sfx).
\end{equation}
A central feature of random lozenge tilings is that they exhibit boundary-induced phase transitions. Depending on the shape of the domain, they can admit frozen regions, where the associated height function is flat almost deterministically, and liquid regions, where the height function appears more rough and random and all three types
of lozenges are asymptotically present. More precisely, the liquid region is defined as
\begin{equation}\label{e:defLP}
\sfL(\sfP):=\{(\sft, \sfx)\in \sfP: 0< p_{\La}(\sft,\sfx), p_{\Lb}(\sft,\sfx), p_{\Lc}(\sft,\sfx)<1\},
\end{equation}
and the rest of the domain is called the frozen region. In particular, in Figure \ref{Figure_simulations} we clearly see frozen regions near the boundaries of the hexagon: in each such region only one type of lozenges is present. Our first result is an explicit rational parameterization of the \emph{arctic curve} separating frozen and liquid regions.

\begin{figure}
\begin{center}
 \includegraphics[height=5cm]{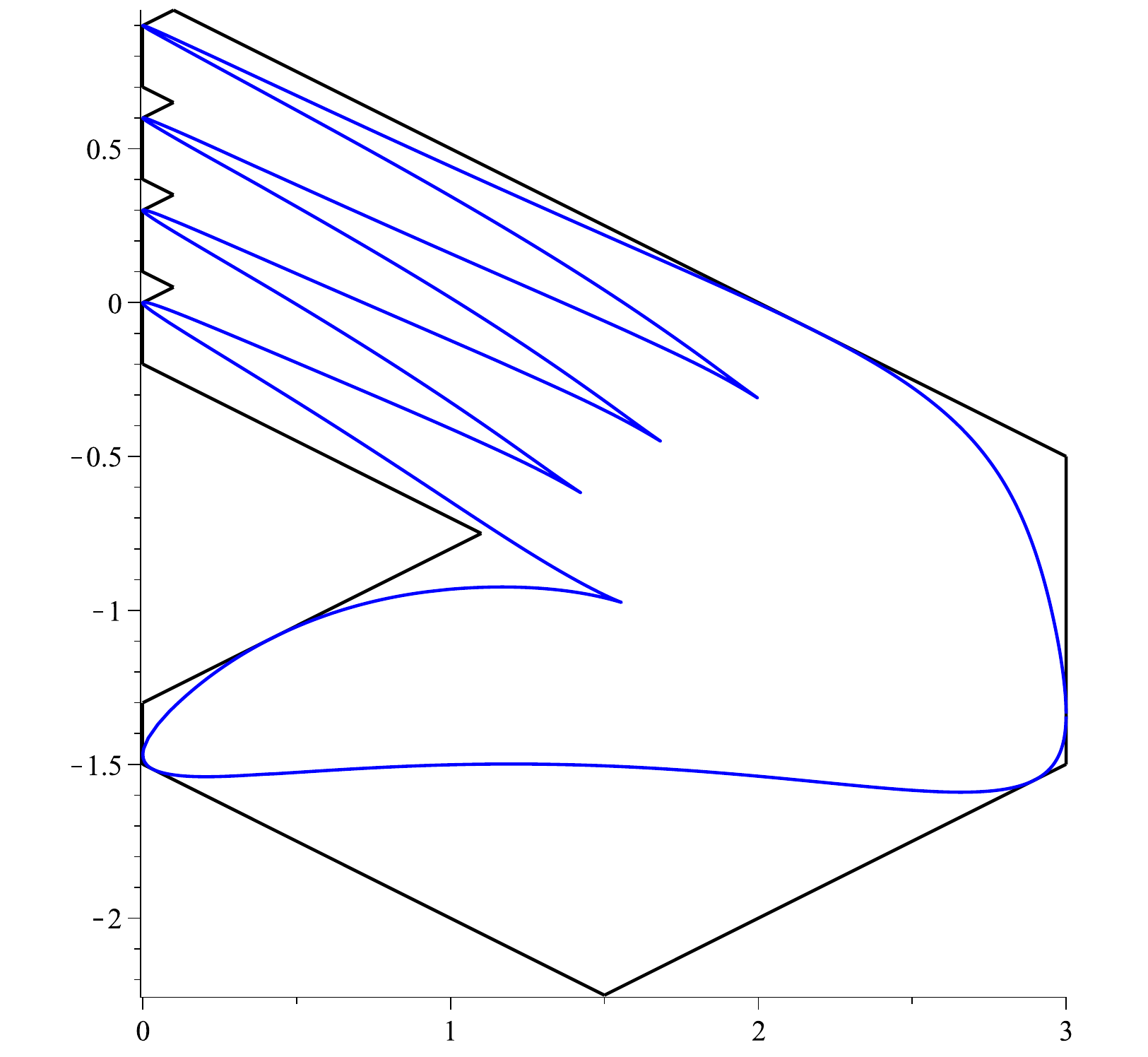}\qquad \qquad\qquad \qquad
\includegraphics[height=5cm]{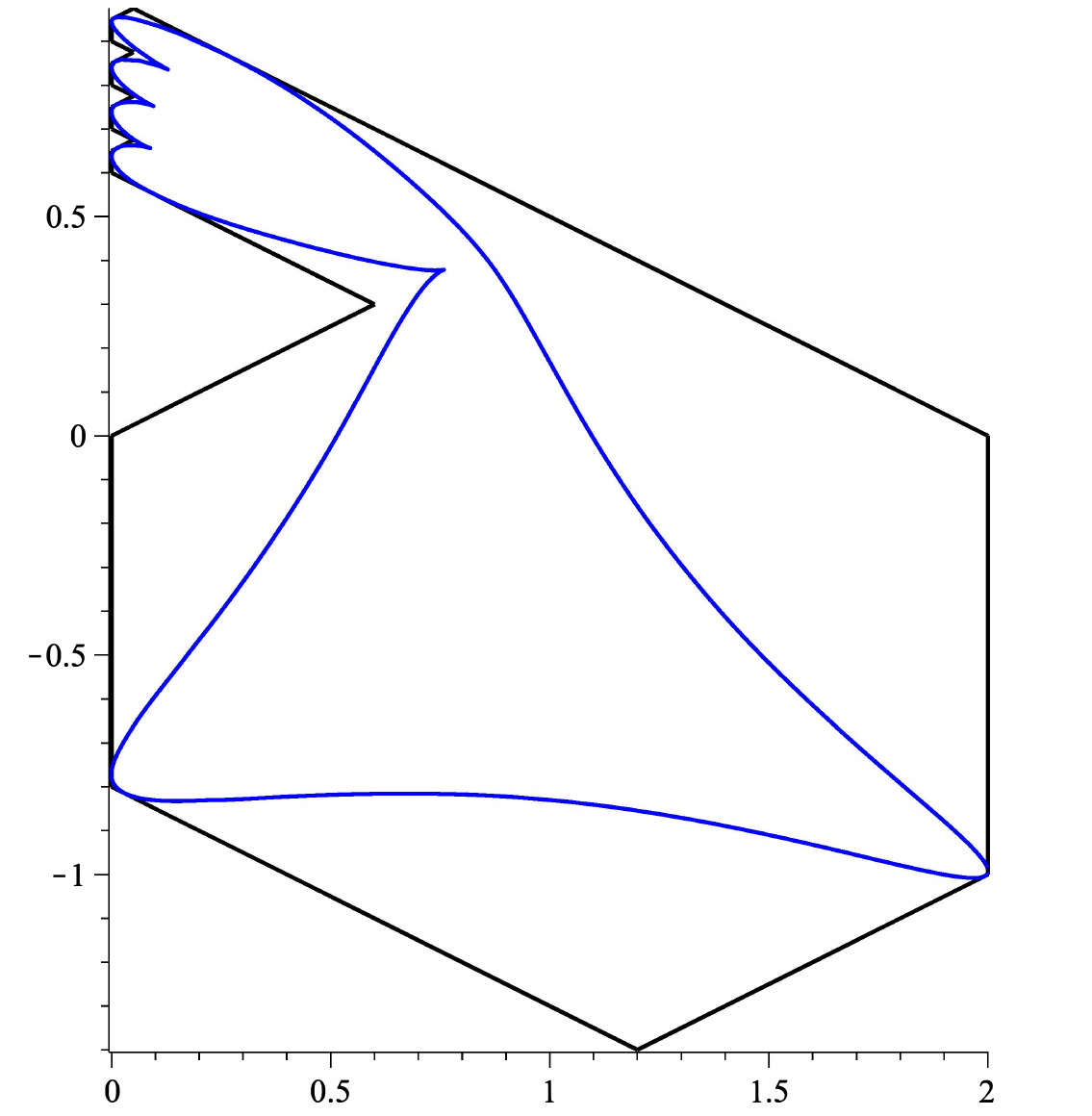}
 \caption{Arctic curves via Theorem \ref{t:arctic}. Left panel: $\sfq=\tfrac{1}{10}$, $\kappa=\frac{\ii}{10}$. Right panel: $\sfq=20$,  $\kappa= \frac{\ii}{10}$. \label{Figure_Frozen}}
 \end{center}
 \end{figure}

\begin{theorem} \label{t:arctic}
Consider the limit regime \eqref{eq_parameter_scaling} and assume that we are in the real or imaginary case, so that $\sfq$ is real and $\kappa$ is either real or purely imaginary. For any $u\in \bC$, let
\begin{equation} \label{eq_algebraic_equation}
f_0(u)
=\frac{(\sfq^{\sfN}-\sfq^{u})(\kappa^2\sfq^{-\sfT}-\sfq^{-u})}{
(\kappa^2 \sfq^{\sfN}-\sfq^{-u})(\sfq^{-\sfT}-\sfq^{u})}\prod_{i=1}^r\frac{(\sfq^{\sfa_i}-\sfq^{u})(\kappa^2 \sfq^{\sfb_i}-\sfq^{-u})}{(\kappa^2 \sfq^{\sfa_i}-\sfq^{-u})
(\sfq^{\sfb_i}-\sfq^{u})}.
\end{equation}
and
\begin{align}\label{e:defUV}
U(u)=\frac{f_0(u)\sfq^{-u}-\kappa^2 \sfq^{u}}{1-f_0(u)},\qquad
V(u)=\frac{\sfq^{-u}-f_0(u)\kappa^2 \sfq^{u}}{1-f_0(u)}.
\end{align}
The arctic  curve separating frozen and liquid regions can be parameterized as $(\sft(u), \sfx(u))$ found from:
\begin{equation}
\label{eq_parameterization_intro}
\sfq^\sft=\frac{V'(u)}{U'(u)}, \qquad
 \sfq^\sfx=\frac{\sfw\pm\sqrt{\sfw^2-4\kappa^2\sfq^{-\sft}}}{2\kappa^2\sfq^{-\sft}},\quad \text{ where }\quad
 \sfw=V(u)-\frac{U(u)}{\sfq^\sft},
\end{equation}
and the parameter $u$ is chosen in such a way that $\sfq^{-u}+\kappa^2 \sfq^u$ belongs to the real line $\bR$.
\end{theorem}
% \xi in Maple file is \sfq^u here. u in Maple file is w \sfq^\sft \kappa^{-2} here
\begin{remark}
 The choice of the sign in \eqref{eq_parameterization_intro} arises from the solution of the quadratic equation in $\sfq^\sfx$:
 \begin{equation*}
  \kappa^2\sfq^{-\sft}  \sfq^\sfx  + \sfq^{-\sfx} =\sfw.
 \end{equation*}
 When $\kappa$ is purely imaginary, we should choose $+$ in \eqref{eq_parameterization_intro} in order to guarantee the positivity of $\sfq^\sfx$. When $\kappa$ is real, the two choices lead to two curves, which are images of each other under the map $(\sfx, \sft)\mapsto (-\sfx+\sft-2\log_\sfq \kappa,\sft)$. One of them is the desired arctic curve, while the second one is outside the polygon $\sfP$.
\end{remark}
\begin{remark}
We are not going to detail the trigonometric case here and in the following theorems. Our approach still applies and, algebraically, all the formulas should remain the same; however, various small features (e.g., the choice of the curve on the complex plane where the parameter $u$ belongs in Theorem \ref{t:arctic}) need to be changed.
\end{remark}

 Examples of the arctic curves drawn using Theorem \ref{t:arctic} are shown in Figure \ref{Figure_Frozen}.

Inside the liquid region, following the general philosophy of \cite{KO_Burgers}, \cite[Lectures 9,10]{VG2020} we encode the triplet of local densities by the complex slope $f$, which is a complex number with positive imaginary part and such that the triangle $(0,1,f)$ has angles $(\pi p_{\La}, \pi p_{\Lb},\pi p_{\Lc})$, as in Figure \ref{Figure_triangle}:
\begin{align}\label{e:localdd}
 \arg f=\pi p_{\La},\quad \arg(f-1)=\pi( p_{\La}+p_{\Lc}).
\end{align}
\begin{figure}
\begin{center}
 \includegraphics[width=0.5\linewidth]{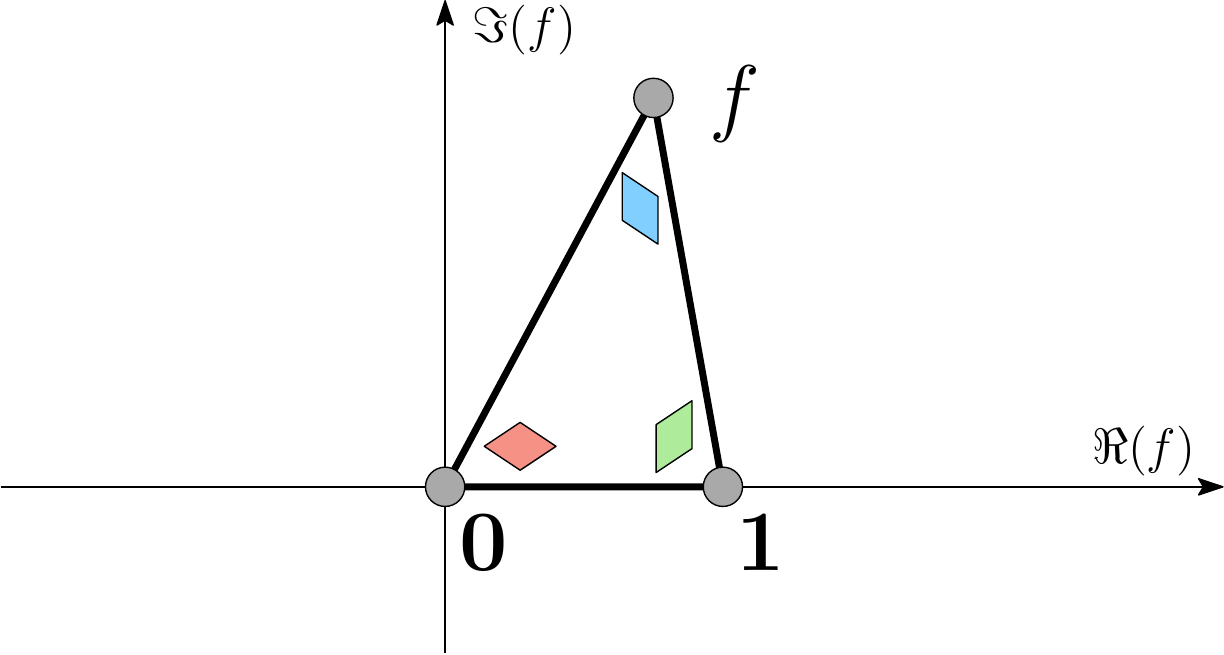}
 \caption{Triangle $(0,1,f)$ has angles $(\pi p_{\protect \La}, \pi p_{\protect \Lb},\pi p_{\protect \Lc})$. \label{Figure_triangle}}
 \end{center}
 \end{figure}

In the next theorem, we state the convergence of the height function, and provide an explicit characterization of the complex slope for the $(q,\kappa)$-distributions on lozenge tilings of trapezoid domains.

\begin{theorem}\label{t:limitshape}
Consider the limit regime \eqref{eq_parameter_scaling} and assume that we are in the real or imaginary case.
Then the height function $h(x,t)$ of random tilings
converges in probability,
\begin{align}\label{e:hlimit2}
\sup_{( \sft, \sfx )\in {\sfP}} |\varepsilon h(\varepsilon^{-1}\sft , \varepsilon^{-1}\sfx )-  \sfh(\sft ,  \sfx )| \to 0.
\end{align}
Further, using  $U(u)$ and  $V(u)$ from \eqref{e:defUV}, consider an equation on an unknown function $u=u(\sft, \sfx)$
\begin{align}\label{e:solveu}
\sfq^{-\sfx}+\kappa^2 \sfq^{\sfx-\sft}=V(u)-\frac{U(u)}{\sfq^\sft}.
\end{align}
Then $(\sft,\sfx)\in \sfP$ belongs to the liquid region  $\sfL(\sfP)$ defined by \eqref{eq_proportions_through_derivatives} and \eqref{e:defLP} if and only if \eqref{e:solveu} has a solution such that $\sfq^{-u}+\kappa^2 \sfq^u$ is non-real. In the latter case the complex slope \eqref{e:localdd} at $( \sft, \sfx )$ is given by
\begin{align}\label{e:ftx}
f_\sft(\sfx)
=\frac{(\sfq^{-u}+\kappa^2 \sfq^u)-(\sfq^{-\sfx}+\kappa^2 \sfq^{\sfx})}{(\sfq^{-u}+\kappa^2 \sfq^u)-(\sfq^{-\sfx+\sft}+\kappa^2\sfq^{\sfx-\sft})}.
\end{align}
\end{theorem}

\begin{remark} \label{Remark_4_solutions}
For any $(\sft,\sfx)\in \sfL(\sfP)$ we show in Propositions \ref{Proposition_number_of_roots} and \ref{Proposition_liquid} that \eqref{e:solveu} has four solutions with non-real $\sfq^u$: $u(\sft,\sfx)$;  $-u(\sft,\sfx)-2\log_\sfq \kappa$; and their complex conjugates.
When we plug them into \eqref{e:ftx}, we get two values for $f_\sft(\sfx)$, which are complex conjugate of each other. This uniquely determines $f_\sft(\sfx)$, since by our definitions the complex slope $f$ is in the upper-half plane.
\end{remark}
\begin{remark}
 As explained in \cite[Section 2.4]{borodin2010q}, the variational problem for the limit shape implies that the complex slope $f$ satisfies a version of the complex Burgers equation inside the liquid region:
\begin{equation}\label{e:Burger}
\frac{\del_\sft  f}{f} -\frac{\del_\sfx f}{1-f}=\ln(\sfq)\frac{\kappa^2 \sfq^{\sfx-\sft}+\sfq^{-\sfx}}{\kappa^2\sfq^{\sfx-\sft}-\sfq^{-\sfx}},\qquad (\sft,\sfx)\in \sfL(\sfP).
\end{equation}
On the other hand, \eqref{e:ftx}, \eqref{e:solveu}, and \eqref{e:defUV} imply\footnote{Expressing $f_0(u)$ through $U(u)$ and $V(u)$ via \eqref{e:defUV} and equating the results, one gets a linear relation between $U(u)$ and $V(u)$. Combining with \eqref{e:solveu}, one gets a system of two linear equations on $U(u)$ and $V(u)$, whose solution gives \eqref{eq_first_integrals}.}
\begin{equation}\label{eq_first_integrals}
U(u)=\sfq^\sft \frac{f\sfq^{-\sfx}-\kappa^2 \sfq^{\sfx-\sft}}{1-f}, \qquad  \qquad
V(u)=\frac{\sfq^{-\sfx}-f\kappa^2 \sfq^{\sfx-\sft}}{1-f}.
\end{equation}
 The right-hand sides of \eqref{eq_first_integrals} are the first integrals of the Burgers equation \eqref{e:Burger}. Definition of $U(u)$ and $V(u)$ of  \eqref{eq_algebraic_equation} \eqref{e:defUV} then implies that there is an algebraic relation between these two first integrals. This generalizes the algebraicity observation of \cite{KO_Burgers} in the $\kappa=0$ case.
\end{remark}

The existence of the non-random limit in \eqref{e:hlimit2} can be alternatively deduced from the variational principle for random tilings of \cite{MR1815214}, see discussion in \cite[Section 2.4]{borodin2010q}. The novel part of Theorem \ref{t:limitshape} is the explicit description of the limit $\sfh(\sft ,  \sfx )$ through \eqref{e:ftx}. In the particular case when $r=1$ and the tiled domain is a hexagon, Theorems \ref{t:arctic} and \ref{t:limitshape} coincide with \cite[Theorem 2.1]{borodin2010q}, but for more general polygons with $r>1$ the formulas are new.
 If we set $q=1$, then these results degenerate into the formulas of \cite{MR3278913} for uniformly random tilings of trapezoids $P$.

\medskip

Our next theorem describes the fluctuations of the height function $h(\varepsilon^{-1}\sft , \varepsilon^{-1}\sfx )$ around its expectation. Let $\bH^+=\{z\in \mathbb C \mid  \Im(z)>0\}$ be the upper half-plane.
The Gaussian free field on $\bH^+$ with zero boundary conditions is a random generalized Gaussian field $\GFF(z)$, $z\in \bH^+$, which has mean $0$ and covariance:
\begin{equation}
\label{eq_GFF_cov}
 \bE\GFF(z) \GFF(w)=K(z,w)=-\frac{1}{2\pi}\ln \left|\frac{z-w}{z-\bar w}\right|.
\end{equation}
Due to the singularity of \eqref{eq_GFF_cov} at $z=w$, the values of $\GFF(z)$ at points are not-well defined, but we can make sense of the integrals of $\GFF(z)$ against (sufficiently smooth) test--measures, see \cite{Sheffield_GFF}, \cite{Werner_Powell_GFF}, \cite{Be_Powell_GFF}, \cite[Lecture 11]{VG2020}  for more details.

We also need a bijection between $\bH^+$ and the liquid region $\sfL(\sfP)$.
\begin{lemma} \label{Lemma_bijection_intro}
 For each $(\sft ,  \sfx )\in\sfL(\sfP)$, define $\bar \Omega(\sft ,  \sfx)=\sfq^{-u}+\kappa^2 \sfq^u$, where $u$ is a non-real solution to \eqref{e:solveu}, chosen so that $\Im(\bar \Omega(\sft ,  \sfx ))>0$. Then $\bar \Omega$ is a bijection between $\sfL(\sfP)$ and $\bH^+$.
\end{lemma}
\begin{remark}

 Continuing Remark \ref{Remark_4_solutions}, we note that by Propositions \ref{Proposition_number_of_roots} and \ref{Proposition_liquid}, $\bar \Omega$ is uniquely determined by \eqref{e:solveu} and $\Im(\bar \Omega(\sft ,  \sfx ))>0$. The choice of $u$ guaranteeing $\Im(f_\sft(\sfx))>0$ might differ by conjugation from the choice of $u$ guaranteeing ${\Im(\bar \Omega(\sft,\sfx))>0}$.
\end{remark}

\begin{theorem}\label{t:GFF}
 Consider the limit regime \eqref{eq_parameter_scaling} and assume that we are in the real or imaginary case. Then inside the liquid region,  $(\sft ,  \sfx )\in \sfL(\sfP)$,
 we have
 $$
  \lim_{\eps\to 0} \sqrt{\pi}\bigl(h(\varepsilon^{-1} \sft ,\varepsilon^{-1} \sfx )-\bE[h(\varepsilon^{-1} \sft, \varepsilon^{-1} \sfx )]\bigr) = \bar \Omega \text{-pullback of GFF  in } \bH^+,
 $$
 in the sense of convergence of the joint moments for pairings with appropriate test-measures.
\end{theorem}
 We refer to Theorem \ref{theorem_GFF_main} for a more precise statement of Theorem \ref{t:GFF}.

\smallskip

 If $q=1$, then  $(q,\kappa)$-distributions turn into the uniform distribution on lozenge tilings of $P$ and Theorem \ref{t:GFF} degenerates into the results of \cite{petrov2015asymptotics, bufetov2018fluctuations}. For general $(q,\kappa)$ and $r=1$, i.e.\ for the hexagons, Theorem \ref{t:GFF} was conjectured in  \cite[Conjecture 8.4.1]{dimitrov2019log}. The same article also contains a partial result establishing Gaussian fluctuations along a single vertical section by using discrete loop equations for log-gases; see also \cite{MR3556288} for another approach. For $\kappa=0$ and $q\ne 1$, i.e.\ for the $q^{\text{volume}}$--weighted lozenge tilings Theorem \ref{t:GFF} verifies for the trapezoidal domains $P$ a general conjecture of Kenyon and Okounkov, cf.\ \cite{KO_Burgers} and \cite[Lectures 11-12]{VG2020}. The only previous Gaussian free field type result for $q^{\text{volume}}$--weighted tilings is \cite{ahn2020global} dealing with plane partitions, which can be obtained from tilings of the hexagons by sending side lengths to infinity while keeping $q$ fixed. For general values of the $(q,\kappa)$ parameters, we are led to a conjecture:

 \begin{conjecture}
 \label{Conjecture}
 The macroscopic fluctuations of the height function in the liquid region for $(q,\kappa)$--distributed random lozenge tilings of arbitrary domains are asymptotically described by the Gaussian Free Field in the complex structure given by either of the first integrals \eqref{eq_first_integrals} of the complex Burgers equation \eqref{e:Burger}.
 \end{conjecture}

 We refer to \cite[Lectures 11-12]{VG2020} for the detailed explanations and heuristics underlying Conjecture \ref{Conjecture} in the special case of the uniform measure, $q=1$.
 The formulas \eqref{eq_first_integrals} imply that the first integrals are holomorphic functions of $\Omega(\sft, \sfx)$ in Theorem \ref{t:GFF}. Hence, these first integrals define the same complex structure in the liquid region as $\Omega$ (or its complex conjugate $\bar \Omega$, depending on the choice of $u$; conjugation of the complex structure does not change the covariance of GFF) and Theorem \ref{t:GFF} verifies Conjecture \ref{Conjecture} for trapezoidal domains $P$. What Conjecture \ref{Conjecture} says is that our result should extend from trapezoids to more general domains.

\bigskip

As we will see in Section \ref{Section_tilings_as_Markov_chain}, the $(q,\kappa)$-distribution \eqref{e:qtiling} on lozenge tilings can be viewed as a Markov process of nonintersecting paths, and the transition probability is given in the form \eqref{e:m1}. Our analysis of $(q,\kappa)$--distributions relies on Theorem \ref{t:loopstudy}, which uses the dynamic loop equations of Theorem \ref{t:loopeq} to obtain an asymptotic expansion for the one-step transition probabilities. We remark that checking the assumptions of Theorem \ref{t:loopstudy} takes efforts and we devote Section \ref{Section_assumptions} to it. Eventually, the deterministic part of the expansion \eqref{eq_decomposition_field} leads to Theorems \ref{t:arctic} and \ref{t:limitshape} after careful analysis and massaging of the formulas in Section \ref{Section_tilings_LLN}. The form of the stochastic part implies that the height fluctuations converge to a Gaussian field, described by a stochastic evolution equation; we analyze this equation and prove that it outputs the Gaussian Free Field in Section \ref{Section_GFF}.

In addition to the dynamical loop equations, the second key ingredient of the proofs is the use of the \emph{complex characteristic flow} which we introduce. The flow is related to the characteristic curves of the first order PDE \eqref{e:Burger}, but has an important difference. Namely, in \eqref{e:Burger} (and its relatives for uniform and $q^{\text{volume}}$-weighted tilings in \cite{KO_Burgers}), the variables $(\sft,\sfx)$ in the equation are real and belong to the liquid region $\sfL(\sfP)$; this complicates the analysis, because the shape of the liquid region is not known a priory. In our approach, we observe that there is a way to extend the definition of the complex slope \eqref{e:localdd} from real $\sfx$ to complex $z$, see \eqref{e:double_complex_slope}, thus arriving at the doubly complex slope. The advantage of the latter is that it  decouples from the notion of the liquid region, but the connection to the first order PDE \eqref{e:Burger} is preserved (with $\sfx$ replaced by complex $z$). We further investigate all the main observables of the random tilings along the complex curves $z^\sft$, $\sft\geq 0$, which arise from the characteristics of the doubly complex version of \eqref{e:Burger}, see Section \ref{Section_Limit_shape_proofs}. We discover that all the formulas are dramatically simplified along these complex curves, which eventually allows us to arrive at the explicit answers of Theorems \ref{t:arctic}, \ref{t:limitshape}, and \ref{t:GFF}. In random matrix setting, the complex characteristic flow has been previously used in \cite{MR1819483, MR2418256,  MR4009708, huang2020edge} to analyze the complex Burgers equation, which describes the evolution of the Stieltjes transform of the empirical eigenvalue density. Comparing with those analyses, our complex Burgers equation \eqref{e:Burger} has extra symmetries \eqref{e:symmetries} and the characteristic flow has singular behavior, which makes the analysis more involved.

\section{Dynamical Loop Equation}\label{s:DLE}

In this section we first give the proof of Theorem \ref{t:loopeq}.  Then we apply the particle-hole duality to produce a version of the transition probability \eqref{e:m1} for descending transitions and derive the corresponding dynamical loop equations.

\subsection{Proof of Theorem \ref{t:loopeq}}

\label{Section_proof_of_le}

\begin{proof}[First proof] The only possible singularities of \eqref{e:sum1} are simple poles at $z=x_i$, $i=1,\dots,n$. Thus, our plan is to fix any $1\leq i\leq n$ and check that the residue of \eqref{e:sum1} at $z=x_i$ vanishes. The first term in \eqref{e:sum1} has a residue at $z=x_i$ if and only if $e_i=0$. The residue is given by
\begin{align}\label{e:r1}
\sum_{\bme: e_i=0}\bP(\bmx+\bme|\bmx) \phi^+(x_i) \frac{b(x_i+\theta)-b(x_i)}{b'(x_i)}\prod_{j:j\neq i}\frac{b(x_i+\theta)-b(x_j+\theta e_j)}{b(x_i)-b(x_j)}.
\end{align}
The second term has a residue at $z=x_i$ if and only if $e_i=1$. The residue is given by
\begin{align}\label{e:r2}
\sum_{\bme: e_i=1}\bP(\bmx+\bme|\bmx) \phi^-(x_i) \frac{b(x_i)-b(x_i+\theta)}{b'(x_i)}\prod_{j:j\neq i}\frac{b(x_i)-b(x_j+\theta e_j)}{b(x_i)-b(x_j)}.
\end{align}

For any pair $\bme^0=(e_1, e_2, \dots, e_{i-1}, 0, e_{i+1}, \dots, e_{n})$ and $\bme^1=(e_1, e_2, \dots, e_{i-1}, 1, e_{i+1}, \dots, e_{n})$,
the ratio of the transition probabilities $\bP(\bmx+\bme^1|\bmx)$ and $\bP(\bmx+\bme^0|\bmx)$ is given by
\begin{align}
\label{eq_ratio_transitions}
\frac{\bP(\bmx+\bme^1|\bmx)}{\bP(\bmx+\bme^0|\bmx)}
=\frac{\phi^+(x_i)}{\phi^-(x_i)}\prod_{j:j\neq i}\frac{b(x_i+ \theta)-b(x_j+\theta e_j)}{b(x_i)-b(x_j+\theta e_j)}.
\end{align}
Plugging \eqref{eq_ratio_transitions} into \eqref{e:r1} and \eqref{e:r2}, we see that the residues at $\bme^0$ and $\bme^1$ cancel out:
\begin{multline*}
\bP(\bmx+\bme^0|\bmx) \phi^+(x_i) \frac{b(x_i+\theta)-b(x_i)}{b'(x_i)} \prod_{j\neq i}\frac{b(x_i+\theta)-b(x_j+\theta e_j)}{b(x_i)-b(x_j)}
\\+
\bP(\bmx+\bme^1|\bmx) \phi^-(x_i)  \frac{b(x_i)-b(x_i+\theta)}{b'(x_i)} \prod_{j:j\neq i}\frac{b(x_i)-b(x_j+\theta e_j)}{b(x_i)-b(x_j)}=0.
\end{multline*}
Summing over $(e_1, e_2, \dots, e_{i-1}, e_{i+1}, \dots, e_{n})\in \{0,1\}^{n-1}$, we conclude that
the residue of \eqref{e:sum1} at $z=x_i$ vanishes.
\end{proof}

\begin{proof}[Second proof]
For any complex numbers $x_1,\dots,x_n\in U$, we denote the partition function
$$
 Z_n(x_1,\dots,x_n)=\sum_{\bme\in\{0,1\}^n}  \prod_{1\leq i<j\leq n}\frac{b(x_i+\theta e_i)-b(x_j+\theta e_j)}{b(x_i)-b(x_j)}\prod_{i=1}^n \phi^+(x_i)^{e_i} \phi^-(x_i)^{1-e_i}.
$$
We claim that $Z_n(x_1,\dots,x_n)$ is a holomorphic function on $U^n$. Indeed, the only possible singularities of $Z_n$ are on the lines $x_i=x_j$. On the other hand, directly from the definition, the function
$$F(x_1,\dots,x_n):=Z_n(x_1,\dots,x_n)  \prod_{1\leq i<j\leq n} (b(x_i)-b(x_j))$$
 is skew-symmetric in $x_1,\dots,x_n$; hence, it vanishes on the lines $x_i-x_j$ and
$$
 \frac{F(x_1,\dots,x_n)}{\prod_{i<j} (x_i-x_j)}
$$
has no singularities. Therefore,
$$
 Z_n(x_1,\dots,x_n)= \frac{F(x_1,\dots,x_n)}{\prod_{i<j} (x_i-x_j)}  \cdot  \prod_{1\leq i<j\leq n}\frac{x_i-x_j}{b(x_i)-b(x_j)}
$$
also has no singularities.

\smallskip

The expectation \eqref{e:sum1} can be identified with a ratio of two partition functions $\frac{Z_{n+1}(x_1,\dots,x_n,z)}{Z_n(x_1,\dots,x_n)}$ and its holomorphicity in $z$ follows from the holomorphicity of the partition functions.
\end{proof}

\begin{remark} \label{Remark_extension}
 One can introduce more advanced weights, which would lead to holomorphic partition functions: in Theorem \ref{t:loopeq} we are allowing the increments $e_i$ of $x_i$ to be in $\{0,1\}$, but more generally one can allow $m$ complex increments: $e_i\in\{v_1,\dots,v_m\}$. Given $m$ holomorphic functions $\phi_1,\dots,\phi_m$, this would lead to the factor $\phi_i(x_i)$ in $\prod_{i=1}^n$ part of the distribution generalizing \eqref{e:m1} whenever $e_i=v_i$ and to $m^n$ possible transitions. The argument for the holomorphicity of the partition function and of the corresponding observable $\frac{Z_{n+1}(x_1,\dots,x_n,z)}{Z_n(x_1,\dots,x_n)}$ remains the same, leading to an extension of Theorem \ref{t:loopeq} with \eqref{e:sum1} being replaced by expectation of the sum of $m$ (rather than two) terms.
\end{remark}

\subsection{Particle-Hole Duality}\label{s:PHD}
In this section, we study descending transitions. Take two signatures $\bmla\in \GT_{n+1}$ and  $\bmmu\in \GT_n$, with interlacing condition $\bmmu\prec \bmla$, as in \eqref{eq_interlacing}. Similarly to \eqref{e:defyi}, we encode them by particle systems
\begin{align*}
x_i=\la_i-(i-1)\theta, \quad 1\leq i\leq n+1, \qquad y_i=\mu_i-(i-1)\theta, \quad 1\leq i\leq n.
\end{align*}
Then $\bmx=(x_1, x_2,\dots, x_{n+1})\in \bW_\theta^{(n+1)}$ and $\bmy=(y_1, y_2,\dots, y_{n})\in \bW_\theta^{(n)}$, and the interlacing condition $\bmmu\prec \bmla$ is equivalent to  $y_i\in\{x_{i+1}+\theta, x_{i+1}+1+\theta, x_{i+1}+2+\theta,\dots, x_i-1, x_i\}$ for $1\leq i\leq n$.

Given $\bmx=(x_1, x_2,\dots, x_{n+1})\in \bW_\theta^{(n+1)}$, we denote the lattice
\begin{align*}
\bL(\bmx)=\bigcup_{i=1}^{n}\{x_{i+1}+\theta, x_{i+1}+1+\theta, x_{i+1}+2+\theta,\dots, x_i-1, x_i\}.
\end{align*}
We fix two functions $b(z)$, $w(z)$ and consider the following transition probability:
\begin{align}\label{e:m2}
\bP(\bmy|\bmx)
=\frac{1}{Z(\bmx)}  \prod_{1\leq i<j\leq n}\bigl(b(y_i)-b(y_j)\bigr)\, \prod_{i=1}^n \left[
w(y_i) \prod\limits_{\begin{smallmatrix}\ell \in \bL(\bmx) \\ \ell>y_i\end{smallmatrix}} \frac{1}{b(\ell)-b(y_i)}\prod\limits_{\begin{smallmatrix}\ell\in \bL(\bmx)\\ \ell<y_i\end{smallmatrix}} \frac{1}{b(y_i)-b(\ell)}\right],
%&\propto
%\prod_{i<j}\frac{b(x_i)-b(x_j)}{b(y_i)-b(y_j)}\prod_{i}\frac{1}{\prod_{\bL\setminus\{x_i\}\ni\ell>x_i} b(\ell)-b(x_i)\prod_{\bL\setminus\{x_i\}\ni\ell<x_i} b(x_i)-b(\ell)}
%\prod_i w(x_i),
\end{align}
where $\bmy=(y_1, y_2,\dots, y_n)\in \bW_\theta^n$ with $y_i\in\{x_{i+1}+\theta, x_{i+1}+1+\theta, x_{i+1}+2+\theta,\dots, x_i-1, x_i\}$ for $1\leq i\leq n$.

We remark that the transition probability \eqref{e:m2} is obtained from the ascending transition \eqref{e:m1} by the particle-hole duality or, equivalently, by transposition of the involved Young diagrams, i.e.\ by mirroring Figure \ref{f:ytop} with respect to the vertical axis. More precisely, starting from \eqref{e:m2} we  do the following three steps:
\begin{itemize}
 \item We replace the particle configurations $\bmx$ and $\bmy$, by the holes, i.e.\ by the complementary configurations $\bL(\bmx)\setminus\{x_1, x_2, \dots, x_n\}$ and $\bL(\bmx)\setminus\{y_1, y_2, \dots, y_n\}$;
 \item We divide all coordinates by $\theta$;
 \item We replace $\theta$ by $\theta^{-1}$.
\end{itemize}
Applying these operations to \eqref{e:m2} we get \eqref{e:m1}, as can be seen through the identities
 (where we denote the configurations of holes as $\bmx'=\bL(\bmx)\setminus \bmx$ and $\bmy'=\bL(\bmx)\setminus \bmy$):
\begin{align}\begin{split}\label{e:pheq}
\prod_{i<j}\bigl( b(x'_i)-b(x'_j)\bigr)&=\frac{\prod_{i<j}\bigl(b(x_i)-b(x_j)\bigr)\prod_{\ell> \ell'\in \bL(\bmx)}\bigl( b(\ell)-b(\ell')\bigr)}{\prod_{i}(\prod_{ \bL(\bmx)\ni \ell>x_i}\bigl(b(\ell)-b(x_i)\bigr) \prod_{\bL(\bmx)\ni \ell<x_i}\bigl(b(x_i)-b(\ell)\bigr)}
\\
\prod_{i<j}\bigl( b(y'_i)-b(y'_j)\bigr)&=\frac{\prod_{i<j}\bigl(b(y_i)-b(y_j)\bigr)\prod_{\ell> \ell'\in \bL(\bmx)} \bigl(b(\ell)-b(\ell')\bigr)}{\prod_{i}(\prod_{ \bL(\bmx)\ni \ell>y_i}\bigl(b(\ell)-b(y_i)\bigr) \prod_{\bL(\bmx)(\bmx)\ni \ell<x_i}\bigl(b(y_i)-b(\ell)\bigr)}
.
\end{split}\end{align}

Because of this connection, Theorem \ref{t:loopeq} can be recast for the ascending process. We record the result in the following theorem, supplied with a self-contained proof.

%
%We remark $\bL(\bmx)\setminus\{x_1, x_2, \dots, x_n\}$ are holes in Figure \ref{f:ytop}.
%Take $y_i=\mu_i-(i-1)\theta$, with $\mu_1\geq \mu_2\geq \dots\geq \mu_n$ and $\bmmu\prec\bmla$. More precisely, $y_i\in\{x_{i+1}+\theta, x_{i+1}+1+\theta, x_{i+1}+2+\theta,\dots, x_i-1, x_i\}$, for $1\leq i\leq n$.
%
%The dynamic from $\bmx$ to $\bmy$ induces a dynamic on the holes $\bL(\bmx)\setminus\{x_1, x_2, \dots, x_n\}$ and $\bL(\bmy)\setminus\{y_1, y_2, \dots, y_n\}$, i.e. each hole either jumps $\theta$ in the negative direction, or stay put. After rescaling by a factor $1/\theta$, we can identify it as an ascending process by adding vertical strips.
%Take any analytic monotone function $b(z)$, we consider the following transition probability
%
%This measure is obtained from \eqref{e:m1} by the particle-hole duality. More precisely,

\begin{theorem}\label{t:loopeq2}
Choose an open set $U\subset \mathbb C$, a parameter $\theta>0$, a particle configuration ${\bmx=(x_1,x_2,\dots,x_{n+1})\in \bW_\theta^{n+1}}$ such that $[x_{n+1},x_1]\subset U$,  two holomorphic functions $\phi^+(z)$, $\phi^-(z)$ on $U$ and a conformal (i.e., holomorphic and injective) function $b(z)$ on $U$. Assume that the random $n$--tuple $\bmy$ is distributed according to the transition probability \eqref{e:m2} with weight function $w(z)$ satisfying
\begin{align}\label{e:wratio}
\frac{w(z+1)}{w(z)}=\frac{\phi^+(z)}{\phi^-(z)}.
\end{align}
 Then the following observable is a holomorphic function of $z\in U$:
\begin{align}\label{e:sum2}
\bE\left[\phi^+(z)\dfrac{\prod\limits_{i=1}^{n+1} \bigl( b(z)-b(x_i) \bigr)}{\prod\limits_{i=1}^{n}\bigl( b(z)-b(y_i)\bigr)}+\phi^-(z)\prod\limits_{\ell\in \bL(\bmx)}\dfrac{\bigl(b(z+1)-b(\ell)\bigr)}{\bigl(b(z)-b(\ell)\bigr)} \cdot \frac{\prod\limits_{i=1}^{n+1}\bigl( b(z)-b(x_i)\bigr)}{\prod\limits_{i=1}^{n}\bigl( b(z+1)-b(y_i)\bigr)}\right].
\end{align}
\end{theorem}
\begin{proof} The possible singularities of the first term in \eqref{e:sum2} are simple poles at the points $z=m\in \{x_{i+1}+\theta, x_{i+1}+1+\theta, x_{i+1}+2+\theta,\dots,  x_i-1\}$; they arise whenever $m=y_i$ for any $1\leq i\leq n$. Note that we excluded $m=x_i$: there is no pole at such point due to the factor  $b(z)-b(x_i)$ in the numerator.

The possible singularities of the second term in \eqref{e:sum2} are also simple poles at the points $z=m\in \{x_{i+1}+\theta, x_{i+1}+1+\theta, x_{i+1}+2+\theta,\dots,  x_i-1\}$; they arise whenever $m=y_i-1$ for any $1\leq i\leq n$. Note that we excluded $m=x_{i+1}+\theta-1$: there is no pole at such point due to the factor  $b(z+1)-b(\ell)$, $\ell=x_{i+1}+\theta$, in the numerator; also note that the denominators in the first product in the second term in \eqref{e:sum2} do not produce additional poles, since for each zero of $b(z)-b(\ell)$ in the denominator, we have a matching zero in the numerator coming either from $b(z+1)-b(\ell+1)$ or from $b(z)-b(x_i)$.

In the rest of the proof we fix $i=1,\dots,n$ and check that the residues of two terms in \eqref{e:sum2} at a point  $z=m\in \{x_{i+1}+\theta, x_{i+1}+1+\theta, x_{i+1}+2+\theta,\dots, x_i-1\}$ differ by a sign and, hence, cancel out. For the first term, the residue is given by
\begin{align}\label{e:r22}
\sum_{\bmy: y_i=m}\bP(\bmy|\bmx)\cdot \phi^+(m) \cdot \frac{b(m)-b(x_i)}{b'(m)}\cdot \frac{\prod\limits_{j\neq i} \bigl(b(m)-b(x_j)\bigr)}{\prod\limits_{j\neq i} \bigl(b(m)-b(y_j)\bigr)},
\end{align}
while for the second term the residue is given by
\begin{align}\label{e:r12}
\sum_{\bmy: y_i=m+1}\bP(\bmy|\bmx)\phi^-(m)  \frac{b(x_i)-b(m)}{b'(m)}\prod_{\ell\in \bL(\bmx)\setminus\{m,m+1\}}\frac{b(m+1)-b(\ell)}{b(m)-b(\ell)}\frac{\prod\limits_{j\neq i}\bigl( b(m)-b(x_j)\bigr)}{\prod\limits_{j\neq i} \bigl(b(m+1)-b(y_j)\bigr)}.
\end{align}

For any pair of particle configurations $\bmy^1=(y_1, y_2, \dots, y_{i-1}, m+1, y_{i+1}, \dots, y_{n})$ and $\bmy^0=(y_1, y_2, \dots, y_{i-1}, m, y_{i+1}, \dots, y_{n})$,
the ratio of the transition probabilities $\bP(\bmy=\bmy^1|\bmx)$ and $\bP(\bmy=\bmy^0|\bmx)$ is given by
\begin{align}\label{e:PPratio}
\frac{\bP(\bmy=\bmy^0|\bmx)}{\bP(\bmy=\bmy^1|\bmx)}
=\frac{w(m)}{w(m+1)}\prod_{\ell\in \bL(\bmx)\setminus\{m,m+1\}}\frac{b(m+1)-b(\ell)}
{ b(m)-b(\ell)}\prod_{j\neq i}\frac{b(m)-b(y_j)}{b(m+1)-b(y_{j})}.
\end{align}

We recall from \eqref{e:wratio} that $w(m)/w(m+1)=\phi^-(m)/\phi^+(m)$.
Plugging \eqref{e:PPratio} into the expressions of \eqref{e:r12} and \eqref{e:r22}, we see that the residuals at $\bmy^0$ and $\bmy^1$ cancel out:
\begin{multline*}
0=\bP(\bmy=\bmy^0|\bmx)   \phi^+(m)   \frac{b(m)-b(x_i)}{b'(m)}  \frac{\prod\limits_{j\neq i} \bigl(b(m)-b(x_j)\bigr)}{\prod\limits_{j\neq i} \bigl(b(m)-b(y_j)\bigr)}\\ +
\bP(\bmy=\bmy^1|\bmx)   \phi^-(m)    \frac{b(x_i)-b(m)}{b'(m)}  \prod_{\ell\in \bL(\bmx)\setminus\{m,m+1\}}\frac{b(m+1)-b(\ell)}{b(m)-b(\ell)}  \frac{\prod\limits_{j\neq i}\bigl( b(m)-b(x_j)\bigr)}{\prod\limits_{j\neq i} \bigl(b(m+1)-b(y_j)\bigr)}.
\end{multline*}
Summing the last identity over $(y_1, y_2, \dots, y_{i-1}, y_{i+1}, \dots, y_{n})$ with $y_j\in\{x_{j+1}+\theta, x_{j+1}+1+\theta,\dots, x_{j}\}$, we conclude that the sum of \eqref{e:r12} and \eqref{e:r22} vanishes and \eqref{e:sum2} has no residue and, hence, no singularity at $z=m$.
\end{proof}

\section{Examples of Dynamical Loop Equations}\label{s:examples}

In this section, we show that many well-known stochastic dynamics, including nonintersecting Bernoulli/Poisson random walks \cite{gorin2019universality,konig2002non, huang2017beta}, $\beta$--corners processes \cite{MR3418747, borodin2015general}, Dyson Brownian motion \cite{MR0148397}, measures on Gelfand-Tsetlin patterns \cite{MR3418747,borodin2016lectures,petrov2015asymptotics}, Macdonald processes \cite{borodin2014macdonald}, and $(q,\kappa)$-distributions on lozenge tilings \cite{borodin2010q} can be transformed to the forms of \eqref{e:m1}, \eqref{e:m2}, and their degenerations. As a corollary, we derive the dynamical loop equations for all these systems.

\subsection{Nonintersecting Bernoulli random walks}

For $n\geq 1$, we consider $n$ independent Bernoulli random walks on $\bZ$, with each walker jumping by $1$ in the positive direction with probability $p\in(0,1)$, or
staying with the complementary probability $1-p$ at each time step,
and condition them to never intersect. If we denote the configuration of walkers at time $t$ as $\bmx(t)=(x_1(t)>x_2(t)>\dots>x_n(t))\in \bZ^n$, then for any $\bme=\{0,1\}^n$, the transition probability $\bP(\bmx(t+1)=\bmx+\bme|\bmx(t)=\bmx)$ is given by
\begin{align}\label{e:defL11}
\bP(\bmx(t+1)=\bmx+\bme|\bmx(t)=\bmx)= \frac{V(\bmx+\bme)}{V(\bmx)}\prod_{i=1}^np^{e_i}(1-p)^{1-e_i},
\end{align}
where $V(\bmx)$ is the Vandermonde determinant
\begin{align}\label{e:Vand}
V(\bmx)=\prod_{1\leq i<j\leq n}(x_j-x_i).
\end{align}
The computation leading to \eqref{e:defL11} can be found in \cite{konig2002non}. Recently, local statistics of nonintersecting Bernoulli random walks were studied in \cite{gorin2019universality}. There are two natural generalizations of \eqref{e:defL11}. First, we can consider a $\theta$ version and, second, we can allow the jump probability $p$ to depend on $x_i$.
Hence, for general $\theta>0$, we fix $\theta=\beta/2$ and define the \emph{$\theta$-nonintersecting Bernoulli random walk} as a discrete time Markov process $\bmx(t)=(x_1(t),x_2(t),\dots,x_n(t))\in \bW^n_\theta$, as in \eqref{e:defWtheta},
with the transition probability  $\bP(\bmx(t+1)=\bmx+\bme|\bmx(t)=\bmx)$ given by
\begin{align}\label{e:defL12}
\bP(\bmx(t+1)=\bmx+\bme|\bmx(t)=\bmx)=\frac{1}{Z(\bmx)} \frac{V(\bmx+\theta \bme)}{V(\bmx)}\prod_{i=1}^n\phi^+(x_i)^{e_i}\phi^-(x_i)^{1-e_i}.
\end{align}

\begin{corollary}\label{c:loop12}
Fix any particle configuration $\bmx=(x_1>x_2>\dots>x_n)\in \bW_\theta^n$ and consider transition probability \eqref{e:defL12}, with weight functions $\phi^+(z),\phi^-(z)$ analytic in a neighborhood of $[x_n,x_1]$. Then the following observable is analytic in a neighborhood of $[x_n,x_1]$
\begin{align*}
\bE\left[\phi^+(z)\prod_{j=1}^n\frac{z-x_j+\theta(1-e_j)}{z-x_j} +\phi^-(z)\prod_{j=1}^n\frac{z-x_j-\theta e_j}{z-x_j}\right].
\end{align*}
\end{corollary}
\begin{proof}
We take $b(z)=z$ in Theorem \ref{t:loopeq}.
\end{proof}

\subsection{Nonintersecting Poisson random walks}
Taking the limit as $p\rightarrow 0$ and scaling to
the continuous time as $\tau=\floor{p^{-1}t}$, for $t>0$,
turns the nonintersecting Bernoulli random
walk \eqref{e:defL11} into the nonintersecting Poisson random walk. This is a continuous time dynamics obtained by conditioning $n$
independent Poisson random walks to never intersect. This Markov chain $\bmx(t)=(x_1(t)>x_2(t)>\dots>x_n(t))\in \bZ^n$ has jump rates
\begin{align}\label{e:defL21}
\bP(\bmx(t+\rd t)=\bmy |\bmx(t)=\bmx)=
\begin{cases}
\frac{V(\bmx+\bme_j)}{V(\bmx)}\rd t+\OO(\rd t^2), & \bmy=\bmx+\bme_j,\\
1-n\rd t+\OO(\rd t^2), & \bmy=\bmx,\\
\OO(\rd t^2), & \text{otherwise}.
\end{cases}
\end{align}
In other words, the nonintersecting Poisson random walk is a continuous time Markov process, with generator
\begin{align}\label{e:defG1}
\cL^n f(\bmx)=\sum_{j=1}^n \frac{V(\bmx+ \bme_j)}{V(\bmx)}(f(\bmx+\bme_j)-f(\bmx)),\quad 1\leq j\leq n.
\end{align}
This nonintersecting process appeared in various settings, cf.\ \cite{konig2002non,OConnell2003,borodin2014anisotropic,GS_TASEP}. More generally, we can also consider a general $\theta$ version (as in \cite{MR3418747,huang2017beta}) and add spatial inhomogeneities in $x$.
Thus, we fix $\theta>0$ and define the \emph{$\theta$-nonintersecting Poisson random walk}, denoted by $\bmx(t)=(x_1(t),x_2(t),\dots,x_n(t))\in \bW^n_\theta$, as a discrete time Markov process
with generator given by
\begin{align}\label{e:defG2}
\cL_\theta^n f(\bmx)= \sum_{j=1}^n \frac{V(\bmx+ \theta\bme_j)}{V(\bmx)}\phi(x_j)\bigl(f(\bmx+\bme_j)-f(\bmx)\bigr).
\end{align}

The following statement can be obtained as a limit of Corollary \ref{c:loop12}; we also provide a direct proof.
\begin{corollary}\label{c:poisson}
Take a particle configuration $\bmx=(x_1>x_2>\dots>x_n)\in \bW_\theta^n$ and consider the generator \eqref{e:defG2}, with weight function $\phi(z)$ analytic in a complex neighborhood of $[x_n,x_1+1]$. Let $G(\bmx)$ be the Stieltjes transform of the particle configuration $\bmx$:
$$
G(\bmx)=\sum_{i=1}^n \frac{1}{z-x_i}.
$$
Then the following observable is analytic in the same neighborhood of $[x_n,x_1+1]$
\begin{equation}
 \label{eq_Obs_Poisson}
 \phantom{{}={}}\theta \cL_\theta^n G(\bmx)+\phi(z)\left(\prod_{j=1}^n\frac{z-x_j+\theta}{z-x_j}-\prod_{j=1}^n\frac{z-1-x_j+\theta}{z-1-x_j}\right).
\end{equation}
\end{corollary}
\begin{proof} Let us show that the following expression in complex variable $u$ has no singularities:
\begin{align}\label{e:Gpoisson}
\sum_{j=1}^n \left[ \frac{V(\bmx+ \theta\bme_j)}{V(\bmx)}\phi(x_j) \frac{u-x_j-\theta}{u-x_j}\right] +\phi(z)\prod_{j=1}^n\frac{u-x_j+\theta}{u-x_j}.
\end{align}
Indeed, the sum might have simple poles at $u=x_i$, $1\leq i\leq n$. However, the residue at such a pole is
\begin{align*}
-\theta \frac{V(\bmx+ \theta\bme_i)}{V(\bmx)}\phi(x_i) +\theta \phi(x_i)\prod_{j: j\neq i}\frac{x_i-x_j+\theta}{x_i-x_j}=0.
\end{align*}
Therefore \eqref{e:Gpoisson} is analytic in a neighborhood of $[x_n, x_1]$. Subtracting \eqref{e:Gpoisson} at $u=z$ and $u=z-1$ and noting
$$
 \theta \cL_\theta^n G(\bmx)=\sum_{j=1}^n \frac{V(\bmx+ \theta\bme_j)}{V(\bmx)}\phi(x_j)\left(\frac{z-x_j-\theta}{z-x_j}-\frac{z-1-x_j-\theta}{z-1-x_j}\right),
$$
we conclude that \eqref{eq_Obs_Poisson} has no singularities.
\end{proof}

The dynamical loop equation for the $\theta$-nonintersecting Poisson random walk with trivial weight $\phi(z)=1$ was discovered in \cite[Corollary 2.10]{huang2017beta}, and used to show that the macroscopic fluctuations of this process with general initial data converge to the Gaussian free field.

\subsection{Dyson Brownian motion}
Consider nonintersecting Poisson random walk \eqref{e:defL21}. Its rescaled version
${\eps^{1/2} (\bmx(\eps^{-1} t)-\eps^{-1} t)}$, $t\geq 0$, converges as $\eps\to 0$ to the Dyson Brownian motion with $\beta=2$, which is a solution to the following stochastic differential equations:
\begin{align}\label{e:defG}
\rd x_i(t)=\rd B_i(t)+ \sum_{j:j\neq i}\frac{\rd t}{x_i(t)-x_j(t)},\quad 1\leq i\leq n,
\end{align}
with independent standard Brownian motions $B_i(t)$. A generalization of \eqref{e:defG} involves replacing $\theta=\frac{\beta}{2}=1$ by an arbitrary positive real number and adding a potential. The resulting diffusion is a solution to the SDE
\begin{align} \label{e:DBM}
\rd x_i(t) = \sqrt{\frac{2}{\beta}} \rd B_i(t) +\sum_{j:j\neq i}\frac{\rd t}{x_i(t)-x_j(t)}+W'(x_i(t))\rd t,\quad 1\leq i\leq n.
\end{align}
 Similarly to \eqref{e:defG}, it should be possible to obtain \eqref{e:DBM} from \eqref{e:defG2} by a diffusive scaling limit (for $\beta=2\theta\geq 1$ and no potential this was proven in \cite[Theorem 3.2]{MR3418747}, see also \cite[Theorem 1.2]{huang2017beta}). The solution to \eqref{e:DBM} is called the $\beta$-Dyson Brownian Motion with potential $W$; it originates in the work of Dyson from 1960s on the evolution of eigenvalues of random matrices, cf.\ \cite{MR0148397}. Recently the Dyson Brownian motion  \eqref{e:DBM} played a central role in the three-step approach to the proofs  of the universal asymptotic behavior of local statistics of eigenvalues for various ensembles of random matrices  in a series of works \cite{MR3687212,MR3914908,MR2919197,MR3372074,MR2810797, adhikari2020dyson}, developed by Erd{\H o}s, Yau and their collaborators.

By taking the diffusive scaling limit from the dynamical loop equations of Corollaries \ref{c:loop12} and \ref{c:poisson}, we get the following dynamical loop equation for $\beta$-Dyson Brownian motion with analytic potential. We also present a direct proof.
\begin{corollary} \label{c:DBM}
Consider the $\beta$-Dyson Brownian motion \eqref{e:DBM}, suppose that the potential $W(z)$ analytic in a complex neighborhood of the real axis, and take any $0\leq t_1 < t_2$. Let
\begin{align*}
m_n(z;t)=\sum_{i=1}^n \frac{1}{z-x_i(t)}.
\end{align*}
Then  the following stochastic integral is an analytic function of $z$ in the same neighborhood
%\begin{equation}
%\int_{t_1}^{t_2} \left[ \rd m_n(z;t)-\sqrt{\frac{2}{\beta}}\sum_{i=1}^n \frac{\rd B_i(t)}{(z-x_i(t))^2}+\frac{1}{2}\del_z (m_n(z;t)+W'(z))^2 \rd t-\frac{2-\beta}{2\beta}\del^2_zm_n(z;t) \rd t \right].
%\end{equation}
\begin{align*}
\int_{t_1}^{t_2} &\left(\rd m_n(z;t)-\sqrt{\frac{2}{\beta}}\sum_{i=1}^n \frac{\rd B_i(t)}{(z-x_i(t))^2}\right.\\
&\left.+\left(\del_z \frac{m_n^2(z;t)}{2}+\del_z[m_n(z;t) W'(z)] +\frac{\beta-2}{2\beta}\del^2_zm_n(z;t)\right)\right) \rd t .
\end{align*}
\end{corollary}
\begin{proof}
By Ito's formula and \eqref{e:DBM}, the integrand simplifies to
\begin{align}
\label{eq_DBM_Ito}
\int_{t_1}^{t_2}\sum_{i=1}^{n} \frac{W'(x_i)-W'(z)-(x_i(t)-z)W''(z)}{(z-x_i(t))^2}\rd t.
\end{align}
Analyticity of $W(z)$ implies that the expression
$$
  \frac{W'(x)-W'(z)-(x-z)W''(z)}{(z-x)^2}
$$
is not singular at $z=x$. Hence,  \eqref{eq_DBM_Ito} is analytic.
%\begin{multline*}
 % \sum_{j\ne i} \frac{dt}{(z-x_i)^2 (x_i-x_j)}=
%(1/2) \sum_{j\ne i} (\frac{dt}{(z-x_i)^2 (x_i-x_j)}+\frac{dt}{(z-x_j)^2 (x_j-x_i)})
%\\=(1/2)\sum_{j\ne i} \frac{1}{x_i-x_j} (\frac{dt}{(z-x_i)^2}-\frac{dt}{(z-x_j)^2 })
%=(1/2)\sum_{j\ne i} \frac{1}{x_i-x_j} (\frac{(z-x_i+z-x_j)(x_i-x_j)dt}{(z-x_i)^2(z-x_j)^2 })
%\\=(1/2)\sum_{j\ne i}  (\frac{(z-x_i+z-x_j)dt}{(z-x_i)^2(z-x_j)^2 })
%=\sum_{j\ne i}\frac{dt}{(z-x_i)(z-x_j)^2 }
%\\=\sum_{i}\frac{dt}{(z-x_i)} \sum_j \frac{1}{(z-x_j)^2 }-\sum_j \frac{dt}{(z-x_j)^3 }
%\\=-\partial_z m_n m_n-\sum_j \frac{dt}{(z-x_j)^3 }
%\end{multline*}
%\begin{align}
%\sum_i \frac{W'(x_i)}{(z-x_i)^2}
%=\sum_{i=1}^{n} \frac{W'(x_i)-W'(z)-(x_i(t)-z)W''(z)}{(z-x_i(t))^2}
%-W'(z)\partial_z m_n-W''(z) m_n
%\end{align}
\end{proof}

Special versions of Corollary \ref{c:DBM} (for trivial potential $W\equiv0$) were used in \cite{MR1819483} and \cite{cabanal2001fluctuations, guionnet2002large} (see also \cite[Section 4.3]{AGZ}) to prove that the macroscopic fluctuations of the $\beta$-Dyson Brownian motion converge to a Gaussian field, which can be further identified  with the Gaussian free field, cf.\ \cite{borodin2014clt2,huang2017beta}. We remark that the loop equations in \cite{cabanal2001fluctuations, guionnet2002large,AGZ} are different from those in Corollary \ref{c:DBM}. Instead of the Stieltjes transforms, polynomial test functions are used, and the corresponding loop equations are used to study the evolution of polynomial functionals of Gaussian large random matrices. For convex potential, it has been proven in \cite{unterberger2018global} that the $\beta$-Dyson Brownian motion converge to a Gaussian field. 

%For trivial potential $W\equiv0$, it was proven in \cite{MR1819483,MR2418256} the global fluctuation of $\beta$-Dyson Brownian motion converges to a Gaussian field, which is identified with the Gaussian free field in \cite{huang2017beta}. For general convex potential $W$, it was proven in \cite{unterberger2018global} the global fluctuation of $\beta$-Dyson Brownian motion converges to a Gaussian field. The mesoscopic central limit theorem $\beta$-Dyson Brownian motion was studied in \cite{adhikari2020dyson, MR4009708,MR3852256}.

\subsection{Measures on Gelfand-Tsetlin Patterns} \label{Section_Jack_example}

A Gelfand-Tsetlin pattern of depth $T$ is an interlacing sequences of signatures
\begin{align}\label{e:GTp}
\emptyset \prec \bmla(1)\prec\bmla(2)\prec \dots \prec \bmla(T-1)\prec \bmla(T).
\end{align}
The patterns were originally introduced in the context of the branching rules for restrictions of irreducible representations of the unitary group $U(N)$ onto the subgroup $U(N-1)$. The characters of these representations are Schur symmetric polynomials and, therefore, the branching rule is equivalent to the following combinatorial formula for the Schur polynomials:
\begin{align}\label{e:schur}
s_\bmla(u_1, u_2, \dots, u_{T})=\sum_{\emptyset \prec \bmla(1)\prec\bmla(2)\prec \dots \prec \bmla(T-1)\prec \bmla(T)=\bmla}u_1^{|\bmla(1)|}u_2^{|\bmla(2)|-|\bmla(1)|}\dots u_{T}^{|\bmla(T)|-|\bmla(T-1)|},
\end{align}
where $|\bmla(n)|$ is the sum of coordinates of a signature $|\bmla(n)|\in \GT_n$, $n=1,2,\dots,T$.
We would like to study the uniform measure on all Gelfand-Tsetlin patterns of depth $T$ with the fixed top row $\bmla(T)$ and refer to \cite{MR3178541} and \cite[Lecture 22]{VG2020} for more details on the interplay between Gelfand-Tsetlin patterns, random tilings, and symmetric functions.
The total number of the Gelfand-Tsetlin patterns with a fixed top row can be computed by combining \eqref{e:schur} with evaluation formulas for the Schur functions (known as Weyl dimension formulas in the representation-theoretic literature):
\begin{align}
\label{eq_GT_partition}
\sum_{\emptyset \prec \bmla(1)\prec\bmla(2)\prec \dots \prec \bmla(T)=\bmla}1
=s_\bmla(1, 1, \dots, 1)=\frac{V(\la_1, \la_2-1, \la_3-2,\dots, \la_{T}-(T-1))}{V(1,2,\dots, T)},
\end{align}
where $V$ is the Vandermonde determinant \eqref{e:Vand}.
We can view the uniform measure on all Gelfand-Tsetlin patterns as a Markov process on signatures with descending transitions from level $T$ down to level $1$.  The transition probability from $(n+1)$-st level to the $n$-level is computed using \eqref{eq_GT_partition} to be
\begin{align}\label{e:transi4}
\bP(\bmla(n)=\bmmu|\bmla(n+1)=\bmla)=\frac{s_\bmmu(\bm1)}{s_\bmla(\bm1)}=\frac{1}{n!}\frac{V(\mu_1,\mu_2-1,\dots,\mu_n-(n-1))}{V(\la_1,\la_2-1, \dots,\la_{n+1}-n)},
\end{align}
where $\bmla=(\la_1, \la_2,\dots,\la_{n+1})\in \GT_{n+1}$, $\bmmu=(\mu_1, \mu_2,\dots,\mu_n)\in \GT_n$ and $\bmmu\prec \bmla$. We encode $\bmla$ and $\bmmu$  as particle systems $\bmx$  and $\bmy$, respectively as in Section \ref{s:PHD} with $\theta=1$: $x_i=\la_i-(i-1)$, $1\leq i\leq n+1$, and $ y_i=\mu_i-(i-1)$, $1\leq i\leq n$. We can rewrite the transition probability \eqref{e:transi4} as
\begin{align}\label{e:transi5}
\bP(\bmy|\bmx)=\frac{1}{n!}\frac{V(\bmy)}{V(\bmx)}.
\end{align}
%
%It corresponds to that
%$\bmx\in [0,m+T)$, $\phi^-(x)=x$, $\phi^+=(m+T-1-x)$, they ensure that if a particle is at $0$, it has to jump to the right, and if a particle is at $m+T-1$ it can not jump.
%\begin{align}
%\bP_\bme= \frac{\Delta(\bmz)}{\Delta (\bmy)}
%\propto\frac{\Delta(\bmx+\bme)/\prod_{i=1}^m  \prod_{\ell=1}^{m+T-1}|\ell-x_i-e_i|}{\Delta(\bmx)/\prod_{i=1}^m  \prod_{\ell=0}^{m+T-1}|\ell-x_i|}
%\propto\frac{\Delta(\bmx+\bme)}{\Delta(\bmx)}\prod_{i=1}^m\left(\frac{m+T-1-x_i}{x_i}\right)^{e_i}
%\end{align}
%
%
%For the $\beta$ version, we have
%We denote the complement of $\bmx$ by $\bmy$, then $\bmy=(y_1<y_2<\dots<y_T)=[0,m+T)\setminus \bmx$, and the complement of $\bmx+\bme$ as $\bmz=(z_1<z_2<\dots<z_{T-1})=[1,m+T)\setminus \bmx+\bme$. $\bmy$ and $\bmz$ are interlaced:
%\begin{align}
%y_i<z_i\leq y_{i+1}
%\end{align}
%We can rewrite $\bP_\bme$ as a transition probability of $\bmy\rightarrow\bmz$.
%
%
%
%Given any particle configuration $\bmx\in \bW^{n+1}$, the transition probability is given by
%\begin{align}
%\bP(\bmy|\bmx)=\frac{V(\bmy)}{V(\bmx)},
%\end{align}
%where $\bmy=(y_1, y_2,\dots, y_{n})\in \bW^{n}$, with $y_i\in\{x_{i+1}, x_{i-1}+1,\dots, x_i-1\}$.
This matches \eqref{e:m2}, with $\theta=1$, $b(z)=z$, $y_i\in\{x_{i+1}+1, x_{i+1}+2, x_{i+1}+2,\dots, x_i\}$, $1\leq i \leq n$, and
\begin{align*}
w(z)=\Gamma(z-x_{n+1})\Gamma(x_1-z+1),\qquad \phi^+(z)=z-x_{n+1},\quad \phi^-(z)=x_1-z.
\end{align*}
More generally, we consider the Jack-deformed Gibbs distributions on Gelfand-Tsetlin patterns, as introduced by \cite[Definition 2.12]{MR3418747}.
For general $\theta>0$, the transition probability is
\begin{align}\label{e:transi4.5}
\bP(\bmla(n)=\bmmu|\bmla(n+1)=\bmla)=\frac{J_{\bmla/\bmmu}(1;\theta)J_\bmmu(\bm1;\theta)}{J_\bmla(\bm1;\theta)}, \qquad \bmmu\in \GT_n,\quad \bmla\in\GT_{n+1},\quad  \bmmu\prec \bmla,
\end{align}
where  $J_\lambda$ is the Jack symmetric polynomial, and $J_{\bmla/\bmmu}$ is its skew version. At $\theta=1$,  $J_\lambda$ turns into the Schur polynomial $s_\lambda$ and $J_{\bmla/\bmmu}(1,1)$ is the indicator of $\bmmu\prec \bmla$; hence, we get back to \eqref{e:transi4}.

We encode the signatures $\bmla, \bmmu$ as particle systems $\bmx, \bmy$, given by  $x_i=\la_i-(i-1)\theta$, $1\leq i\leq n+1$, and $y_i=\mu_i-(i-1)\theta$, $1\leq i\leq n$. This turns the transition probability \eqref{e:transi4.5} into a  special case of \eqref{e:m2}; we recode this statement in the following claim, whose proof is postponed till Appendix \ref{Section_Appendix_examples}.
\begin{claim}\label{c:Jackdescending}
The transition probability \eqref{e:transi4.5} in terms of $\bmx$ and $\bmy$ is proportional to
\begin{align}\label{e:transi5.5}
\bP(\bmy|\bmx)\propto\prod_{1\leq i<j\leq n}(y_i-y_j) \prod_{1\leq i\leq j\leq n}\frac{\Gamma(y_i-x_{j+1})\Gamma(x_i-y_j+\theta)}{\Gamma(y_i-x_{j+1}-\theta+1)\Gamma(x_i-y_j+1)},
\end{align}
which is identified with a  special case of \eqref{e:m2}, with $b(z)=z$ and weight $
w(z)={\Gamma(x_1-z+\theta)}{\Gamma(z-x_{n+1})}$.
\end{claim}

Theorem  \ref{t:loopeq2} gives the dynamical loop equation for transition probabilities generalizing \eqref{e:transi5.5}.

\begin{corollary} \label{Corollary_Jack_corners}
Fix  $\bmx=(x_1, x_2,\dots, x_{n+1})\in \bW_\theta^{n+1}$ and consider the following transition probability
\begin{align}\label{e:Gelf11}
\bP(\bmy|\bmx)\propto \prod_{1\leq i<j\leq n}(y_i-y_j)\prod_{1\leq i\leq j\leq n}\frac{\Gamma(y_i-x_{j+1})\Gamma(x_i-y_j+\theta)}{\Gamma(y_i-x_{j+1}-\theta+1)\Gamma(x_i-y_j+1)}\prod_{i=1}^nw(y_i),
\end{align}
where $y_i\in\{x_{i+1}+\theta,x_{i+1}+\theta+1,\dots, x_{i}\}$, $1\leq i\leq n$, and
the weight function $w(z)$ satisfies
\begin{align*}
\frac{w(z+1)}{w(z)}=\frac{\phi^+(z)}{\phi^-(z)}.
\end{align*}
If we assume that $\phi^+(z),\phi^-(z)$ are holomorphic in a complex neighborhood of $[x_{n+1}+\theta, x_1]$,
then the following observable is holomorphic in the same neighborhood:
\begin{align}\label{e:Jackloop}
\bE\left[\phi^+(z)\frac{\prod\limits_{i=1}^{n+1} (z-x_i)}{\prod\limits_{i=1}^n (z-y_i)}-\phi^-(z)\frac{\prod\limits_{i=1}^{n+1}(z-x_i-\theta+1)}{\prod\limits_{i=1}^n (z-y_i+1)}\right].
\end{align}
\end{corollary}
\begin{proof}
 Using Claim \ref{c:Jackdescending} and Theorem \ref{t:loopeq2} with $b(z)=z$, we obtain analyticity of the following observable:
 \begin{multline}\label{e:sum2_Jack}
\bE\left[\phi^+(z)\cdot (z-x_{n+1})\dfrac{\prod\limits_{i=1}^{n+1} \bigl( z-x_i \bigr)}{\prod\limits_{i=1}^{n}\bigl( z-y_i\bigr)}\right. \\\left.+\phi^-(z)\cdot (x_1-z+\theta-1)  \frac{\prod\limits_{i=1}^{n+1}\bigl( z-x_i\bigr)}{\prod\limits_{i=1}^{n}\bigl( z+1-y_i\bigr)} \prod\limits_{\begin{smallmatrix}\ell\in \{x_{i+1}+\theta,x_{i+1}+\theta+1,\dots, x_{i}\}\\ i=1,\dots,n\end{smallmatrix}}\dfrac{\bigl(z+1-\ell\bigr)}{\bigl(z-\ell\bigr)} \right].
\end{multline}
Computing the telescoping product over $\ell$ and cancelling the factors, we convert \eqref{e:sum2_Jack} into
 \begin{equation*}
 (z-x_{n+1}) \cdot \bE\left[\phi^+(z) \dfrac{\prod\limits_{i=1}^{n+1} \bigl( z-x_i \bigr)}{\prod\limits_{i=1}^{n}\bigl( z-y_i\bigr)} -\phi^-(z) \frac{\prod\limits_{i=1}^{n+1}\bigl( z-x_i-\theta+1\bigr)}{\prod\limits_{i=1}^{n}\bigl( z+1-y_i\bigr)} \right].
\end{equation*}
The last expression differs from \eqref{e:Jackloop} by  $(z-x_{n+1})$ factor, hence, we have shown so far only that \eqref{e:Jackloop} is holomorphic except, possibly, at $z=x_{n+1}$. However, neither of the terms under expectation in \eqref{e:Jackloop} have a singularity at $z=x_{n+1}$. Thus, the sum is holomorphic also at $z=x_{n+1}$.
\end{proof}

\subsection{$\beta$--Corners processes}

Let us consider again the uniform measure on Gelfand-Tsetlin patterns \eqref{eq_GT_partition}, \eqref{e:transi4} and rescale the discrete integer lattice to the continuous space via  $\bmx=\floor{\tilde \bmx/\delta}$, $\delta\to 0$. After the limit, the transition probability from the particle configuration $\tilde \bmx=(\tilde x_1>\tilde x_2>\dots>\tilde x_{n+1})\in \bR^{n+1}$ at level $n+1$ to the particle configuration $\tilde \bmy=(\tilde y_1>\tilde y_2>\dots>\tilde y_{n})\in \bR^{n}$ at level $n$ is given by
\begin{align}\label{e:defL51}
\bP(\tilde \bmy|\tilde \bmx)=\frac{1}{n!}\frac{V(\tilde \bmy)}{V(\tilde \bmx)},
\end{align}
with additional requirement that $\tilde \bmy$ interlaces with $\tilde \bmx$, i.e. $\tilde x_{i+1}<\tilde y_i<\tilde x_i$ for $1\leq i\leq n$. The same limit can be performed for each $\theta>0$ in transition probabilities \eqref{e:transi5.5} and \eqref{e:Gelf11}, resulting in
\begin{equation}\label{eq_beta_corners}
\bP(\tilde \bmy|\tilde \bmx)\propto \prod_{1\leq i\leq j\leq n}(\tilde y_i-\tilde y_j)\prod_{1\leq i\leq j\leq n}(\tilde y_i-\tilde x_{i+1})^{\theta-1}(\tilde x_i-\tilde y_j)^{\theta-1}\prod_{i=1}^n e^{W(\tilde y_i)},
\end{equation}
again with interlacing $\tilde \bmx$ and $\tilde \bmy$. The distributions of the form \eqref{eq_beta_corners} for specific choices of $W(y)$ and their discrete approximations were recently introduced in \cite{MR3418747, borodin2015general} under the name of $\beta$--corners processes with $\beta=2\theta$, see also \cite{Forrester_Rains}. As shown in \cite{Neretin} with ideas going back as far as \cite{Gelfand_Naimark}, at $\beta=1,2,4$, the transition probabilities of this type appear in joint distributions of eigenvalues of corners of self-adjoint real/complex/quaternionic matrices with uniformly random eigenvectors.

Here is the dynamical loop equation for the $\beta$--corners process with analytic potential.

\begin{corollary}\label{Corollary_Loop_corners}
Fix $\tilde \bmx=(\tilde x_1>\tilde x_2>\dots>\tilde x_{n+1})\in \bR^{n+1}$ and consider the following transition probability:
\begin{align*}\begin{split}
&\bP(\tilde \bmy|\tilde \bmx)\propto \prod_{1\leq i<j\leq n}(\tilde y_i-\tilde y_j)\prod_{1\leq i\leq j\leq n}(\tilde y_i-\tilde x_{j+1})^{\theta-1}\prod_{1\leq i\leq j\leq n}(\tilde x_{i}-\tilde y_j)^{\theta-1}\prod_{i=1}^n e^{W(\tilde y_i)},\\
& \tilde \bmy=(\tilde y_1> \tilde y_2>\dots> \tilde y_{n})\in \bR^{n},\quad \tilde x_{i+1}<\tilde y_i< \tilde x_i,\quad 1\leq i\leq n.
\end{split}\end{align*}
Suppose that $\theta>0$ and the weight function $W(z)$ is holomorphic in a complex neighborhood of $[\tilde x_{n+1},\tilde x_1]$. Then
the following function of $\tilde z$ is holomorphic in the same neighborhood:
\begin{equation} \label{eq_observable_corners_process}
\bE\left[\frac{\prod\limits_{j=1}^{n+1} (\tilde z-\tilde x_j)}{\prod\limits_{j=1}^n (\tilde z-\tilde y_j)}\left(W'(\tilde z)+\sum_{j=1}^{n+1} \frac{\theta-1}{\tilde z-\tilde x_j}+\sum_{j=1}^{n}\frac{1}{\tilde z-y_j}\right)\right].
\end{equation}
\end{corollary}
\begin{proof}
  Let us compute the scaling limit for the dynamical loop equation \eqref{e:Jackloop}, setting  $\bmx=\floor{\tilde \bmx/\delta}$,  $\bmy=\floor{\tilde \bmy/\delta}$, $z=\tilde z/\delta$, $w(z)=\exp( W(\delta z))$. We can further set
  $$
   \phi^-(z)=1,\qquad \phi^+(z)=\frac{w(z+1)}{w(z)}=e^{\delta W'(\delta z)+\OO(\delta^2)}.
  $$
   Multiplying  \eqref{e:Jackloop} by $\delta^{-1}$, for $\tilde z$ outside $[\tilde x_{n+1},\tilde x_1]$ we get
\begin{multline} \label{eq_corners_limit}
\delta^{-1}\bE\left[\frac{\prod\limits_{j=1}^{n+1} (\tilde z-\tilde x_j)}{\prod\limits_{j=1}^n (\tilde z-\tilde y_j)}\left(\exp\bigl( W(\tilde z+\delta)-W(\tilde z)\bigr)-\frac{\prod\limits_{j=1}^{n+1} (\tilde z-\tilde x_j+\delta(1-\theta))}{\prod\limits_{j=1}^{n+1}( \tilde z-\tilde x_j)}\frac{\prod\limits_{j=1}^{n}( \tilde z-\tilde y_j)}{\prod\limits_{j=1}^n(\tilde z-\tilde y_j+\delta)}\right)\right]\\
=\bE\left[\frac{\prod\limits_{j=1}^{n+1} (\tilde z-\tilde x_j)}{\prod\limits_{j=1}^n (\tilde z-\tilde y_j)}\left(W'(\tilde z)+\sum_{j=1}^{n+1}\frac{\theta-1}{\tilde z-\tilde x_j}+ \sum_{j=1}^{n}\frac{1}{\tilde z-\tilde y_j}+\OO(\delta)
\right)\right].
\end{multline}
Let us denote through $f_\delta(\tilde z)$ the expression of \eqref{eq_corners_limit} and by $f_0(\tilde z)$ its $\delta\to 0$ limit, which matches \eqref{eq_observable_corners_process}. Choose a positively oriented complex contour $\gamma$ enclosing $[\tilde x_{n+1},\tilde x_1]$, such that $f_\delta(\tilde z)$ is holomorphic inside $\gamma$ by Corollary \ref{Corollary_Jack_corners}, and use the Cauchy integral formula to write for $\tilde z$ inside $\gamma$:
$$
 f_\delta(\tilde z)=\frac{1}{2\pi \ii} \int_\gamma \frac{f_\delta(u)}{u-\tilde z} \rd u.
$$
Using  the computation \eqref{eq_corners_limit}, we send $\delta\to 0$ in the last identity and get the Cauchy integral formula for $f_0(\tilde z)$, which implies that $f_0(\tilde z)$ is holomorphic.
\end{proof}
%\begin{remark} Checking Corollary \ref{Corollary_Loop_corners} directly without relying on Corollary \ref{Corollary_Jack_corners} is not a simple task. For instance, here is what the computation boils down to at $n=1$ and $\theta>1$. In this case, we need to consider the density $(y-x_2)^{\theta-1}(x_1-y)^{\theta-1}e^{W(y)}$ for $x_1 < y<x_2$ and show the holomorphicity of the observable
%\begin{equation}
%\label{eq_corners_observable_n1}
%\bE\left[\frac{(z-x_1)(z-x_2)}{z-y}\left(W'(z)+\frac{\theta-1}{z-x_1}+\frac{\theta-1}{z-x_1}+\frac{1}{z-y}\right)\right].
%\end{equation}
%Integrating by parts, we have
%
%\begin{align*}
%&\phantom{{}={}}\bE\left[\frac{(z-x_1)(z-x_2)}{z-y}\left(\frac{1}{z-y}\right)\right]
%= \bE\left[\del_y \frac{(z-x_1)(z-x_2)}{z-y}\right]
%\\&=-\int \frac{(z-x_1)(z-x_2)}{z-y} \del_y ((y-x_1)^{\theta-1}(x_2-y)^{\theta-1}e^{W(y)})\rd y\\
%&=-\bE\left[ \frac{(z-x_1)(z-x_2)}{z-y} \left(\frac{\theta-1}{y-x_1}+\frac{\theta-1}{y-x_2}+W'(y)\right)\right].
%\end{align*}
%Which can be rewritten as
%$$
%\bE\left[ \frac{(z-x_1)(z-x_2)}{z-y} \left(W'(y)+(\theta-1)\frac{z-y +y-x_1}{(z-x_1)(y-x_1)}+(\theta-1)\frac{z-y+ y-x_2}{(z-x_2)(y-x_2)}+\frac{1}{z-y}\right)\right]=0
%$$
%Subtracting the last expression from \eqref{eq_corners_observable_n1}, it remains to show the holomorphicity of
%$$
%\bE\left[ (z-x_1)(z-x_2)\frac{W'(z)-W'(y)}{z-y}+(1-\theta)\frac{z-x_2}{y-x_1}+(1-\theta)\frac{z-x_1}{y-x_2} \right],
%$$
%which is true, because the expression under the expectation has no singularities in $z$.
%\end{remark}

\subsection{Macdonald processes} \label{Section_Macdonald_example}
We use the following notations:
\begin{align}\label{e:bracketf}
(a;q)_\infty=\prod_{i=1}^\infty(1-aq^{i-1}),\quad f(u)=\frac{(tu;q)_\infty}{(qu;q)_{\infty}},\quad
\Gamma_q(x)=(1-q)^{1-x}\frac{(q;q)_\infty}{(q^x;q)_\infty}.
\end{align}
\subsubsection{Macdonald polynomials and specializations}
The next class of Markov chains is built out of the specializations of Macdonald polynomials. We refer to \cite[Section VI]{macdonald1998symmetric} and \cite{borodin2014macdonald} for definitions and properties of Macdonald symmetric functions and use them as a black box in this section. Macdonald symmetric functions $P$ and $Q$ are indexed by
partitions and implicitly depend on two parameters $q,t\in(0,1)$. The coefficients for the symmetric functions are in $\bQ[q,t]$. Macdonald symmetric functions $P_\bmla(\bmx;q,t)$ and $Q_\bmla(\bmx;q,t)$ are elements of the algebra $\Lambda$ of the symmetric functions in infinitely many variables $(x_i)_{i=1}^\infty$ uniquely determined by the following two properties:
\begin{enumerate}
\item $P_\bmla$, $|\bm\lambda|=m$, can be expressed in terms of the monomial symmetric functions via a strictly upper unitriangular transition matrix:
$$
P_\bmla=m_\bmla+\sum_{\bmmu<\bmla\in \bY_m}R_{\bmla\bmmu}P_{\bmmu},
$$
where $R_{\bmla \bmmu}$ are functions of $q,t$ and $\bmmu<\bmla$ is comparison in the dominance order on the set $\bY_m$ of all partitions of $m$ (equivalently, Young diagrams with $m$ boxes).
\item They are pairwise orthogonal with respect to the scalar product defined on the power sums via
$$
\langle p_\bmla,p_\bmmu\rangle_{q,t}=\delta_{\bmla\bmmu}z_\bmla(q,t), \qquad p_\lambda=\prod_{i=1}^{\infty} p_{\lambda_i}, \qquad z_\bmla(q,t)=\prod_{i\geq 1}i^{m_i}(m_i)!\prod_{i=1}^{\ell(\bmla)}\frac{1-q^{\la_i}}{1-t^{\la_i}},
$$
where $\bmla=1^{m_1}2^{m_2}\dots$, i.e.\ $m_i$ is the multiplicity of $i$ in $\bmla$, $\ell(\bmla)$ is the number of rows, and $p_k=(x_1)^k+(x_2)^k+\dots$, $k\geq 1$.
\end{enumerate}
We further define $Q_\bmla=\frac{P_\bmla}{\langle P_\bmla, P_\bmla\rangle_{q,t}}.$ Finally, the skew Macdonald polynomials $P_{\bmla/\bmmu}$ and $Q_{\bmla/\bmmu}$ are defined through the expansions:
\begin{align*}
 P_\bmla(x_1,x_2,\dots, y_1,y_2,\dots; q,t)&=\sum_{\bmmu} P_{\bmla/\bmmu}(x_1,x_1,\dots;q,t) P_\bmmu(y_1,y_2,\dots; q,t),\\
 Q_\bmla(x_1,x_2,\dots, y_1,y_2,\dots; q,t)&=\sum_{\bmmu} Q_{\bmla/\bmmu}(x_1,x_1,\dots;q,t) Q_\bmmu(y_1,y_2,\dots; q,t).
\end{align*}

Computations in the algebra of symmetric functions $\Lambda$ can be converted into numeric identities by means of \emph{specializations}, which are algebra homomorphism from $\Lambda$ to the set of complex numbers. Specialization $\rho$ is uniquely determined by its values on any set of algebraic generators of $\Lambda$ and we use $(p_k)_{k=1}^{\infty}$ as such generators. The value of $\rho$ on a symmetric function $f$ is denoted $f(\rho)$.
Given two specializations $\rho,\rho'$, we define their union $(\rho,\rho')$ through the formula:
$$
p_k(\rho,\rho')=p_k(\rho)+p_k(\rho'),\quad k\geq 1.
$$
A specialization $\rho$ is called \emph{Macdonald nonnegative} if its values on all (skew) Macdonald symmetric functions are non-negative, i.e., if for all partitions $\bm\lambda$ and $\bm\mu$,
\begin{align*}
P_{\bmla/\bmmu}(\rho;q,t)\geq 0,
\end{align*}
The description of Macdonald nonnegative specializations was conjectured in \cite{kerov2003asymptotic} and proven in \cite{matveev2019macdonald}:
\begin{theorem}[\cite{matveev2019macdonald}] \label{Theorem_Macdonald_positive}
For any fixed $q,t\in(0,1)$, Macdonald nonnegative specializations can be parameterized by triplets $(\bm\alpha=\{\al_i\}_{i\geq 1}, \bm\beta=\{\beta_i\}_{i\geq 1},\gamma)$ of nonnegative numbers satisfying $\sum_{i=1}^\infty(\al_i+\beta_i)<\infty$. The specialization $\rho$ corresponding to a triplet $(\bm\alpha, \bm \beta, \gamma)$ is defined by for $k\geq 2$
$$
 p_1(\rho)=\sum_{i=1}^\infty \alpha_i+ \frac{1-q}{1-t}\left(\gamma+ \sum_{i=1}^{\infty}\beta_i\right) , \qquad p_k(\rho)=\sum_{i=1}^{\infty} (\alpha_i)^k + (-1)^{k-1}\frac{1-q^k}{1-t^k} \sum_{i=1}^{\infty} (\beta_i)^k.
$$
\end{theorem}

Following \cite[Section 2.3]{borodin2014macdonald}, we create Markov chains out of the Macdonald-positive specializations:
\begin{definition} \label{Definition_Macdonald_ascending}
 Given two specializations $\rho$ and $\rho'$ we define the ascending transition through
\begin{equation}
\label{eq_ascending_transition}
 \bP(\bmla\mid \bmmu)= \frac{1}{\Pi(\rho;\rho')} \frac{P_\bmla(\rho; q,t)}{P_\bmmu(\rho; q,t)} Q_{\bmla/\bmmu}(\rho'; q,t),
\end{equation}
where $\bmla$ and $\bmmu$ are partitions with $\bmmu\subset\bmla$ and $\Pi(\rho;\rho')$ is the result of applying $\rho$ to the $x_i$ variables and $\rho'$ to the $y_j$ variables in the infinite product
$$
 \Pi=\prod_{i,j\geq 1} \frac{(tx_i y_j;q)_\infty}{(x_i y_j;q)_\infty}.
$$
\end{definition}
\begin{definition} \label{Definition_Macdonald_descending}
 Given two specializations $\rho$ and $\rho'$ we define the descending transition through
\begin{equation*}
 \bP(\bmmu\mid \bmla)= \frac{P_\bmmu(\rho; q,t)}{P_\bmla(\rho,\rho'; q,t)} Q_{\bmla/\bmmu}(\rho'; q,t),
\end{equation*}
where $\bmla$ and $\bmmu$ are partitions with $\bmmu\subset\bmla$.
\end{definition}

\subsubsection{Loop equations for the ascending process}

For particular choices of $\rho$ and $\rho'$, the ascending process of Definition \ref{Definition_Macdonald_ascending} fits into our formalism of dynamic loop equations. Namely, we fix $n=1,2,\dots$ and set in the notations of Theorem \ref{Theorem_Macdonald_positive}
\begin{equation} \label{eq_principal_ascending}
 \rho:\, \alpha_i=t^{i-1},\, 1\leq i \leq n;\qquad \rho': \beta_1=b,
\end{equation}
with all other parameters set to $0$.

\begin{claim}\label{c:mdensity2}
The transition probability of Definition \ref{Definition_Macdonald_ascending} under the specializations \eqref{eq_principal_ascending} is non-degenerate only for partitions $\bmla,\bmmu$ with at most $n$ parts, i.e.\ $\bmla=(\lambda_1\geq \lambda_2\geq\dots\geq\lambda_n)$ and $\bmmu=(\mu_1\geq \mu_2\geq\dots\geq\mu_n)$. Further, if we  set $t=q^{\theta}$ and identify
\begin{equation}\label{e:x1}
\bmx=(x_1,x_2,\dots, x_{n})\in \bW_\theta^{n},\quad x_i=\mu_i-\theta(i-1),\quad x_i+e_i=\lambda_i-\theta(i-1), \quad1\leq i\leq n,
\end{equation}
then $e_i\in\{0,1\}$ and the transition probability is proportional to
\begin{align}\label{e:mdensity2}
\bP(\bmx+\bme|\bmx)\propto
b^{\sum_{i=1}^ne_i}
\prod_{1\leq i<j\leq n}\frac{q^{x_i+\theta e_i}-q^{x_j+\theta e_j}}{q^{x_i}-q^{x_j}},
\end{align}
which is a special case of \eqref{e:m1}, with $b(z)=q^z$ and $\phi^+(z)=b,\phi^-(z)=1$.
\end{claim}

By taking $q\rightarrow 1$, Macdonald symmetric polynomials recover the Jack symmetric polynomials with parameter $\theta$. The global fluctuations of these measures related to Jack symmetric polynomials have been previously studied in \cite{dolkega2016gaussian, moll2015random, huang2021law} by different methods. We postpone the proof of Claim \ref{c:mdensity2} till Appendix \ref{Section_Appendix_examples}.
Theorem  \ref{t:loopeq} gives the dynamical loop equation for transition probabilities generalizing the ascending Macdonald process \eqref{e:mdensity2}.
\begin{corollary}
Fix $\bmx=(x_1>x_2>\dots>x_n)\in \bW_\theta^n$, and consider the following transition probability
\begin{align*}
\bP(\bmx+\bme|\bmx)\propto
\prod_{1\leq i<j\leq n}\frac{q^{x_i+\theta e_i}-q^{x_j+\theta e_j}}{q^{x_i}-q^{x_j}}\prod_{i=1}^n \phi^+(x_i)^{e_i} \phi^-(x_i)^{1-e_i}, \quad \bme\in\{0,1\}^n.
\end{align*}
Suppose that the weighs $\phi^+(z),\phi^-(z)$ are holomorphic functions in a neighborhood of $[x_n,x_1]$. Then
the following observable is holomorphic in the same neighborhood
\begin{align*}
&\bE\left[\phi^+(z)\prod_{j=1}^n
\frac{q^{z+\theta}-q^{x_j+\theta e_j}}{q^z-q^{x_i}}+\phi^-(z)\prod_{j=1}^n\frac{q^z-q^{x_j+\theta e_j}}{q^z-q^{x_j}}
\right].
\end{align*}
\end{corollary}
\begin{proof}
 This is Theorem  \ref{t:loopeq} with $b(z)=q^z$.
\end{proof}

\subsubsection{Loop equations for the descending process}
For particular choices of $\rho$ and $\rho'$, the descending process of Definition \ref{Definition_Macdonald_descending} also fits into our formalism of dynamic loop equations. Namely, we fix $n=1,2,\dots$ and set in the notations of Theorem \ref{Theorem_Macdonald_positive}
\begin{equation} \label{eq_principal_descending}
 \rho:\, \alpha_i=t^{i-1},\, 1\leq i \leq n;\qquad \rho': \alpha_1=t^n,
\end{equation}
with all other parameters setting to $0$.
\begin{claim}\label{c:Mdescending}
The transition probability of Definition \ref{Definition_Macdonald_descending} under the specializations \eqref{eq_principal_descending} is non-degenerate only for partitions $\bmla$ with at most  $(n+1)$ parts, i.e.\ $\bmla=(\lambda_1\geq \lambda_2\geq\dots\geq\lambda_{n+1})$,  and $\bmmu$ with at most $n$ parts $\bmmu=(\mu_1\geq \mu_2\geq\dots\geq\mu_n)$, and such that $\bmmu\prec\bmla$. Further, if we set $t=q^{\theta}$ and denote
\begin{align*}%\label{e:bmx}
&\bmx=(x_1,x_2,\dots, x_{n+1})\in \bW_\theta^{n+1}, &x_i=\la_i-\theta(i-1), \quad&1\leq i\leq n+1,\\
&\bmy=(y_1,y_2,\dots, y_{n})\in \bW_\theta^{n}, &y_i=\mu_i-\theta(i-1), \quad &1\leq i\leq n,
\end{align*}
then the transition probability is proportional to
\begin{align}\label{e:Mdescend}
%\bP(\bmy|\bmx)\propto
\prod_{1\leq i<j\leq n}
(q^{-y_i}-q^{-y_j})
\prod_{1\leq i\leq j\leq n}\frac{\Gamma_q(y_i-x_{j+1})}{\Gamma_q(y_i-x_{j+1}+1-\theta)}
\frac{\Gamma_q(x_i-y_j+\theta)}{\Gamma_q(x_i-y_j+1)}\prod_{i=1}^n q^{-((\theta-1)(n-i)+\theta)y_i} ,
\end{align}
%\begin{align}\label{e:Mdescend}
%\bP(\bmy|\bmx)
%\propto \prod_{1\leq i<j\leq n}
%(1-q^{y_i-y_j})
%\prod_{1\leq i\leq j\leq n}\frac{\Gamma_q(y_i-x_{j+1})}{\Gamma_q(y_i-x_{j+1}+1-\theta)}
%\frac{\Gamma_q(x_i-y_j+\theta)}{\Gamma_q(x_i-y_j+1)}\prod_{i=1}^n q^{-\theta (n-i+1)y_i} ,
%\end{align}
which is a  special case of \eqref{e:m2}, with $b(z)=q^{-z}$ and weight
\begin{equation} \label{eq_weight_Macdonald}
w(z)=q^{-z^2/2+z (x_{n+1}+1/2)}{\Gamma_q(x_1-z+\theta)}{\Gamma_q(z-x_{n+1})}.
\end{equation}
\end{claim}
We postpone the proof of Claim \ref{c:Mdescending} till Appendix \ref{Section_Appendix_examples}. Note that for special values of $\theta$, the weights \eqref{e:Mdescend} can be dramatically simplified. For instance, at $\theta=1$ all $\Gamma_q$ factors cancel out and the $\prod_{1\leq i\leq j\leq n}$ part disappears from the formula. For $\theta=2$, the formula is also simple:
\begin{equation}\label{e:Mdescend_th_2}
\bP(\bmy|\bmx)
\propto \prod_{1\leq i<j\leq n}
(q^{-y_i}-q^{-y_j})
\prod_{1\leq i\leq j\leq n}\bigl[(1-q^{y_i-x_{j+1}-1})(1-q^{x_i-y_j+1})\bigr]\prod_{i=1}^n q^{-(n-i+2)y_i}.
\end{equation}
Theorem  \ref{t:loopeq2} gives the dynamical loop equation for transition probabilities generalizing \eqref{e:Mdescend}.
\begin{corollary} \label{Cor_Macdonald_descending_obs}
Fix $\bmx=(x_1, x_2,\dots, x_{n+1})\in \bW_\theta^{n+1}$ and  consider the following transition probability
\begin{align*}
%\bP(\bmy|\bmx) \propto
\prod_{1\leq i<j\leq n}
(q^{-y_i}-q^{-y_j})
\prod_{1\leq i\leq j\leq n}\frac{\Gamma_q(y_i-x_{j+1})}{\Gamma_q(y_i-x_{j+1}+1-\theta)}
\frac{\Gamma_q(x_i-y_j+\theta)}{\Gamma_q(x_i-y_j+1)}\prod_{i=1}^n q^{-((\theta-1)(n-i)+\theta)y_i}w(y_i),
\end{align*}
with $y_i\in\{x_{i+1}+\theta, x_{i+1}+1+\theta, x_{i+1}+2+\theta,\dots, x_i-1, x_i\}$ for $1\leq i\leq n$. Suppose that
\begin{align*}
\frac{w(z+1)}{w(z)}=\frac{\phi^+(z)}{\phi^-(z)},
\end{align*}
with $\phi^+(z)$ and $\phi^-(z)$ holomorphic in a complex neighborhood of $[x_{n+1}, x_1]$.
Then the following observable is holomorphic in the same neighborhood:
\begin{align*}
\bE\left[
\phi^+(z)\frac{\prod\limits_{j=1}^{n+1} (q^{-z}-q^{-x_j})}{\prod\limits_{j=1}^n(q^{-z}-q^{-y_j})}-q^{(\theta-1)(n+1)}\phi^-(z)\frac{\prod\limits_{j=1}^{n+1}(q^{-z} -q^{-x_i+1-\theta})}{\prod\limits_{j=1}^n(q^{-z-1}-q^{-y_j})}
\right].
\end{align*}
\end{corollary}
\begin{proof}
 Using Claim \ref{c:Mdescending} and Theorem \ref{t:loopeq2} with $b(z)=q^{-z}$, we obtain the holomorphicity of
\begin{multline}
\label{eq_x1}
\bE\left[(1-q^{z-x_{n+1}})\phi^+(z)\dfrac{\prod\limits_{i=1}^{n+1} \bigl(q^{-z}-q^{-x_i} \bigr)}{\prod\limits_{i=1}^{n}\bigl( q^{-z}-q^{-y_i}\bigr)}\right.\\+\left.(1-q^{x_1-z+\theta-1}) q^{z-x_{n+1}} \phi^-(z)  \prod\limits_{\ell\in \bL(\bmx)}\dfrac{\bigl(q^{-z-1}-q^{-\ell}\bigr)}{\bigl(q^{-z}-q^{-\ell}\bigr)} \cdot \frac{\prod\limits_{i=1}^{n+1}\bigl( q^{-z}-q^{-x_i}\bigr)}{\prod\limits_{i=1}^{n}\bigl( q^{-z-1}-q^{-y_i}\bigr)}\right].
\end{multline}
Let us simplify the product over $\ell$ in the second term:
\begin{equation*}
 \prod_{i=1}^n\, \prod_{\ell\in \{x_{i+1}+\theta, x_{i+1}+1+\theta, \dots, x_i-1, x_i\}} \dfrac{q^{-1}\bigl(q^{-z}-q^{1-\ell}\bigr)}{\bigl(q^{-z}-q^{-\ell}\bigr)}= \prod_{i=1}^n\left[ q^{-(x_i-x_{i+1}+1-\theta)} \frac{q^{-z}-q^{-x_{i+1}+1-\theta}}{q^{-z}-q^{-x_i}}\right].
\end{equation*}
Therefore, \eqref{eq_x1} is
\begin{equation*}
(1-q^{z-x_{n+1}})\bE\left[\phi^+(z)\dfrac{\prod\limits_{i=1}^{n+1} \bigl(q^{-z}-q^{-x_i} \bigr)}{\prod\limits_{i=1}^{n}\bigl( q^{-z}-q^{-y_i}\bigr)}\right.\\-\left.q^{(\theta-1)(n+1)}  \phi^-(z)   \frac{ \prod\limits_{i=1}^{n+1}  (q^{-z}-q^{-x_{i}+1-\theta}) }{\prod\limits_{i=1}^{n}\bigl( q^{-z-1}-q^{-y_i}\bigr)}\right].
\end{equation*}
It remains to note that neither of the terms under expectation has a singularity at $z=x_{n+1}$ and, hence, we can divide by $(1-q^{z-x_{n+1}})$ and the result is still holomorphic.
\end{proof}

\subsection{$(q,\kappa)$-distributions on tilings and Koornwinder polynomials}

\label{Section_tilings_as_Markov_chain}

In this section we first explain how $(q,\kappa)$--distributions on tilings of Section \ref{s:heightf} fit into the formalism of dynamic loop equations. Then we show that these distributions are particular cases of more general Markov chains which are obtained from quasi-branching rules for Koornwinder symmetric polynomials.

\subsubsection{Weighted lozenge tilings}

\label{Section_weighted_tilings_loop}

We deal with random lozenge tilings of a trapezoid, as in Figures \ref{Fig_trapezoid} and \ref{Fig_trapezoid_MC}. The probability of each tiling $\cT$ is proportional to the weight given in terms of \La\, lozenges, as in \eqref{e:qtiling}, \eqref{eq_lozenge_weight}:
 \begin{align} \label{eq_lozenge_weight_copy}
 w(\cT)=\prod_{\La\in \cT}w(\La),\qquad w(\La)=\kappa q^{x-t/2}- \kappa^{-1}q^{-x+t/2}=\kappa q^{\tilde x}-\kappa^{-1} q^{-\tilde x}.
 \end{align}
We further identify a tiling with a collection of $N$ non-intersecting random walks $\bmx(t)=(x_1(t)>\dots>x_n(t))$, $0\leq t \leq T$, by tracing the trajectories of \Lb\, and \Lc\, lozenges, as in Figure \ref{Fig_trapezoid_MC}. The trapezoid is uniquely determined by specifying $N$, $T$, arbitrary initial configuration $\bmx(0)$, and prefixed ending configuration $\bmx(T)=(N-1,N-2,\dots,0)$. The measure \eqref{eq_lozenge_weight_copy} leads to a Markov chain structures on $\bmx(t)$ with explicit transition probabilities.

 \begin{figure}[t]
\begin{center}
 \includegraphics[scale=0.4]{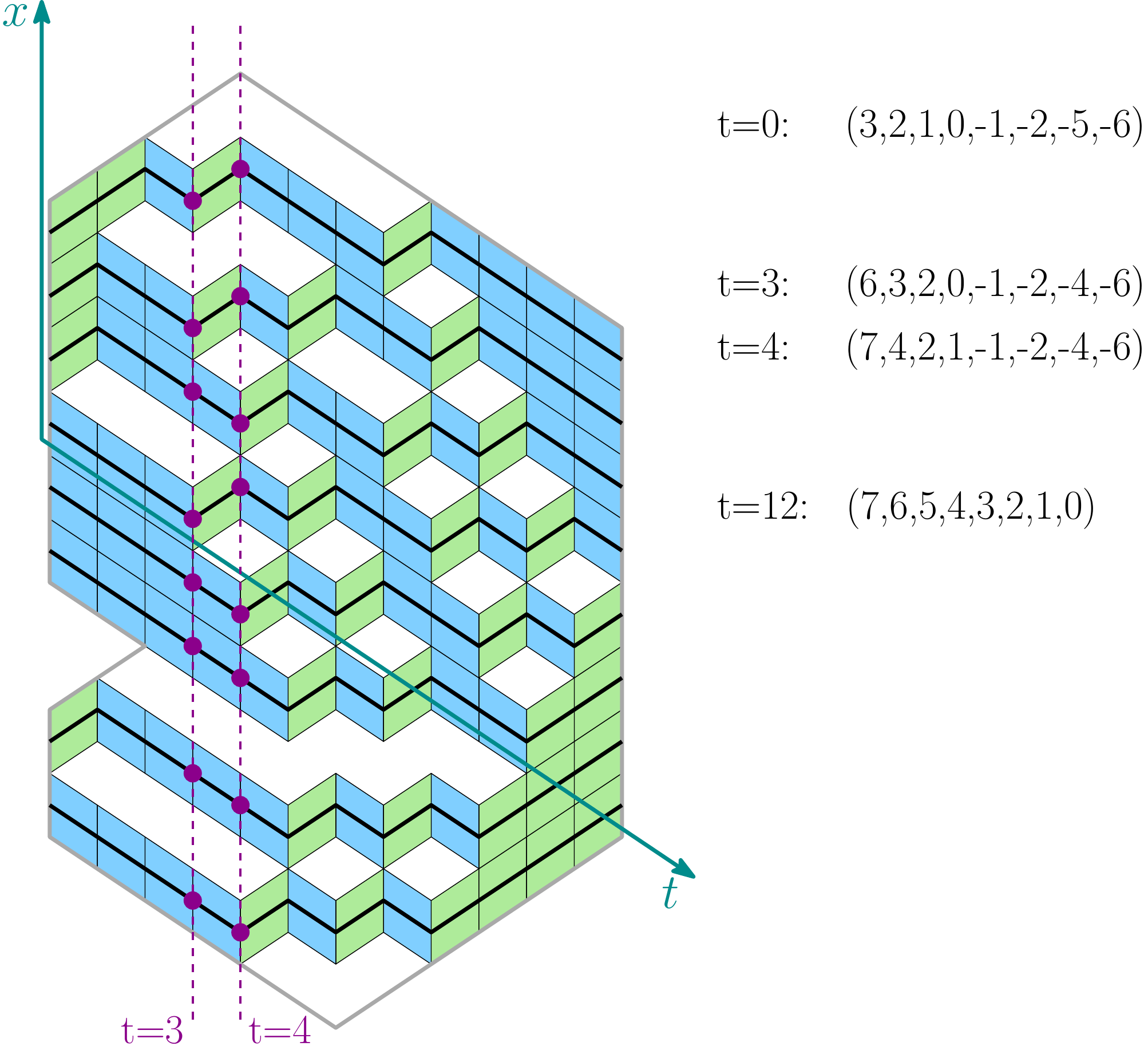}
 \caption{Tracking positions of particles in non-intersecting paths, tilings turn into a Markov chain in  $\bW^{N}$.   }
 \label{Fig_trapezoid_MC}
 \end{center}
 \end{figure}

\begin{theorem} \label{Theorem_transition_qk}
$\{\bmx(t)\}_{0\leq t\leq T}$ corresponding to the $(q,\kappa)$-distributions on lozenge tilings  with weights \eqref{eq_lozenge_weight_copy} is a Markov chain with transition probability proportional to
\begin{align}\label{e:qtransit}
\bP(\bmx(t+1)=\bmx+\bme| \bmx(t)=\bmx)
\propto \prod_{1\leq i<j\leq N} \frac{b_{t}(x_i+e_i)-b_{t}(x_j+e_j)}{b_t(x_i)-b_t(x_j)}\prod_{i=1}^N\phi_t^+(x_i)^{e_i} \phi_t^-(x_i)^{1-e_i},
\end{align}
where $\bme\in\{0,1\}^N$ and
$$b_t(x)=q^{-x}+\kappa^2 q^{x-t},$$
$$
\phi_t^+(x)=q^{T+N-1-t}(1-q^{x-N+1})(1-\kappa^2 q^{x-T+1}), \qquad \phi_t^-(x)=-(1-q^{x+T-t})(1-\kappa^2 q^{x+N-t}).
$$
\end{theorem}

\smallskip

We remark that the weights $\phi_t^\pm(x)$ encode a part of the boundary of the tiled trapezoid. More precisely, $\phi_t^+(x)$ has a zero at $x=N-1$ and $\phi_t^-(x)$ has a zero at $t-x=T$. Therefore, if at time $t$ a particle is at $N-1$, it will stay; if a particle is at $t-T$, it will jump one unit to the right.
\begin{proof}[Proof of Theorem \ref{Theorem_transition_qk}] Up to shift $x\mapsto x+S$, this is \cite[Proposition 4.2]{borodin2010q}. Let us, however, present a sketch of the proof. The Markov property that ``given present, future and past are independent'' immediately follows from the definition of the system. Further, let $Z(x_1,\dots,x_N; T)$ denote the partition function, i.e.\ this is the total sum of weights \eqref{eq_lozenge_weight_copy} over the tilings corresponding to non-intersecting paths $\bmx(t)=(x_1(t)>\dots>x_n(t))$, $0\leq t \leq T$, with $\bmx(0)=(x_1,\dots,x_N)$ and $\bmx(T)=(N-1,N-2,\dots,0)$. We can express
\begin{equation}
\label{eq_transition_partition}
 \bP(\bmx(t+1)=\bmx+\bme| \bmx(t)=\bmx)=\frac{Z(\bmx+\bme; T-t-1)}{Z(\bmx; T-t)} \prod_{\La} \left(\kappa q^{x(\La)-t/2}- \kappa^{-1}q^{-x(\La)+t/2}\right),
\end{equation}
where the product is taken over all horizontal lozenges on the vertical line with ordinate $t+1$; the positions of these lozenges complement the configuration $\bmx+\bme$. The key observation is that $Z(x_1,\dots,x_N; T)$ has an explicit product formula: this computation relying on explicit evaluation of certain determinants is present in \cite[Appendix]{borodin2010q}; another way to get the same answer is through explicit evaluation of principally-specialized Koornwinder polynomials, as we explain in Theorem \ref{Theorem_qRacah_partition} below. Plugging the result into \eqref{eq_transition_partition}, we arrive at \eqref{e:qtransit}.
\end{proof}

Theorem  \ref{t:loopeq} gives the dynamical loop equation for transition probabilities \eqref{e:qtransit}.
\begin{corollary}
Fix $\bmx=(x_1>x_2>\dots>x_n)\in \bW^n$ and $0\leq t<T$. Consider transition probability \eqref{e:qtransit}.
Then the following observable is holomorphic in a complex neighborhood of $[x_n,x_1]$,
\begin{multline*}
\bE\left[q^{T+N-1-t}(1-q^{z-N+1})(1-\kappa^2 q^{z-T+1})\prod_{j=1}^n\frac{b_t(z+\theta)-b_t(x_j+\theta e_j)}{b_t(z)-b_t(x_j)}\right.\\-\left.(1-q^{z+T-t})(1-\kappa^2 q^{z+N-t})\prod_{j=1}^n\frac{b_t(z)-b_t(x_j+\theta e_j)}{b_t(z)-b_t(x_j)}\right],
\end{multline*}
where $b_t(z)=q^{-z}+\kappa^2 q^{z-t}$.
\end{corollary}

\subsubsection{Quasi-branching of Koornwinder symmetric polynomials}

\label{Section_quasi_branching}

Following the notations in \cite[Definition 2]{rains2005bc}, we denote $K_{\bmla}^{(n)}(x_1,\dots,x_n; q,t;t_0,t_1,t_2,t_3)$ the Koorwinder symmetric (Laurent) polynomial in $n$ variables $x_1,\dots,x_n$. We explained in Sections \ref{Section_Jack_example} and \ref{Section_Macdonald_example} how branching rules for Schur, Jack, and Macdonald polynomials lead to measures on Gelfand-Tsetlin patterns, which fit into our formalism of dynamical loop equations. Generally speaking, the branching rules for Koornwinder polynomials are much more complicated, and do not lead to any simple measures on Gelfand--Tsetlin patterns. However, when we deal with a special principal specialization, there is a dramatic simplification known as \emph{quasi-branching rule}. The following result is \cite[Theorem 5.21]{rains2005bc}.
\begin{theorem} \label{Theorem_Koornwinder_branching}
For any partition $\bmla\in \GT_{n+1}$ with $\lambda_{n+1}\geq 0$,  we have
\begin{multline}\label{e:branchrule}
\frac{K_{\bmla}^{(n+1)}(x_1, \dots, x_{n},t_0;\, q,t;\, t_0,t_1,t_2,t_3)}{K_{\bmla}^{(n+1)}(q^{n}t_0, q^{n-1}t_0, \dots, t_0;\, q,t;\, t_0,t_1,t_2,t_3)}\\
=\sum_{\bmmu\prec \bmla}\psi_{\bmla\setminus\bmmu}^{(i)}(t^{n+1};\, q,t,t^{n}\sqrt{t_0t_1t_2t_3/qt})\frac{K_\bmmu^{(n)}(x_1,\dots, x_{n};\, q,t;\, t_0t,t_1,t_2,t_3)}{K_\bmmu^{(n)}(q^{n}t_0, q^{n-1}t_0, \dots,qt_0;\, q,t;\, t_0t,t_1,t_2,t_3)},
\end{multline}
where the coefficients are given by
\begin{align}\begin{split}\label{e:phii}
\psi_{\bmla\setminus\bmmu}^{(i)}(u; q,t,s)
&=(u/t)^{|\bmla|-|\bmmu|}t^{n(\bmmu)-n(\bmla)}
\frac{C_\bmla^0(s^2qt/u;q,t)C_\bmmu^0(u/t;q,t)}{C_\bmla^0(u;q,t)C_\bmmu^0(s^2qt/u;q,t)}\\
&\prod_{(i,j)\in \bmla\atop \lambda_j'=\mu_j'}\frac{1-q^{\la_i+j-1}t^{-\la_j'-i+3}s^2}{1-q^{\mu_i-j+1}t^{\mu_j'-i}}
\prod_{(i,j)\in \bmla\atop \lambda_j'\neq\mu_j'}\frac{1-q^{\la_i-j}t^{\la_j'-i+1}}{1-q^{\mu_i+j}t^{-\mu_j'-i+1}s^2}\\
&\prod_{(i,j)\in \bmmu\atop \lambda_j'=\mu_j'}\frac{1-q^{\la_i-j+1}t^{\la_j'-i}}{1-q^{\mu_i+j-1}t^{2-\mu_j'-i}s^2}
\prod_{(i,j)\in \bmmu\atop \lambda_j'\neq\mu_j'}\frac{1-q^{\la_i+j}t^{2-\la_j'-i}s^2}{1-q^{\mu_i-j}t^{\mu_j'-i+1}},
\end{split}
\end{align}
with $|\bmla|=\sum_{i=1}^{n+1} \lambda_i$, $n(\bm\lambda)=\sum_{i=1}^{n+1} \lambda_i(\lambda_i-1)/2$, and similarly for $\bmmu$. Also
$$
C_\bmmu^0(x;q,t)=\prod_{(i,j)\in \bmmu}(1-q^{j-1}t^{1-i}x)=\prod_{i\geq 1}(t^{1-i}x;q)_{\mu_i}.
$$
\end{theorem}
By taking principle specialization on both sides of \eqref{e:branchrule}, i.e. $x_i=q^{n+1-i}t_0$ with $1\leq i\leq n$, we get a probability measure over Gelfand-Tsetlin patterns, with the transition probability given by
\begin{equation}\label{e:MKdensity}
\bP(\bmmu|\bmla)=\psi_{\bmla\setminus\bmmu}^{(i)}(t^{n+1}; q,t,t^{n}\sqrt{t_0t_1t_2t_3/qt}), \qquad \bmla\in \GT_{n+1},\, \lambda_{n+1}\geq 0,\, \bmmu\in \GT_n,\, \bmmu\prec\bmla.
\end{equation}

We denote $t=q^\theta$ for some $\theta>0$, and encode $\bmla$ and $\bmmu$ as particle systems $\bmx$ and $\bmy$, with
\begin{align*}
&\bmx=(x_1,x_2,\dots, x_{n+1})\in \bW_\theta^{n+1}, &x_i=\la_i-\theta(i-1), \quad &1\leq i\leq n+1,\\
&\bmy=(y_1,y_2,\dots, y_{n+1})\in \bW_\theta^{n}, &y_i=\mu_i-\theta(i-1), \quad &1\leq i\leq n.
\end{align*}
The interlacing condition $\bmmu\prec \bmla$ becomes $x_{i+1}+\theta \leq y_i\leq x_i$ for $1\leq i\leq n$. The transition probability \eqref{e:MKdensity} becomes a special case of \eqref{e:m2}, as Claim \ref{c:MKdescending} (whose proof is postponed till Appendix \ref{Section_Appendix_examples}) shows:

\begin{claim}\label{c:MKdescending}
Let $q^{v/2}=t^{n}\sqrt{t_0t_1t_2t_3/qt}$. We can rewrite the transition probability \eqref{e:MKdensity} as
\begin{align}\begin{split}\label{e:KMdensityhi}
\bP(\bmy|\bmx)
&\propto\prod_{1\leq i<k\leq n}(q^{-y_i}-q^{-y_k})\prod_{1\leq i\leq k\leq n}\frac{\Gamma_q(y_i-x_{k+1})\Gamma_{q}(x_i-y_k+\theta)}{\Gamma_q(x_i-y_k+1)\Gamma_q(y_i-x_{k+1}-\theta+1)}\\
&\times \prod_{1\leq i\leq k\leq n}(1-q^{y_i+y_k+v})\left[\prod_{i=1}^n \prod_{k=1}^{n+1} \frac{\Gamma_q(y_i+x_k+v+\theta)}{\Gamma_q(y_i+x_k+1+v)}\right]\prod_{i=1}^nq^{-((\theta-1)(n-i)+\theta)y_i},
\end{split}\end{align}
%\begin{align}\begin{split}\label{e:KMdensityhi}
%\bP(\bmy|\bmx)
%&\propto\prod_{1\leq i<k\leq n}(1-q^{y_i-y_k})\prod_{1\leq i\leq k\leq n}\frac{\Gamma_q(y_i-x_{k+1})\Gamma_{q}(x_i-y_k+\theta)}{\Gamma_q(x_i-y_k+1)\Gamma_q(y_i-x_{k+1}-\theta+1)}\\
%&\prod_{1\leq i\leq k\leq n}(1-q^{y_i+y_k+v})\prod_{i\leq n,k\leq n+1}\frac{\Gamma_q(y_i+x_k+v+\theta)}{\Gamma_q(y_i+x_k+1+v)}\prod_i q^{-\theta \sum_i (n-(i-1))y_i}\\
%\end{split}\end{align}
which is a  special case of \eqref{e:m2}, with $b(z)=q^{-z}+q^{z+v}$ and weight
\begin{align*}
w(z)=q^{-z^2/2+z (x_{n+1}+1/2)}{\Gamma_q(x_1-z+\theta)}{\Gamma_q(z-x_{n+1})}\frac{\Gamma_q(z+x_1+v+\theta)}{\Gamma_q(z+x_{n+1}+v+1)}.
\end{align*}

\end{claim}

\begin{corollary}
Fix $\bmx=(x_1, x_2,\dots, x_{n+1})\in \bW_\theta^{n+1}$ and  consider the following transition probability
\begin{align}\begin{split}\label{e:KMdensityhi_gen}
\bP(\bmy|\bmx)
&\propto\prod_{1\leq i<k\leq n}(q^{-y_i}-q^{-y_k})\prod_{1\leq i\leq k\leq n}\frac{\Gamma_q(y_i-x_{k+1})\Gamma_{q}(x_i-y_k+\theta)}{\Gamma_q(x_i-y_k+1)\Gamma_q(y_i-x_{k+1}-\theta+1)}\\
&\times \prod_{1\leq i\leq k\leq n}(1-q^{y_i+y_k+v})\left[\prod_{i=1}^n \prod_{k=1}^{n+1} \frac{\Gamma_q(y_i+x_k+v+\theta)}{\Gamma_q(y_i+x_k+1+v)}\right]\prod_{i=1}^nq^{-((\theta-1)(n-i)+\theta)y_i} w(y_i),
\end{split}\end{align}
with $y_i\in\{x_{i+1}+\theta, x_{i+1}+1+\theta, x_{i+1}+2+\theta,\dots, x_i-1, x_i\}$ for $1\leq i\leq n$. Suppose that
\begin{align*}
\frac{w(z+1)}{w(z)}=\frac{\phi^+(z)}{\phi^-(z)},
\end{align*}
with $\phi^+(z)$ and $\phi^-(z)$ holomorphic in a complex neighborhood of $[x_{n+1}, x_1]$.
Then the following observable is holomorphic in the same neighborhood:
\begin{align*}
&\bE\left[
\phi^+(z)\frac{\prod\limits_{j=1}^{n+1} (q^{-z}+q^{z+v}-q^{-x_j}-q^{x_j+v})}{\prod\limits_{j=1}^n(q^{-z}+q^{z+v}-q^{-y_j}-q^{y_i+v})}\right.\\
&\left.-q^{(\theta-1)(n+1)}\phi^-(z)\frac{\prod\limits_{j=1}^{n+1}(q^{-z}+q^{z+v} -q^{-x_i+1-\theta}-q^{x_i+\theta-1+v})}{\prod\limits_{j=1}^n(q^{-z-1}+q^{z+1+v}-q^{-y_j}-q^{y_j+v})}
\right].
\end{align*}
\end{corollary}
 We omit the proof, as it is very similar to Corollaries \ref{Corollary_Jack_corners} and \ref{Cor_Macdonald_descending_obs}.

%Then Theorem 0.5. gives the following identity
%\begin{align}
%\sum_{\bmla^{(n)}\prec \bmla^{(n+1)}}\prod_i (q^{-y_i}-q^{y_i+v})\prod_{i<j} (b(y_i)-b(y_j))
%=Z_n \prod_{i<j}(b'(x_i)-b'(x_j))
%\end{align}

\subsubsection{Koornwinder polynomials at $q=t$ and tilings}

In this section we explain that $q=t$ case of the measures of Section \ref{Section_quasi_branching} matches the measures on tilings of Section \ref{Section_weighted_tilings_loop} after the particle-hole involution. We start from a statement of independent interest.

\begin{theorem} \label{Theorem_qRacah_partition}
 For any  $\bmla=(\lambda_1\geq\lambda_2\geq\dots\geq\lambda_n)\in\GT_n$, the following identity of Laurent polynomials in $\sigma$ holds:
 \begin{equation}
 \label{eq_qRacah_sum}
  \sum_{\emptyset=\bmla^{(0)}\prec\dots\prec \bmla^{(n)}=\bmla} \prod_{k=1}^{n-1} \prod_{i=1}^k \left( \sigma q^{\lambda^{(k)}_i+\frac{k+1}{2}-i}-\frac{1}{\sigma  q^{\lambda^{(k)}_i+\frac{k+1}{2}-i}}\right)=
  Z_n\prod_{ i<j}\bigl[x(\lambda_i-i)-x(\lambda_j-j)\bigr],
 \end{equation}
 where
 $$
  x(\lambda_i-i)=\sigma q^{\lambda_i-i+\frac{n+1}{2}}+\frac{1}{\sigma q^{\lambda_i-i+\frac{n+1}{2}}},
 $$
 $$
  Z_n= (-1)^{-n(n-1)/2} \prod_{i=1}^n \frac{q^{(i-\frac{n+1}{2})i}}{(q;q)_{i-1}}.
 $$
\end{theorem}

\begin{remark} An equivalent form of the same statement is
 \begin{multline}
 \label{eq_qRacah_sum_2}
  \sum_{\emptyset=\bmla^{(0)}\prec\dots\prec \bmla^{(n)}=\bmla} \prod_{k=1}^{n-1} \prod_{i=1}^k \frac{\sigma q^{\lambda^{(k)}_i+\frac{k+1}{2}-i}-\frac{1}{\sigma  q^{\lambda^{(k)}_i+\frac{k+1}{2}-i}}}{\sigma q^{\frac{k+1}{2}-i}-\frac{1}{\sigma  q^{\frac{k+1}{2}-i}}}=
 \prod_{1\leq i<j\leq n}\frac{x(\lambda_i-i)-x(\lambda_j-j)}{x(-i)-x(-j)}.
 \end{multline}
 \end{remark}
 On one hand, the sum in Theorem \ref{Theorem_qRacah_partition} is readily identified with a sum over lozenge tilings, where $\lambda_i^{(k)}-i$ encode the positions of horizontal lozenges and the weights match \eqref{eq_lozenge_weight_copy} after a change of notations; by particle-hole involution it also gives rise to \eqref{e:qtransit}, as we explained in the proof of Theorem \ref{Theorem_transition_qk}. On the other hand, the proof\footnote{An alternative proof is hidden in \cite[Appendix]{borodin2010q}. A possibility of our present proof can be hinted  by \cite[Theorem 7.5.1]{Be}.} of Theorem \ref{Theorem_qRacah_partition} can be obtained by combining the $q=t$ version of Theorem \ref{Theorem_Koornwinder_branching} with known explicit evaluations of Koornwinder polynomials at principal specialization; we detail this in Appendix \ref{Section_Appendix_examples}.

Theorem \ref{Theorem_qRacah_partition} has an interesting corollary of representation-theoretic flavor:

\begin{corollary} \label{Corollary_orthogonal_partition}
 Take a signature $\bmla=(\lambda_1\geq\lambda_2\geq\dots\geq\lambda_n\geq 0)$. We have\footnote{$\mathrm{Dim}_{SO(2n)}(\lambda_1,\dots,\lambda_n)$ is the dimension of the irreducible representation of group $SO(2n)$ spelled out in the same notation at \url{https://en.wikipedia.org/wiki/Representations_of_classical_Lie_groups} or in \cite[Section 2.1]{Borodin_Kuan_O}.}
 \begin{multline}
 \label{eq_orthogonal_dimension}
  \sum_{\emptyset=\bmla^{(0)}\prec\bmla^{(1)}\prec\dots\prec \bmla^{(n)}=\bmla} \prod_{k=1}^{n-1} \prod_{i=1}^k \frac{\lambda^{(k)}_i+\frac{n+k}{2}-i}{\frac{n+k}{2}-i} \\=
 \prod_{1\leq i<j\leq n}\frac{\lambda_i-i-(\lambda_j-j)}{j-i}\cdot \prod_{1\leq i<j\leq n}\frac{2n+ \lambda_i-i+\lambda_j-j}{2n-j-i}=\mathrm{Dim}_{SO(2n)}(\lambda_1,\dots,\lambda_n).
 \end{multline}
\end{corollary}
\begin{proof}
 Take a signature $\bmla=(\lambda_1\geq\lambda_2\geq\dots\geq\lambda_n)\in\GT_n$. We claim that for any $\mathfrak s$,
 \begin{multline}
 \label{eq_Racah_sum}
  \sum_{\emptyset=\bmla^{(0)}\prec\bmla^{(1)}\prec\dots\prec \bmla^{(n)}=\bmla} \prod_{k=1}^{n-1} \prod_{i=1}^k \frac{\mathfrak s+ \lambda^{(k)}_i+\frac{k+1}{2}-i}{\mathfrak s+ \frac{k+1}{2}-i}\\ =
 \prod_{1\leq i<j\leq n}\frac{\lambda_i-i-(\lambda_j-j)}{(j-i)}\cdot \prod_{1\leq i<j\leq n}\frac{2\mathfrak s +n+1+ \lambda_i-i+\lambda_j-j}{2\mathfrak s +n+1-j-i}.
 \end{multline}
Indeed, notice that
 $$
  x(a)-x(b)=\frac{1}{\sigma q^{a+\frac{n+1}{2}}} \left(\sigma q^{a+\frac{n+1}{2}}-\sigma q^{b+\frac{n+1}{2}}\right) \left(\sigma q^{a+\frac{n+1}{2}}-\frac{1}{\sigma q^{b+\frac{n+1}{2}}}\right).
 $$
Setting $\sigma=q^{\mathfrak s}$ and sending $q\to 1$, \eqref{eq_qRacah_sum_2} turns into \eqref{eq_Racah_sum}. In particular, taking $\mathfrak s=\frac{n-1}{2}$ in \eqref{eq_Racah_sum} and demanding that $\lambda_n\geq 0$, we get \eqref{eq_orthogonal_dimension}.
\end{proof}

\begin{remark} Grigori Olshanski noticed that \eqref{eq_orthogonal_dimension} and \eqref{eq_Racah_sum} can be deduced from the (multiplicity-free) branching rules for restrictions of irreducible representations of $SO(2n)$ onto the subgroup $SO(2n-1)$, thus, avoiding the machinery of Koornwinder polynomials. The plan is as follows. We first notice that \eqref{eq_Racah_sum} for one particular value of $\mathfrak s$ implies the same identity for all $\mathfrak s$: indeed, shifting by integer values of $\mathfrak s$ we are getting the same statement for shifted $\lambda_i$; on the other hand, we are proving a polynomial identity, hence, its validity at infinitely many integer shifts of a single $\mathfrak s$ implies the validity at all $\mathfrak s$.

Next, we are checking \eqref{eq_Racah_sum} by induction in $n$ for a particular value of $\mathfrak s$. This check is reduced to the statement that the dimension of an irreducible representation of $SO(2n)$ is the sum of dimensions of irreducible representations of $SO(2n-1)$ into which it splits.
\end{remark}

It would be interesting to produce a direct representation-theoretic proof of the identity \eqref{eq_orthogonal_dimension}.

\section{Analysis of Dynamical Loop Equation}\label{s:analysis}

\subsection{Statement of the asymptotic theorem}

The aim of present Section \ref{s:analysis} is to analyze the transition probability $\bP(\bmy=\bmx+\bme|\bmx)=a_\bme$ of \eqref{e:m1}:
\begin{align}\label{e:m1ccopy}
a_\bme=\frac{1}{Z(\bmx)} \prod_{1\leq i<j\leq n}\frac{b(x_i+\theta e_i)-b(x_j+\theta e_j)}{b(x_i)-b(x_j)}\prod_{i=1}^n \phi^+(x_i)^{e_i} \phi^-(x_i)^{1-e_i}, \quad \bme\in\{0,1\}^n,
\end{align}
where $\bmx\in \bW_\theta^n$, as defined in \eqref{e:defWtheta}; in particular, $x_i\geq x_{i+1}+\theta$ for $i=1,2,\dots, n-1$. The prefactor $\frac{1}{Z(\bmx)}$ is chosen so that $\sum_{\bme} a_\bme=1$. We encode the particle configurations $\bmx$, $\bmy$ in \eqref{e:m1ccopy} by their (smoothed) empirical densities:
\begin{align}
\label{eq_rho_x_def}
\rho(s;\bmx)=\frac{1}{\theta}\sum_{i=1}^n \bm1(\varepsilon x_i \leq s\leq \varepsilon x_i+\varepsilon \theta ), \qquad \rho(s;\bmy)=\frac{1}{\theta}\sum_{i=1}^n \bm1(\varepsilon y_i \leq s\leq \varepsilon y_i+\varepsilon \theta ).
\end{align}

We impose several conditions on the asymptotic behavior of various ingredients of \eqref{e:m1ccopy} as $n\to\infty$:

\begin{assumption}\label{a:asymp}
Fix constants $\sfN>0$ and $\bl, \br$.
We introduce a small parameter $\varepsilon=\sfN/n\ll 1$, and assume the particle configuration $\bmx=(x_1,x_2,\dots, x_n)\in \bW_\theta^n$ satisfies $\bl\leq \varepsilon x_{n}\leq \varepsilon x_1\leq \br-\varepsilon$. We further assume
\begin{enumerate}
\item There exists an open complex neighborhood $\Lambda$ of $[\bl,\br]$, such that for $z=\varepsilon \xi$,
\begin{align}\label{e:asymp}
b(\xi)=\sfb(z)\quad \text{ and }\quad
\phi^\pm(\xi)={\varphi}^{\pm}(z), \quad z\in \Lambda.
\end{align}
where $\sfb(z)$ and $\varphi^{\pm}(z)$ are holomorphic functions on $\Lambda$. In addition, we require $\sfb(z)$ to be conformal (injective and biholomorphic) and $\sfb(\bar{z})=\overline{\sfb(z)}$.
\item The functions $\sfb(z)$, $[\partial_z \sfb(z)]^{-1}$, and $\varphi^\pm(z)$  are uniformly bounded, namely there exists a universal\footnote{By ``universal'' we mean not depending on $n$, $\eps$, or $(x_1,\dots,x_n)$.} constant $C>0$ such that
\begin{align}\label{e:derbba}
 | \sfb(z)|\leq C,
 \quad
 |\del_z \sfb(z)|\geq 1/C,\
\quad
|\varphi^\pm(z)|\leq C,\quad z\in  \Lambda.
\end{align}
\end{enumerate}
\end{assumption}

\begin{remark}
In the rest of this section, we view $\sfN, \bl, \br, C$ in Assumption \ref{a:asymp} as fixed constants. On the other hand, in our applications the functions $\sfb, \varphi^\pm$ might depend on $\varepsilon$.
\end{remark}
\begin{remark}
 Injectivity of $\sfb$ implies that it is a (real) monotone function on $[\bl, \br]$. Without loss of generality, throughout this section we silently assume that $\sfb(x)$ is an increasing function of $x\in[\bl,\br]$.
\end{remark}
Throughout this section we analyze \eqref{e:m1ccopy} in two separate situations:
\begin{enumerate}
 \item {\bf Probabilistic case.} The weight functions $\varphi^\pm$ in \eqref{e:asymp} are real analytic and positive on $(\bl, \br)$:
   \begin{equation}
    \label{e:probabilistic_case}
    \varphi^\pm(\bar z)=\overline{\varphi^\pm(z)}, \quad z\in \Lambda; \quad \text{ and } \quad \varphi^\pm(x)>0, \quad x\in(\bl,\br).
   \end{equation}
    In this situation \eqref{e:m1ccopy} is a probability measure.
 \item {\bf Complex case.}   The weight functions $\varphi^\pm$ might be complex, but we additionally assume
 \begin{align}\label{e:totalabsmass}
\sum_\bme |a_\bme|\leq C,
\end{align}
for a universal constant $C$.
\end{enumerate}
While eventually we are only interested in the former case, the latter appears in the intermediate steps of our analysis: in the proofs we deform the weights $\phi^\pm$ in \eqref{e:m1ccopy} so that \eqref{e:m1ccopy} becomes a complex-valued measure.

\bigskip

Our last assumption for the asymptotic analysis of \eqref{e:m1ccopy} involves the following functions defined in terms of the empirical density \eqref{eq_rho_x_def} and the functions $\phi^\pm(z)$ of Assumption \ref{a:asymp}:
\begin{equation}
\label{e:B_function}
\cB(z)=\cG(z)\varphi^+(z)+\varphi^-(z), \qquad \qquad  \cG(z)=\exp\left[\theta\int_{\bl}^{\br}\frac{\sfb'(z)\rho(s;\bmx)}{\sfb(z)-\sfb(s)}\rd s\right].
\end{equation}
Note that $\cB(z)$ is a holomorphic function for $z\in \Lambda\setminus  [\bl,\br]$.
%We also denote
%\begin{align}\label{e:ladelta}
%\Lambda_\delta=\{z\in \Lambda: \dist(z, [\bl,\br])>\delta\}.
%\end{align}
\begin{assumption}\label{a:stable}
For any compact set $S\subset \bigl(\Lambda\setminus [\bl,\br]\bigr)$ there  exists a small universal constant $c=c(S)>0$ such that  for $z\in S$ we have $c<|\cB(z)|<c^{-1}$. Moreover, for any closed contour $\omega\subset  \bigl(\Lambda\setminus [\bl,\br]\bigr)$, we have
\begin{align}\label{e:aB0}
\frac{1}{2\pi \ri}\oint_{\omega}\frac{\del_z\cB(z)}{\cB(z)}\rd z=0,
\end{align}
which implies that there exists a well-defined single-valued branch of the function $\ln \cB(z)$ in $\Lambda\setminus [\bl,\br]$.
%\end{enumerate}
\end{assumption}
The last assumption plays the same role in our asymptotic analysis as the non-criticality condition in the analysis of the (discrete) $\beta$-ensembles, cf.\  \cite{MR3668648,MR3010191, borot-guionnet2,MR3351052,bourgade2021optimal}. Here is our central asymptotic result:

\begin{theorem}\label{t:loopstudy}
Consider transition probability \eqref{e:m1ccopy} with parameters satisfying Assumptions \ref{a:asymp} and \ref{a:stable} for all small enough $\eps$ and in the probabilistic case \eqref{e:probabilistic_case}. Then for any $z\in \Lambda\setminus [\bl,\br]$ we have as $\eps\to 0$:
 \begin{align}\label{e:dmg}
\notag\frac{1}{\varepsilon}\int_{\bl}^{\br} \frac{\sfb'(z)(\rho(s;\bmy)-\rho(s;\bmx))}{\sfb(z)-\sfb(s)}\rd s
&=\frac{1}{2\pi \ri\theta}\oint_{\cin}\frac{\ln \cB(w)\sfb'(w) \sfb'(z)\rd w}{(\sfb(w)-\sfb(z))^2}\\
&+
\frac{\varepsilon}{2\pi \ri\theta }\oint_{\cin}  \cE^{(2)}(w,z)\frac{\sfb'(w)\sfb'(z)\rd w}{(\sfb(w)-\sfb(z))^2}
\,
+\Delta\cM(z)
\,+\, \OO\left(\varepsilon^2 \right),
\end{align}
where the contour $\cin\subset \Lambda$ encloses $[\bl, \br]$, but not $z$, and
\begin{align*}
\cE^{(2)}(w,z)
&=\frac{\varphi^+(w) \cG(w)}{\cB(w)}\Biggl(\frac{\theta \sfb'(w)}{\sfb(z)-\sfb(w)}+\frac{\theta}{2\pi \ri}\oint_{\omega_-'}\frac{\ln\cB(u)\sfb'(w)\sfb'(u)\rd u}{(\sfb(u)-\sfb(w))^2}\\
&+
\frac{\theta^2}{2}\int_\bl^\br \frac{\sfb''(w)(\sfb(w)-\sfb(s))-(\sfb'(w))^2-\sfb'(w)\sfb'(s)}{(\sfb(w)-\sfb(s))^2} \rho(s;\bmx)\rd s\Biggr)-\frac{\del_{w} \cB(w)}{2\cB(w)},
\end{align*}
where the contour $\omega'_-\subset \Lambda$ encloses  $[\bl,\br]$, but not $w$.

Moreover, $\Delta\cM(z)$ are mean $0$ random variables such that $\{\varepsilon^{-1/2}\Delta \cM(z)\}_{z\in \Lambda\setminus[\bl,\br]}$ are asymptotically Gaussian with covariance given by
\begin{align}
\notag \bE\left[\frac{ \Delta \cM(z_1)}{\varepsilon^{1/2}}\cdot \frac{\Delta \cM(z_2)}{\varepsilon^{1/2} }\right]
=&\frac{1}{2\pi \ri \theta}\oint_{\cin} \frac{\cG(w)\varphi^+(w)}{\cB(w)}\frac{\sfb'(w)\sfb'(z_1)}{(\sfb(w)-\sfb(z_1))^2} \frac{\sfb'(w)\sfb'(z_2)}{(\sfb(w)-\sfb(z_2))^2} \rd w\\ &+o(1) \label{e:covT},
\end{align}
where the contour $\cin\subset \Lambda$ encloses $[\bl, \br]$, but not $z_1, z_2$. The higher order joint moments of $\{\varepsilon^{-1/2}\Delta \cM(z)\}_{z\in \Lambda\setminus[\bl,\br]}$ converge as $\eps\to 0$ to the Gaussian joint moments.

 The implicit constants in the $\OO(\eps^2)$, in $o(1)$, and in higher order joint moments are uniform in all the involved parameters and particle configurations $\bmx$ satisfying Assumption \ref{a:asymp} and \ref{a:stable}, as long as the constants $C$ and $c$ of Assumptions  \ref{a:asymp} and \ref{a:stable} are fixed and $z$ (or $z_1,z_2$) belongs to a compact subset of $\Lambda\setminus [\bl,\br]$.
\end{theorem}

\begin{remark} \label{Remark_uniformity}
In the rest of this paper, when we write a quantity $\cE(z;\bmx)$ is uniformly small, it refers to that for any compact subset $F\subset \Lambda\setminus [\bl,\br]$
\begin{align}
\lim_{\varepsilon\rightarrow 0} \sup_{\begin{smallmatrix} z \in F, \\ \text{$\bmx$ satisfying Assumption \ref{a:asymp} and \ref{a:stable}}\end{smallmatrix}}\cE(z;\bmx)\rightarrow 0,
\end{align}
as long as the constants $C$ and $c$ of Assumptions  \ref{a:asymp} and \ref{a:stable} are fixed.
\end{remark}

\begin{remark}
 When we apply Theorem \ref{t:loopstudy} in our study of random lozenge tilings, the exact expression for $\cE^{(2)}(w,z)$ is never used. However, our method of the proof of Theorem \ref{t:loopstudy} provides this expression and, hence, we decided to include it for potential future applications.
\end{remark}
\begin{remark}
 Dividing \eqref{e:dmg} by $b'(z)$ and changing the variables $u=b(z)$, we can equivalently rewrite it as
  \begin{align}\label{e:dmg_2}
\notag\frac{1}{\varepsilon}\int_{\bl}^{\br} \frac{\rho(s;\bmy)-\rho(s;\bmx)}{u-\sfb(s)}\rd s
=\frac{1}{2\pi \ri\theta }\oint_{\cin}\frac{\ln \cB(w)\sfb'(w) \rd w}{(\sfb(w)-u)^2}\,&+\,
\frac{\varepsilon}{2\pi \ri\theta}\oint_{\cin}  \cE^{(2)}(w,z)\frac{\sfb'(w)\rd w}{(\sfb(w)-u)^2}
\,\\
&+\Delta\cM'(u)
\,+\, \OO\left(\varepsilon^2 \right),
\end{align}
with an appropriately covariance for the Gaussian process $\cM'(u)$. One advantage of the form \eqref{e:dmg_2} is that we can take $u$ in it  to be an arbitrary complex number outside $[\sfb(\bl),\sfb(\br)]$;  see also Remarks \ref{r:cAud} and \ref{r:cAud_second_order}.
\end{remark}

In general words, Theorem \ref{t:loopstudy} shows that, conditional on $\bmx$, the empirical density of the particle configuration $\bmy=\bmx+\bme$ (distributed with weights \eqref{e:m1ccopy}) consists of two parts: a computable deterministic part and a random fluctuation, which is asymptotically a Gaussian field. However, instead of directly studying the difference of the empirical densities, we study the difference of their centered and modified Stieltjes transforms:
\begin{align}\label{e:mt00}
\int_\bl^\br \frac{\sfb'(z) (\rho(s;\bmy)-\rho(s;\bmx))}{\sfb(z)-\sfb(s)}\rd s.
\end{align}
We remark that \eqref{e:mt00} is different from the standard definition of Stieltjes transform. However, for any function $g(z)$ which is analytic in a neighborhood of $[\bl,\br]$, we have
\begin{align}
\label{e:linear_statistic_contour}
\frac{1}{2\pi\ri}\oint_\omega g(z) \left[\int_\bl^\br \frac{\sfb'(z)\rho(s;\bmy)}{\sfb(z)-\sfb(s)}\rd s\right] \rd z
=\int_\bl^\br g(s)\rho(s;\bmy)\rd s,
\end{align}
where $\omega$ is a contour enclosing $[\bl,\br]$. Hence, Theorem \ref{t:loopstudy} can be used to deduce the asymptotic behavior of the pairings of $\rho(s;\bmy)$ with arbitrary analytic test functions.

\bigskip

The rest of this section is devoted to the proof of Theorem \ref{t:loopstudy}. For a reader familiar with the arguments in \cite{MR3668648}, some parts of the proof might look similar, so let us emphasize two important differences. First, as a starting point in \cite{MR3668648} a concentration bound on the empirical measure was used; this bound was produced by a separate method unrelated to the loop equations. In our present approach such a bound is no longer needed and we rely on the dynamical loop equation and nothing else. Second, an asymptotic expansion of the discrete loop equations in \cite{MR3668648} was directly leading to the expectation of the centered Stieltjes transform, which could then be used to identify the first term in an analogue of \eqref{e:dmg}. This is different for our dynamic loop equations: instead the expectation of exponentiated Stieltjes transform appears. There is no direct connection between $\bE \xi$ and $\bE \exp(\xi)$ for a random variable $\xi$, hence, this new feature leads to a significant difference in the argument.

\subsection{Plan of the proof for Theorem \ref{t:loopstudy}}\label{s:heur}
In this subsection we present the general scheme for the proof of Theorem \ref{t:loopstudy} omitting technical details which are going to be reconstructed in the following subsections. We rely on the dynamical loop equation of Theorem \ref{t:loopeq}, which states that the expression
\begin{align}\label{e:C_def}\begin{split}
\cC(z):=\bE\left[\varphi^+(z)\prod_{i=1}^n\frac{\sfb(z+\varepsilon\theta)-\sfb(\varepsilon x_i+ \varepsilon\theta e_i)}{\sfb(z)-\sfb(\varepsilon x_i)}+\varphi^-(z)\prod_{i=1}^n\frac{\sfb(z)-\sfb(\varepsilon x_i+\varepsilon\theta e_i)}{\sfb(z)-\sfb(\varepsilon x_i)}\right]
\end{split}\end{align}
is a holomorphic function of $z\in \Lambda$.

\smallskip

{\bf Step 1.} We use \eqref{e:C_def} to analyze the following quantity (which directly appears in its second term):
\begin{equation}\label{e:A_def}
\cA(z):=\bE\left[\prod_{i=1}^n\frac{\sfb(z)-\sfb(\varepsilon x_i+ \varepsilon\theta  e_i)}{\sfb(z)-\sfb(\varepsilon x_i)}\right].
%\varepsilon(\xi)&=\sum_{\bme}\frac{V(\bmx+\bme)}{V(\bmx)}\prod_{i=1}^n  \phi^+(x_i)^{e_i}\phi^-(x_i)^{1-e_i} \phi^+(\xi)
%\left(\prod_{i=1}^n \frac{\xi+1-x_i-e_i}{\xi-x_i}-\prod_{i=1}^n \frac{\xi-x_i-e_i}{\xi-x_i-1}\right)\\
\end{equation}
We remark that under Assumption \ref{a:asymp}, $\cA(z)$ is analytic on $\Lambda\setminus [\bl,\br]$.
This quantity encodes the information about the measure $\rho(s;\bmy)-\rho(s;\bmx)$, which we want to understand. Our approach is to use holomorphicity of \eqref{e:C_def} to express $\cA(z)$ explicitly in terms of  contour integrals of $\cB(z)$ defined in \eqref{e:B_function}. For that we expand $\cC(z)$ as
\begin{align}\label{e:cCdec}
\cC(z)=\cA(z)\cB(z)+\OO(\varepsilon).
\end{align}

Note that in \eqref{e:cCdec} we know explicitly $\cB(z)$, and in addition we know that $\cC(z)$ is holomorphic (but we do not have any formulas for $\cC(z)$). A priory, this does not give enough information to extract anything about $\cA(z)$, but the situation changes when we impose Assumption \ref{a:stable}. More precisely, adopting Assumption \ref{a:stable}, we can decompose $\cB(z)$ into two parts
\begin{align}
\label{e:B_decomp_1}
\cB(z)=e^{h_+(\sfb(z))} e^{-h_-(\sfb(z))},
\end{align}
where $h_+(\sfb(z))$ is analytic in a neighborhood of $[\bl,\br]$, and $h_-(u)$ is analytic in a neighborhood of $u=\infty$ with $h_-(u)=\OO(1/u)$ as $u$ approaches $\infty$. \eqref{e:B_decomp_1} is a version of the Wiener-Hopf decomposition, adapted to our setting. In more detail, computing the difference of the integrals as the residue at $w=z$, we have
\begin{align}\label{e:B_decomp_2}
\ln \cB(z)=\frac{1}{2\pi\ri}\oint_{\omega_+}\frac{\ln \cB(w)\sfb'(w)\rd w}{\sfb(w)-\sfb(z)}
-\frac{1}{2\pi\ri}\oint_{\omega_-}\frac{\ln \cB(w)\sfb'(w)\rd w}{\sfb(w)-\sfb(z)},
\end{align}
where a positively oriented contour $\omega_+\subset \Lambda$ encloses $[\bl,\br]$ and $z$; a positively oriented contour $\omega_-\subset \Lambda$ encloses $[\bl,\br]$ but not $z$. We can rewrite \eqref{e:B_decomp_2} as an expansion \eqref{e:B_decomp_1} with
\begin{equation}
\label{e:B_decomp_3}
 h_+(u):=\frac{1}{2\pi\ri}\oint_{\omega_+}\frac{\ln \cB(w)\sfb'(w)\rd w}{\sfb(w)-u}, \qquad h_-(u):=\frac{1}{2\pi\ri}\oint_{\omega_-}\frac{\ln \cB(w)\sfb'(w)\rd w}{\sfb(w)-u}.
\end{equation}
Note that the particular choices of $\omega_\pm$ are not important in \eqref{e:B_decomp_2}, as long as $z$ is inside $\omega_+$ and outside $\omega_-$, while $[\bl,\br]$ is inside both contours. However, in order for the definition \eqref{e:B_decomp_3} to make sense, we need to fix these contours: $\omega_-$ should be a tight loop around $[\bl, \br]$, while $\omega_+$ should be a larger contour, enclosing $\omega_-$ and staying inside $\Lambda$, cf.\ Figure \ref{f:B_contours}. We treat \eqref{e:B_decomp_3} as the definition of holomorphic $h_+(u)$ for $u$ inside $b^{-1}(\omega_+)$ and as the definition of holomorphic $h_-(u)$ for $u$ outside $b^{-1}(\omega_-)$. The identity \eqref{e:B_decomp_1} is valid for $z$ in the annulus enclosed by $\omega_\pm$.

\begin{figure}[t]
\begin{center}
 \includegraphics[width=0.5\linewidth]{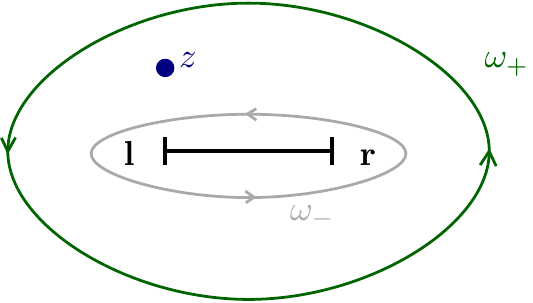}
 \caption{Contours $\omega_\pm$ and $z$ in the annulus between them.
 \label{f:B_contours}}
 \end{center}
 \end{figure}

\medskip

To get the first order asymptotics of $\cA(z)$ we multiply both sides of \eqref{e:cCdec} by $e^{-h_+(\sfb(z))}$, and notice that the left-hand side remains analytic in a complex neighborhood of $[\bl,\br]$. Then we can recover $\cA(z)$ through a contour integral: using the above contours $\omega_\pm$ we have
\begin{align}\begin{split}\label{e:dec}
0&=\frac{1}{2\pi\ri}\oint_{\omega_-}\frac{\cC(w)e^{-h_+(\sfb(w))} \sfb'(w)\rd w}{\sfb(w)-\sfb(z)}=\frac{1}{2\pi\ri}\oint_{\omega_-}\frac{(\cA(w)e^{-h_-(\sfb(w))}+\OO(\varepsilon)) \sfb'(w)\rd w}{\sfb(w)-\sfb(z)}\\
&=-\cA(z)e^{-h_-(\sfb(z))}+\frac{1}{2\pi\ri}\oint_{\omega_+}\frac{\cA(w)e^{-h_-(\sfb(w))} \sfb'(w)\rd w}{\sfb(w)-\sfb(z)} +\OO(\varepsilon).
\end{split}\end{align}
We claim that the last contour integral over $\omega_+$ is equal to $1$. Indeed, to see that we change the variables $u=b(w)$, resulting in
$$
 \frac{1}{2\pi\ri}\oint_{b(\omega_+)}\frac{e^{-h_-(u)} }{u-\sfb(z)} \bE\left[\prod_{i=1}^n\frac{u-\sfb(\varepsilon x_i+ \varepsilon\theta  e_i)}{u-\sfb(\varepsilon x_i)}\right]\,  \rd u.
$$
The last integrand has no singularities outside $b(\omega_+)$ and, therefore, the integral is equal to the residue at $u=\infty$. This residue is $1$, because $h_-(u)=1+\OO(\tfrac{1}{u})$ and the expression under expectation is $1+\OO(\tfrac{1}{u})$ as $u\to\infty$. Hence, \eqref{e:dec} gives the first order asymptotics of $\cA(z)$:
\begin{align}
\label{e:A_first_order}
\ln \cA(z)=h_-(\sfb(z))+\OO(\varepsilon).
\end{align}

The formula \eqref{e:A_first_order} is developed in further detail in Proposition \ref{p:firstod0}. With additional efforts (but using similar ideas) the higher order asymptotics of $\ln \cA(z)$ can be obtained by further expanding the $\OO(\varepsilon)$ error term in \eqref{e:cCdec}; we do this for the first two orders in Proposition \ref{p:second}.

\smallskip

{\bf Step 2.} The next step is to extract the probabilistic information --- the first two (deterministic) terms in the right-hand side of \eqref{e:dmg} --- from the asymptotic expansion of $\cA(z)$. Observe that $\cA(z)$ can be transformed as follows:
\begin{align}
\notag \bE\left[\prod_{i=1}^n\frac{\sfb(z)-\sfb(\varepsilon x_i+ \varepsilon\theta  e_i)}{\sfb(z)-\sfb(\varepsilon x_i)}\right]&=
\bE\left[\prod_{i=1}^n\left(1+\frac{\sfb(\varepsilon x_i)-\sfb(\varepsilon x_i+ \varepsilon\theta  e_i)}{\sfb(z)-\sfb(\varepsilon x_i)}\right)\right]\\&=\bE\left[ \exp\left(-\eps \theta \sum_{i=1}^n \frac{e_i \sfb'(\eps x_i)}{\sfb(z)-\sfb(\varepsilon x_i)}  +\oo(1)\right)\right]. \label{e:A_exp}
%\varepsilon(\xi)&=\sum_{\bme}\frac{V(\bmx+\bme)}{V(\bmx)}\prod_{i=1}^n  \phi^+(x_i)^{e_i}\phi^-(x_i)^{1-e_i} \phi^+(\xi)
%\left(\prod_{i=1}^n \frac{\xi+1-x_i-e_i}{\xi-x_i}-\prod_{i=1}^n \frac{\xi-x_i-e_i}{\xi-x_i-1}\right)\\
\end{align}
On the other hand, in order to get the deterministic terms in the right-hand side of  \eqref{e:dmg}, we need to analyze
\begin{align}
\notag \bE\left[ \frac{1}{\varepsilon}\int_{\bl}^{\br} \frac{\sfb'(z)(\rho(s;\bmy)-\rho(s;\bmx))}{\sfb(z)-\sfb(s)}\rd s\right]&=  \bE\left[\sum_{i=1}^n \left(\frac{\sfb'(z) }{\sfb(z)-\sfb(\varepsilon x_i+\eps e_i)}- \frac{\sfb'(z) }{\sfb(z)-\sfb(\varepsilon x_i)}\right)+\oo(1)\right]\\
 &=  \bE\left[-\eps  \sum_{i=1}^n \frac{ e_i \sfb'(\eps x_i)\sfb'(z)}{(\sfb(z)-\sfb(\varepsilon x_i))^2}+\oo(1)\right] \label{e:empirical_increment} .
\end{align}
There are two differences between the expectation in \eqref{e:A_exp} which we already know and the expectation in \eqref{e:empirical_increment} which we need to find. First, the sum over $i$ in the latter is (up to constant factors) the $z$--derivative of the sum over $i$ in the former. Second, the sum is being exponentiated in \eqref{e:A_exp}. Generally speaking, there is no direct connection between $\bE \xi$ and $\bE \exp(\xi)$ for a random variable $\xi$, which makes the second difference a non-trivial difficulty.

We overcome this technical obstacle in Proposition \ref{p:second}. For that we notice that the logarithmic derivative $\partial_z \ln \cA(z)$ can be interpreted as the expectation of the form \eqref{e:empirical_increment}, but with respect to a deformed measure. Further, we offset the deformation  by inserting  certain explicit additional complex factors in $\phi^\pm$ in \eqref{e:m1ccopy}. Because the analysis of Step 1 did not rely on positivity of the weights, it continues to hold for the resulting new complex measure, eventually leading to the evaluation of the first two terms of \eqref{e:dmg}.

\smallskip

{\bf Step 3.} The last step is to prove asymptotic Gaussianity of the random variable $\Delta\cM(z)$ in \eqref{e:dmg} and compute its covariance. This is done in Proposition \ref{p:cov}. In more detail, we notice that the computation of the Laplace transform of $\Delta\cM(z)$ is equivalent to the computation of the deterministic terms in \eqref{e:dmg} for a slightly deformed measure. The developments of Steps 1 and 2 continue to be valid for this measure. Hence, we find the leading asymptotic behavior of the Laplace transforms of $\Delta \cM(z)$, which matches the Laplace transforms of Gaussian random variables with covariance given by \eqref{e:covT}.

\bigskip

The rest of this section adds necessary details to the three steps argument sketched above.

\subsection{Properties of $\cA(z)$ and $\cC(z)$}\label{s:simplfyA}

We start our detailed proof of Theorem \ref{t:loopstudy} by collecting estimates and asymptotic expansions for $\cA(z)$ and $\cC(z)$ from \eqref{e:A_def} and \eqref{e:C_def}, respectively. We start from  $\cA(z)$.

\begin{proposition}\label{p:simplfyA}
We work under Assumption \ref{a:asymp}, either in probabilistic case with \eqref{e:probabilistic_case} or in the complex case with \eqref{e:totalabsmass}. Fix any  compact subset $S\subset \Lambda\setminus [\bl,\br]$. Then
\begin{align}\label{e:Aupbb}
|\cA(z)|\leq C, \qquad z\in S,
\end{align}
and
\begin{align}
\notag&\phantom{{}={}} \del_z \cA(z)
 =\bE\Biggl[\prod_{i=1}^n  \frac{\sfb(z)-\sfb(\varepsilon x_j+\varepsilon \theta e_i) }{\sfb(z)-\sfb(\varepsilon x_j)}\notag \\
 &
  \times  \left(\frac{\theta}{\varepsilon}\int_\bl^\br \frac{\sfb'(z)(\rho(s;\bmy)-\rho(s;\bmx))}{\sfb(z)-\sfb(s)} \rd s
-\frac{\theta}{2}\int_\bl^\br \frac{ \sfb'(z)\sfb'(s)(\rho(s;\bmy)-\rho(s;\bmx))}{(\sfb(z)-\sfb(s))^2} \rd s +\OO(\varepsilon^2)\right)\Biggr]\notag \\
&= \bE\left[\prod_{i=1}^n  \frac{\sfb(z)-\sfb(\varepsilon x_j+\varepsilon \theta e_i) }{\sfb(z)-\sfb(\varepsilon x_j)} \cdot \left(\frac{\theta}{\varepsilon}\int_\bl^\br \frac{\sfb'(z)(\rho(s;\bmy)-\rho(s;\bmx))\rd s}{\sfb(z)-\sfb(s)}
+\OO(\varepsilon)\right)\right], \label{e:dzAexp}
%&=\tilde \bE\left[\frac{1}{\varepsilon}\int \frac{\sfb'(z)(\rho(x;\bmy)-\rho(x;\bmx))\rd x}{z-\sfb(x)}
%+\int \frac{\theta \sfb'(w)\sfb'(x)(\rho(x;\bmy)-\rho(x;\bmx))\rd x}{2(z-\sfb(x))^2} +\OO(\varepsilon^2)\right]
\end{align}
where the constant $C$ and implicit constants in $\OO(\cdot)$ depend only on $S$.
\end{proposition}

\begin{proof}
Whenever $\dist(z,[\bl,\br])>\delta$, we have $\sfb(z)-\sfb(\varepsilon x_i)\geq \delta/C$ from \eqref{e:derbba}. Hence,
\begin{align}\label{e:uppbb}
\left|\prod_{i=1}^n \frac{\sfb(z)-\sfb(\varepsilon x_i+ \varepsilon \theta e_i)}{\sfb(z)-\sfb(\varepsilon x_i)}\right|
\leq \prod_{i=1}^n \left(1+\frac{|\sfb(\varepsilon x_i+\varepsilon \theta e_i)-\sfb(\varepsilon x_i)|}{\delta/C}\right)\leq \left(1+\frac{\OO(\varepsilon)}{\delta/C}\right)^n\leq C_\delta,
\end{align}
where we used $n=\OO(\eps^{-1})$. Then \eqref{e:Aupbb} follows from \eqref{e:uppbb} and \eqref{e:totalabsmass}.

The derivative of $\cA(z)$ is given by
\begin{align}\begin{split}\label{e:diffmean0}
\del_z \cA(z)&=\del_z\bE\left[\prod_{i=1}^n\frac{\sfb(z)-\sfb(\varepsilon x_i+\varepsilon\theta e_i )}{\sfb(z)-\sfb(\varepsilon x_i)}\right]=\del_z\bE\left[\prod_{i=1}^n\left(\frac{\sfb(z)-\sfb(\varepsilon x_i+\varepsilon\theta )}{\sfb(z)-\sfb(\varepsilon x_i)}\right)^{e_i}\right]\\
&=\bE\left[\prod_{i=1}^n\left(\frac{\sfb(z)-\sfb(\varepsilon x_i+\varepsilon \theta )}{\sfb(z)-\sfb(\varepsilon x_i)}\right)^{e_i}
\cdot \sum_{i=1}^n e_i\left(\frac{\sfb'(z)}{\sfb(z)-\sfb(\varepsilon x_i+\varepsilon \theta)}-\frac{\sfb'(z)}{\sfb(z)-\sfb(\varepsilon x_i)}\right)\right].
%\sum_ie_i\frac{\sfb'(z)(\sfb(\varepsilon x_i+\varepsilon\theta)-\sfb(\varepsilon x_i))}{(\sfb(z)-\sfb(\varepsilon x_i))(\sfb(z)-\sfb(\varepsilon x_i+\varepsilon\theta))}
\end{split}\end{align}
We use the following identity
\begin{align}\label{e:expaa}
f(x+\varepsilon \theta)-f(x)=\frac{1}{\varepsilon}\int_x^{x+\varepsilon\theta}(f(y+\varepsilon)-f(y)) \rd y
-\frac{1}{2}\int_x^{x+\varepsilon\theta}(f'(y+\varepsilon)-f'(y)) \rd y +\OO(\varepsilon^3),
\end{align}
where the implicit constant in $\OO(\eps^3)$ is uniform over thrice differentiable functions $f(x)$ satisfying a uniform bound on the third derivative of the form $\sup_{y\in [x,x+2\varepsilon]} |f^{(3)}(y)|\leq C$.
Using \eqref{e:expaa} with \mbox{$f(x)=\sfb'(z)/(\sfb(z)-\sfb(x))$}, we can rewrite
\begin{align}\begin{split}%\label{e:diffmean}
&\sum_{i=1}^n e_i\left(\frac{\sfb'(z)}{\sfb(z)-\sfb(\varepsilon x_i+\varepsilon \theta)}-\frac{\sfb'(z)}{\sfb(z)-\sfb(\varepsilon x_i)}\right)\\
&=\frac{\theta}{\varepsilon}\int_\bl^\br \frac{\sfb'(z)(\rho(s;\bmy)-\rho(s;\bmx))}{\sfb(z)-\sfb(s)}\rd s
-\frac{\theta}{2}\int_\bl^\br \frac{ \sfb'(z)\sfb'(s)(\rho(s;\bmy)-\rho(s;\bmx))}{(\sfb(z)-\sfb(s))^2}\rd s +\OO(\varepsilon^2)\\
&=\frac{\theta}{\varepsilon}\int_\bl^\br \frac{\sfb'(z)(\rho(s;\bmy)-\rho(s;\bmx))}{\sfb(z)-\sfb(s)}\rd s
+\OO(\varepsilon),
\end{split}\end{align}
where we bound the middle term in the second line by $\OO(\varepsilon)$.
\qedhere
\end{proof}

Our next task is to rewrite $\cC(z)$ in \eqref{e:C_def} as a product of $\cA(z)$  and $\cB(z)$ with a small error; in other words, we would like to give a rigorous justification to \eqref{e:cCdec}.
\begin{proposition}\label{p:simplfyC}
Under Assumption \ref{a:asymp} (either in probabilistic case with \eqref{e:probabilistic_case} or in the complex case with \eqref{e:totalabsmass}), $\cC(z)$ satisfies
\begin{align}\label{e:ABCprecise}
\cC(z)=\cA(z)\cB(z)\left(1+\varepsilon \cE^{(1)}(z)\right), \qquad z\in \Lambda\setminus [\bl,\br],
\end{align}
where $\cE^{(1)}(z)$ is given by (here $\cG(z)=\exp\left[\theta\int_{\bl}^{\br}\frac{\sfb'(z)\rho(s;\bmx)}{\sfb(z)-\sfb(s)}\rd s\right]$, as in \eqref{e:B_function}):
\begin{align}\begin{split}\label{e:defcE}
\cE^{(1)}(z)=\frac{1}{\cA(z)\cB(z)}\bE\left[\prod_{i=1}^n \frac{\sfb(z)-\sfb(\varepsilon x_i+\varepsilon\theta e_i)}{\sfb(z)-\sfb(\varepsilon x_i)}\left( \varphi^+(z)\cG(z) \cE^{(0)}(z)+\OO(\varepsilon)\right)\right],
\end{split}\end{align}
\begin{align}\begin{split}\label{e:cE00}
\cE^{(0)}(z)&=\frac{\theta^2}{\eps} \int_\bl^\br \frac{\sfb'(z)(\rho(s;\bmy)-\rho(s;\bmx))}{\sfb(z)-\sfb(s)}\rd s\\
&+\frac{\theta^2}{2}\int_\bl^\br \frac{\sfb''(z)(\sfb(z)-\sfb(s))-(\sfb'(z))^2-\sfb'(z)\sfb'(s)}{(\sfb(z)-\sfb(s))^2} \rho(s;\bmx)\rd s.
\end{split}\end{align}
If we assume that $z$ varies in a compact subset of $\Lambda\subset [\bl,\br]$, then the implicit constant in the $\OO(\eps)$ is uniform in $z$ and  $|\cE^{(0)}(z)|$, $| \cA(z)\cB(z)\cE^{(1)}(z)|$ are uniformly bounded.
\end{proposition}

\begin{proof}
We rewrite $\cC(z)$ as
\begin{align}\label{e:defcCz}
\cC(z)&=\bE\left[\prod_{i=1}^n\frac{\sfb(z)-b(\varepsilon x_i+\varepsilon\theta e_i)}{\sfb(z)-\sfb(\varepsilon x_i)}\left(\varphi^+(z)\prod_{i=1}^n\frac{\sfb(z+\varepsilon\theta)-\sfb(\varepsilon x_i+\varepsilon \theta e_i)}{\sfb(z)-\sfb(\varepsilon x_i+\varepsilon\theta e_i)}+\varphi^-(z)\right)\right].
\end{align}
We would like to approximate the last product with $\cB(z)$. For that we Taylor expand:
\begin{align}
&\notag\phantom{{}={}}\prod_{i=1}^n\frac{\sfb(z+\varepsilon\theta)-\sfb(\varepsilon x_i+\varepsilon \theta e_i)}{\sfb(z)-\sfb(\varepsilon x_i+\varepsilon\theta e_i)}
=\prod_{i=1}^n\left(1+\frac{\sfb(z+\varepsilon\theta)-\sfb(z)}{\sfb(z)-\sfb(\varepsilon x_i+\varepsilon\theta e_i)}\right)\\
&\notag=\exp\left[\sum_{i=1}^n \log\left(1+\frac{\sfb(z+\varepsilon\theta)-\sfb(z)}{\sfb(z)-\sfb(\varepsilon x_i+\varepsilon\theta e_i)}\right)\right]\\
&\notag=\exp\left[\sum_{i=1}^n \log\left(1+\frac{\varepsilon\theta\sfb'(z)+\tfrac{1}{2}(\varepsilon\theta)^2\sfb''(z)+\OO(\varepsilon^3)}{\sfb(z)-\sfb(\varepsilon x_i+\varepsilon\theta e_i)}\right)\right]\\
&=\exp\left[\sum_{i=1}^n \left(\frac{\varepsilon\theta\sfb'(z)}{\sfb(z)-\sfb(\varepsilon x_i+\varepsilon\theta e_i)}+\frac{1}{2}\frac{(\varepsilon\theta)^2\sfb''(z)}{\sfb(z)-\sfb(\varepsilon x_i)}-\frac{1}{2}\frac{(\varepsilon \theta \sfb'(z))^2}{(\sfb(z)-\sfb(\varepsilon x_i))^2}+\OO(\varepsilon^3)\right)\right].\label{e:tayl1}
\end{align}
For any $f$ with uniformly bounded second derivative, we have the following identity
\begin{align}\label{e:expaa2}
f(x+\varepsilon \theta e_i)=\int_x^{x+\varepsilon\theta}\left(\frac{f(y)}{\varepsilon \theta} -\frac{f'(y)}{2}\right)\rd y
+\frac{e_i}{\varepsilon}\int_x^{x+\varepsilon\theta}\bigl(f(y+\varepsilon)-f(y)\bigr) \rd y +\OO(\varepsilon^2).
\end{align}
Using \eqref{e:expaa2} with $f(x)=\sfb'(z)/(\sfb(z)-\sfb(x))$, we rewrite the first term in the right-hand side of \eqref{e:tayl1} as
\begin{align}\begin{split}\label{e:tayl2}
\sum_{i=1}^n \frac{\varepsilon\theta\sfb'(z)}{\sfb(z)-\sfb(\varepsilon x_i+\varepsilon\theta e_i)}
&=\theta\int_\bl^\br\frac{\sfb'(z)\rho(s;\bmx)}{\sfb(z)-\sfb(s)}-\frac{\varepsilon \theta^2 }{2}\int_\bl^\br\frac{\sfb'(z)\sfb'(s)\rho(s;\bmx)}{(\sfb(z)-\sfb(s))^2}\rd s\\
&+\theta^2 \int_\bl^\br \frac{\sfb'(z)(\rho(s;\bmy)-\rho(s;\bmx))}{\sfb(z)-\sfb(s)}\rd s+\OO(\varepsilon^2).
\end{split}\end{align}
We also transform the last two terms in the right-hand side of \eqref{e:tayl1} into integrals of $\rho(s;\bmx)$:
\begin{align}\begin{split}\label{e:tayl5}
&\phantom{{}={}}\frac{1}{2} \sum_{i=1}^n\left[ \frac{(\varepsilon\theta)^2\sfb''(z)}{\sfb(z)-\sfb(\varepsilon x_i)}-\frac{(\varepsilon \theta \sfb'(z))^2}{(\sfb(z)-\sfb(\varepsilon x_i))^2}\right]\\
&=\frac{\varepsilon}{2}\int_\bl^\br \left(\frac{\theta^2\sfb''(z)}{\sfb(z)-\sfb(s)}-\frac{(\theta\sfb'(z))^2}{(\sfb(z)-\sfb(s))^2}\right)\rho(s;\bmx)\rd s+\OO(\varepsilon^2).
\end{split}\end{align}

By plugging \eqref{e:tayl2}, \eqref{e:tayl5} into  \eqref{e:tayl1}, we get
\begin{align}\label{e:tayl6}
\prod_{i=1}^n\frac{\sfb(z+\varepsilon\theta)-\sfb(\varepsilon x_i+\varepsilon \theta e_i)}{\sfb(z)-\sfb(\varepsilon x_i+\varepsilon\theta e_i)}
=\cG(z)\exp\left[\varepsilon \cE^{(0)}(z)+\OO(\varepsilon^2)\right],
\end{align}
where $\cG(z)=\exp\left[\theta\int_{\bl}^{\br}\frac{\sfb'(z)\rho(s;\bmx)}{\sfb(z)-\sfb(s)}\rd s\right]$, corresponding to the first term on the right-hand side of \eqref{e:tayl2}, and $\cE^{(0)}(z)$ is
\begin{align*}
\theta^2 \int_\bl^\br \frac{\sfb'(z)(\rho(s;\bmy)-\rho(s;\bmx))}{\varepsilon(\sfb(z)-\sfb(s))}\rd s+\frac{\theta^2}{2}\int_\bl^\br \frac{\sfb''(z)(\sfb(z)-\sfb(s))-(\sfb'(z))^2-\sfb'(z)\sfb'(s)}{(\sfb(z)-\sfb(s))^2} \rho(s;\bmx)\rd s.
\end{align*}

By plugging  \eqref{e:tayl6} into \eqref{e:defcCz}, we transform $\cC(z)$ as
\begin{align*}\begin{split}
\cC(z)&=\bE\left[\prod_{i=1}^n\frac{\sfb(z)-b(\varepsilon x_i+\varepsilon\theta e_i)}{\sfb(z)-\sfb(\varepsilon x_i)}\left(\varphi^+(z)\prod_{i=1}^n\frac{\sfb(z+\varepsilon\theta)-\sfb(\varepsilon x_i+\varepsilon \theta e_i)}{\sfb(z)-\sfb(\varepsilon x_i+\varepsilon\theta e_i)}+\varphi^-(z)\right)\right].\\
&=\bE\left[\prod_{i=1}^n \frac{\sfb(z)-b(\varepsilon x_i+\varepsilon\theta e_i)}{\sfb(z)-\sfb(\varepsilon x_i)}\left(\varphi^+(z)\cG(z)e^{\varepsilon \cE^{(0)}(z)}+\varphi^-(z)+\OO(\varepsilon^2)\right)\right]\\
&=\bE\left[\prod_{i=1}^n \frac{\sfb(z)-b(\varepsilon x_i+\varepsilon\theta e_i)}{\sfb(z)-\sfb(\varepsilon x_i)}\left(\cB(z)+\varepsilon \varphi^+(z)\cG(z) \cE^{(0)}(z)+\OO(\varepsilon^2)\right)\right],
\end{split}\end{align*}
which leads to \eqref{e:defcE}. The uniformity bounds for $\OO(\eps^2)$, $|\cE^{(0)}(z)|$, and $| \cA(z)\cB(z)\cE^{(1)}(z)|$ routinely follow from the definitions and we omit their justifications.
\end{proof}

\subsection{Asymptotic expansion of $\cA(z)$}\label{s:first}
In this section we derive the second order asymptotic expansion of the quantity $\cA(z)$ as defined in \eqref{e:A_def} in terms of contour integrals of $\cB(z)$.
\begin{proposition}\label{p:firstod0}
Under Assumptions \ref{a:asymp} and \ref{a:stable} (either in probabilistic case with \eqref{e:probabilistic_case} or in the complex case with \eqref{e:totalabsmass}), there is a branch of $\ln \cA(z)$ such that
\begin{align}\label{e:yao10}
\ln \cA(z)
&=\frac{1}{2\pi \ri}\oint_{\cin}\frac{\ln \cB(w)\sfb'(w)}{\sfb(w)-\sfb(z)}\rd w
+\OO\left(\varepsilon \right), \qquad z\in \Lambda\setminus [\bl,\br],
\end{align}
where $\cin\subset \Lambda$ is an arbitrary positively oriented simple contour enclosing  $[\bl,\br]$, but not $z$.  The error term $\OO(\eps)$ can be chosen to be uniform over $z$ in a compact subset of $\Lambda\setminus [\bl, \br]$.
\end{proposition}
\begin{proof}  We start by fixing two positively oriented contours $\omega_\pm\subset \Lambda$, such that they both enclose $[\bl, \br]$, are inside $\Lambda$, and $\omega_-$ is inside $\omega_+$. Throughout the proof, we deal with $z$ inside the annulus bounded by these contours,  as in Figure \ref{f:B_contours}. Thanks to Assumption \ref{a:stable}, $\ln \cB(z)$ is well defined on $\Lambda$, and we can introduce two functions:
\begin{equation}
\label{e:B_decomp_4}
 h_+(u):=\frac{1}{2\pi\ri}\oint_{\omega_+}\frac{\ln \cB(w)\sfb'(w)}{\sfb(w)-u} \rd w, \qquad \qquad h_-(u):=\frac{1}{2\pi\ri}\oint_{\omega_-}\frac{\ln \cB(w)\sfb'(w)}{\sfb(w)-u} \rd w.
\end{equation}
Computing the integral as a residue, we have
\begin{align}\label{e:decomp}
\ln \cB(z)=\frac{1}{2\pi\ri}\oint_{\omega_+}\frac{\ln \cB(w)\sfb'(w)}{\sfb(w)-\sfb(z)} \rd w
-\frac{1}{2\pi\ri}\oint_{\omega_-}\frac{\ln \cB(w)\sfb'(w)}{\sfb(w)-\sfb(z)}\rd w= h_+(\sfb(z))-h_-(\sfb(z)).
\end{align}
From the construction, $h_+(\sfb(z))$ is holomorphic for $z$ inside the contour $\omega_+$ and $h_-(\sfb(z))$ is holomorphic for $z$ outside the contour $\omega_-$.

Next, we claim that when $z$ varies in the annulus enclosed by $\omega_\pm$, the functions  $h_+(\sfb(z))$ and $h_-(\sfb(z))$ are uniformly bounded by a constant  $C$ which depends only on $\omega_\pm$. Indeed, to prove the claim it is sufficient to notice that $\omega_-$ can be made slightly smaller and $\omega_+$ can be made slightly larger without changing the values of the integrals in \eqref{e:B_decomp_4}. Once we do this change, the integrand is uniformly bounded and so is the integral.

The above boundeness allows us to multiply both sides of \eqref{e:ABCprecise} by $\exp(-h_+(\sfb(z)))$  and use the decomposition \eqref{e:decomp} to get
\begin{align}\label{e:ABCorg}
\cC(z)e^{-h_+(\sfb(z))}=\cA(z)e^{-h_-(\sfb(z))}+\OO(\varepsilon).
\end{align}
Since the left-hand side of \eqref{e:ABCorg} is holomorphic inside $\omega_+$, we can use a contour integral to get rid of it. Integrating \eqref{e:ABCorg} over $\omega_-$, we get
\begin{align}\begin{split}\label{e:getridC}
0&=\frac{1}{2\pi\ri}\oint_{\omega_-}\frac{\cC(w)e^{-h_+(\sfb(w))} \sfb'(w)}{\sfb(w)-\sfb(z)}\rd w=\frac{1}{2\pi\ri}\oint_{\omega_-}\frac{(\cA(w)e^{-h_-(\sfb(w))}+\OO(\varepsilon)) \sfb'(w)}{\sfb(w)-\sfb(z)}\rd w\\
&=\frac{1}{2\pi\ri}\oint_{\omega_-}\frac{\cA(w)e^{-h_-(\sfb(w))} \sfb'(w)}{\sfb(w)-\sfb(z)}\rd w+\OO(\varepsilon).
\end{split}\end{align}
We recall the expression of $\cA(w)$ from \eqref{e:A_def} and notice that  it is  a meromorphic function of $\sfb(w)$:
\begin{align}\label{e:deftA}
\cA(w)= {\mathbf A}(\sfb(w)),\qquad
{\mathbf A}(u)=\bE\left[\prod_{i=1}^n\frac{u-\sfb(\varepsilon x_i+ \varepsilon\theta  e_i)}{u-\sfb(\varepsilon x_i)}\right].
\end{align}
 In particular ${\mathbf A}(u)$ is analytic outside the contour $\sfb(\omega_-)$. Deforming the integration contour to $\omega_+$, collecting the residue at $w=z$, and then changing the variables $u=b(w)$, we transform  \eqref{e:getridC} into
\begin{equation}\label{e:getridC2}
\OO(\eps)=\frac{1}{2\pi\ri}\oint_{\omega_-}\frac{\cA(w)e^{-h_-(\sfb(w))} \sfb'(w)}{\sfb(w)-\sfb(z)} \rd w
=\frac{1}{2\pi\ri}\oint_{\sfb(\omega_+)}\frac{{\mathbf A}(u)e^{-h_-(u)}}{u-\sfb(z)} \rd u- \cA(z)e^{-h_-(\sfb(z))}.
\end{equation}
The $u$--integrand in the right-hand side of \eqref{e:getridC2} does not have singularities outside the integration contour and, therefore, the integral evaluates as a residue at $u=\infty$. This residue is $1$, because ${\mathbf A}(u)=1+\oo(1)$ and $h_-(u)\to 0$ as $u\to\infty$ directly from the definitions of these functions. Thus, \eqref{e:getridC2} simplifies to
$$
 \cA(z)e^{-h_-(\sfb(z))}= 1+ \OO(\eps).\qedhere
$$
\end{proof}
\begin{remark}\label{r:cAud}
 We can rewrite the result of Proposition \ref{p:firstod0} in the form
 \begin{equation}
 \label{e:first_order_extended}
 {\mathbf A}(u)=e^{h_-(u)}+\OO(\varepsilon).
 \end{equation}
 So far we proved this statement for $u\in \sfb( \Lambda\setminus [\bl,\br])$, however, it can readily be extended to all $u$ outside $[\sfb(\bl),\sfb(\br)]$. Indeed, to see that we notice that both sides of \eqref{e:first_order_extended} are holomorphic outside $\sfb(\omega_-)$, use Cauchy integral formula to rewrite them as contour integrals over $\sfb(\omega_-)$, and send $\eps\to 0$ in this integral form. In particular, \eqref{e:first_order_extended} implies that $\ln {\mathbf A}(u)$ is well-defined outside $[\sfb(\bl), \sfb(\br)]$.
\end{remark}

We proceed to the second order asymptotic expansion of $\cA(z)$.
\begin{proposition}\label{p:firstod}
Under Assumptions \ref{a:asymp} and \ref{a:stable} (either in probabilistic case with \eqref{e:probabilistic_case} or in the complex case with \eqref{e:totalabsmass}), we have
\begin{equation}\label{e:yao1}
\frac{\del_z \cA(z)}{\cA(z)}=\frac{1}{2\pi \ri}\oint_{\cin}\frac{\ln \cB(w) \sfb'(z)\sfb'(w)}{(\sfb(w)-\sfb(z))^2}\rd w+
\frac{\varepsilon}{2\pi \ri}\oint_{\cin}  \cE^{(1)}(w)\frac{\sfb'(w)\sfb'(z)}{(\sfb(w)-\sfb(z))^2}\rd w
+\OO\left(\varepsilon^2 \right),
\end{equation}
for  $z\in \Lambda\setminus[\bl,\br]$; where the contour $\cin\subset \Lambda$ encloses $[\bl, \br]$, but not $z$. $\cE^{(1)}(w)$ is as in \eqref{e:defcE} and has asymptotic behavior
\begin{align}\begin{split}\label{e:defnewc1}
\cE^{(1)}(w)&=\frac{\varphi^+(w)\cG(w)}{\cB(w)}\left(\frac{\theta}{2\pi \ri}\oint_{\omega'_-}\frac{\ln \cB(u)\sfb'(w)\sfb'(u)}{(\sfb(u)-\sfb(w))^2}\rd u\right.\\
&\left.+\frac{\theta^2}{2}\int \frac{\sfb''(z)(\sfb(z)-\sfb(s))-(\sfb'(z))^2-\sfb'(z)\sfb'(s)}{(\sfb(z)-\sfb(s))^2} \rho(s;\bmx)\rd s\right)+\OO(\varepsilon),
\end{split}\end{align}
where the contour $\omega'_-\subset \Lambda$ encloses  $[\bl, \br]$, but not $w$. If we assume that $z$ belongs to a compact subset of $\Lambda\setminus [\bl,\br]$ and fix the contours $\omega_-$, $\omega'_-$, then the implicit constants in $\OO(\cdot)$ errors can be chosen uniformly.
\end{proposition}

\begin{proof} We recall the result of Proposition \ref{p:simplfyC}  and its logarithmic derivative:
\begin{align}\begin{split}
\cC(z)&=\cA(z)\cB(z)\left(1+\varepsilon \cE^{(1)}(z)\right), \\
 \frac{\del_z \cC(z)}{\cC(z)}&=\frac{\del_z \cA(z)}{\cA(z)}+\frac{\del_z\cB(z)}{\cB(z)}+\del_z \ln\bigl(1+\varepsilon \cE^{(1)}(z)\bigr).\label{e:dzC}
\end{split}\end{align}

Thanks to Assumption \ref{a:stable}, Proposition \ref{p:simplfyC}, and Proposition \ref{p:firstod0}, whenever $z$ belongs to a compact subset of $\Lambda\setminus[\bl,\br]$, we have
$$c\leq |\cA(z)|\leq 1/c, \qquad c\leq |\cB(z)|\leq 1/c, \qquad c\leq |\cE^{(1)}(z)|\leq 1/c,$$
for a constant $c$ depending on the choice of this compact set, but not on $z$. Hence, all terms in the right-hand side of \eqref{e:dzC} are holomorphic for $z\in \Lambda\setminus[\bl,\br]$.

\smallskip

\begin{claim} If $\eps$ is small enough, then $\del_z \cC(z)/\cC(z)$ is holomorphic for all $z\in \Lambda$ (including $[\bl,\br]$).
\end{claim}

\noindent To prove the claim, we note that $\cC(z)$ is holomorphic by Theorem \ref{t:loopeq}, hence, we only need to show that it has no zeros. Outside $[\bl,\br]$ there are no zeros by \eqref{e:dzC} and we only focus on a neighborhood of $[\bl,\br]$. By the Cauchy's argument principle, the total number of zeros is $\cC(z)$ in a neighborhood of $[\bl,\br]$ can be computed as
\begin{equation}
\label{e:aC0}
 \frac{1}{2\pi\ri}\oint_{\omega}\frac{\del_z \cC(z)}{\cC(z)}\rd z,
\end{equation}
where $\omega$ is a contour enclosing this neighborhood. Hence, we need to prove that \eqref{e:aC0} vanishes. Using \eqref{e:dzC}, Assumption \ref{a:stable}, and definition \eqref{e:deftA} we transform \eqref{e:aC0} as
\begin{align}
 \notag\frac{1}{2\pi \ri}\oint_{\omega}\left(\frac{\del_z \cA(z)}{\cA(z)}+\frac{\del_z \cB(z)}{\cB(z)}+\del_z\ln (1+\varepsilon  \cE^{(1)}(z))\right)\rd z&=\frac{1}{2\pi \ri}\oint_{\omega}\frac{\del_z \cA(z)}{\cA(z)}\rd z+\OO(\varepsilon)\\
&=\frac{1}{2\pi \ri}\oint_{\sfb(\omega)}\frac{\del_u {\mathbf A}(u)}{\mathbf A(u)}\rd u+\OO(\varepsilon). \label{e:zeros_count}
\end{align}
Remark \ref{r:cAud} guarantees that $\del_u {\mathbf A}(u)/\mathbf A(u)=\del_u \ln \mathbf A(u)$ is holomorphic outside $\omega$; therefore, we can compute the contour integral in the right-hand side of \eqref{e:zeros_count} as the residue at the infinity. Since ${\mathbf A}(u)=1+\frac{c}{u}+\dots$ as $u\to\infty$, we have $\del_u \ln {\mathbf A}(u)\sim -c/u^2$ and the residue vanishes. We conclude that \eqref{e:zeros_count} is $\OO(\eps)$ and, hence, so is \eqref{e:aC0}. Because this number is an integer counting the number of zeros, it has to vanish for small $\eps$. The claim is proven.

\bigskip

Because of the claim, we can use a contour integral to get rid of $\del_z\cC(z)/\cC(z)$ and recover $\del_z \cA(z)/\cA(z)$. We take a contour $\cin\subset \Lambda$ enclosing $[\bl, \br]$ and $z$ outside this contour. By performing the same contour integral on both sides of \eqref{e:dzC}, using Assumption \ref{a:stable} and the claim we get
\begin{align}\begin{split}\label{e:expza}
&\phantom{{}={}}-\frac{1}{2\pi \ri}\oint_{\cin}\frac{\del_w \cA(w)}{\cA(w)}\frac{\sfb'(z)\rd w}{\sfb(w)-\sfb(z)}\\
&=
\frac{1}{2\pi \ri}\oint_{\cin}\left(\frac{\del_w \cB(w)}{\cB(w)}+\del_w\ln(1+\varepsilon \cE^{(1)}(w))-\frac{\del_w \cC(w)}{\cC(w)}\right)\frac{\sfb'(z)\rd w}{\sfb(w)-\sfb(z)}\\
&=\frac{1}{2\pi \ri}\oint_{\cin}\frac{\ln \cB(w)\sfb'(w)\sfb'(z)\rd w}{(\sfb(w)-\sfb(z))^2}+
\frac{1}{2\pi \ri}\oint_{\cin}\frac{\ln(1+\varepsilon \cE^{(1)}(w))\sfb'(w)\sfb'(z)\rd w}{(\sfb(w)-\sfb(z))^2}\\
&=\frac{1}{2\pi \ri}\oint_{\cin}\frac{\ln \cB(w) \sfb'(z)\sfb'(w)\rd w}{(\sfb(w)-\sfb(z))^2}+
\frac{\varepsilon}{2\pi \ri}\oint_{\cin}  \cE^{(1)}(w)\frac{\sfb'(w)\sfb'(z)\rd w}{(\sfb(w)-\sfb(z))^2}
+\OO\left(\varepsilon^2 \right).
\end{split}\end{align}
Here, in the second line, we used that  $\del_w \cC(w)/\cC(w)$ is holomorphic inside the contour $\cin$, thus, the corresponding contour integrals vanish. Let us transform the left-hand side of \eqref{e:expza}. For that we change the variable $u=\sfb(w)$ and use ${\mathbf A}(u)=\cA(w)$ from \eqref{e:deftA} to get:
\begin{align}
\begin{split}\label{e:dAdz2}
-\frac{1}{2\pi \ri}\oint_{\cin}\frac{\del_w \cA(w)}{\cA(w)}\frac{\sfb'(z)\rd w}{\sfb(w)-\sfb(z)}
&=\frac{\del_z \cA(z)}{\cA(z)}-\frac{1}{2\pi \ri}\oint_{\cout}\frac{\del_w \cA(w)}{\cA(w)}\frac{\sfb'(z)\rd w}{\sfb(w)-\sfb(z)}\\
&=\frac{\del_z \cA(z)}{\cA(z)}-\frac{1}{2\pi \ri}\oint_{\sfb(\cout)}\frac{\del_u {\mathbf A}(u)}{{\mathbf A}(u)}\frac{\sfb'(z)\rd u}{u-\sfb(z)}
=\frac{\del_z \cA(z)}{\cA(z)},
\end{split}\end{align}
where the contour $\cout\subset \Lambda$ encloses $[\bl, \br]$ and $z$; for the last inequality, we used the same computation of the integral as the residue at $u=\infty$, as for the integral in the right-hand side of \eqref{e:zeros_count}.

The combination of \eqref{e:expza} with \eqref{e:dAdz2} gives \eqref{e:yao1} and it remains to prove the asymptotic expansion \eqref{e:defnewc1} for $\cE^{(1)}(w)$. By plugging \eqref{e:cE00} into \eqref{e:defcE}, we have the following expression for $\cE^{(1)}(w)$ with $w\in \Lambda\setminus[\bl,\br]$
\begin{align}\begin{split}\label{e:cE1est2}
\cE^{(1)}(w)=\frac{\theta \varphi^+(w)\cG(w)}{\cA(w)\cB(w)}\bE\left[\prod_{i=1}^n \frac{\sfb(w)-\sfb(\varepsilon x_i+\varepsilon\theta e_i)}{\sfb(w)-\sfb(\varepsilon x_i)}\left( \frac{\theta}{\eps}\int_\bl^\br \frac{\sfb'(w)(\rho(s;\bmy)-\rho(s;\bmx))}{\sfb(w)-\sfb(s)}\rd s\right)\right]\\
+\frac{\varphi^+(w)\cG(w)}{\cB(w)}
\frac{\theta^2}{2}\int_\bl^\br \frac{\sfb''(z)(\sfb(z)-\sfb(s))-(\sfb'(z))^2-\sfb'(z)\sfb'(s)}{(\sfb(z)-\sfb(s))^2} \rho(s;\bmx)\rd s
+\OO(\varepsilon),
\end{split}\end{align}
where the uniform control over $\OO(\eps)$ term relies on $|\cA(w)|\geq c$ and  $|\cB(w)|\geq c$ from Proposition \ref{p:firstod0} and Assumption \ref{a:stable}, respectively. The second term in the right-hand side of \eqref{e:cE1est2} matches the second line in \eqref{e:defnewc1} and it remains to study the first term on the righthand side of \eqref{e:cE1est2} as $\eps\to 0$.
It follows from combining \eqref{e:dzAexp} with upper bound \eqref{e:Aupbb} that
\begin{align}\label{e:aterm2}
(\del_w \ln \cA(w)) \cA(w)&=  \bE\left[\prod_{i=1}^n  \frac{\sfb(w)-\sfb(\varepsilon x_j+\varepsilon \theta e_i) }{\sfb(w)-\sfb(\varepsilon x_j)}  \left(\frac{\theta}{\varepsilon}\int_\bl^\br \frac{\sfb'(w)(\rho(s;\bmy)-\rho(s;\bmx))}{\sfb(w)-\sfb(s)}\rd s
\right)\right]+\OO(\varepsilon).
\end{align}
On the other hand, differentiating \eqref{e:yao10} (we can differentiate the asymptotic expansion, because we deal with holomorphic functions and can represent derivative through the Cauchy integral formula), we get
\begin{equation}\label{e:aterm3}
\del_w \ln \cA(w)=\frac{1}{2\pi \ri}\oint_{\omega'_-}\frac{\ln \cB(u)\sfb'(u)\sfb'(w)}{(\sfb(u)-\sfb(w))^2}\rd u+\OO(\varepsilon),
\end{equation}
where the contour $\omega'_-\in \Lambda$ encloses $[\bl, \br]$, but not $w$. Combining \eqref{e:aterm2} with \eqref{e:aterm3} gives  the desired expression for the first term in the right-hand side of \eqref{e:cE1est2}, which matches the first line in \eqref{e:defnewc1}.
\end{proof}
\begin{remark} \label{r:cAud_second_order}
 Similarly to Remark \ref{r:cAud}, we can rewrite the statement of Proposition \ref{p:firstod} in terms of the function ${\mathbf A}(u)$ of \eqref{e:deftA}:
 \begin{equation}\label{e:yao1_second_form}
\frac{\del_u {\mathbf A}(u)}{{\mathbf A}(u)}=\frac{1}{2\pi \ri}\oint_{\cin}\frac{\ln \cB(w) \sfb'(w)}{(\sfb(w)-u)^2}\rd w+
\frac{\varepsilon}{2\pi \ri}\oint_{\cin}  \cE^{(1)}(w)\frac{\sfb'(w)}{(\sfb(w)-u)^2}\rd w
+\OO\left(\varepsilon^2 \right).
\end{equation}
The advantage of \eqref{e:yao1_second_form} is that $u$ can be taken in it to be an arbitrary point in $\mathbb C\setminus [\sfb(\bl),\sfb(\br)]$, as long as $\cin$ is chosen so that $u$ is outside $\sfb(\omega_-)$. The proof of this extension is the same as in Remark \ref{r:cAud}.
\end{remark}

\subsection{Mean Estimate}\label{s:mean}
In this subsection we compute the first two terms in the right-hand side of \eqref{e:dmg} in Theorem \ref{t:loopstudy}.
\begin{proposition}\label{p:second}
Under Assumptions \ref{a:asymp}, \ref{a:stable} in the probabilistic case \eqref{e:probabilistic_case} we have for $z\in\Lambda\setminus[\bl,\br]$:
\begin{align*}
\bE\left[\frac{\theta}{\varepsilon}\int_\bl^\br \frac{\sfb'(z)(\rho(s;\bmy)-\rho(s;\bmx))}{\sfb(z)-\sfb(s)} \rd s\right]
&=\frac{1}{2\pi \ri}\oint_{\cin}\frac{\ln \cB(w)\sfb'(z)\sfb'(w)}{(\sfb(w)-\sfb(z))^2}\rd w\\
&+
\frac{\varepsilon}{2\pi \ri}\oint_{\cin}  \cE^{(2)}(w,z)\frac{\sfb'(w)\sfb'(z)}{(\sfb(w)-\sfb(z))^2} \rd w
+\OO\left(\varepsilon^2 \right),
\end{align*}
where the contour $\cin\subset \Lambda$ encloses $[\bl, \br]$, but not $z$, and
\begin{align}\begin{split}\label{e:cE2hi}
\cE^{(2)}(w,z)&=\frac{\varphi^+(w)\cG(w)}{\cB(w)}\left(\frac{\theta}{2\pi \ri}\oint_{\omega'_-}\frac{\ln\cB(u) \sfb'(w)\sfb'(u) }{(\sfb(u)-\sfb(w))^2}\rd u +\frac{\theta \sfb'(w)}{\sfb(z)-\sfb(w)}\right.\\
&\quad \left.+\frac{\theta^2}{2}\int_\bl^\br \frac{\sfb''(z)(\sfb(z)-\sfb(s))-(\sfb'(z))^2-\sfb'(z)\sfb'(s)}{(\sfb(z)-\sfb(s))^2} \rho(s;\bmx)\rd s\right)
-\frac{\del_{w} \cB(w)}{2\cB(w)},
\end{split}\end{align}
the contour $\omega'_-\subset \Lambda$ encloses  $[\bl,\br]$, but not $w$. If we assume that $z$ belongs to a compact subset of $\Lambda\setminus[\bl,\br]$, then the error $\OO(\eps^2)$ is uniformly (in $z$) small.
\end{proposition}

\begin{proof}
We fix an arbitrarily  small $\eta>0$,  take $v$ with $\dist(v,[\bl,\br])>2\eta$, and consider a deformation of the transition probability $\bP(\bmx+\bme|\bmx)=a_\bme$ of \eqref{e:m1ccopy}:
\begin{align}\label{e:tae0}
\tilde a_\bme:= \frac{a_\bme \prod\limits_{i=1}^n  \left(\frac{\sfb(v)-\sfb(\varepsilon x_j) }{\sfb(v)-\sfb(\varepsilon x_j+\varepsilon \theta)} \right)^{e_i}}{\sum\limits_{\bme'\in\{0,1\}^n} a_{\bme'} \prod\limits_{i=1}^n  \left(\frac{\sfb(v)-\sfb(\varepsilon x_j )}{\sfb(v)-\sfb(\varepsilon x_j+\varepsilon\theta)}\right)^{e'_i}}.
\end{align}
Transition probability $\tilde a_\bme$ is again in the form \eqref{e:m1ccopy} but with modified function $\varphi^+(z)$ in \eqref{e:asymp}:
\begin{align}\label{e:defvarphi}
\tilde  \varphi^+(z)=\varphi^+(z) \cdot \frac{\sfb(v)-\sfb(z)}{\sfb(v)-\sfb(z+\varepsilon \theta)}.
\end{align}

\begin{claim}\label{c:checkC}
The transition probabilities $\tilde a_\bme$ satisfy Assumptions \ref{a:asymp} and \ref{a:stable} in the complex case \eqref{e:totalabsmass} with $\Lambda$ replaced by $\Lambda\cap \{z\in \mathbb C\mid \dist(z,[\bl,\br])<\eta\}$.
\end{claim}
We postpone proving Claim \ref{c:checkC} until after we finish the proof of Proposition \ref{p:second}. We let $\tilde\bE$ denote expectation for the transition probability $\tilde a_\bme$ and define the quantities $\tilde  \cA(z)$ and $\tilde  \cB(z)$ through
\begin{align}\begin{split}\label{e:defhB}
\tilde  \cA(z)
&=\sum_{\bme\in\{0,1\}^n}\tilde  a_\bme\prod_{i=1}^n\frac{\sfb(z)-\sfb(\varepsilon x_i+ \varepsilon\theta  e_i)}{\sfb(z)-\sfb(\varepsilon x_i)}=\tilde \bE\left[\prod_{i=1}^n\frac{\sfb(z)-\sfb(\varepsilon x_i+ \varepsilon\theta  e_i)}{\sfb(z)-\sfb(\varepsilon x_i)}\right],\\
\tilde  \cB(z)&=\cG(z)\tilde  \varphi^+(z)+ \varphi^-(z), \qquad  \qquad \cG(z)=\exp\left[\theta\int_{\bl}^{\br}\frac{\sfb'(z)\rho(s;\bmx)}{\sfb(z)-\sfb(s)}\rd s\right].
\end{split}\end{align}
We can rewrite $\ln \tilde \cB(z)$ in terms of $\ln \cB(z)$ from \eqref{e:B_function} with a small error:
\begin{align}\begin{split}\label{e:changeBB}
\ln \tilde \cB(z)
&=\ln \left(\cG(z) \varphi^+(z) \frac{\sfb(v)-\sfb(z)}{\sfb(v)-\sfb(z+\varepsilon \theta)}+\varphi^-(z)\right)\\
&=\ln \cB(z)
+\ln \left(1+\frac{\cG(z) \varphi^+(z)}{\cB(z)} \cdot  \frac{\sfb(z+\varepsilon \theta)-\sfb(z)}{\sfb(v)-\sfb(z+\varepsilon \theta)}\right)\\
&=\ln \cB(z)+\frac{\varepsilon \theta\cG(z) \varphi^+(z)}{\cB(z)}\cdot \frac{\sfb'(z)}{\sfb(v)-\sfb(z)}+\OO(\varepsilon^2).
\end{split}\end{align}
Applying Proposition \ref{p:firstod}, we have
\begin{align}\label{e:dtAz'}
\del_{z}\ln \tilde \cA(z)
=\frac{1}{2\pi \ri}\oint_{\cin}\frac{\ln \tilde \cB(w)\sfb'(w)\sfb'(z)}{(\sfb(w)-\sfb(z))^2}\rd w+
\frac{\varepsilon}{2\pi \ri}\oint_{\cin} \frac{ \tilde \cE^{(1)}(w) \sfb'(w)\sfb'(z)}{(\sfb(w)-\sfb(z))^2}\rd w
+\OO\left(\varepsilon^2 \right),
\end{align}
where the contour $\cin\subset \Lambda\cap \{z\in \mathbb C\mid \dist(z,[\bl,\br])<\eta\}$ encloses $[\bl, \br]$, but not $v$ or $z$. Here
\begin{align}\begin{split}\label{e:rcE1}
\tilde \cE^{(1)}(w)&=\frac{\tilde \varphi^+(w)\cG(w)}{\tilde \cB(w)}\left(\frac{\theta}{2\pi \ri}\oint_{\omega'_-}\frac{\ln \tilde \cB(u)\sfb'(w)\sfb'(u)}{(\sfb(u)-\sfb(w))^2}\rd u\right.\\
&\left.+\frac{\theta^2}{2}\int_\bl^\br \frac{\sfb''(z)(\sfb(z)-\sfb(s))-(\sfb'(z))^2-\sfb'(z)\sfb'(s)}{(\sfb(z)-\sfb(s))^2} \rho(s;\bmx)\rd s\right)+\OO(\varepsilon),
\end{split}\end{align}
where the contour $\omega'_-\subset  \Lambda\cap \{z\in \mathbb C\mid \dist(z,[\bl,\br])<\eta\}$ encloses  $[\bl, \br]$, but not $w$.
Using \eqref{e:changeBB}, we rewrite
\begin{align}\begin{split}\label{e:dtAzpcopy}
\del_{z}\ln \tilde \cA(z)
&=\frac{1}{2\pi \ri}\oint_{\cin}\frac{\ln \cB(w)\sfb'(w)\sfb'(z)}{(\sfb(w)-\sfb(z))^2}\rd w\\
&+
\frac{\varepsilon}{2\pi \ri}\oint_{\cin} \left(\cE^{(1)}(w)+\frac{\theta \cG(w)\varphi^+(w)}{\cB(w)}\frac{\sfb'(w)}{\sfb(v)-\sfb(w)}\right)\frac{ \sfb'(w)\sfb'(z)}{(\sfb(w)-\sfb(z))^2}\rd w
+\OO\left(\varepsilon^2 \right),
\end{split}\end{align}
where $\cE^{(1)}(w)$, as in \eqref{e:defnewc1}, is obtained from $\tilde \cE^{(1)}(w)$ by replacing $\tilde \varphi^+$ and $\tilde \cB$ with $\varphi^+$ and  $\cB$, respectively.

Note that we used Proposition \ref{p:firstod} under Claim \ref{c:checkC} to prove \eqref{e:dtAzpcopy}. Hence, $z$ in \eqref{e:dtAzpcopy} should belong to the set $\Lambda\cap \{z\in \mathbb C\mid \dist(z,[\bl,\br])<\eta\}$. However, we can extend to all $z$ in $\Lambda$ outside $\omega_-$ by using Remark \ref{r:cAud_second_order}; this is important for us, because we later would like to set $z=v$.

On the other hand, using \eqref{e:dzAexp} we have
\begin{align}\label{e:secondod}
\del_{z}\ln \tilde \cA(z)=\bE\left[\int_\bl^\br \frac{\theta}{\varepsilon}\frac{\sfb'(z)(\rho(s;\bmy)-\rho(s;\bmx))}{\sfb(z)-\sfb(x)}\rd s
-\frac{\theta}{2} \frac{ \sfb'(z)\sfb'(s)(\rho(s;\bmy)-\rho(s;\bmx))}{(\sfb(z)-\sfb(s))^2}\rd s +\OO(\varepsilon^2)\right],
\end{align}
where we emphasize that the right-hand side involves the expectation $\bE$ with respect to the original un-deformed measure.

By comparing the leading terms in \eqref{e:dtAzpcopy} and \eqref{e:secondod}, we obtain
\begin{align}\label{e:recall1}
\bE\left[\frac{\theta}{\varepsilon}\int_\bl^\br \frac{\sfb'(z)(\rho(s;\bmy)-\rho(s;\bmx))}{\sfb(z)-\sfb(s)}\rd s\right]
=\frac{1}{2\pi \ri}\oint_{\cin}\frac{\ln  \cB(w)\sfb'(w)\sfb'(z)}{(\sfb(w)-\sfb(z))^2}\rd w+\OO\left(\varepsilon\right),
\end{align}
where the contour $\omega_-$ encloses $[\bl,\br]$, but not $z$.
Introducing another contour $\omega_-'\subset\Lambda$, which contains $\omega_-$, but not $z$, and using \eqref{e:recall1}, we can rewrite the second term on the righthand side of \eqref{e:secondod} as a contour integral
\begin{align}\begin{split}
&\phantom{{}={}}\bE\biggl[\frac{\theta}{2}\int_\bl^\br \frac{ \sfb'(z)\sfb'(s)(\rho(s;\bmy)-\rho(s;\bmx))}{(\sfb(z)-\sfb(s))^2}\rd s \biggr]\\
&=\frac{1}{2\pi\ri}\oint_{\omega'_-} \frac{\sfb'(z)\sfb'(u)}{(\sfb(u)-\sfb(z))^2}\bE\left[\frac{\theta}{2}\int_\bl^\br \frac{\sfb'(u)(\rho(s;\bmy)-\rho(s;\bmx))}{\sfb(u)-\sfb(s)} \rd s \right] \rd u\\
&=\frac{\varepsilon}{4\pi \ri}\oint_{\omega'_-}\frac{\sfb'(z)\sfb'(u)}{(\sfb(u)-\sfb(z))^2}\left(\frac{1}{2\pi \ri}\oint_{\cin}\frac{\ln  \cB(w)\sfb'(w)\sfb'(u)}{(\sfb(w)-\sfb(u))^2} \rd w
 +\OO(\varepsilon)\right)\rd u. \label{e:secondtt}
 \end{split}\end{align}
We first integrate in $u$ by using an identity valid for any holomorphic function $f(u)$:
 $$
  \oint\limits_{\text{around } w} \frac{f(u) \sfb'(u)}{(\sfb(w)-\sfb(u))^2} \rd u=   \oint\limits_{\text{around } \sfb(w)} \frac{f(\sfb^{-1}(a)) }{(\sfb(w)-a)^2} \rd a= 2\pi \ii    \del_a \Bigl[f(\sfb^{-1}(a))\Bigr]_{a=\sfb(w)}=2\pi \ii    \frac{f'(w)}{\sfb'(w)}.
 $$
Thus, taking $f(u)=\sfb'(z) \sfb'(u)/ (\sfb(u)-\sfb(z))^2$ in the above identity, we simplify \eqref{e:secondtt} to
 \begin{align} \label{e:secondtt_2}
\frac{\varepsilon}{4\pi \ri}\oint_{\cin}\del_u\left[\frac{\sfb'(z) \sfb'(u)}{(\sfb(u)-\sfb(z))^2}\right]_{u=w}   \ln \cB(w) \rd w
=-\frac{\varepsilon}{4\pi \ri}\oint_{\cin}\frac{\sfb'(z)\sfb'(w)}{(\sfb(w)-\sfb(z))^2}     \frac{\del_w \cB(w)}{\cB(w)} \rd w +\OO(\varepsilon^2),
\end{align}
where we integrated by parts in the last identity. Plugging \eqref{e:secondtt}, \eqref{e:secondtt_2} into \eqref{e:secondod}, and rearranging, we get
\begin{align*}\begin{split}
\bE\left[\frac{\theta}{\varepsilon}\int \frac{\sfb'(z)(\rho(s;\bmy)-\rho(s;\bmx))\rd s}{\sfb(z)-\sfb(s)}\right]
&=\frac{1}{2\pi \ri}\oint_{\cin}\frac{\ln \cB(w)\sfb'(z)\sfb'(w)\rd w}{(\sfb(w)-\sfb(z))^2}\\
&+
\frac{\varepsilon}{2\pi \ri}\oint_{\cin}  \cE^{(2)}(w,z)\frac{\sfb'(w)\sfb'(z)\rd w}{(\sfb(w)-\sfb(z))^2}
+\OO\left(\varepsilon^2 \right),
\end{split}\end{align*}
where, with $\cE^{(1)}(w)$ as in \eqref{e:defnewc1}, we have:
$$
\cE^{(2)}(w,z):=\cE^{(1)}(w)+\frac{\theta \cG(w)\varphi^+(w)}{\cB(w)}\frac{\sfb'(w)}{\sfb(z)-\sfb(w)}-\frac{\del_{w} \cB(w)}{2\cB(w)}. \qedhere
$$
\end{proof}

\begin{proof}[Proof of Claim \ref{c:checkC}] In addition to $\tilde a_\bme$ we need another deformation $\hat a_\bme$ of the transition  probability:
\begin{align}\label{e:tae0_2}
\hat a_\bme:= \frac{a_\bme \prod\limits_{i=1}^n  \left(\frac{\sfb(v)-\sfb(\varepsilon x_j) }{\sfb(v)-\sfb(\varepsilon x_j+\varepsilon \theta)} \frac{\sfb(\bar v)-\sfb(\varepsilon x_j) }{\sfb(\bar v)-\sfb(\varepsilon x_j+\varepsilon \theta)}\right)^{e_i}}{\sum\limits_{\bme'\in\{0,1\}^n} a_{\bme'} \prod\limits_{i=1}^n  \left(\frac{\sfb(v)-\sfb(\varepsilon x_j )}{\sfb(v)-\sfb(\varepsilon x_j+\varepsilon\theta)}\frac{\sfb(\bar v)-\sfb(\varepsilon x_j) }{\sfb(\bar v)-\sfb(\varepsilon x_j+\varepsilon \theta)}\right)^{e'_i}}.
\end{align}
This transition probability is again of the form \eqref{e:m1ccopy} with modified function $\varphi^+(z)$ in \eqref{e:asymp}:
\begin{align}\label{e:defvarphi_2}
\hat \varphi^+(z)=\varphi^+(z)\cdot  \frac{\sfb(v)-\sfb(z)}{\sfb(v)-\sfb(z+\varepsilon \theta)}\cdot \frac{\sfb(\bar v)-\sfb(z)}{\sfb(\bar v)-\sfb(z+\varepsilon \theta)}.
\end{align}
Similarly to \eqref{e:defhB}, we define the quantity $\hat \cA(z)$ and $\hat \cB(z)$ for the measure $\hat a_\bme$ from \eqref{e:tae0} as
\begin{align}\begin{split}\label{e:defhB_2}
\hat  \cA(z)
&=\sum_{\bme\in\{0,1\}^n}\hat  a_\bme\prod_{i=1}^n\frac{\sfb(z)-\sfb(\varepsilon x_i+ \varepsilon\theta  e_i)}{\sfb(z)-\sfb(\varepsilon x_i)}=\hat \bE\left[\prod_{i=1}^n\frac{\sfb(z)-\sfb(\varepsilon x_i+ \varepsilon\theta  e_i)}{\sfb(z)-\sfb(\varepsilon x_i)}\right],\\
\hat  \cB(z)&=\cG(z)\hat  \varphi^+(z)+ \varphi^-(z), \qquad  \qquad \cG(z)=\exp\left[\theta\int_{\bl}^{\br}\frac{\sfb'(z)\rho(s;\bmx)}{\sfb(z)-\sfb(s)}\rd s\right].
\end{split}\end{align}
By our construction, $\hat a_\bme$ is a bona-fide real positive measure, and we are in the probabilistic case \eqref{e:probabilistic_case}. Our choice $\dist(v,[\bl,\br])>{2\eta}$, implies that Assumption \ref{a:asymp} holds for $\hat a_\bme$ with $\Lambda$ replaced by $\Lambda\cap\{z\in\mathbb C\mid \dist(z,[\bl,\br])<\eta\}$. Moreover, similarly to \eqref{e:changeBB}, we have
\begin{align*}
\ln \hat \cB(z)=\ln\cB(z)+\OO(\varepsilon).
\end{align*}
It follows that $\ln \hat\cB(z)$ is well defined and Assumption \ref{a:stable} holds for $\hat \cB(z)$ and small enough $\eps$.

We can check that $\tilde a_\bme$ satisfies Assumptions \ref{a:asymp} and \ref{a:stable} in exactly the same way as we just did for $\hat a_\bme$. The non-trivial step is to verify the bound \eqref{e:totalabsmass}: the difficulty is that the weights are complex and, therefore, we need to check that the denominator in \eqref{e:tae0} is bounded away from $0$. For that we use Proposition \ref{p:firstod0} for the transition probability $\hat a_\bme$, which gives that
\begin{align}\begin{split}\label{e:lnhatA}
\frac{1}{2\pi \ri}\oint_{\cin}\frac{\ln \hat\cB(w)\sfb'(w)\rd w}{\sfb(w)-\sfb(\bar v)}
+\OO(\varepsilon)&= \ln \hat \cA(\bar v)
=\ln \left[\sum_{\bme\in\{0,1\}^n}\hat a_\bme\prod_{i=1}^n\left(\frac{\sfb(\bar v)-\sfb(\varepsilon x_i+ \varepsilon\theta )}{\sfb(\bar v)-\sfb(\varepsilon x_i)}\right)^{e_i}\right]\\ &
=
\ln \left[\frac{\sum\limits_{\bme\in\{0,1\}^n} a_\bme \prod\limits_{i=1}^n  \left(\frac{\sfb(v)-\sfb(\varepsilon x_j) }{\sfb(v)-\sfb(\varepsilon x_j+\varepsilon \theta)} \right)^{e_i}}{\sum\limits_{\bme\in\{0,1\}^n} a_\bme \prod\limits_{i=1}^n  \left(\frac{\sfb(v)-\sfb(\varepsilon x_j )}{\sfb(v)-\sfb(\varepsilon x_j+\varepsilon\theta)}\frac{\sfb(\bar v)-\sfb(\varepsilon x_j) }{\sfb(\bar v)-\sfb(\varepsilon x_j+\varepsilon \theta)}\right)^{e_i}}\right].
\end{split}\end{align}
In particular, for small $\eps$ the expression on the right-hand side of \eqref{e:lnhatA} is a bounded number. For the denominator in the last term, since $ |\sfb(v)-\sfb(\varepsilon x_i+\varepsilon \theta)|\geq \eta/2C$, we have
\begin{align}\label{e:denobound}
\sum_{\bme\in\{0,1\}^n} a_\bme \left(1-\frac{\OO(\varepsilon)}{\eta}\right)^{n} \leq {\sum_{\bme\in\{0,1\}^n} a_\bme \prod_{i=1}^n  \left|\frac{\sfb(v)-\sfb(\varepsilon x_j )}{\sfb(v)-\sfb(\varepsilon x_j+\varepsilon\theta)}\right|^{2e_i}}\leq \sum_{\bme\in\{0,1\}^n} a_\bme \left(1+\frac{\OO(\varepsilon)}{\eta}\right)^{n}.
\end{align}
Hence, the denominator in the right-hand side of \eqref{e:lnhatA} is bounded away from $0$ and $\infty$ and, therefore, so is the numerator. The latter numerator matches the denominator in \eqref{e:tae0}.
\end{proof}

\subsection{Covariance Estimate}\label{s:Cov}

In this subsection we study the stochastic term $\Delta \cM(v)$ in the right-hand side of \eqref{e:dmg} in Theorem \ref{t:loopstudy}.

\begin{proposition} \label{p:cov}
Under Assumptions \ref{a:asymp}, \ref{a:stable} in the probabilistic case \eqref{e:probabilistic_case}, let us define
\begin{align}\label{e:dMt}
\Delta \cM(v):
=\frac{1}{\varepsilon}\int_\bl^\br \frac{\sfb'(v)(\rho(s;\bmx+\bme)-\bE[\rho(s;\bmx+\bme)])}{\sfb(v)-\sfb(s)} \rd s,
\end{align}
where the expectation is with respect to the transition probability \eqref{e:m1ccopy}. Then as $\eps\to 0$ the random field $\{\varepsilon^{-1/2}\Delta \cM(v)\}_{v\in \Lambda\setminus[\bl,\br]}$ is asymptotically Gaussian with zero mean and covariance given by
\begin{equation}
\label{e:assymp_covariance}
\lim_{\eps\to 0}\bE\left[\frac{\Delta \cM(v_1)}{\varepsilon^{1/2} }, \frac{ \Delta \cM(v_2)}{\varepsilon^{1/2}}\right]
=\frac{1}{2\pi \ri \theta }\oint_{\cin} \frac{\cG(w)\varphi^+(w)}{\cB(w)}\frac{\sfb'(w)\sfb'(v_1)}{(\sfb(w)-\sfb(v_1))^2} \frac{\sfb'(w)\sfb'(v_2)}{(\sfb(w)-\sfb(v_2))^2} \rd w,
\end{equation}
where the contour $\cin\subset \Lambda$ encloses $[\bl, \br]$ but not $v_1, v_2$.  If we assume that $v_1$ and $v_2$ belong to a compact subset of $\Lambda\setminus[\bl,\br]$, then the convergence is uniform in $v_1$ and $v_2$. The joint moments of $\{\varepsilon^{-1/2} \Delta \cM(v)\}_{v\in\Lambda\setminus [\bl,\br]}$ converge to the Gaussian joint moments (uniformly in $v$, if we assume that $v$ belongs to a compact subset of $\Lambda\setminus[\bl,\br]$).
\end{proposition}

\begin{proof}
We study random variables
\begin{align*}
\left\{\frac{1}{\varepsilon^{3/2}}\int_\bl^\br \frac{\sfb'(v)(\rho(s;\bmx+\bme)-\rho(s;\bmx))}{\sfb(v)-\sfb(s)}\rd s\right\}_{v\in \Lambda\setminus[\bl,\br]},
\end{align*}
which differ from $ \varepsilon^{-1/2} \Delta \cM(v)$ by a deterministic shift. We let
\begin{align}\label{e:defchi}
\chi(z, v):=\frac{1}{\varepsilon^{2}\theta}\int_{z}^{z+\varepsilon\theta} \left(\frac{\sfb'(v)}{\sfb(v)-\sfb(s+\varepsilon)}-\frac{\sfb'(v)}{\sfb(v)-\sfb(s)}\right)\rd s
=\frac{\sfb'(v)\sfb'(z)}{(\sfb(v)-\sfb(z))^2}+\OO(\varepsilon).
\end{align}
We rewrite the random variables of interest as
\begin{align}\begin{split}\label{e:defchi2}
\frac{1}{\varepsilon^{3/2}}\int_\bl^\br\frac{\sfb'(v)(\rho(s;\bmx+\bme)-\rho(s;\bmx))}{\sfb(v)-\sfb(s)}\rd s
= \varepsilon^{1/2} \sum_{i=1}^n e_i \chi(\varepsilon x_i, v).
\end{split}\end{align}
%From the construction, as a function of $v$, $\chi(\varepsilon x_i,v)$ is real analytic. Linear combinations of $\Re [\chi(\varepsilon x_i,v)]$ and $\Im[ \chi(\varepsilon x_i,v)]$, can be written as linear combinations of $\chi(\varepsilon x_i,v)$ and $\chi(\varepsilon x_i,\bar v)$.

In the following, we use Proposition \ref{p:firstod} to compute the Laplace transform (moment generating function) of $\varepsilon^{1/2} \sum_{i=1}^n  e_i \chi(\varepsilon x_i, v)$. For that we need a deformed version of the transition probability $\bP(\bmx+\bme|\bmx)$.  We fix any  complex numbers $\bms=(s_1, s_2,\dots, s_p)\in\bC^p$, and complex numbers $\bmv=(v_1, v_2, \dots, v_p)\in (\bC\setminus[\bl,\br])^p$
and define a deformed version of $\bP(\bmx+\bme|\bmx)=a_\bme$ from \eqref{e:m1ccopy}:
\begin{align}\label{e:defLtnew}
\bP^{\bms,\bmv}(\bmx+\bme|\bmx)\deq \frac{a_\bme  \exp\left(\varepsilon^{1/2} \sum\limits_{k=1}^p \sum\limits_{i=1}^n  s_k
e_i \chi(\varepsilon x_i;v_k)\right)}{\sum\limits_{\bme'\in\{0,1\}^n}a_{\bme'} \exp\left(\varepsilon^{1/2} \sum\limits_{k=1}^p \sum\limits_{i=1}^n  s_k
e'_i \chi(\varepsilon x_i;v_k)\right)}.
\end{align}
The resulting probability measure is still in the form of \eqref{e:m1ccopy}, but with $\varphi^+(z)$ changed to
\begin{align}\label{e:defvphitnew}
\varphi_{\bmv,\bms}^+(z)= \varphi^+(z) \exp\left(\varepsilon^{1/2} \sum_{k=1}^ps_k \chi(z; v_k)\right).
\end{align}
The expectation with respect to the deformed measure $\bP^{\bms, \bmv}$ can be rewritten as an expectation with respect to the original transition probability $\bP(\bmx+\bme|\bmx)$ in the following way: for any random variable $\xi:\{0,1\}^n\to \mathbb C$, we have
\begin{align}\label{e:rewrite}
\bE_{ \bP^{\bms, \bmv}}\left[ \xi\right]
=\frac{\bE\left[ \xi  \exp\left(\varepsilon^{1/2} \sum\limits_{k=1}^p \sum\limits_{i=1}^n   s_k e_i\chi(\varepsilon x_i; v_k)\right)\right]}{\bE\left[\exp\left( \varepsilon^{1/2} \sum\limits_{k=1}^p \sum\limits_{i=1}^n s_k  e_i\chi(\varepsilon x_i;v_k)\right)\right]}.
\end{align}
The last identity implies that the joint Laplace transform of $\varepsilon^{1/2}\sum_{i=1}^n e_i \chi(\varepsilon x_i;v)$ over $v=v_1,v_2,\dots,v_k$ can be computed as an integral:
\begin{align}\begin{split}\label{e:momentg}
\ln \bE\left[\exp\left(\varepsilon^{1/2} \sum_{k=1}^p  s_k\sum_{i=1}^n
\chi(\varepsilon x_i;v_k)\right)\right]
&=\int_0^1 \del_t \ln \bE\left[\exp\left(t \varepsilon^{1/2} \sum_{k=1}^p  s_k\sum_{i=1}^n
\chi(\varepsilon x_i;v_k)\right)\right]\rd t\\
&=\int_0^1\bE_{\bP^{\bmv, t\bms}} \left[\varepsilon^{1/2} \sum_{k=1}^p  s_k\sum_{i=1}^n
\chi(\varepsilon x_i;v_k)\right]\rd t,
\end{split}\end{align}
where $t\bms=(ts_1,\dots,t s_p)$. We would like to apply Proposition \ref{p:second} to the right-hand side of \eqref{e:momentg} and need to check the conditions of this theorem. In general, $\bP^{\bmv, t\bms}$ is a complex measure and it is unclear how to verify the condition \eqref{e:totalabsmass}. However, we claim that it is sufficient for us to restrict our attention to the case when $\bP^{\bmv, t\bms}$ is a bona fide (real positive) measure. Indeed, the law of $\sum_{i=1}^n
\chi(\varepsilon x_i;v)$ is uniquely determined by its real and imaginary parts. The joint Laplace transform of the former and the latter is accessed through the identity:
\begin{multline} \label{e:Re_Im}
\bE \exp\left(r_1 \Re \left(\sum_{i=1}^n
\chi(\varepsilon x_i;v)\right)+ r_2  \Im \left(\sum_{i=1}^n
\chi(\varepsilon x_i;v)\right)\right)\\ =  \bE \exp\left(\frac{r_1-\ri r_2}{2}\sum_{i=1}^n
\chi(\varepsilon x_i;v)+ \frac{r_1+\ri r_2}{2}  \sum_{i=1}^n
\chi(\varepsilon x_i;\bar v)\right).
\end{multline}
If we choose $r_1$ and $r_2$ to be real (by the uniqueness theorem for the Laplace transform, this is sufficient for identifying the joint law of $\Re \bigl(\sum_{i=1}^n
\chi(\varepsilon x_i;v)\bigr)$ and $\Im \bigl(\sum_{i=1}^n
\chi(\varepsilon x_i;v)\bigr)$), then the random variable in the left-hand side of \eqref{e:Re_Im} is a positive real number, and the measures $\bP^{\bmv, t\bms}$ appearing in the integral \eqref{e:momentg} are also real-positive. Using a similar identity for the joint Laplace transforms of real and imaginary parts of several $\sum_{i=1}^n
\chi(\varepsilon x_i;v_k)$, we conclude that it is sufficient for us to deal in \eqref{e:momentg} with real-positive measures $\bP^{\bmv, t\bms}$.

For positive measures $\bP^{\bmv, t\bms}$, checking Assumptions  \ref{a:asymp} and \ref{a:stable} is straightforward (cf.\  Claim \ref{c:checkC}) and Proposition \ref{p:second} applies. Let
\begin{align*}
\cB^{\bmv,t \bms}(z)=\cG(z)\varphi_{\bmv,t\bms}^+(z)+\varphi^-(z).
\end{align*}
We rewrite $\ln \cB^{\bmv, t\bms}(z)$ in terms of $\ln \cB(z)$ from \eqref{e:B_function} with a small error:
\begin{align}\begin{split}\label{e:bsvexp}
\ln \cB^{\bmv, t\bms}(w)
&=\ln \left(\cG(z)\varphi^+(w) \exp\left(t \varepsilon^{1/2}\sum_{k=1}^p  s_k\chi(w, v_k)\right)+\varphi^-(w)\right)\\
&=\ln \cB(w)+\ln \left(1+\frac{\cG(z)\varphi^+(w)\left( \exp\left(t \varepsilon^{1/2}\sum_{k=1}^p  s_k\chi(w, v_k)\right)-1\right)}{\cB(w)}\right)\\
&=\ln \cB(w)+\varepsilon^{1/2} \frac{t \cG(z)\varphi^+(w)}{\cB(w)}\sum_{k=1}^p  s_k\chi(w; v_k)+\OO\left(\varepsilon\right).
\end{split}\end{align}
Combining Proposition \ref{p:second} (where we use only the first term of the asymptotic expansion) with \eqref{e:bsvexp} and \eqref{e:defchi}, we get
\begin{align}\begin{split}\label{e:EPT}
&\bE_{\bP^{\bmv,t\bms}}\left[\frac{\theta}{\varepsilon}\int_\bl^\br \frac{\sfb'(z)(\rho(s;\bmx+\bme)-\rho(s;\bmx))}{\sfb(z)-\sfb(s)}\rd s\right]
=\frac{1}{2\pi \ri}\oint_{\cin}\frac{\ln \cB^{\bmv,t\bms}(w)\sfb'(w)\sfb'(z)}{(\sfb(w)-\sfb(z))^2}\rd w+
\OO\left(\varepsilon \right)\\
&=\frac{1}{2\pi \ri}\oint_{\cin}\frac{\ln \cB(w)\sfb'(w)\sfb'(z)}{(\sfb(w)-\sfb(z))^2}\rd w+\varepsilon^{1/2}\sum_{k=1}^p\frac{ ts_k}{2\pi \ri}\frac{\cG(z)\varphi^+(w)}{\cB(w)}  \frac{\sfb'(w)\sfb'(z)\chi(w; v_k)}{(\sfb(w)-\sfb(z))^2}\rd w+\OO(\varepsilon)\\
&=\frac{1}{2\pi \ri}\oint_{\cin}\frac{\ln \cB(w)\sfb'(w)\sfb'(z)}{(\sfb(w)-\sfb(z))^2}\rd w+\varepsilon^{1/2}\sum_{k=1}^p\frac{ ts_k}{2\pi \ri}\frac{\cG(z)\varphi^+(w)}{\cB(w)}  \frac{\sfb'(w)\sfb'(z)\sfb'(w)\sfb'(v_k)\rd w}{(\sfb(w)-\sfb(z))^2(\sfb(w)-\sfb(v_k))^2}+\OO(\varepsilon),
\end{split}\end{align}
where the contour $\omega_-\in \Lambda$ encloses $[\bl,\br]$, but not $z, v_1, v_2, \dots, v_p$. Multiplying by $\varepsilon^{1/2}$, the left-hand side of \eqref{e:EPT} becomes $\varepsilon^{1/2}\sum_{i=1}^n  e_i \chi(\varepsilon x_i, z)$, as seen from \eqref{e:defchi2}. Thus, specializing $z=v_1,v_2,\dots, v_p$ in \eqref{e:EPT} and summing, we get
\begin{align}\begin{split}\label{e:EPT2}
&\bE_{\bP^{\bmv,t\bms}}\left[\varepsilon^{1/2} \sum_{k=1}^p s_k\sum_{i=1}^n  e_i \chi(\varepsilon x_i, v_k)\right]
=\frac{1}{\theta \varepsilon^{1/2}}\sum_{k=1}^p \frac{s_k}{2\pi \ri}\oint_{\cin}\frac{\ln \cB(w)\sfb'(w)\sfb'(v_k)}{(\sfb(w)-\sfb(v_k))^2}\rd w\\
&+\frac{1}{\theta}\sum_{\ell=1}^p \sum_{k=1}^p\frac{ ts_k s_\ell}{2\pi \ri}\oint_{\cin}\frac{\cG(z)\varphi^+(w)}{\cB(w)}  \frac{\sfb'(w)\sfb'(v_k)\sfb'(w)\sfb'(v_\ell)}{(\sfb(w)-\sfb(v_k))^2(\sfb(w)-\sfb(v_\ell))^2}\rd w+\OO(\varepsilon^{1/2}),
\end{split}\end{align}
where the contour $\omega_-\subset \Lambda$ encloses $[\bl,\br]$, but not $ v_1, v_2, \dots, v_p$.
Plugging \eqref{e:EPT2} into \eqref{e:momentg}, and integrating, we get the joint Laplace transform of $\varepsilon^{1/2} \sum_{i=1}^n
\chi(\varepsilon x_i;v_k)$, $k=1,\dots,p$:
\begin{align}\begin{split}\label{e:momentg2}
&\ln \bE\left[\exp\left(\varepsilon^{1/2} \sum_{k=1}^p  s_k\sum_{i=1}^n
\chi(\varepsilon x_i;v_k)\right)\right]
=\frac{1}{\theta \varepsilon^{1/2}} \sum_{k=1}^p \frac{s_k}{2\pi \ri}\oint_{\cin}\frac{\ln \cB(w)\sfb'(w)\sfb'(v_k)\rd w}{(\sfb(w)-\sfb(v_k))^2}\\
&+\frac{1}{\theta}\sum_{\ell=1}^p \sum_{k=1}^p\frac{ s_k s_\ell}{4\pi \ri}\oint_{\cin}\frac{\cG(z)\varphi^+(w)}{\cB(w)}  \frac{\sfb'(w)\sfb'(v_k)\sfb'(w)\sfb'(v_\ell)}{(\sfb(w)-\sfb(v_k))^2(\sfb(w)-\sfb(v_\ell))^2}\rd w +\OO(\varepsilon^{1/2}),
\end{split}\end{align}
where the contour $\omega_-\subset \Lambda$ encloses $[\bl, \br]$, but not $ v_1, v_2, \dots, v_p$.

The first term in the right-hand side of \eqref{e:momentg2} corresponds to the expected value of $\sum_{i=1}^n
\chi(\varepsilon x_i;v_k)$ and disappears when we center these random variables. The second term matches the Laplace transform of the Gaussian vector with covariance given by \eqref{e:assymp_covariance}. Thus, it only remains to show that the convergence of the Laplace transforms in \eqref{e:momentg2} implies the convergence of the moments claimed in Proposition \ref{p:cov}. This is the content of the following abstract lemma, whose proof we omit.

\begin{lemma} Consider a family of real random variables $\{\xi^\eps_i\}_{i\in I}$ depending on a parameter $\eps>0$ and indexed by some set $I$. Suppose that there exists another family $\{\xi_i\}_{i\in I}$, such that for any $p=1,2,\dots$, indices $i_1,\dots,i_p\in I$, and real numbers $r_1,\dots,r_p$ we have
\begin{equation}
\label{e:Laplace_convergence}
 \lim_{\eps\to 0} \bE\left[ \exp\left(\sum_{k=1}^p r_k \xi^{\eps}_{i_k} \right)\right] =  \bE\left[ \exp\left(\sum_{k=1}^p r_k \xi_{i_k} \right)\right],
\end{equation}
where for each fixed $p$  and a compact set $\mathcal R \subset \mathbb R^p$, the convergence in \eqref{e:Laplace_convergence} is uniform over $r_1,\dots,r_p$ and over the choices of $i_1,\dots,i_p\in I$. In addition, suppose that the real random variables $\{\xi_i\}_{i\in I}$ satisfy a uniform bound: there exists a continuous function $f(r)$ such that
$$
 \bE \exp( r \xi_i) \leq  f(r),\qquad  \quad r\in\mathbb R, \quad i\in I.
$$
Then $\lim\limits_{\eps\to 0}\{\xi^\eps_i\}_{i\in I}=\{\xi_i\}_{i\in I}$ in the sense of moments uniformly in $i\in I$: for any $i_1,\dots,i_q\in I$, we have
$$
 \lim_{\eps\to 0} \bE\left[ \prod_{k=1}^q \xi_{i_k}^\eps\right]= \bE\left[ \prod_{k=1}^q \xi_{i_k}\right],
$$
where for each fixed $q$ the convergence is uniform over possible choices of $i_1,\dots,i_q\in I$.
\end{lemma}

This finishes the proof of Proposition \ref{p:cov} and, thus, also of Theorem \ref{t:loopstudy}.
\end{proof}
%We can conclude that the martingale difference terms $\Delta M_t(z)$ after rescaling by a factor $\sqrt n$, converges to a Gaussian field, with covariance structure
%\begin{align}
%\cov\left[\sqrt n \Delta M_t(z),\sqrt{n}\Delta M_t(z) \right]
%=\frac{1}{2\pi \ri}\oint_{\cin} \frac{f_t(w;\beta,t)}{f_t(w;\beta,t)+1}\frac{1}{(v_1-w)^2}\frac{1}{(v_2-w)^2}\rd w.
%%=-\frac{1}{2\pi \ri}\oint_{\cin} \frac{1}{f_t(w;\beta,t)+1}\frac{1}{(z_1-w)^2}\frac{1}{(z_2-w)^2}\rd w.
%\end{align}
%where the contour $\omega_-$ encloses $z^{-1}(\{x:(x,t+)\in{} \})$, but not $v_0, v_1$, and

\section{Law of Large Numbers for random tilings}

\label{Section_tilings_LLN}

The goal of this section is to prove Theorems \ref{t:arctic} and \ref{t:limitshape}.

\subsection{Setup}

\label{Section_tilings_setup}

We recall the trapezoid domain $\sfP$ from Section \ref{s:heightf}, with right side $[0,\sfN]$, width $\sfT$, and left sides $[\sfa_i, \sfb_i]$ for $1\leq i\leq r$, where $-\sfT\leq \sfa_1<\sfb_1<\sfa_2<\sfb_2<\dots<\sfa_r<\sfb_r\leq \sfN$, and
 $\sum_{i=1}^r(\sfb_i-\sfa_i)=\sfN$; see Figure \ref{Fig_trapezoid}.

As in Section \ref{Section_weighted_tilings_loop}, we ignore \La\, lozenges in tilings and trace only the lozenges of the remaining two types \Lb\, and \Lc. Then these lozenges form $N$ {non-intersecting paths} traversing the trapezoid from the left boundary to the right boundary, as shown in the right panel of Figure \ref{Fig_trapezoid} and in Figure \ref{Fig_trapezoid_MC}. At any time $0\leq t\leq T$, the nonintersecting paths correspond to an $N$ point configuration $\bmx(t)=(x_1(t)>x_2(t)>\dots>x_N(t))\in \bW^N$. Theorem \ref{Theorem_transition_qk} yields that for the $(q,\kappa)$-distribution \eqref{e:qtiling} on lozenge tilings of a trapezoidal domain, $\{\bmx(t)\}_{0\leq t\leq T}$ form a Markov chain with transition probability
\begin{align}\label{e:qtransit_2}
\bP(\bmx(t+1)=\bmx+\bme| \bmx(t)=\bmx)
\propto \prod_{i<j} \frac{b_{t}(x_i+e_i)-b_{t}(x_j+e_j)}{b_t(x_i)-b_t(x_j)}\prod_{i=1}^n\phi_t^+(x_i)^{e_i} \phi_t^-(x_i)^{1-e_i},
\end{align}
where $\bme=(e_1, e_2,\cdots, e_N)\in\{0,1\}^N$,
\begin{align}
\label{e:bt_def}
b_t(x)=q^{-x}+\kappa^2 q^{x-t},
\end{align}
and the weights
\begin{align*}
\phi_t^+(x)=q^{T+N-1-t}(1-q^{x-N+1})(1-\kappa^2 q^{x-T+1}), \quad \phi_t^-(x)=-(1-q^{x+T-t})(1-\kappa^2 q^{x+N-t}).
\end{align*}
At this point, we can forget about the lozenge tilings: our task is to study the Markov chain $\bmx(t)$ with the above transition probabilities and with initial condition
\begin{equation}
\label{e:initial_condition}
 \bmx(0)= (b_r-1,b_r-2,\dots,a_r+1,a_r,\,\, \dots,                                              \,\, b_1-1, b_1-2,\dots, a_1+1, a_1).
\end{equation}
Note that we need to impose restrictions on the $\bmx(0)$ and parameters $q$ and $\kappa$ in order to guarantee the positivity of probabilities \eqref{e:qtransit_2}; as in Section \ref{s:heightf} there are three possible cases: imaginary, real, and trigonometric.

We introduce a small parameter $\varepsilon>0$, and set
\begin{align}\label{e:defpp}
N=\varepsilon^{-1}\sfN, \quad T=\varepsilon^{-1}\sfT, \quad q=\sfq^{\varepsilon},\quad a_i=\varepsilon^{-1}\sfa_i, \quad b_i=\varepsilon^{-1}\sfb_i, \quad 1\leq i\leq r.
\end{align}
As in \eqref{eq_rho_x_def}, we encode $\bmx(t)$ through its smoothed empirical density:
\begin{align}
\label{eq_rho_x_t_def}
\rho(s;\bmx(t))=\sum_{i=1}^N \bm1(\varepsilon x_i \leq s\leq \varepsilon x_i+\varepsilon).
\end{align}
Note that the parameters $N$, $a_i$, $b_i$ and time $t$ should be integers. Being somewhat sloppy, we are going to ignore this and use the formulas \eqref{e:defpp} and \eqref{eq_rho_x_t_def} for the values of $t$ and $\eps$, which might lead to non-integer results. Formally, one should add integer parts to the formulas in these situations, but, for the brevity of notations, we are not going to do this. We also define
\begin{align}\label{e:defbt}
\sfb_\sft(z)=\sfq^{-z}+\kappa^2 \sfq^{z-\sft},
\end{align}
which is the version of \eqref{e:bt_def} with rescaled variables.

\begin{theorem} \label{Theorem_LLN_abstract} Consider the Markov chain $\bmx(t)=(x_1(t)>x_2(t)>\dots>x_N(t))$, $t=0,1,\dots,T$, on ordered $N$-tuples of integers with transition probabilities \eqref{e:qtransit_2}, a (deterministic) initial condition $\bmx(0)$ such that  $N>x_1(0)>\dots>x_N(0)\geq T$.  Suppose that the data depends on a small parameter $\eps>0$ in such a way that $N=\varepsilon^{-1}\sfN$, $T=\varepsilon^{-1}\sfT$, $q=\sfq^{\varepsilon}$. Fix $\delta>0$ and additionally suppose that the parameters are such that $\sfb_\sft(z)$ is real and satisfies $|\partial_z \sfb_\sft(z)|>\delta$ for $z\in [\max(\eps x_N(0),\sft-\sfT),\min(\eps x_1(0)+\sft,\sfN)]$ and $0\leq \sft\leq \sfT$. Finally, suppose that the measures $\rho(x;\bmx(0))dx$ weakly converge to a measure $\rho_0$.

 Then for each $0\leq \sft \leq \sfT$ the smoothed empirical measures converge to a deterministic limit
\begin{equation}
\label{e:weak_convergence}
 \lim_{\eps\to 0} \rho(\, \cdot\, ; \bmx(\eps^{-1} t))= \rho_\sft(\, \cdot\, ), \qquad \text{ weakly, in probability,}
\end{equation}
which means that for each continuous function $f(s)$, we have
$$
 \lim_{\eps\to 0} \int_{\mathbb R} f(s) \rho(s; \bmx(\eps^{-1} \sft)) \rd s = \int_{\mathbb R} f(s) \rho_\sft(s) \rd s, \qquad \text{ in probability.}
$$
The density $\rho_\sft(s)$ is supported inside $[\sft-\sfT, \sfN]$ and can be determined by either of the equivalent descriptions of Propositions \ref{Proposition_LLN_Burgers}, \ref{Proposition_LLN_equation}, or \ref{Proposition_LLN_invariant}.
\end{theorem}
\begin{remark} \label{Remark_two_forms_IC}
 When the initial condition $\bmx(0)$ is given by \eqref{e:initial_condition} with parameters scaling as in \eqref{e:defpp}, $\rho_0$ has the density
 $$
  \rho_0(s)=\begin{cases} 1, & s\in \cup_{i=1}^r [\sfa_i,\sfb_i],\\ 0, & \text{otherwise}.\end{cases}
 $$
 We also have a product formula for the exponentiated (modified) Stieltjes transform:
 \begin{multline} \label{e:complex_slope_comp}\begin{split}
   \exp\left( \int_{-\sfT}^{\sfN} \frac{\sfb'_0(z)\rho_0(s)}{\sfb_0(z)-\sfb_0(s)}\rd s\right)=\exp\left(\sum_{i=1}^r \int_{\sfa_i}^{\sfb_i} \frac{(-\kappa^{-2}\sfq^{-z}+\sfq^{z})\sfq^s \ln \sfq}{(\sfq^s-\sfq^{z})(\kappa^{-2}\sfq^{-z}- \sfq^{s})}\rd s\right)
   \\=\exp\left(\sum_{i=1}^r \int_{\sfa_i}^{\sfb_i} \left(\frac{1}{\sfq^{z}-\sfq^s}+\frac{1}{\sfq^s-\kappa^{-2}\sfq^{-z}}\right)\rd \sfq^s\right)=\prod_{i=1}^r  \left( \frac{1-\sfq^{\sfa_i-z}}{1-\sfq^{\sfb_i-z}} \cdot \frac{1-\kappa^{2}\sfq^{z+\sfb_i}}{1-\kappa^{2}\sfq^{z+\sfa_i}} \right).
\end{split} \end{multline}
\end{remark}

\begin{proposition} \label{Proposition_LLN_Burgers}
The densities $\rho_\sft(s)$ in Theorem \ref{Theorem_LLN_abstract} are uniquely determined by  the complex slope function\footnote{If we take $z$ to be a real number in the support of $\rho_\sft$, then this is the same complex slope as in \eqref{e:ftx}. However, in  \eqref{e:double_complex_slope} we instead allow complex arguments $z$.}
\begin{equation}
\label{e:double_complex_slope}
 f_\sft(z):= \exp\left(-\int_{\sft-\sfT}^{\sfN} \frac{\sfb'_\sft(z)(1-\rho_\sft(s))}{\sfb_\sft(z)-\sfb_\sft(s)}\rd s\right),
\end{equation}
where $z$ can be any complex number outside the set $\sfb_\sft^{-1}(\sfb_\sft([\sft-\sfT,\sfN]))$. Further, for each such $z$, $f_\sft(z)$ solves a partial differential equation (a relative of the complex Burgers equation):
\begin{align}\label{e:Burgers_first_appearance}
\del_\sft \ln f_\sft(z)+\del_z \ln\bigl(1-f_\sft(z)\bigr)=\ln(\sfq)\frac{\kappa^2 \sfq^{z-\sft}+\sfq^{-z}}{\kappa^2\sfq^{z-\sft}-\sfq^{-z}}.
\end{align}
\end{proposition}
\begin{remark}
 The density $1-\rho_\sft(s)$ has a transparent meaning in terms of the lozenge tilings: while $\rho_{\sft}(s)$ is the combined density of the lozenges of types \Lb\, and \Lc, $1-\rho_\sft(s)$ is the density of \La\, lozenges. By a computation similar to \eqref{e:complex_slope_comp} the complex slope function can be also rewritten as
\begin{equation}
\label{e:double_complex_slope_2}
 f_\sft(z)= \frac{ (1-\sfq^{\sfN-z})(1-\kappa^2 \sfq^{z-\sfT})}{(1-\sfq^{-z-\sfT+\sft})(1-\kappa^2 \sfq^{z+\sfN-\sft})} \exp\left( \int_{\sft-\sfT}^{\sfN} \frac{\sfb'_\sft(z)\rho_\sft(s)}{\sfb_\sft(z)-\sfb_\sft(s)}\rd s\right).
\end{equation}
\end{remark}

\begin{proposition} \label{Proposition_LLN_equation}
The complex slope functions \eqref{e:double_complex_slope} for the densities $\rho_\sft(s)$ in Theorem \ref{Theorem_LLN_abstract} can be computed as
 \begin{equation}
 \label{e:slope_through_u}
 f_\sft(z)
=\frac{(\sfq^{-u}+\kappa^2 \sfq^u)-(\sfq^{-z}+\kappa^2 \sfq^{z})}{(\sfq^{-u}+\kappa^2 \sfq^u)-(\sfq^{-z+\sft}+\kappa^2\sfq^{z-\sft})},
\end{equation}
where the function $u=u(\sft,z)$ solves the equation involving the complex slope of the initial condition $\rho_0$:
\begin{equation}
\label{e:slope_equation}
\sfq^{-z}+\kappa^2 \sfq^{z-\sft}=\frac{\sfq^{-u}-f_0(u)\kappa^2 \sfq^{u}  -\sfq^{-\sft}  (f_0(u)\sfq^{-u}-\kappa^2 \sfq^{u})}{1-f_0(u)}.
\end{equation}
\end{proposition}
\begin{remark} \label{remark_u_choice}The equation \eqref{e:slope_equation} might have several solutions and one should choose the correct branch of $u(\sft,z)$: the exact specification of this choice depends on $\sfq$ and $\kappa$ we work with. In Step 5 of Section \ref{Section_Limit_shape_proofs} we explain in details the specification under Assumption \ref{a:para}. In this case, if we deal with $z$ satisfying $0<\Im z<- \tfrac{\pi}{\ln \sfq}$, $\Re z \leq \sft/2-\log_\sfq \kappa$ (the values at all other $z$ can be reconstructed from these using symmetries of $f_\sft(z)$), then we can choose the unique $u$ such that
 $$
  0<\Im u<- \tfrac{\pi}{\ln \sfq},{ }  \Re u< -\log_\sfq \kappa, \text { and }\Im\left[\sfq^{-u}+\kappa^2 \sfq^{u}-(\sfq^{-\sft}-1) \frac{f_0(u)\sfq^{-u}-\kappa^2 \sfq^{u}}{1-f_0(u)}\right]<0.
 $$
\end{remark}

\begin{remark}
 Combination of \eqref{e:complex_slope_comp} and \eqref{e:double_complex_slope_2} gives $f_0(u)$ in \eqref{e:slope_equation} as an explicit product involving parameters \eqref{e:defpp} encoding the trapezoid. In this situation \eqref{e:slope_equation} turns into an algebraic equation in terms of the variable $\sfq^u$. This equation (after renaming $z$ into $x$) coincides with \eqref{e:solveu}.
\end{remark}

\begin{proposition}\label{Proposition_LLN_invariant}
Define a function $\cS_\sft(w)$, through the relation
\begin{align}\label{e:U_definition}
 \cS_\sft(\sfb_\sft(z))=\sfq^\sft \frac{f_\sft(z)\sfq^{-z}-\kappa^2 \sfq^{z-\sft}}{1-f_\sft(z)},\qquad  w=\sfb_\sft(z),
\end{align}
where $f_\sft(z)$ is as in \eqref{e:double_complex_slope}. Then:
\begin{enumerate}
 \item For each $0\leq \sft\leq \sfT$ there exist a compactly supported measure $\mu_\sft$ on $\mathbb R$ and a real constant $C_\sft$, such that
\begin{align}\label{e:U_measure}
\cS_\sft(w)=\frac{w}{\sfq^{-\sfT}-\sfq^{-\sft}}+C_\sft-\int_{\mathbb R}\frac{\rd \mu_\sft(s)}{w-s}.
\end{align}
\item The functions $\cS_\sft(w)$ for different values of $\sft$ are linked by the functional relation
\begin{equation}
\label{e:U_functional_statement}
\cS_\sft^{(-1)}(v)+\sfq^{\sft} v=\cS_0^{(-1)}(v)+v,
\end{equation}
where $\cS_\sft^{(-1)}(v)$ is the functional inverse of $\cS_\sft$, which satisfies $\cS_\sft^{(-1)}(v)=(\sfq^{-\sfT}-\sfq^{-\sft})v+\oo(v)$ as $v\to\infty$.
\end{enumerate}
\end{proposition}
\begin{remark}
 The function $\cS_\sft$ will be produced using the first integral \eqref{e:defu_1} for the characteristics of the PDE \eqref{e:Burgers_first_appearance}, but we could also use another first integral \eqref{e:defv_1} or, perhaps, the ratio of them. In each situation we would be getting a different version of \eqref{e:U_functional_statement} and a different definition of $\mu_\sft$. Finding the best form of the functional relation is an interesting open problem.
\end{remark}
\begin{remark}
 The relation \eqref{e:U_functional_statement} has similarities with \cite[Eq.\ (2.9) and Proposition 3.4]{Bu_G_quantized}, which describes the limit shape in the case of the uniform measure on tilings (which can be obtained by setting $q=\sfq=1$ in \eqref{eq_lozenge_weight}). The same paper also described how such formulas can be transformed to the formulas for the free projection in the free probability theory by the means of the Markov-Krein transform. It would be interesting to find an analogue of such transform in the present $(q,\kappa)$--setting.
\end{remark}

Several possible choices of $\sfq$ and $\kappa$ make $b_t(x)$ a monotone function of $x$ in Theorem \ref{Theorem_LLN_abstract}, which guarantees that \eqref{e:qtransit_2} is positive and corresponds to the condition $|\partial_z \sfb_\sft(z)|>\delta$ in the statement of the theorem. The choices correspond to real, imaginary, and trigonometric cases outlined in Section \ref{s:heightf}. In our proofs, we are only going to detail (a subcase of) the real case by imposing the following restriction:

\begin{assumption}\label{a:para}
We require $0<\sfq<1$ and we require $\kappa$ to be a positive real number such that for some fixed $\delta>0$ and all small enough $\eps>0$ we have
$$\kappa^2>\max_{0\leq \sft \leq \sfT}\left[\sfq^{-2\min(\eps x_1(0)+\sft/2,\sfN-\sft/2)}\right]+\delta
=\sfq^{-\sfN-\varepsilon x_1(0)}+\delta
.$$
\end{assumption}

The arguments for other possible choices of $\sfq$ and $\kappa$ in the imaginary and real cases are similar. The trigonometric case needs additional efforts and explanations. In order to keep our presentation concise, we excluded it from Theorem \ref{Theorem_LLN_abstract} by requiring $\sfb_\sft(x)$ to be a real function of real $x$ there.

\bigskip

The proofs of Theorem \ref{Theorem_LLN_abstract} and Propositions \ref{Proposition_LLN_Burgers}, \ref{Proposition_LLN_equation},  \ref{Proposition_LLN_invariant} are presented in the rest of this section: in Subsection \ref{Section_assumptions} we check that Theorem \ref{t:loopstudy} is applicable by verifying its assumptions; in Subsection \ref{Section_massaging} we transform the conclusion of Theorem \ref{t:loopstudy} into the form reminiscent to the PDE \eqref{e:Burgers_first_appearance}; in Subsection \ref{Section_Limit_shape_proofs} we solve this PDE and show that  Propositions \ref{Proposition_LLN_Burgers}, \ref{Proposition_LLN_equation}, and \ref{Proposition_LLN_invariant} are equivalent; in Subsection \ref{Section_proof_of_abstract_LLN} we finish the proof of Theorem \ref{Theorem_LLN_abstract}; and in Subsection \ref{Section_proof_of_LLN_intro} we deduce Theorems \ref{t:arctic} and \ref{t:limitshape} as its corollaries.

\subsection{Assumptions in Theorem \ref{t:loopstudy}}

\label{Section_assumptions}

Our proofs in this section rely on Theorem \ref{t:loopstudy} with $\theta=1$ applied to transition probabilities \eqref{e:qtransit_2}. Here we check that the required assumptions are satisfied and record the conclusion of the theorem. We start by introducing various notations.

We use three sets of functions:
\begin{align}
\notag&\phi_t^+(x)=q^{T+N-1-t}(1-q^{x-N+1})(1-\kappa^2 q^{x-T+1}), & &\phi_t^-(x)=-(1-q^{x+T-t})(1-\kappa^2 q^{x+N-t}),
\\
\notag &\varphi^+_\sft(z)=\phi_{\sft/\eps}^+(z/\eps), & & \varphi^-_\sft(z)=\phi_{\sft/\eps}^-(z/\eps),
\\
&\tilde \varphi^+_\sft(z)=\sfq^{\sfT+\sfN-\sft}(1-\sfq^{z-\sfN})(1-\kappa^2 \sfq^{z-\sfT}),&
&\tilde \varphi_\sft^-(z)=-(1-\sfq^{z+\sfT-\sft})(1-\kappa^2 \sfq^{z+\sfN-\sft}). \label{e:phi_functions_defs}
\end{align}
The relation between $\phi_t^\pm$ and $\varphi_\sft^\pm$ is the same as between $\phi^\pm$ and $\varphi^\pm$ in the setting of Theorem \ref{t:loopstudy}. On the other hand, $\tilde \varphi_\sft^\pm$ do not have shifts by $1$ appearing in $\phi_t^\pm$. Let us also define
$$
{\psi^{+}_\sft}(z):=\ln(\sfq) \left( -\sfq^{\sfT+\sfN-\sft}+\kappa^2 \sfq^{2z-\sft} \right). %=-\ln(\sfq) \left(\frac{\sfq^{z-\sfN}}{1-\sfq^{z-\sfN}}+\frac{\kappa^2\sfq^{z-\sfT}}{1-\kappa^2\sfq^{z-\sfT}}+1\right)\tilde \varphi^+_\sft(z).
$$
Then we have
$$
 \varphi^+_\sft(z)=\tilde \varphi^+_\sft(z)+ \eps \psi^+_\sft(z)+ \OO(\eps^2),\qquad \qquad \varphi^-_\sft(z)=\tilde \varphi^-_\sft(z).
$$
Note that the function $\sfb_\sft(z)=\sfq^{-z}+\kappa^2 \sfq^{z-\sft}$ is symmetric under $z\mapsto -z+\sft-2\log_\sfq\kappa$:
\begin{align}\label{e:btsym}
\sfb_\sft(z)=\sfb_\sft(-z+\sft-2\log_\sfq\kappa),\quad \sfb'_\sft(z)=-\sfb'(-z+\sft-2\log_\sfq\kappa).
\end{align}
An advantage of the functions $\tilde \varphi^\pm_\sft(z)$ over $\varphi^\pm(z)$ is that they agree with this symmetry:
\begin{align}\label{e:phtsim}
\frac{\tilde \varphi_\sft^-(-z+\sft-2\log_\sfq \kappa)}{\tilde \varphi_\sft^+(z)}=\frac{\tilde \varphi_\sft^+(-z+\sft-2\log_\sfq \kappa)}{\tilde \varphi_\sft^-(z)}=-\frac{\sfq^{\sft-2z}}{\kappa^2}.
\end{align}

\bigskip

Next, we recall the smoothed empirical density $\rho(s;\bmx(t))$ of $\bmx(t)$ given by \eqref{eq_rho_x_t_def} and note its support:
\begin{equation}\label{e:defatbt}
\supp(\rho(s;\bmx(\sft/\eps)))\subset[\bl(\sft),\br(\sft)],\quad \bl(\sft):=\max(\eps x_N(0), \sft -\sfT),\quad \br(\sft):=\min(\sfN, \eps x_1(0)+\sft+\eps).
\end{equation}
We need several integral transforms of $\rho(s;\bmx(t))$ defined for complex $z$ outside $\sfb_\sft^{-1}(\sfb_\sft([\bl(\sft),\br(\sft)]))$:
\begin{equation}
\label{e:x28}
 \cG_\sft(z):=\exp\left(\int_{\bl(\sft)}^{\br(\sft)} \frac{\sfb_\sft'(z)\rho(s;\bmx(\sft/\eps))}{\sfb_\sft(z)-\sfb_\sft(s)}\rd s \right), \qquad \tilde f_\sft(z):=-\cG_{\sft}(z)\frac{\tilde \varphi_\sft^+(z)}{\tilde \varphi_\sft^-(z)},
\end{equation}
\begin{equation} \label{e:x4}
\cB_\sft(z):=\cG_\sft(z) \varphi_\sft^+(z)+\varphi_\sft^-(z), \qquad
\tilde\cB_\sft(z):=\cG_\sft(z)\tilde \varphi_\sft^+(z)+\tilde \varphi_\sft^-(z)=\tilde \varphi_\sft^-(z)(1-\tilde f_\sft(z)).
\end{equation}
Let us emphasize that since $\bmx(t)$ is random, so are all the above integral transforms. $\cB_\sft(z)$ is the function which we eventually need, because it enters Assumption \ref{a:stable}. However, it is simpler to analyze $\tilde f_\sft(z)$ and $\tilde \cB_\sft(z)$, because of the following symmetries under $z\mapsto -z+\sft-2\log_\sfq\kappa$ transformations: using \eqref{e:btsym} we obtain
$$
 \cG_\sft(-z+\sft-2\log_\sfq\kappa)= \frac{1}{ \cG_\sft(z)};
$$
further, using \eqref{e:phtsim} we get
$$
  \tilde f_\sft(-z+\sft-2\log_\sfq\kappa)=-\cG_{\sft}(-z+\sft-2\log_\sfq\kappa)\frac{\tilde \varphi_\sft^+(-z+\sft-2\log_\sfq\kappa)}{\tilde \varphi_\sft^-(-z+\sft-2\log_\sfq\kappa)}=-\frac{1}{\cG_{\sft}(z)}\frac{\tilde \varphi_\sft^-(z)}{\tilde \varphi_\sft^+(z)}=\frac{1}{ \tilde f_\sft(z)};
$$
and, finally, again using \eqref{e:phtsim} we get
\begin{align} \label{e:x13}
\notag \tilde \cB_\sft(-z+\sft-2\log_\sfq\kappa)&=\cG_\sft(-z+\sft-2\log_\sfq\kappa)\tilde \varphi_\sft^+(-z+\sft-2\log_\sfq\kappa)+\tilde \varphi_\sft^-(-z+\sft-2\log_\sfq\kappa)\\&=-\frac{\sfq^{\sft-2z}}{\kappa^2 \cG_\sft(z)}\left(\tilde \varphi_\sft^-(z)+ \cG_\sft(z)\tilde \varphi_\sft^+(z)  \right)= -\frac{\sfq^{\sft-2z}}{\kappa^2 \cG_\sft(z)} \tilde \cB_\sft(z).
\end{align}

\begin{proposition} \label{Proposition_off_criticality} In the framework of Theorem \ref{Theorem_LLN_abstract}, for each $t=0,1,2\dots,T-1$ transition probabilities \eqref{e:qtransit_2} satisfy the requirements of Theorem \ref{t:loopstudy}, which means that Assumptions \ref{a:asymp} and \ref{a:stable} hold almost surely for $\bmx=\bmx(t)$.
\end{proposition}
\begin{proof}
 We fix $t$ with $\sft=\eps t$ and set in Theorem \ref{t:loopstudy} $\theta=1$, $b=b_t$, $\sfb=\sfb_\sft$, $\phi^\pm=\phi_t^\pm$, $\varphi^\pm=\varphi_\sft^\pm$, $\bl=\bl(\sft)$, $\br=\bf(\sft)$, and $\cB(z)=\cB_\sft(z)$. For $\bmx$ we are going to choose all possible values of $\bmx(t)$ which arise with positive probabilities. For the set $\Lambda$ we choose a large positive constant $C$ and define
 \begin{equation}
 \label{e:x10}
  \Lambda=\left\{z\in\mathbb C \, \Big|\, -\tfrac{\pi}{\ln(\sfq)}<\Im(z)<\tfrac{\pi}{\ln(\sfq)}, \quad -C< \Re(z)< \tfrac{\sft}{2}-\log_\sfq(\kappa)-1/C\right\}.
 \end{equation}

 Checking Assumption \ref{a:asymp} is straightforward and we only focus on Assumption \ref{a:stable}. For that we first check the conditions of the latter assumption for the function $\tilde \cB_\sft(z)$ rather than $\cB_\sft(z)$. For the check we need to identify all the zeros and singularities of $\tilde \cB_\sft(z)$, which is split into several steps.

\smallskip

\noindent {\bf Step 1.} We would like to understand the solutions to $\tilde f_t(z)=1$ for $z$ such that $\sfq^z$ is outside the set $[\sfq^{\br},\sfq^\bl]\cup [\sfq^{\sft-\bl}\kappa^{-2}, \sfq^{\sft-\br} \kappa^{-2}]$. We claim that the solutions are at $z$ satisfying $\sfq^z=\pm \sfq^{\sft/2} \kappa^{-1}$ and $\tilde f_t'(z)\ne 0$ at these points. First, note that for such values of $\sfq^z$ we have $\sfb'_\sft(z)=0$,  $\cG_\sft(z)=1$, $\tilde \varphi^-_\sft(z)=-\tilde \varphi^+_\sft(z)$, and therefore $\tilde f_t(z)=1$. In order to show that there are no other solutions, we rewrite $\cG_\sft(z)$ and $\tilde f_\sft(z)$ using the same computation as in \eqref{e:complex_slope_comp} and $[\bl,\br]\subset[\sft-\sfT,\sfN]$:
\begin{align}
\label{e:x3}\cG_\sft(z)&=\exp\left(\int_{\bl}^{\br} \frac{\sfb_\sft'(z)\rho(s;\bmx)}{\sfb_\sft(z)-\sfb_\sft(s)}\rd s\right)
=\exp\left( \int_{\bl}^{\br} \left(\frac{\rd \sfq^s}{\sfq^z-\sfq^s}+\frac{\rd \sfq^s}{\sfq^s-\sfq^{-z+\sft-2\log_\sfq \kappa}}\right)\rho(s;\bmx)\right),\\ \label{e:x7}
 -\frac{\tilde \varphi_\sft^+(z)}{\tilde \varphi_\sft^-(z)}&=\frac{(1-\sfq^{-z+\sfN})(1-\kappa^2 \sfq^{z-\sfT})}{(1-\sfq^{-z-\sfT+\sft})(1-\kappa^2 \sfq^{z+\sfN-\sft})}=\exp\left(- \int_{\sft-\sfT}^{\sfN} \left(\frac{\rd \sfq^s}{\sfq^z-\sfq^s}+\frac{\rd \sfq^s}{\sfq^s-\sfq^{-z+\sft-2\log_\sfq \kappa}}\right)\right),\\
 \tilde f_t(z)&=-\frac{\tilde \varphi_\sft^+(z)}{\tilde \varphi_\sft^-(z)}\tilde \cG_\sft(z)=\exp\left(\int_{\sft-\sfT}^{\sfN} \left(\frac{\rd \sfq^s}{\sfq^s-\sfq^z}-\frac{\rd \sfq^s}{\sfq^s-\sfq^{-z+\sft-2\log_\sfq \kappa}}\right)(1-\rho(s;\bmx))\right). \label{e:x2}
\end{align}
Denote
$$
 G(Z)=\int_{\sfq^N}^{\sfq^{\sfT-\sft}} \frac{ 1-\rho(\log_\sfq(S);\bmx)}{Z-S} \rd S,
$$
and notice that under the change of variable $Z=\sfq^z$, we have
\begin{equation}\label{e:x6}
 \tilde f_t(z)=\exp\left(G\Bigl(Z\Bigr)-G\Bigl(\frac{\sfq^{t}\kappa^{-2}}{Z}\Bigr)\right) .
\end{equation}
For real $S$ and complex $Z=Z_1+\ii Z_2$, we have
$$
 \Im\left(\frac{1}{Z-S}\right)=-\frac{Z_2}{(S-Z_1)^2 + Z_2^2}.
$$
Hence, for $Z$ with positive imaginary part we have
$$
 0> \Im G(Z) > -\int_{-\infty}^{\infty} \frac{Z_2}{(S-Z_1)^2 + Z_2^2} \rd S=-\pi.
$$
Similarly, for $Z$ with negative imaginary part we have $0<\Im G(Z)<\pi$. Therefore, the expression under exponent in \eqref{e:x6} has imaginary part between $-2\pi$ and $0$ if $\Im(Z)>0$ and has imaginary part netween $0$ and $2\pi$ if $\Im(Z)<0$. We conclude that for $\tilde f_t(z)=1$ it is necessary for $Z$ to be real.

At this point the argument splits into two cases. Recall that we work under Assumption \ref{a:para} of large real $\kappa$. The first case is that $\kappa^2> \sfq^{\sft-2\sfN}$, which means that the intervals $[\sfq^{\sfT} \kappa^{-2}, \sfq^{\sft-\sfN} \kappa^{-2}]$ and $[\sfq^\sfN,\sfq^{\sft-\sfT}]$ are disjoint (the intervals are images of each other under $Z\leftrightarrow \tfrac{\sfq^{t}\kappa^{-2}}{Z}$). The second case is when these two intervals intersect (note, however, that we always have the ordering $\sfq^{\sfT} \kappa^{-2}<\sfq^\sfN$ and $\sfq^{\sft-\sfN} \kappa^{-2}<\sfq^{\sft-\sfT}$).

In the first case we note that directly from \eqref{e:x6} $\tilde f_t(z)$ is a decreasing function of $Z\in\mathbb R_{<0}$; hence, there is at most one solution to $\tilde f_t(z)=1$ for $Z\in\mathbb R_{<0}$ and this solution is at $Z=-\sfq^{\sft/2} \kappa^{-1}$. Simultaneously, \eqref{e:x6} implies that the logarithmic derivative (and hence, also the usual derivative) of $\tilde f_t$  is bounded away from $0$ at this $Z$. Furthermore, $\tilde f_t(Z)$ is a decreasing function of $Z\in [\sfq^{\sft-\sfN} \kappa^{-2},  \sfq^\sfN]$; hence, there is at most one solution to $\tilde f_t(z)=1$ on this segment and this solution is at  $Z=\sfq^{\sft/2} \kappa^{-1}$, with non-zero derivative of $\tilde f_t$. There are no other real $Z$ leading to solutions of $\tilde f_t(z)=1$: For $0<Z<\sfq^{\sfT} \kappa^{-2}$ both terms in the exponent in \eqref{e:x6} are negative; for $Z>\sfq^{\sft-\sfT}$ both terms in the exponent are positive; for $Z\in [\sfq^{\sfT} \kappa^{-2}, \sfq^{\sft-\sfN} \kappa^{-2}]\setminus [\sfq^{\sft-\bl}\kappa^{-2}, \sfq^{\sft-\br} \kappa^{-2}]$ and for $Z\in [\sfq^\sfN,\sfq^{\sft-\sfT}]\setminus [\sfq^\br,\sfq^\bl]$ the values of $\tilde f_t(z)$ are negative, because both $\cG_\sft(z)$ and $\tfrac{\tilde \varphi_\sft^+(z)}{\tilde \varphi_\sft^-(z)}$ are positive.

In the second case of intersecting  intervals $[\sfq^{\sfT} \kappa^{-2}, \sfq^{\sft-\sfN} \kappa^{-2}]$ and $[\sfq^\sfN,\sfq^{\sft-\sfT}]$,  the argument remains mostly the same, except for the new case $Z\in [\sfq^{\sfN},  \sfq^{\sft-\sfN} \kappa^{-2}]$ replacing the previous $Z\in [\sfq^{\sft-\sfN} \kappa^{-2},\sfq^{\sfN}]$. In this new case, we can no longer rely on \eqref{e:x6}, because we are at the singularities of the integrands and the meaning of the integrals are unclear; however we can still rely on \eqref{e:x3} because $Z$ is between the intervals $[\sfq^{\sft-\bl}\kappa^{-2}, \sfq^{\sft-\br} \kappa^{-2}]$ and $[\sfq^{\br},\sfq^\bl]$, we can also use the first formula of \eqref{e:x7}. We, thus, modify \eqref{e:x6} into
\begin{equation}\label{e:x8}
 \tilde f_t(z)=\frac{H(Z)}{H\Bigl(\frac{\sfq^{t}\kappa^{-2}}{Z}\Bigr)},
\end{equation}
where $Z=\sfq^z$, and $H(Z)$ is
\begin{equation}
\label{e:x9}
\frac{Z-\sfq^{-\sfN}}{\sfq^{-\sfT+\sft}-Z}\exp\left( \int_{\bl}^{\br} \frac{\rd \sfq^s}{Z-\sfq^s} \rho(s;\bmx)\right)=\exp\left(\ln\left[\frac{Z-\sfq^{-\sfN}}{\sfq^{\br}-Z}\right]-\int_{\sft-\sfT}^{\br} \frac{\rd \sfq^s}{Z-\sfq^s} (1-\rho(s;\bmx))\right).
\end{equation}
Notice that for $Z\in [\sfq^{\sfN},  \sfq^{\sft-\sfN} \kappa^{-2}]$
\begin{align*}
 \partial_Z \ln H(Z)&=\frac{1}{Z-\sfq^{-\sfN}}+ \frac{1}{\sfq^{\br}-Z}+\int_{\sft-\sfT}^{\br} \frac{\rd \sfq^s}{(Z-\sfq^s)^2} (1-\rho(s;\bmx))\\
 &\geq \frac{1}{Z-\sfq^{-\sfN}}+ \frac{1}{\sfq^{\br}-Z}+\int_{\sft-\sfT}^{\br} \frac{\rd \sfq^s}{(Z-\sfq^s)^2}
 \\&=\frac{1}{Z-\sfq^{-\sfN}}+  \frac{1}{\sfq^{\sft-\sfT}-Z}>0.
\end{align*}
We conclude that $H(Z)$ is increasing in $Z$ and so is $H\Bigl(\frac{\sfq^{t}\kappa^{-2}}{Z}\Bigr)^{-1}$. Therefore, $\tilde f_t(z)$ is increasing in $Z$, and the equation $\tilde f_t(z)=1$ is solved on $Z\in [\sfq^{\sfN},  \sfq^{\sft-\sfN} \kappa^{-2}]$ only by $Z=\sfq^{\sft/2} \kappa^{-1}$. The above bounds on the derivative of $\ln H(Z)$ imply that the derivative of $\tilde f_t$ at $Z=\sfq^{\sft/2} \kappa^{-1}$ does not vanish.

\medskip

\noindent {\bf Step 2.} We rewrite $\tilde \cB_\sft(z)= \mathfrak B_\sft(\sfq^z)$ and analyze the zeros and singularities of $\mathfrak B_\sft(Z)$. We claim that
\begin{itemize}
 \item All singularities of  $\mathfrak B_\sft(Z)$ are inside the set $[\sfq^{\br},\sfq^{\bl}]\cup [\sfq^{-\bl+\sft} \kappa^{-2}, \sfq^{-\br+\sft} \kappa^{-2}]$.
 \item The zeros of $\mathfrak B_{\sft}(Z)$ outside this set are at points $Z=\pm \sfq^{\sft/2} \kappa^{-1}$.
\end{itemize}
Observe that the singularities of $\mathfrak B_\sft(\sfq^z)$ are a subset of singularities of $\cG_\sft(z)$, and the latter are singularities of the integral in \eqref{e:x3}.

For the zeros, we note that according to the definition \eqref{e:x4}, $\mathfrak B_\sft(\sfq^z)$ vanishes whenever either $\tilde \varphi_\sft^+(z)=\tilde \varphi_\sft^-(z)=0$ or $\tilde f_\sft(z)=1$. In the first case, since zeros of $\tilde \varphi_\sft^+(z)$ are at $\sfq^z\in\{\sfq^\sfN, \sfq^\sfT\kappa^{-2}\}$ and zeros of $\tilde \varphi_\sft^-(z)$ are at $\sfq^z\in\{\sfq^{\sft-\sfT}, \sfq^{\sft-\sfN}\kappa^{-2}\}$, we should have either $\sfq^{z}=\sfq^\sfN= \sfq^{\sft-\sfN}\kappa^{-2}$ or $\sfq^z=\sfq^\sfT\kappa^{-2}=\sfq^{\sft-\sfT}$; in both situations the equality requires special values of $\kappa$ and implies $w=\sfq^z= \sfq^{\sft/2} \kappa^{-1}$. The second case $\tilde f_\sft(z)=1$ was analyzed in Step 1.

\medskip

{\bf Step 3.} We now check Assumption \ref{a:stable} for $\tilde \cB_\sft(z)$. In Step 2 we identified all zeros and singularities of $\tilde \cB_\sft(z)$. Comparing with \eqref{e:x10} we conclude that they are all outside $\Lambda\setminus [\bl,\br]$; we recall $\Lambda$ from \eqref{e:x10}. Hence, $\tilde \cB_\sft(z)$ is bounded away from $0$ and $\infty$ on compact subsets of  $\Lambda\setminus [\bl,\br]$, as claimed in the assumption (the constant $c$ can be chosen to be independent of $\bmx$). It remains to check the second condition, i.e.\ we want to verify
\begin{align}\label{e:x11}
\frac{1}{2\pi \ri}\oint_{\omega}\frac{\tilde \cB'_\sft(z)}{\tilde \cB_\sft(z)}\rd z\stackrel{?}{=}0,
\end{align}
where $\omega$ is an arbitrary simple contour inside $\Lambda\setminus [\bl,\br]$, enclosing the segment $[\bl,\br]$.
Let $\omega'$ be an image of $\omega$ under the map $z\mapsto -z+ \sft-2\log_\sfq \kappa$
and assume that both $\omega$ and $\omega'$ are positively oriented.
Then \eqref{e:x11} can be transformed using \eqref{e:x13} and its $z$--derivative into
\begin{equation}
\label{e:x12}
-\frac{1}{2\pi \ri}\oint_{\omega'}\frac{\tilde \cB'_\sft(-z+ \sft-2\log_\sfq \kappa)}{\tilde \cB_\sft(-z+ \sft-2\log_\sfq \kappa)}\rd z=\frac{1}{2\pi \ri}\oint_{\omega'}\frac{\tilde \cB'_\sft(z)}{\tilde \cB_\sft(z)}\rd z+ \frac{1}{2\pi \ri}\oint_{\omega'} \partial_z\ln \left(\frac{\sfq^{\sft-2z}}{\kappa^2 \cG_\sft(z)}\right) \rd z.
\end{equation}
We claim that the second integral in the right-hand side vanishes. Indeed, the definition of $\cG_\sft(z)$ readily implies that $ \ln \left(\frac{\sfq^{\sft-2z}}{\kappa^2 \cG_\sft(z)}\right)$ is a holomorphic function on the contour $\omega'$; hence, the contour integral of its derivative vanishes.
Combining \eqref{e:x11} with its equivalent form \eqref{e:x12}, we rewrite the desired identity as
\begin{align*}%\label{e:x13}
\frac{1}{2\pi \ri}\oint_{\omega\cup \omega'}\frac{\tilde \cB'_\sft(z)}{\tilde \cB_\sft(z)}\rd z\stackrel{?}{=}0,
\end{align*}
As in Step 2, we change variables and rewrite $\tilde \cB(z)=\mathfrak B_\sft(Z)$, $Z=\sfq^z$. We have,
\begin{equation}
\label{e:x123}
\frac{1}{2\pi \ri}\oint_{\omega\cup \omega'}\frac{\del_z\tilde \cB_\sft(z)}{\tilde \cB_\sft(z)}\rd z=\frac{1}{2\pi \ri}\oint_{\omega\cup \omega'}\frac{\del_z\mathfrak B_\sft(\sfq^z)}{\mathfrak B_\sft(\sfq^z)}\rd z
=\frac{1}{2\pi \ri}\oint_{\sfq^{\omega}\cup \sfq^{\omega'}}\frac{\mathfrak B'_\sft(Z)}{\mathfrak B_\sft(Z)}\rd Z.
\end{equation}
We compute the last integral as (minus) the sum of the residues outside the integration contour. These residues are:
\begin{itemize}
 \item Simple zeros of $\mathfrak B_\sft(Z)$ at points $Z=\pm \sfq^{\sft/2} \kappa^{-1}$ as identified in Steps 1 and 2, lead to two poles in the integrand, each one coming with residue $-1$ for the integrand.
 \item At $Z=0$ the function $\mathfrak B_\sft(Z)$ is holomorphic and there is no residue for the integrand.
 \item As $Z\to \infty$ the function $\mathfrak B_\sft(Z)$ grows as $c Z^2$ for a certain non-zero constant $c$, giving the residue $2$ for the integrand.
\end{itemize}
Altogether, the residues cancel out and \eqref{e:x123} integrates to $0$.

\medskip

{\bf Step 4.} We come back to $\cB_\sft(z)$ as in \eqref{e:x4} and check that it satisfies Assumption \ref{a:stable}. We have
\begin{equation}
\label{e:x14}
 \cB_\sft(z)=\tilde \cB_\sft(z)+\eps\psi^+_\sft(z)\cG_\sft(z)+\OO(\eps^2).
\end{equation}
Hence, because $\tilde \cB_\sft(z)$ is bounded away from $0$ and $\infty$ on compact subsets of $\Lambda\setminus[\bl,\br]$, so is $\cB_\sft(z)$ for small enough $\eps$.

Next, we note that the value of the contour integral $\frac{1}{2\pi \ri}\oint_{\omega}\frac{\cB'_\sft(z)}{\cB_\sft(z)}\rd z$ is an integer, because this is the integral of the derivative of the multi-valued function $\frac{1}{2\pi\ii}\ln (\cB_\sft(z))$, whose values at a point differ by integers from each other (equivalently, the integral counts the winding number of the curve $\cB_\sft(\omega)$). On the other hand, \eqref{e:x14} and the result of Step 3 imply that
$$
  \frac{1}{2\pi \ri}\oint_{\omega}\frac{\cB'_\sft(z)}{\cB_\sft(z)}\rd z=\frac{1}{2\pi \ri}\oint_{\omega}\frac{\tilde \cB'_\sft(z)}{\tilde \cB_\sft(z)}\rd z+\OO(\eps)=\OO(\eps).
$$
Hence, for small $\eps$ the contour integral vanishes.
%&=\exp\left(\int_{\sft-\sfT}^{\sfN} \frac{\rho(s;\bmx) }{\sfq^z-\sfq^s}  \rd \sfq^s+\int_{\sft-\sfN-2\log_\sfq\kappa}^{\sfT-2\log_\sfq\kappa}  \frac{\rho(\sft-2\log_\sfq \kappa-s;\bmx)}{\sfq^{\sft-2\log_\sfq \kappa-s}-\sfq^{-z+\sft-2\log_\sfq \kappa}} \rd \sfq^{\sft-2\log_\sfq \kappa-s}\right)=
%\\
%&=\exp\left(\int_{\sft-\sfT}^{\sfN} \frac{\rho(s;\bmx) }{\sfq^z-\sfq^s} \rd \sfq^s+\int_{\sft-\sfN-2\log_\sfq\kappa}^{\sfT-2\log_\sfq\kappa}  %\frac{\rho(\sft-2\log_\sfq \kappa-s;\bmx)}{\sfq^z-\sfq^{s}} \rd \sfq^s+\sfN\ln \sfq\right).
%\end{split}\end{align}
%\textcolor{green}{[Continue checking! - also the sign might be wrong?]}
%\begin{multline}
%\ln q \int (\tilde \rho_t(t-2\log_q \kappa-x) q^{-x} dx)/(q^{-x}-q^{-z})
%\\=\ln q \int \tilde \rho_t(t-2\log_q \kappa-x) dx+\ln q \int (\tilde \rho_t(t-2\log_q \kappa-x) q^{-z} dx)/(q^{-x}-q^{-z})
%\\=N\ln q +\ln q \int (\tilde \rho_t(t-2\log_q \kappa-x) q^{x} dx)/(q^{z}-q^{x})
%\\=N\ln q + \int (\tilde \rho_t(t-2\log_q \kappa-x) dq^x)/(q^{z}-q^{x}),
%\end{multline}
\end{proof}

\begin{remark} The proof of \eqref{e:x11} in Proposition \ref{Proposition_off_criticality} might look mysterious and one can be wondering about simpler reasons to expect this identity to be true. If we replace $\tilde \cB_\sft(z)$ by $B_\sft(z):=\tilde \varphi^-_\sft(z)(1-f_\sft(z))$, where we used the form of \eqref{e:x4} with complex slope of \eqref{e:double_complex_slope}, \eqref{e:double_complex_slope_2}, then the vanishing of the contour integral can be related to the version of the complex Burger's equation of \eqref{e:Burgers_first_appearance}. Indeed, this equation gives
\begin{align*}
&\phantom{{}={}}\partial_z \ln B_\sft(z)
=-\partial_\sft \ln f_\sft +\partial_z \ln \tilde \varphi_\sft^-+\ln \sfq \frac{\kappa^2 \sfq^{z-\sft}+\sfq^{-z}}{\kappa^2 \sfq^{z-\sft}-\sfq^{-z}}
\\&=-\partial_\sft \left[\int_{\bl}^{\br} \frac{\sfb'_\sft(z)\rho_\sft(s)}{\sfb_\sft(z)-\sfb_\sft(s)}\rd s \right]-\partial_\sft \ln \tilde\varphi^+_\sft +\partial_\sft \ln \tilde \varphi^-_\sft +\partial_z \ln \tilde\varphi_\sft^-+\ln \sfq \frac{\kappa^2 \sfq^{z-\sft}+\sfq^{-z}}{\kappa^2 \sfq^{z-\sft}-\sfq^{-z}}.
\end{align*}
The last term and $-\partial_\sft \ln \tilde \varphi^+_\sft$ are analytic in a neighborhood of $[\bl,\br]$, thus, their contour integrals vanish. An important cancellation is $\partial_\sft \ln \tilde\varphi^-_\sft +\partial_z \ln \tilde\varphi_\sft^-=0$, because $\tilde\varphi_\sft^-$ is a function of $z-\sft$.
The only thing left is the $\rd s$--integral; its contour integral gives the time derivative of the total mass of $\rho_t$, which is zero.
\end{remark}

Let us record the conclusion of Theorem \ref{t:loopstudy}. In our present setting \eqref{e:dmg} implies that for $\sft=t\eps$, $\Lambda_\sft$ given by $\Lambda$ of \eqref{e:x10}, and $z\in \Lambda_\sft\setminus [\bl(\sft),\br(\sft)]$ we have
\begin{multline}\label{e:dmg_tilings}
\frac{1}{\varepsilon}\int_{\bl(\sft)}^{\br(\sft)} \frac{\sfb'_\sft(z)(\rho(s;\bmx(t+1))-\rho(s;\bmx(t)))}{\sfb_\sft(z)-\sfb_\sft(s)}\rd s
\\=\frac{1}{2\pi \ri}\oint_{\cin}\frac{\ln \cB_\sft(w)\sfb'_\sft(w) \sfb'_\sft(z)}{(\sfb_\sft(w)-\sfb_\sft(z))^2}\rd w\,+\eps R_t(z)
+\Delta\cM_t(z)+\, \OO\left(\varepsilon^2 \right),
\end{multline}
where the contour $\cin\subset \Lambda_\sft$ encloses $[\bl(\sft), \br(\sft)]$, but not $z$; $R_t(z)$ is a bounded deterministic term; and conditionally on $\bmx(0), \bmx(1),\dots,\bmx(t)$, $\Delta\cM_t(z)$ are mean $0$  random variables such that $\{\varepsilon^{-1/2}\Delta \cM_t(z)\}_{z\in \Lambda\setminus[\bl,\br]}$ are asymptotically Gaussian with covariance  given by
\begin{align}\label{e:covT_tilings}
\lim_{\eps\to 0}\bE\left[\frac{ \Delta \cM_t(z_1)}{\varepsilon^{1/2}}, \frac{\Delta \cM_t(z_2)}{\varepsilon^{1/2} }\right]
&=\frac{1}{2\pi \ri}\oint_{\cin} \frac{\cG_\sft(w)\varphi^+_\sft(w)}{\cB_\sft(w)}\frac{\sfb'_\sft(w)\sfb'(z_1)}{(\sfb_\sft(w)-\sfb_\sft(z_1))^2} \frac{\sfb'_\sft(w)\sfb'_\sft(z_2)}{(\sfb_\sft(w)-\sfb_\sft(z_2))^2} \rd w,
\end{align}
with $z_1$ and $z_2$ outside the integration contour.

\subsection{Application of Theorem \ref{t:loopstudy}} \label{Section_massaging}

The aim of this subsection is to transform the result of Theorem \ref{t:loopstudy} for lozenge tilings, as recorded in \eqref{e:dmg_tilings}, \eqref{e:covT_tilings}, into the following statement:

\begin{proposition} \label{Proposition_tilings_increment} In the framework of Theorem \ref{Theorem_LLN_abstract} and using the notations of Sections \ref{Section_tilings_setup}, \ref{Section_assumptions},
  for $\sft=t\eps$ we have as $\eps\to 0$:
 \begin{multline}\label{e:dmg_tilings_2}
\frac{1}{\varepsilon}\left[\int_{\bl(\sft)}^{\br(\sft)} \frac{\sfb'_{\sft+\eps}(z) \rho(s;\bmx(t+1))}{\sfb_{\sft+\eps}(z)-\sfb_{\sft+\eps}(s)}\rd s -\int_{\bl(\sft)}^{\br(\sft)} \frac{\sfb'_\sft(z) \rho(s;\bmx(t))}{\sfb_\sft(z)-\sfb_\sft(s)}\rd s\right]
\\=-\del_z \ln\tilde\cB_{\sft}(z)+\ln(\sfq)\frac{\kappa^2 \sfq^{z-\sft}+\sfq^{-z}}{\kappa^2\sfq^{z-\sft}-\sfq^{-z}}-\del_\sft \ln \tilde \varphi_\sft^+(z)+\eps R_t(z)
+\Delta\cM_t(z) + O(\eps^2),
\end{multline}
where the contour $\cin\subset \Lambda_\sft$ encloses $[\bl(\sft), \br(\sft)]$, but not $z$; and  $R_t(z)$ is a term (we do not claim that it is the same as in \eqref{e:dmg_tilings}) which can be expressed as an integral of a uniformly bounded function against the measure $\rho(s,\bmx(t)) \rd s$. After conditioning on $\bmx(0), \bmx(1),\dots,\bmx(t)$, the terms $\Delta\cM_t(z)$ are mean $0$  random variables such that $\{\varepsilon^{-1/2}\Delta \cM_t(z)\}_{z\in \Lambda_\sft\setminus[\bl(\sft),\br(\sft)]}$ are asymptotically Gaussian with covariance  given by
\begin{multline}\label{e:covT_tilings_2}
\frac{\bE[\Delta \cM_t(z_1)\Delta \cM_t(z_2)|\bmx(t)]}{\varepsilon}\\ = \frac{1}{2\pi \ri}\oint_{\cin} \frac{\tilde f_\sft(w)}{\tilde f_\sft(w)-1}\frac{\sfb_\sft'(w)\sfb_\sft'(z_1)}{(\sfb_\sft(w)-\sfb_\sft(z_1))^2} \frac{\sfb_\sft'(w)\sfb_\sft'(z_2)}{(\sfb_\sft(w)-\sfb_\sft(z_2))^2} \rd w+\oo(1),
\end{multline}
with $z_1$ and $z_2$ outside the integration contour. The asymptotic Gaussianity is in the convergence of moments sense, as in Theorem \ref{t:loopstudy}. $O(\eps^2)$ is in the sense of moments, uniformly in $\sft$ and $z\in \Lambda_\sft$ bounded away from $[\bl(\sft),\br(\sft)]$.
\end{proposition}
\begin{proof}
 The left-hand side of \eqref{e:dmg_tilings_2} differs from the left-hand side of \eqref{e:dmg_tilings} by the term
\begin{multline} \label{e:x27}
 \frac{1}{\varepsilon}\int_{\bl(\sft)}^{\br(\sft)} \left(\frac{\sfb_{\sft+\varepsilon}'(z)}{\sfb_{\sft+\varepsilon}(z)-\sfb_{\sft+\varepsilon}(s)}- \frac{\sfb_\sft'(z)}{\sfb_\sft(z)-\sfb_\sft(s)}\right) \rho(s;\bmx(t+1))\rd s
 \\=\int_{\bl(\sft)}^{\br(\sft)} \partial_t\left(\frac{\sfb_\sft'(z)}{\sfb_\sft(z)-\sfb_\sft(s)}\right) \rho(s;\bmx(t))\rd s+\OO(\eps).
\end{multline}
The asymptotic covariance for the term $\Delta\cM_t(z)$ in \eqref{e:dmg_tilings_2} and in \eqref{e:dmg_tilings} is the same: we only used the identity
$$
 \frac{\cG_\sft(w)\varphi^+_\sft(w)}{\cB_\sft(w)}=\frac{\tilde f_\sft(w)}{\tilde f_\sft(w)-1}+\OO(\eps)
$$
to rewrite the covariance \eqref{e:covT_tilings} as \eqref{e:covT_tilings_2}. This identity follows from $\cB_\sft(w)=\tilde \cB_\sft(w)+\OO(\eps)$,  $\varphi^-_\sft(w)=\tilde\varphi^-_\sft(w)$ and definitions \eqref{e:x28}, \eqref{e:x4}.

We further simplify the leading term in the right-hand side \eqref{e:dmg_tilings}, integrating by parts and then deforming the integration contour through $w=z$. Denoting through $\cout$ a contour including both $[\bl(\sft),\br(\sft)]$ and $z$, we have
\begin{multline}
\label{e:x33}
\frac{1}{2\pi \ri}\oint_{\cin}\frac{\ln \cB_\sft(w)\sfb'_\sft(w) \sfb'_\sft(z)}{(\sfb_\sft(w)-\sfb_\sft(z))^2}\rd w=-
\frac{1}{2\pi \ri}\oint_{\cin}\ln \cB_\sft(w) \partial_w \left[\frac{ \sfb'_\sft(z)}{\sfb_\sft(w)-\sfb_\sft(z)}\right]\rd w\\= \frac{1}{2\pi \ri}\oint_{\cin} \frac{ \partial_w \ln \cB_\sft(w)\sfb'_\sft(z)}{\sfb_\sft(w)-\sfb_\sft(z)}\rd w= \frac{1}{2\pi \ri}\oint_{\cout}\frac{ \partial_w \ln \cB_\sft(w) \sfb'_\sft(z)}{\sfb_\sft(w)-\sfb_\sft(z)}\rd w-   \partial_z \ln \cB_\sft(z).
\end{multline}
We replace  $\cB_\sft(w)=\tilde \cB_\sft(w)+\OO(\eps)$ throughout the last formula, accumulating $\OO(\eps)$ total error. The second term  $-\partial_z \ln \cB_\sft(z)$ then turns into the first term in the right-hand side of \eqref{e:dmg_tilings_2} and it remains to deal with
\begin{align}\begin{split}
\label{e:x30}
  &\phantom{{}={}}\frac{1}{2\pi \ri}\oint_{\cout}\left[ \partial_w \ln \tilde \cB_\sft(w)\right] \cdot \frac{ \sfb'_\sft(z)}{\sfb_\sft(w)-\sfb_\sft(z)}\rd w\\
  &=-\frac{1}{2\pi \ri}\oint_{\cout}\left[ \partial_w \ln \tilde \cB_\sft(w)\right] \cdot \left( \frac{\rd \sfq^w}{\sfq^z-\sfq^w}+\frac{\rd \sfq^w}{\sfq^w- \sfq^{-z+\sft-2\log_\sfq \kappa}}  \right).
\end{split}\end{align}
It is helpful to make a change of variables $w=-w'+\sft-2\log_\sfq \kappa$ in the second integral \eqref{e:x30}; denoting the image of $\cout$ contour by $\omega'$ (both are positively oriented) and using \eqref{e:x13}, we have
\begin{align*}
&\frac{1}{2\pi \ri}\oint_{\cout}\left[ \partial_w \ln \tilde \cB_\sft(w)\right] \frac{\rd \sfq^w}{\sfq^w- \sfq^{-z+\sft-2\log_\sfq \kappa}}\\
 &=\frac{1}{2\pi \ri}\oint_{\omega'}\left[ -\partial_{w'} \ln \tilde \cB_\sft(-w'+\sft-2\log_\sfq \kappa)\right] \cdot \frac{\rd \sfq^{-w'+\sft-2\log_\sfq \kappa}}{\sfq^{-w'+\sft-2\log_\sfq \kappa}- \sfq^{-z+\sft-2\log_\sfq \kappa}}\\
  &=\frac{1}{2\pi \ri}\oint_{\omega'}\left[-\partial_{w'} \ln\tilde \cB_\sft(w') +2\ln \sfq+\partial_{w'}\ln \cG_\sft(w') \right] \cdot \left(-\rd w' - \frac{\rd \sfq^{w'}}{\sfq^{z}- \sfq^{w'}}\right).
\end{align*}
The $\rd w'$ integral vanishes and so does the integral of $2\ln\sfq \frac{\rd \sfq^{w'}}{\sfq^{z}- \sfq^{w'}}$. Hence, renaming back $w'$ into $w$, we convert \eqref{e:x30} into
\begin{equation}
\label{e:x31}
 \frac{1}{2\pi \ri}\oint_{\cout\cup\omega'} \partial_w \ln \tilde \cB_\sft(w)  \frac{\rd \sfq^w}{\sfq^w-\sfq^z}- \frac{1}{2\pi \ri}\oint_{\omega'} \partial_{w}\ln \cG_\sft(w)  \frac{\rd \sfq^w}{\sfq^w-\sfq^z}.
\end{equation}
We compute two integrals in \eqref{e:x31} separately. For the first integral, we recall the notation $\tilde \cB_\sft(w)= \mathfrak B_\sft(\sfq^w)$ and make a change of variables $W=\sfq^w$, converting the integral into
$$
 \frac{\ln\sfq}{2\pi \ri}\oint_{\sfq^{\cout}\cup\sfq^{\omega'}}  W \left[\partial_W \ln \mathfrak B_\sft(W)\right]  \frac{\rd W}{W-\sfq^z}.
$$
We compute the last integral as minus the sum of the residues outside the integration contour. This is similar to Steps 2 and 3 in the proof of Proposition \ref{Proposition_off_criticality}. The residues are at $Z=\pm \sfq^{\sft/2}\kappa^{-1}$ and at $Z=\infty$. At each of the former points, $\partial_W \ln \mathfrak B_\sft(W)$ has a simple pole of residue $1$; hence, these points give contribution:
\begin{equation}
\label{e:x32}
% -\ln \sfq \frac{\sfq^{\sft/2}\kappa^{-1}}{\sfq^{\sft/2}\kappa^{-1}+\sfq^z}-\ln \sfq \frac{ \sfq^{\sft/2}\kappa^{-1}}{ \sfq^{\sft/2}\kappa^{-1}-\sfq^z}=
-2\ln\sfq \left( \frac{\sfq^{\sft}\kappa^{-2}}{\sfq^{\sft}\kappa^{-2}-\sfq^{2z}}\right)=-2\ln \sfq-2\ln\sfq \left( \frac{\sfq^{2z}}{\sfq^{\sft}\kappa^{-2}-\sfq^{2z}}\right).
\end{equation}
At large $W$, $\mathfrak B_\sft(W)$ grows as $c W^2$ and therefore, the integrand behaves as $\frac{2}{W}$, resulting in the contribution $2\ln\sfq$. Summing with \eqref{e:x32} and also noticing that $\partial_\sft \ln \tilde \varphi^+_\sft(z)=-\ln\sfq$, we conclude that
$$
  \frac{1}{2\pi \ri}\oint_{\cout\cup\omega'} \partial_w \ln \tilde \cB_\sft(w)  \frac{\rd \sfq^w}{\sfq^w-\sfq^z}=-\ln\sfq \left( \frac{\sfq^{-z}+\kappa^{2}\sfq^{z-\sft}}{\sfq^{-z}-\kappa^{2}\sfq^{z-\sft}}\right)+ \partial_\sft \ln \tilde \varphi^+_\sft(z).
$$
Combining with $-\partial_z \ln \cB_\sft(z)$ from \eqref{e:x33} we obtain all three first terms in \eqref{e:dmg_tilings_2}. It remains to show that \eqref{e:x27} cancels with the second integral in  \eqref{e:x31}. We transform the latter using the formula for $\cG_\sft(w)$ in \eqref{e:x3}:
\begin{align*}
&\phantom{{}={}} \frac{1}{2\pi \ri}\oint_{\omega'} \partial_{w}\ln \cG_\sft(w)  \frac{\rd \sfq^w}{\sfq^w-\sfq^z}\\
&=
 \frac{1}{2\pi \ri}\oint_{\omega'} \partial_{w} \left[
\int_{\bl(\sft)}^{\br(\sft)} \left(\frac{\rd \sfq^s}{\sfq^w-\sfq^s}+\frac{\rd \sfq^s}{\sfq^s-\sfq^{-w+\sft-2\log_\sfq \kappa}}\right)\rho(s;\bmx(t))\right] \frac{\rd \sfq^w}{\sfq^w-\sfq^z}.
\end{align*}
In the terms of the last formula the first one integrates to $0$ because there are no $w$--singularities inside the $\omega'$ contour. For the second term we swap the order of integration and integrate over $w$ first. Changing the variables $W=\sfq^w$ and then evaluating the integral as (minus) sum of the residues outside the integration contour, we have
\begin{align*}
 & \frac{1}{2\pi \ri} \int_{\omega'} \partial_{w} \left[\frac{1}{\sfq^s-\sfq^{-w+\sft-2\log_\sfq \kappa}} \right] \frac{\rd \sfq^w}{\sfq^w-\sfq^z}=
- \frac{1}{2\pi \ri} \int_{\omega'} \frac{\ln \sfq \cdot \sfq^{-w+\sft-2\log_\sfq \kappa}}{(\sfq^s-\sfq^{-w+\sft-2\log_\sfq \kappa})^2}\cdot \frac{\rd \sfq^w}{\sfq^w-\sfq^z}\\&=
- \frac{1}{2\pi \ri} \int_{\sfq^{\omega'}} \frac{W \ln \sfq \cdot \sfq^{\sft-2\log_\sfq \kappa}}{(W \sfq^s-\sfq^{\sft-2\log_\sfq \kappa})^2}\cdot \frac{\rd W}{W-\sfq^z}=
\frac{\sfq^{z} \ln \sfq \cdot \sfq^{\sft-2\log_\sfq \kappa}}{(\sfq^z \sfq^s-\sfq^{\sft-2\log_\sfq \kappa})^2}=\frac{ \ln \sfq \cdot \sfq^{-z+\sft-2\log_\sfq \kappa}}{( \sfq^s-\sfq^{-z+\sft-2\log_\sfq \kappa})^2}.
\end{align*}
We conclude that the second integral in \eqref{e:x31} is
$$
 \int_{\bl(\sft)}^{\br(\sft)} \frac{ \ln \sfq \cdot \sfq^{-z+\sft-2\log_\sfq \kappa}}{( \sfq^s-\sfq^{-z+\sft-2\log_\sfq \kappa})^2} \rho(s;\bmx(t)) \rd \sfq^s.
$$
On the other hand, the integral in \eqref{e:x27} is readily transformed to the same expression:
$$
 \int_{\bl(\sft)}^{\br(\sft)} \partial_t\ \left(\frac{\rd \sfq^s}{\sfq^z-\sfq^s}+\frac{\rd \sfq^s}{\sfq^s-\sfq^{-z+\sft-2\log_\sfq \kappa}}\right) \rho(s;\bmx(t))=  \int_{\bl(\sft)}^{\br(\sft)}  \frac{\ln \sfq \cdot \sfq^{-z+\sft-2\log_\sfq \kappa}\cdot }{(\sfq^s-\sfq^{-z+\sft-2\log_\sfq \kappa})^2} \rho(s;\bmx(t))\rd \sfq^s. \qedhere
$$
%$$
% \ln \tilde \cB_\sft(-z+\sft-2\log_\sfq\kappa)= \ln(-1) +\ln(\sfq) (\sft-2z) -\ln(\kappa^2) -\ln \cG_\sft(z) +\ln \tilde \cB_\sft(z).
%$$
\end{proof}

\subsection{Equivalent descriptions of the limit shape}

\label{Section_Limit_shape_proofs}

In this subsection we show that the descriptions of $\rho_\sft$ given in Propositions \ref{Proposition_LLN_Burgers}, \ref{Proposition_LLN_equation}, and \ref{Proposition_LLN_invariant} are equivalent.

\begin{theorem} \label{Theorem_three_descriptions}
   Take a family $\rho_\sft$, $0\leq \sft \leq \sfT$ of measures of total mass $\sfN$, supported on $[\sft-\sfT,\sfN]$, and of density $\rho_\sft(s)$ at most $1$. Suppose that $\rho_0$ is given, while $\rho_\sft$, $\sft>0$ are unknown. If these measures are descibed by either of Proposition \ref{Proposition_LLN_Burgers}, Proposition \ref{Proposition_LLN_equation}, or Proposition \ref{Proposition_LLN_invariant}, the resulting family $\rho_\sft$ is the same.
\end{theorem}

\begin{proof} {\bf Step 1.} We start by explaining that the complex slope $f_\sft(z)$ defined through \eqref{e:double_complex_slope} uniquely determines the density $\rho_\sft(s)$. Indeed, for $x\in[\sft-\sfT,\sfN]$ and $\eps>0$ we write
\begin{align*}
&\phantom{{}={}}\arg f_\sft(x+\ii \eps)=
 -\Im\left(\int_{\sft-\sfT}^{\sfN} \frac{\sfb'_\sft(x+\ii \eps)(1-\rho_\sft(s))}{\sfb_\sft(x+\ii \eps)-\sfb_\sft(s)}\rd s\right)
 \\&= -\int_{\sft-\sfT}^{\sfN}  \Im\left( \frac{\sfb'_\sft(x+\ii \eps)}{[\sfb_\sft(x)-\sfb_\sft(s)]+[\sfb_\sft(x+\ii \eps)-\sfb_\sft(x)]}\right)(1-\rho_\sft(s))\rd s\\
 &= -\int_{\sft-\sfT}^{\sfN}  \Im\left( \frac{\sfb'_\sft(x+\ii \eps)[\sfb_\sft(x)-\sfb_\sft(s)]-\sfb'_\sft(x+\ii \eps)[\sfb_\sft(x+\ii \eps)-\sfb_\sft(x)]}{[\sfb_\sft(x)-\sfb_\sft(s)]^2-[\sfb_\sft(x+\ii \eps)-\sfb_\sft(x)]^2}\right)(1-\rho_\sft(s))\rd s.
\end{align*}
Let us send $\eps\to 0+$ in the last formula. Assuming $\sfb'_\sft(x)\ne 0$, we have (uniformly over $s\in\mathbb [\sft-\sfT,\sfN]$)
\begin{align*}
\Im\left( \frac{\sfb'_\sft(x+\ii \eps)[\sfb_\sft(x)-\sfb_\sft(s)]}{[\sfb_\sft(x)-\sfb_\sft(s)]^2-[\sfb_\sft(x+\ii \eps)-\sfb_\sft(x)]^2}\right)= \frac{\eps \sfb''_\sft(x)[\sfb_\sft(x)-\sfb_\sft(s)]+\OO(\varepsilon^2)}{[\sfb_\sft(x)-\sfb_\sft(s)]^2+\eps^2 \sfb'_\sft(x)^2+\OO(\varepsilon^3)},\\
 \Im\left( \frac{-\sfb'_\sft(x+\ii \eps)[\sfb_\sft(x+\ii \eps)-\sfb_\sft(x)]}{[\sfb_\sft(x)-\sfb_\sft(s)]^2-[\sfb_\sft(x+\ii \eps)-\sfb_\sft(x)]^2}\right)= \frac{-\eps(\sfb'_\sft(x))^2+\OO(\varepsilon^2)}{[\sfb_\sft(x)-\sfb_\sft(s)]^2+\eps^2\sfb'_\sft(x)^2+\OO(\varepsilon^3)}.
\end{align*}
Hence, using the fact
$$
 \int_{-\infty}^{\infty} \frac{\eps }{y^2+\eps^2} \rd y = \pi,\quad y=\frac{\sfb_\sft(x)-\sfb_\sft(s)}{\sfb_\sft'(x)},
$$
and that the integral is sharply concentrated around $0$ as $\eps\to 0$, we conclude that
\begin{equation*}
 \lim_{\eps\to 0+} \frac{1}{\pi} \arg f_\sft(x+\ii \eps)= 1-\rho_\sft(x).
\end{equation*}

\noindent {\bf Step 2.} Next, we take $f_\sft(z)$ satisfying the PDE of Proposition \ref{Proposition_LLN_Burgers}, which we rewrite as
\begin{align}\label{e:Burgers_repeat}
\frac{\del_\sft  f_\sft(z)}{f_\sft(z)}+\frac{\del_z f_\sft(z)}{f_\sft(z)-1}=\ln(\sfq)\frac{\kappa^2 \sfq^{z-\sft}+\sfq^{-z}}{\kappa^2\sfq^{z-\sft}-\sfq^{-z}}.
\end{align}
Our goal is to solve\footnote{We only search for solutions of the form \eqref{e:double_complex_slope} and do not investigate possibility of other solutions.}   this PDE arriving at the expression of Proposition \ref{Proposition_LLN_equation}. For that we are going to use the characteristics method.

The method guides to search for the three--dimensional hypersurface of triplets $(\sft, z, f_\sft(z))$ solving \eqref{e:Burgers_repeat} as a union of one-dimensional characteristic curves $\sft(\tau)$, $z(\tau)$, $f(\tau)$, with each curve solving the differential equation
$$
 \del_\tau \sft = \frac{1}{f},\qquad \del_\tau z = \frac{1}{f-1}, \qquad \del_\tau f= \ln(\sfq)\frac{\kappa^2 \sfq^{z-\sft}+\sfq^{-z}}{\kappa^2\sfq^{z-\sft}-\sfq^{-z}}.
$$
Because we only care about the hypersurface itself, we can replace $\tau$ by any other parameter. The key simplifying observation, leading to the relatively explicit expressions for $f_\sft(z)$ in our case, is that we are able to choose $\sft$ as the parameter. Hence, we are looking for a two-dimensional family of curves\footnote{Somewhat abusing the notations, the upper index $\sft$ in $z^{\sft}$ and $f^\sft$ indicates the time coordinate they depend on; on the other hand, $\sfq^\sft$ stays for the number $\sfq$ raised to the power $\sft$.} $(z^\sft, f^\sft)\in\mathbb C^{2}$, $\sft\geq 0$, which we are going to enumerate by their initial conditions $(z^0,f^0)=(u,f_0(u))$. We choose the $z$--component of the curve to be solving the differential equation
\begin{equation}\label{e:z_DE}
 \del_\sft z^\sft=\frac{f_\sft(z^\sft)}{f_\sft(z^\sft)-1},\quad \sft\geq 0,\qquad z^0=u.
\end{equation}
 The equation for the $f$--component can be read from the fact that $f^\sft$ comes from substituting $z=z_\sft$ into the function $f_\sft(z)$ solving \eqref{e:Burgers_repeat}:
 \begin{align}
  \notag\del_\sft f^\sft&= \del_\sft [f_\sft(z^\sft)]= \del_\sft f_\sft(z)\Big|_{z=z^\sft} + \del_z f_\sft(z)\Big|_{z=z^\sft} \del_\sft z^\sft= f_\sft(z^\sft) \left[\frac{\del_\sft f_\sft(z)}{f_\sft(z)}+\frac{\del_z f_\sft(z)}{f_\sft(z)-1} \right]_{z=z^\sft}
  \\  \label{e:f_DE} &=  \ln(\sfq) f^\sft \frac{\kappa^2 \sfq^{z^\sft-\sft}+\sfq^{-z^\sft}}{\kappa^2\sfq^{z^\sft-\sft}-\sfq^{-z^\sft}}, \quad \sft\geq 0,\qquad f^0=f_0(u).
 \end{align}
We further introduce a domain
\begin{align}\label{e:dfbD}
\bD_\sft=\left\{z\in \bC:\, 0\leq \Im z\leq - \tfrac{\pi}{\ln \sfq},\, \Re z \leq \sft/2-\log_\sfq \kappa \right\},
\end{align}
and notice that $\sfb_\sft$ is a bijection between this domain and closed lower half-plane, which is a conformal bijection between the interior of this domain denoted $\bD_\sft^\circ$ and open lower half-plane.
Thanks to the relations (which are all immediate from the definition):
\begin{equation}
\label{e:symmetries}
f_\sft(-z+\sft-2\log_\sfq \kappa)=\frac{1}{f_\sft(z)}, \qquad f_\sft(\bar z)=\overline{f_\sft(z)}, \qquad  f_\sft\left(z+\tfrac{2\pi\ii}{\ln \sfq}\right)=f_\sft(z),
\end{equation}
we can restrict our study of $f_\sft(z)$ to $\bD_\sft$: the values of $f_\sft(z)$ outside $\bD_\sft$ can be determined from its value inside $\bD_\sft$. Hence, we can summarize our task: for each $z\in \bD_\sft$, we would like to find an initial condition $u$ and corresponding solution to the following system of differential equations
\begin{equation}
\label{e:DE_combined}
 \begin{cases} \del_\sft z^\sft=\frac{f^\sft}{f^\sft-1}, &z^0=u,\\
                       \del_\sft f^\sft= \ln(\sfq) f^\sft \frac{\kappa^2 \sfq^{z^\sft-\sft}+\sfq^{-z^\sft}}{\kappa^2\sfq^{z^\sft-\sft}-\sfq^{-z^\sft}},\qquad\qquad &  f^0=f_0(u),
 \end{cases}
\end{equation}
 such that $z^\sft=z$; once they are found, we obtain $f_\sft(z)=f^\sft$.

 The system of differential equations \eqref{e:DE_combined} has two first integrals, which are expressions staying constant along the trajectories $(\sft,z^\sft,f^\sft)$, cf.\ \cite[Theorem 2.1]{borodin2010q}. As can be verified by straightforward differentiation, they are given by
\begin{align}\label{e:defu_1}
U(u)&=\sfq^\sft \frac{f^\sft\sfq^{-z^\sft}-\kappa^2 \sfq^{z^\sft-\sft}}{1-f^\sft}
=\frac{\sfq^\sft}{2} \left(-(\sfq^{-z^\sft}+\kappa^2 \sfq^{z^\sft-\sft}) +(\sfq^{-z^\sft}-\kappa^2 \sfq^{z^\sft-\sft})\frac{1+f^\sft}{1-f^\sft}\right),\\
V(u)&=\frac{\sfq^{-z^\sft}-f^\sft\kappa^2 \sfq^{z^\sft-\sft}}{1-f^\sft}=\frac{1}{2}\left((\sfq^{-z^\sft}+\kappa^2 \sfq^{z^\sft-\sft}) +(\sfq^{-z^\sft}-\kappa^2 \sfq^{z^\sft-\sft})\frac{1+f^\sft}{1-f^\sft}\right).\label{e:defv_1}
\end{align}
Because trajectories are parameterized by initial conditions $u=z^0$, the values of the first integrals in \eqref{e:defu_1}, \eqref{e:defv_1} depend on $u$. This dependence can be made explicit by plugging $\sft=0$ into \eqref{e:defu_1}, \eqref{e:defv_1}:
\begin{equation}\label{e:defuv_2}
U(u)=\frac{f_0(u)\sfq^{-u}-\kappa^2 \sfq^{u}}{1-f_0(u)},\qquad
V(u)=\frac{\sfq^{-u}-f_0(u)\kappa^2 \sfq^{u}}{1-f_0(u)}.
\end{equation}

Combining \eqref{e:defu_1} and \eqref{e:defv_1} and recalling $\sfb_\sft(z_\sft)=\sfq^{-z_\sft}+\kappa^2 \sfq^{z_\sft-\sft}$, we express $z^\sft$ in terms of $\sft$ and $u$:
\begin{align}\label{e:ztsolver_1}
\sfb_\sft(z_\sft)=V(u)-\frac{U(u)}{\sfq^\sft}.
\end{align}
If we treat \eqref{e:ztsolver_1} as an equation determining $u$ as a function of $\sft$ and $z=z^\sft$, then it matches \eqref{e:slope_equation}. Once we find $u$, it is also helpful to subtract \eqref{e:ztsolver_1} from its $\sft=0$ version resulting in
\begin{align}\label{e:ztsolver_2}
\frac{\sfb_\sft(z_\sft)-\sfb_0(u)}{1-\sfq^{-\sft}}=U(u).
\end{align}
Equating the expressions for $U(u)$ of \eqref{e:ztsolver_2} with \eqref{e:defu_1}, we can express $f^\sft$ through $z^\sft$, $u$, and $\sft$. Replacing $f_\sft(z)=f^\sft$ and $z=z^\sft$, we arrive at the desired \eqref{e:slope_through_u} of Proposition \ref{Proposition_LLN_equation}. Note that \eqref{e:ztsolver_1} might have several $u$ satisfying it (or none) and we have not yet explained how to choose the correct $u$ --- this will be done in Step 5.

\medskip

\noindent {\bf Step 3.} Further, we present an alternative way to package the characteristics method  for finding $f_\sft(z)$, which solves the PDE of Proposition \ref{Proposition_LLN_Burgers}. This will lead us to the expression of Proposition \ref{Proposition_LLN_invariant}.

Guided by \eqref{e:defu_1} and using symmetries \eqref{e:symmetries} we notice that the expression
$$
 \sfq^\sft \frac{f_\sft(z)\sfq^{-z}-\kappa^2 \sfq^{z-\sft}}{1-f_\sft(z)}
$$
is unchanged under $z\mapsto z+\frac{2\pi \ii}{\ln \sfq}$ and $z\mapsto -z +\sft -2\ln (\kappa)$. Therefore, there exists a function $\cS_\sft(w)$ of complex argument $w$, such that
\begin{equation}
\label{e:R_def}
 \cS_\sft(\sfb_\sft(z))= \sfq^\sft \frac{f_\sft(z)\sfq^{-z}-\kappa^2 \sfq^{z-\sft}}{1-f_\sft(z)},
\end{equation}
for all $z$ such that $\sfb_\sft(z)$ is outside $\sfb_\sft([\sft -\sfT, \sfN])$. The function $\cS_\sft(w)$ is holomorphic in $w$, except, perhaps, at $w\in \sfb_\sft([\sft -\sfT, \sfN])$. Notice that once we know $(\sft,z, \cS_\sft(\sfb_\sft(z)))$, we also know $(\sft,z,f_\sft(z))$ and vice-versa.

\smallskip

We now come back to \eqref{e:DE_combined} and introduce $w^\sft:=\sfb_\sft(z^\sft)$ and $\cS^\sft= \cS_\sft (w^\sft)$. Let us describe the evolution of the triplet $(\sft,w^\sft,\cS^\sft)$, $\sft\geq 0$. Notice that $U(u)$ in \eqref{e:defu_1} becomes $ \cS^\sft$. Therefore, the value of $\cS^{\sft}$  does not depend on $\sft$. On the other hand, \eqref{e:ztsolver_2} becomes
\begin{equation}
\label{e:x17}
 w^\sft=w^0+(1-\sfq^{-\sft}) U(u)=w^0-(\sfq^{-\sft}-1) \cS^0.
\end{equation}
Therefore, the equation $ \cS^\sft= \cS^0$ can be rewritten as
\begin{equation}
\label{e:invariant_equation_rewritten}
  \cS_\sft \left(w_0  - (\sfq^{-\sft}-1)\cS_0(w_0)\right)= \cS_0(w_0).
\end{equation}
If we now rename $\cS_0(w_0)=v$, then we get
\begin{equation*}
  \cS_0^{(-1)}(v)  - (\sfq^{-\sft}-1)v= \cS_\sft^{(-1)}(v),
\end{equation*}
which is precisely \eqref{e:U_functional_statement} of Proposition \ref{Proposition_LLN_invariant}.

\smallskip

{\bf Step 4.} We are now going to show that the function $\cS_\sft(w)$ defined in the previous step by \eqref{e:R_def} has the form \eqref{e:U_measure} for some measure $\mu_\sft$. For that we would like to check that $\cS_\sft(w)$ is a Nevanlinna function:
\begin{equation}\label{e:Nevanlinna}
\Im \cS_\sft(w) < 0\quad  \text{ whenever }\Im w<0.
\end{equation}
 Recall that $\sfb_\sft(z)$ is a conformal bijection between the interior $\bD_\sft^\circ$ of the domain of \eqref{e:dfbD} and lower half-plane. Hence, \eqref{e:Nevanlinna} is equivalent to
 \begin{equation} \label{e:Nevanlinna_restated}
  \Im \mathbf U_\sft(z)<0\quad \text{ whenever } z\in \bD_\sft^\circ,\qquad \text{ where } \mathbf U_\sft(z)= \sfq^\sft \frac{f_\sft(z)\sfq^{-z}-\kappa^2 \sfq^{z-\sft}}{1-f_\sft(z)}.
 \end{equation}
 It is helpful to also use an alternative representation of $\mathbf U_\sft(z)$:
 \begin{equation}
 \label{e:x15}
 \mathbf U_\sft(z)=\frac{\sfq^\sft}{2} \left(-(\sfq^{-z}+\kappa^2 \sfq^{z-\sft}) +(\sfq^{-z}-\kappa^2 \sfq^{z-\sft})\frac{1+f_\sft(z)}{1-f_\sft(z)}\right).
\end{equation}
For $z$ on the right boundary of $\bD_\sft$, i.e. $\Re z= \sft /2 -\log_\sfq \kappa$, we have  $|f_\sft(z)|=1$, because the integrand in \eqref{e:double_complex_slope} is purely imaginary.  Further, $\sfq^{-z}+\kappa^2 \sfq^{z-\sft}$ is real and $\sfq^{-z}-\kappa^2 \sfq^{z-\sft}$ is purely imaginary. Plugging into \eqref{e:x15} we conclude that  $\mathbf U_\sft(z)$ is real on the right boundary of $\bD_\sft$.

For $z$ on the bottom boundary of $\bD_\sft$ by the argument of Step 1, $\Im f_\sft(z)\geq 0$ and therefore also $\Im \left[\frac{1+f_\sft(z)}{1-f_\sft(z)}\right]\geq 0$. On the other hand, $\sfq^{-z}+\kappa^2 \sfq^{z-\sft}$ is real and $\sfq^{-z}-\kappa^2 \sfq^{z-\sft}$ is negative real. Plugging into \eqref{e:x15} we conclude that $\Im \mathbf U_\sft(z)\leq 0$ on the bottom boundary of $\bD_\sft$. Note that for the points on the bottom boundary where  $\rho_\sft(s)\in\{0,1\}$, the same argument shows that $\mathbf U_\sft(s)$ is real.

For $z$ on the top boundary of $\bD_\sft$, the integrand in \eqref{e:double_complex_slope} is real (and non-singular); therefore, $f_\sft(z)$ is real, and so are $\sfq^{-z}+\kappa^2 \sfq^{z-\sft}$ and $\sfq^{-z}-\kappa^2 \sfq^{z-\sft}$. Plugging into \eqref{e:x15} we conclude that $\Im \mathbf U_\sft(z)=0$ on the top boundary of $\bD_\sft$.

$\bD_\sft$ does not have a left boundary, but we should also study the case $\Re z\to -\infty$. We investigate the integrand in \eqref{e:double_complex_slope} using
\begin{align*}
 \frac{\sfb'_\sft(z)}{\sfb_\sft(z)-\sfb_\sft(s)}=\ln \sfq + \OO(\sfq^{-z}), \qquad \Re z\to -\infty,
\end{align*}
Since the total mass of the measure $\rho_\sft$ is $\sfN$, we get $f_\sft(z)=\sfq^{\sft-\sfT}+ \OO(\sfq^{-z})$ and therefore,
\begin{equation} \label{e:x16}
  \mathbf U_\sft(z)=\frac{\kappa^2\sfq^{z-\sft} }{\sfq^{-\sfT}-\sfq^{-\sft}} + \OO(1),   \qquad \Re z\to -\infty.
\end{equation}
Inside $\bD_\sft$ we have $0<\Im z< -\frac{\pi}{\ln \sfq}$, leading to the imaginary part of the last expression being negative.

The function $\Im \mathbf U_\sft(z)$ is harmonic in $\bD_\sft^\circ$. We have shown that it is non-positive on the boundary of this domain and as $\Re z\to -\infty$. Using the maximum principle, we conclude that \eqref{e:Nevanlinna_restated} and, thus, also \eqref{e:Nevanlinna} holds.

Using the representation theorem for Nevanlinna functions (we refer to Appendix A in \cite{behrndt2020boundary} for a detailed discussion),  \eqref{e:Nevanlinna}  implies that
\begin{equation}
\label{e:Nev_representation}
 \cS_\sft(w)= \alpha w + \beta - \int_{\mathbb R} \left(\frac{1}{w-s}+\frac{s}{1+s^2} \right) \mu_\sft(ds),
\end{equation}
for a certain measure $\mu_\sft$ on $\mathbb R$ and real constants $\alpha>0$ and $\beta$. Our definitions imply that $\mathbf U_\sft(z)$ does not have any singularities for large complex $z$; therefore, $\cS_\sft(w)$ is holomorphic for all large enough $w$ and the measure $\mu_{\sft} (ds)$ should be compactly supported. Hence, $\int_{\mathbb R} \frac{s}{1+s^2}\mu_\sft(ds)$ can be absorbed into $\beta$. On the other hand, \eqref{e:x16} implies that for large $w$
$$
 \cS_\sft(w)=\frac{w}{\sfq^{-\sfT}-\sfq^{-\sft}}+\oo(w).
$$
This fixes the constant $\alpha$ in \eqref{e:Nev_representation} and turns it into \eqref{e:U_measure}.

\medskip

\noindent {\bf Step 5.} In the previous steps we explained how using the values of $z^\sft$ or $w^\sft$ one can reconstruct the values of $f_\sft(z^\sft)$ and $\cS_\sft(w^\sft)$, respectively. In this step we analyze the values which $z^\sft$ and $w^\sft$ take.

We define the liquid region of the measure $\rho_\sft$ through:
\begin{equation}
\label{e:Liquid_def_2}
 \mathcal L_\sft=\{s\in[\sft-\sfT, \sfN]\mid 0<\rho_\sft(s) < 1\}.
\end{equation}
Observe that by the argument of Step 4, the only points of the boundary of $\bD_0$ where $\mathbf U_\sft(z)$ is non-real are those of $\mathcal L_\sft$.

Take a point $u\in\bD_0^\circ$ and consider the evolution $z^\sft=z^{\sft,u}$ solving \eqref{e:DE_combined}. Let $\sft(u)$ be the first time when $z^{\sft,u}$ hits the boundary of $\bD_\sft$. Note that $z^{\sft(u),u}\in \mathcal L_\sft$, because $\mathbf U_\sft(z^{\sft(u),u})=\mathbf U_0(u)$ is non-real.

\begin{lemma} \label{Lemma_bijectivity}
 For each $\sft\geq 0$, the map $u\mapsto z^{\sft,u}$ is a conformal bijection between $\{u\in \bD_0\mid \sft<\sft(u)\}$ and $\bD_\sft^\circ$, which continuously extends to a homeomorphism between the closures of these two domains.
\end{lemma}
\begin{proof} Instead of dealing with $u\mapsto z^{\sft,u}$ directly, we consider the map $\sfb_0(u)\mapsto \sfb_\sft(z^{\sft,u})$; in the notations of Step 4 this is the map $w_0\mapsto w^\sft$, which is given explicitly by \eqref{e:x17}:
\begin{equation}
\label{e:x18}
 w^0\mapsto w^0-(\sfq^{-\sft}-1) \cS^0.
\end{equation}
According to Step 4, $\cS^0=\cS_0(w^0)$ is a complex number with negative imaginary part. Hence, $w^\sft$, as a function of $\sft$, has increasing imaginary part and the time $\sft(\sfb_\sft^{-1}(w))$ is the first time when $w^\sft$ becomes real. Since $\sfb_\sft$ is a bijection between $\bD_\sft^\circ$ and (open) lower halfplane, the lemma becomes equivalent to the study of the map \eqref{e:x18} between the set $\{w\in \mathbb C\mid \Im(w-(\sfq^{-\sft}-1) \cS_0(w))\leq 0\}$ and the closed lower halfplane $\{w\in \mathbb C\mid \Im(w)\leq 0\}$. Using the representation \eqref{e:U_measure}, we rewrite the map as
\begin{equation}
\label{e:x50}
  w\mapsto  w\frac{\sfq^{-\sfT}-\sfq^{-\sft}}{\sfq^{-\sfT}-1}-  (\sfq^{-\sft}-1)C_0 + (\sfq^{-\sft}-1) \int_{\mathbb R}\frac{\rd \mu_0(s)}{w-s}.
\end{equation}
The maps of this type (up to rescalings and real shifts, which preserve the lower halfplane and therefore are irrelevant) has been studied in the free probability theory before and the desired statements on being homeomorphism and being comformal is \cite[Lemma 4]{Biane_convlution}. This finishes the proof of Lemma \ref{Lemma_bijectivity}.
\end{proof}
\begin{remark} \label{Remark_extension_to_the_boundary}
 The statement of Lemma \ref{Lemma_bijectivity} leads to the definition of the map $u\mapsto z^{\sft,u}$ on the boundary of $\{u\in \bD_0\mid \sft<\sft(u)\}$; equivalently, we can define the map $w_0\mapsto w^{\sft}$ on the boundary of the domain $\{w\in \mathbb C\mid \Im(w-(\sfq^{-\sft}-1) \cS_0(w))\leq 0\}$. By continuity, these maps are defined by exactly the same formulas: $z^{\sft,u}$ can be found solving \eqref{e:DE_combined}, while $w^{\sft}$ is given by \eqref{e:x17} or \eqref{e:x18}.

 We also observe that the boundary of $\{w\in \mathbb C\mid \Im(w-(\sfq^{-\sft}-1) \cS_0(w))\leq 0\}$ naturally splits into two parts: the part with strictly negative imaginary part of $w$ is being mapped by $w_0\mapsto w^{\sft}$ into the liquid region $\sfb_\sft(\mathcal L_\sft)$ and the part on the real axis is mapped to the complement $\mathbb R\setminus \sfb_\sft(\mathcal L_\sft)$.
\end{remark}

 The role of Lemma \ref{Lemma_bijectivity} is two-fold. First, it guarantees that the procedures of Steps 2 and 3 work for all relevant values of $z$ or $w$, because all of them are possible values for $z^\sft$ or $w^\sft$. Second, it explains which $u$ solving \eqref{e:slope_equation} we should choose: this must be the unique $u$ such that:
 $$
  u\in \bD_0^\circ,\qquad \text { and }\qquad  \Im\left[\sfb_0(u)- (\sfq^{-\sft}-1) \cS_0(\sfb_0(u))\right]<0,
 $$
 which is precisely the condition of Remark \ref{remark_u_choice}.

\medskip

{\bf Step 6.} At this point we have verified that if $\rho_\sft$ is described by the PDE of Proposition \ref{Proposition_LLN_Burgers}, then it can be found through the formulas of either Proposition \ref{Proposition_LLN_equation} or Proposition \ref{Proposition_LLN_invariant}. For the sake of completeness, we also would like to check that if $f_\sft(z)$ is found from the formulas of the latter two propositions, then it solves the PDE \eqref{e:Burgers_first_appearance}. We start from Proposition \ref{Proposition_LLN_equation}

Let $u=u(z,t)$ and $f_\sft(z)$ be obtained by the procedure of Proposition \ref{Proposition_LLN_equation}. We use functions $U(u)$ and $V(u)$ as defined by \eqref{e:defuv_2}. We claim that
\begin{equation}
\label{e:x22}
 U(u(z,\sft))=\sfq^\sft \frac{f_\sft(z)\sfq^{-z}-\kappa^2 \sfq^{z-\sft}}{1-f_\sft(z)}, \qquad V(u(z,\sft))=\frac{\sfq^{-z}-f_\sft(z)\kappa^2 \sfq^{z-\sft}}{1-f_\sft(z)}.
\end{equation}
To prove the claim we notice that \eqref{e:slope_equation} can be rewritten as
\begin{equation}
\label{e:x19}
\sfq^{-z}+\kappa^2 \sfq^{z-\sft}=V(u)-\sfq^{-\sft} U(u).
\end{equation}
On the other hand, expressing $f_0(u)$ from the definitions of $U(u)$ and $V(u)$ in \eqref{e:defuv_2} and equating the results, we get
\begin{equation}
\label{e:x20}
 \frac{U(u)+\kappa^2 \sfq^{u}}{\sfq^{-u}+U(u)}=\frac{\sfq^{-u}-V(u)}{\kappa^2 \sfq^{u}-V(u)}; \qquad \text{Equivalently: } \quad \sfq^{-u}+\kappa^2 \sfq^{u}=V(u)-U(u).
\end{equation}
%\begin{equation}
% (U(u)+\kappa^2 \sfq^{u})(\kappa^2 \sfq^{u}-V(u)) =(\sfq^{-u}+U(u)) (\sfq^{-u}-V(u)),
%\end{equation}
$\kappa^2 \sfq^{u}+\sfq^{-u}$ arising in the last formula, can be reexpressed using \eqref{e:slope_through_u} as
\begin{equation}
\label{e:x21}
 \sfq^{-u}+\kappa^2 \sfq^u
=\frac{\sfq^{-z}+\kappa^2 \sfq^{z}-f_\sft(z)(\sfq^{-z+\sft}+\kappa^2\sfq^{z-\sft})}{1-f_\sft(z)} ,
\end{equation}
Combining \eqref{e:x19}, \eqref{e:x20}, and \eqref{e:x21} we get a system of two linear equations on $U(u)$, $V(u)$, which is solved by \eqref{e:x22}.
%$$
%  V-U=\frac{\sfq^{-z}-f_\sft(z)\kappa^2 \sfq^{z-\sft} -\sfq^\sft(f_\sft(z)\sfq^{-z}-\kappa^2 \sfq^{z-\sft})}{1-f_\sft(z)}
%$$

Next, we take \eqref{e:x22} and differentiate it in $\sft$ and $z$, getting four identities:
%$$
% U' \partial_z u(z,t)=\sfq^\sft \frac{( \partial_z f_\sft(z)\sfq^{-z}-\ln\sfq f_\sft(z)\sfq^{-z} -\ln\sfq\kappa^2 \sfq^{z-\sft})(1-f_\sft(z))+(f_\sft(z)\sfq^{-z}-\kappa^2 \sfq^{z-\sft})\partial_z f_\sft(z)}{(1-f_\sft(z))^2}
%$$
%$$
% U' \partial_\sft u(z,\sft) =  \frac{(\partial_\sft f_\sft(z)\sfq^{\sft-z}+f_\sft(z)\sfq^{\sft-z}\ln\sfq)(1-f_\sft(z))+(f_\sft(z)\sfq^{\sft-z}-\kappa^2 \sfq^{z})\partial_\sft f_\sft(z)}{(1-f_\sft(z))^2}
%$$
%$$
% V'(u(z,\sft)) \partial_z u(z,t) =\frac{-\sfq^{-z}\ln\sfq  -\partial_z f_\sft(z)\kappa^2 \sfq^{z-\sft}-f_\sft(z)\kappa^2 \sfq^{z-\sft}\ln\sfq-f_\sft(z)(-\sfq^{-z}\ln\sfq  -\partial_z f_\sft(z)\kappa^2 \sfq^{z-\sft}-f_\sft(z)\kappa^2 \sfq^{z-\sft}\ln\sfq)+(\sfq^{-z}-f_\sft(z)\kappa^2 \sfq^{z-\sft})\partial_z f_\sft(z)}{(1-f_\sft(z))^2}
%$$
%$$
% V'(u(z,\sft)) \partial_\sft u(z,\sft)=\frac{(-\partial_\sft f_\sft(z)\kappa^2 \sfq^{z-\sft}+f_\sft(z)\kappa^2 \sfq^{z-\sft}\ln\sfq)(1-f_\sft(z))+(\sfq^{-z}-f_\sft(z)\kappa^2 \sfq^{z-\sft})\partial_\sft f_\sft(z)}{(1-f_\sft(z))^2}
%$$
\begin{align*}
 U'(u(z,\sft)) \partial_z u(z,\sft)&=\frac{\partial_z f_\sft(z)(\sfq^{\sft-z}-\kappa^2 \sfq^{z})- f_\sft(z)\sfq^{\sft-z}(1-f_\sft(z))\ln\sfq -\kappa^2 \sfq^{z}(1-f_\sft(z))\ln\sfq}{(1-f_\sft(z))^2},
 \\
 U'(u(z,\sft)) \partial_\sft u(z,\sft) &=  \frac{\partial_\sft f_\sft(z)(\sfq^{\sft-z}-\kappa^2 \sfq^{z})+f_\sft(z)\sfq^{\sft-z}(1-f_\sft(z))\ln\sfq}{(1-f_\sft(z))^2},
 \\
 V'(u(z,\sft)) \partial_z u(z,\sft) &=\frac{\partial_z f_\sft(z)(\sfq^{-z}-\kappa^2 \sfq^{z-\sft})-\sfq^{-z}(1-f_\sft(z))\ln\sfq  -f_\sft(z)\kappa^2 \sfq^{z-\sft}(1-f_\sft(z))\ln\sfq}{(1-f_\sft(z))^2},
 \\
 V'(u(z,\sft)) \partial_\sft u(z,\sft)&=\frac{\partial_\sft f_\sft(z)(\sfq^{-z}-\kappa^2 \sfq^{z-\sft})+f_\sft(z)\kappa^2 \sfq^{z-\sft}(1-f_\sft(z))\ln\sfq }{(1-f_\sft(z))^2}.
\end{align*}
The product of the left-hand sides for the first and forth identities is the same as the product for the second and third identities. Therefore,
{
\begin{multline*}
 \bigl[\partial_z f_\sft(z)(\sfq^{\sft-z}-\kappa^2 \sfq^{z})- f_\sft(z)\sfq^{\sft-z}(1-f_\sft(z))\ln\sfq -\kappa^2 \sfq^{z}(1-f_\sft(z))\ln\sfq\bigr] \\
 \times \bigl[\partial_\sft f_\sft(z)(\sfq^{-z}-\kappa^2 \sfq^{z-\sft})+f_\sft(z)\kappa^2 \sfq^{z-\sft}(1-f_\sft(z))\ln\sfq\bigr]
 \\=\bigl[\partial_\sft f_\sft(z)(\sfq^{\sft-z}-\kappa^2 \sfq^{z})+f_\sft(z)\sfq^{\sft-z}(1-f_\sft(z))\ln\sfq\bigr]\\
 \times \bigl[\partial_z f_\sft(z)(\sfq^{-z}-\kappa^2 \sfq^{z-\sft})-\sfq^{-z}(1-f_\sft(z))\ln\sfq  -f_\sft(z)\kappa^2 \sfq^{z-\sft}(1-f_\sft(z))\ln\sfq\bigr],
\end{multline*}
}
which simplifies to
%{\small
%\begin{multline}
% \partial_z f_\sft(z)(\sfq^{\sft-z}-\kappa^2 \sfq^z)f_\sft(z)\kappa^2 \sfq^{z-\sft}(1-f_\sft(z))\ln\sfq- f_\sft(z)\sfq^{\sft-z}(1-f_\sft(z)) \partial_\sft f_\sft(z)(\sfq^{-z}-\kappa^2 \sfq^{z-\sft})\ln\sfq \\-\kappa^2 \sfq^{z}(1-f_\sft(z))\ln\sfq \bigl[\partial_\sft f_\sft(z)(\sfq^{-z}-\kappa^2 \sfq^{z-\sft})+f_\sft(z)\kappa^2 \sfq^{z-\sft}(1-f_\sft(z))\ln\sfq\bigr]
% \\=-\partial_\sft f_\sft(z)(\sfq^{\sft-z}-\kappa^2 \sfq^{z})\bigl[\sfq^{-z}(1-f_\sft(z))  +f_\sft(z)\kappa^2 \sfq^{z-\sft}(1-f_\sft(z))\bigr]\ln\sfq
% \\+f_\sft(z)\sfq^{\sft-z}(1-f_\sft(z))\ln\sfq\bigl[\partial_z f_\sft(z)(\sfq^{-z}-\kappa^2 \sfq^{z-\sft})-\sfq^{-z}(1-f_\sft(z))\ln\sfq\bigr]
%\end{multline}
%}
\begin{multline*}
 -[\partial_z f_\sft(z)]f_\sft(z)(1-f_\sft(z))(\sfq^{-z}-\kappa^2 \sfq^{z-\sft})(\sfq^{\sft-z}-\kappa^2 \sfq^{z})\ln\sfq
 \\+ [\partial_\sft f_\sft(z)](1-f_\sft(z))^2 (\sfq^{-z}-\kappa^2 \sfq^{z-\sft})(\sfq^{\sft-z}-\kappa^2\sfq^z) \ln\sfq\\=
 f_\sft(z) (1-f_\sft(z))^2(\kappa^4 \sfq^{2z-\sft}- \sfq^{\sft-2z})(\ln\sfq)^2.
\end{multline*}
Dividing by $f_\sft(z)(1-f_\sft(z))^2(\sfq^{-z}-\kappa^2 \sfq^{z-\sft})(\sfq^{\sft-z}-\kappa^2 \sfq^{z})\ln\sfq$, we arrive at \eqref{e:Burgers_first_appearance}.

\bigskip

{\bf Step 7.} Finally, we assume that $f_\sft(z)$ is obtained by the procedure of Proposition \ref{Proposition_LLN_invariant}. Our task is to show that it satisfies the PDE \eqref{e:Burgers_first_appearance}. Recasting \eqref{e:U_functional_statement} as \eqref{e:invariant_equation_rewritten}, the procedure can be summarized as follows. Given $(\sft,z)$, we first find $w=w(\sft,z)$ such that
\begin{equation}
\label{e:x23}
 w-(\sfq^{-\sft}-1) \cS_0(w)=\sfq^{-z}+\kappa^2 \sfq^{z-\sft}.
\end{equation}
Then we find $f_\sft(z)$ by solving
$$
 \cS_0(w)=\sfq^\sft \frac{f_\sft(z)\sfq^{-z}-\kappa^2 \sfq^{z-\sft}}{1-f_\sft(z)},
$$
which results in
\begin{equation}
\label{e:x26}
  f_\sft(z)=\frac{\cS_0(w)+\kappa^2 \sfq^{z}}{\cS_0(w)+\sfq^{\sft-z}}, \qquad 1-f_\sft(z)=\frac{\sfq^{\sft-z}-\kappa^2 \sfq^{z}}{\cS_0(w)+\sfq^{\sft-z}}.
\end{equation}
Differentiating \eqref{e:x23} in $z$ and $\sft$, we get:
\begin{align}
\notag [ \partial_z w] (1- (\sfq^{-\sft}-1)\cS'_0(w))&=-\sfq^{-z}\ln\sfq +\kappa^2 \sfq^{z-\sft}\ln\sfq,\\ [\partial_\sft w] (1- (\sfq^{-\sft}-1)\cS'_0(w))&=-\kappa^2 \sfq^{z-\sft} \ln \sfq- \sfq^{-\sft} \cS_0(w) \ln \sfq.\label{e:x25}
\end{align}
On the other hand, differentiating \eqref{e:x26}, we get
\begin{multline}
\label{e:x24}
 \partial_t \ln f_\sft(z)+\partial_z \ln (1-f_\sft(z))\\= \frac{\cS'_0(w) [\partial_t w]}{\cS_0(w)+\kappa^2 \sfq^{z}}-\frac{\cS'_0(w) [\partial_t w]+\sfq^{\sft-z}\ln\sfq}{\cS_0(w)+\sfq^{\sft-z}}+\frac{-\sfq^{\sft-z}\ln\sfq-\kappa^2 \sfq^z\ln\sfq}{\sfq^{\sft-z}-\kappa^2 \sfq^{z}}-\frac{\cS_0'(w)[\partial_z w]-\sfq^{\sft-z}\ln\sfq}{\cS_0(w)+\sfq^{\sft-z}}
 \\= \cS'_0(w)\left(\frac{ [\partial_t w]}{\cS_0(w)+\kappa^2 \sfq^{z}}-\frac{[\partial_t w]+[\partial_z w]}{\cS_0(w)+\sfq^{\sft-z}}\right)-\frac{\sfq^{-z}+\kappa^2 \sfq^{z-\sft}}{\sfq^{-z}-\kappa^2 \sfq^{z-\sft}}\ln\sfq.
\end{multline}
The last term in the right-hand side matches the right-hand side in \eqref{e:Burgers_first_appearance}. Hence, it remains to show that the first term in the right-hand side of \eqref{e:x24} vanishes. We check:
$$
  [\partial_t w](\cS_0(w)+\sfq^{\sft-z})\stackrel{?}{=} ([\partial_t w]+[\partial_z w])(\cS_0(w)+\kappa^2 \sfq^{z});
$$
$$
 \text{Equivalently:}\qquad  [\partial_t w](\sfq^{\sft-z}-\kappa^2 \sfq^{z})\stackrel{?}{=} [\partial_z w](\cS_0(w)+\kappa^2 \sfq^{z}).
$$
The last identity is checked by taking the ratio of two identities in \eqref{e:x25}.
\end{proof}

\subsection{Proof of Theorem \ref{Theorem_LLN_abstract}}
\label{Section_proof_of_abstract_LLN}

The measures $\rho(\cdot; \bmx(t))$ of densities $\rho(s;\bmx(t))$ as given by \eqref{eq_rho_x_t_def}, are originally defined for integer values of $t$ and we now extend the definition to all $t\in [0,T]$ by linear interpolation between integers.

Let $C([0,\sfT], \mathscr M_1(\bR)])$ denote the space of continuous functions from $[0,\sft]$ to the set of probability measures on $\mathbb R$ equipped with the weak topology of measures. Then $\rho(\cdot; \bmx(\eps^{-1} \sft))$ represents a random element of this space. In fact, there is a compact set $\mathcal K\subset C([0,\sfT], \mathscr M_1(\bR)])$, such that the distribution of $\rho(\cdot; \bmx(\eps^{-1} \sft))_{0\leq \sft \leq \sfT}$ is supported on $\mathcal K$ --- this is because measures $\rho(\cdot; \bmx(\eps^{-1} \sft))$ are supported inside $[\sft-\sfT, \sfN]$ and the dependence on $\sft$ is Lipschitz (each component of $\bmx(t)$ jumps at most by $1$ when $t$ grows by $1$), cf.\ \cite[Lemma 4.3.13]{AGZ}.

Since the space of probability measures on a compact set is compact, we conclude that the stochastic processes $\rho(\cdot; \bmx(\eps^{-1} \sft))_{0\leq \sft\leq \sfT}$ have subsequential limits in distribution as $\eps\to 0$. We let $(\rho_\sft)_{0\leq \sft\leq \sfT}$ be one of the limiting points. Our task is to show that $\rho_\sft$ are described by Propositions \ref{Proposition_LLN_Burgers}, \ref{Proposition_LLN_equation}, \ref{Proposition_LLN_invariant}. (Implying, in particular, that all the limiting points are the same and, therefore, coincide with $\eps\to 0$ limit.) Note that for each $0\leq \sft\leq \sfT$, $\rho_\sft$ is an absolutely continuous measures of density at most $1$, because so were the prelimit measures.

\bigskip

For any $\sft\in [0,\sfT]\cap \eps\bZ$, using Proposition \ref{Proposition_tilings_increment}, we have
\begin{multline}\label{e:martingale_equation}
\int_{\bl(\sft)}^{\br(\sft)} \frac{\sfb'_{\sft}(z) \rho(s;\bmx(\eps^{-1}\sft))}{\sfb_{\sft}(z)-\sfb_{\sft}(s)}\rd s= \int_{\bl(0)}^{\br(0)} \frac{\sfb'_{0}(z) \rho(s;\bmx(0))}{\sfb_{0}(z)-\sfb_{0}(s)}\rd s  \\+\varepsilon \widetilde \cM_\sft(z)+
\varepsilon\sum_{\tau \in[0,\sft)\cap \eps \bZ}\left[-\del_z \ln\tilde\cB_{\tau}(z)+\ln(\sfq)\frac{\kappa^2 \sfq^{z-\tau}+\sfq^{-z}}{\kappa^2\sfq^{z-\tau}-\sfq^{-z}}-\del_\tau \ln \tilde \varphi_\tau^+(z)\right]
+\OO\left(\varepsilon \right),
\end{multline}
where $\widetilde\cM_\sft(z)$ is a martingale given by
\begin{align}\label{e:deftMt}
\widetilde\cM_\sft(z)\deq \sum_{\tau \in [0, \sft)\cap \eps\bZ}  \Delta \cM_{\eps^{-1}\tau}(z).
\end{align}
We can estimate the $\widetilde\cM_\sft(z)$ term by Doob/Kolmogorov's inequality for martingales:
$$
 {\rm Prob} \left(\sup_{0\leq \sft\leq \sfT} |\eps \widetilde\cM_\sft(z)|>\lambda \right)\leq \frac{\eps^2}{\lambda^2}\sum_{\tau \in [0, \sfT)\cap \eps\bZ}  \bE [\Delta \cM_{\eps^{-1}\tau}(z)]^2, \qquad \lambda>0.
$$
Using \eqref{e:covT_tilings_2}, each term in the last sum is $\OO(\eps)$. There are $\OO(\eps^{-1})$ terms and therefore the probability decays as $\OO(\eps^2/\lambda^2)$. We conclude that the martingale part in \eqref{e:martingale_equation} goes to $0$ in probability as $\eps\to 0$. Hence, $\eps\to 0$ limit of \eqref{e:martingale_equation} gives an integral equation for $\rho_\sft$:
\begin{align}\begin{split}\label{e:LLN_integral_equation}
\int_{\bl(\sft)}^{\br(\sft)} \frac{\sfb'_{\sft}(z) \rho_\sft(s)}{\sfb_{\sft}(z)-\sfb_{\sft}(s)}\rd s&= \int_{\bl(0)}^{\br(0)} \frac{\sfb'_{0}(z) \rho_0(s)}{\sfb_{0}(z)-\sfb_{0}(s)}\rd s \\
& + \int_0^{\sft} \left[-\del_z \ln\widehat \cB_{\tau}(z)+\ln(\sfq)\frac{\kappa^2 \sfq^{z-\tau}+\sfq^{-z}}{\kappa^2\sfq^{z-\tau}-\sfq^{-z}}-\del_\tau \ln \tilde \varphi_\tau^+(z)\right] \rd \tau,
\end{split}\end{align}
where $\widehat \cB_{\tau}(z)$ is the limit of $\tilde \cB_{\tau}(z)$ given by $\widehat \cB_{\tau}(z)=\tilde \varphi_\tau^-(z)(1-f_\tau(z))$ using the notation \eqref{e:double_complex_slope}. We also recall that
$$
 f_\sft(z)=-\frac{\tilde\varphi^+_\sft(z)}{\tilde\varphi^-_\sft(z)} \exp\left(\int_{\bl(\sft)}^{\br(\sft)} \frac{\sfb'_{\sft}(z) \rho_\sft(s)}{\sfb_{\sft}(z)-\sfb_{\sft}(s)}\rd s\right).
$$
Hence, differentiating \eqref{e:LLN_integral_equation} in $\sft$, we get
\begin{align*}
 &\phantom{{}={}}\partial_\sft \ln f_\sft(z)- \partial_\sft \ln \tilde\varphi^+_\sft(z) +\partial_\sft \tilde\varphi^-_\sft(z)\\
 &=-\partial_z \ln \varphi_\sft^-(z)-\partial_z \ln (1-f_\sft(z))
 +\ln(\sfq)\frac{\kappa^2 \sfq^{z-\sft}+\sfq^{-z}}{\kappa^2\sfq^{z-\sft}-\sfq^{-z}}-\partial_\sft \ln \tilde \varphi_\sft^+(z).
\end{align*}
Cancelling $-\partial_\sft \ln \tilde \varphi_\sft^+(z)$ terms and noting that $\partial_\sft \tilde\varphi^-_\sft(z)=-\partial_z \ln \varphi_\sft^-(z)$ from the definition \eqref{e:phi_functions_defs}, we arrive at the PDE \eqref{e:Burgers_first_appearance} of Proposition \ref{Proposition_LLN_Burgers}. Theorem \ref{Theorem_three_descriptions} then yields that the measures $\rho_\sft$, $0\leq \sft\leq \sfT$, are given by the procedures of Propositions \ref{Proposition_LLN_equation} or \ref{Proposition_LLN_invariant}.

\subsection{Proof of Theorems \ref{t:arctic} and \ref{t:limitshape}}
\label{Section_proof_of_LLN_intro}

In Theorem \ref{Theorem_LLN_abstract} we identified the limit shape for lozenge tilings of the trapezoids and in this subsection we explain that the description of this theorem matches those of Theorems \ref{t:arctic} and \ref{t:limitshape} in the introduction.

We start by analyzing the equation \eqref{e:solveu}. Note that this is the same equation as \eqref{e:slope_equation}; however the variable $z$ is a complex number outside the support of $\rho_\sft$ in the latter, while a similar variable $x$ in the former is a real number, which might belong to the support of $\rho_\sft$. Let us recast the equation. We have $f_0(u)$ is given by
\begin{equation} \label{e:initial_condition_two_ways}
\frac{(\sfq^{\sfN}-\sfq^{u})(\kappa^2\sfq^{-\sfT}-\sfq^{-u})}{
(\kappa^2 \sfq^{\sfN}-\sfq^{-u})(\sfq^{-\sfT}-\sfq^{u})}\prod_{i=1}^r\frac{(\sfq^{\sfa_i}-\sfq^{u})(\kappa^2 \sfq^{\sfb_i}-\sfq^{-u})}{(\kappa^2 \sfq^{\sfa_i}-\sfq^{-u})
(\sfq^{\sfb_i}-\sfq^{u})}=\exp\left(-\int_{\sft-\sfT}^{\sfN} \frac{\sfb'_\sft(u)(1-\rho_0(s))}{\sfb_\sft(u)-\sfb_\sft(s)}\rd s\right),
\end{equation}
where we used Remark \ref{Remark_two_forms_IC} to match two forms of $f_0(u)$. Let us introduce the new variable $\uu=\sfq^u$, so that
$$
 f_0(u)=g_0(\sfq^u), \qquad g_0(\uu)=\sfq^{-\sfT}\frac{(\uu-\sfq^{\sfN})(\uu -\kappa^{-2}\sfq^{\sfT})}{
(\uu -\kappa^{-2} \sfq^{-\sfN})(\uu-\sfq^{-\sfT})}\prod_{i=1}^r\frac{(\uu - \sfq^{\sfa_i})(\uu - \kappa^{-2} \sfq^{-\sfb_i})}{(\uu -\kappa^{-2} \sfq^{-\sfa_i})
(\uu-\sfq^{\sfb_i})}.
$$
(We used $\sum_{i=1}^r (\sfb_i-\sfa_i)=\sfN$ in transformation of $f_0$ into $g_0$.)
Using \eqref{e:defUV}, the equation \eqref{e:solveu} has the form
\begin{align}\label{e:solveu_2}
\sfq^{-\sfx}+\kappa^2 \sfq^{\sfx-\sft}=\frac{(\sfq^{-\sft}-1)( \kappa^2 \uu-\uu^{-1})}{1-g_0(\uu)}+\kappa^2 \uu +\uu^{-1}\sfq^{-\sft}.
\end{align}

\begin{proposition} \label{Proposition_number_of_roots} Take $\sft\in [0,\sfT]$ and $x\in [\sft -\sfT, \sfN]$. For the equation \eqref{e:solveu_2} either:
\begin{itemize}
\item all its solutions $\uu$ have real $\uu^{-1}+\kappa^2 \uu$; or
\item there is a unique quadruple of complex solutions $(\uu,\bar \uu, \kappa^{-2} \uu^{-1},  \kappa^{-2} \bar \uu^{-1})$ with complex $\uu^{-1}+\kappa^2 \uu$ and all other solutions have real $\uu^{-1}+\kappa^2 \uu$.
\end{itemize}
\end{proposition}
\begin{proof} Recall that we still only detail the case of $0<\sfq<1$ and large real $\kappa$ as in Assumption \ref{a:para}.

 Multiplying \eqref{e:solveu_2} by $\uu(1-g_0(\uu))$ and then by all denominators in the definition of $g_0(\uu)$ we arrive at a polynomial equation of degree $2r+4$ with real coefficients, which has at most $2r+4$ complex solutions. Note that the equation is unchanged if we replace $\uu$ by $\kappa^{-2} \uu^{-1}$ ($g_0(\uu)$ turns into $\tfrac{1}{g_0(\uu)}$ under this transformation). Hence, all its complex solutions come in quadruples $(\uu,\bar \uu, \kappa^{-2} \uu^{-1},  \kappa^{-2} \bar \uu^{-1})$. If $\uu$ is non-real, then all these numbers are distinct unless $\uu=\kappa^{-2} \bar \uu^{-1}$. In the latter case $\uu^{-1}+\kappa^{-2} \uu$ is real. Hence, non-real $\uu^{-1}+\kappa^{-2}\uu$ leads to four distinct non-real solutions to \eqref{e:solveu_2}. Therefore, the statement of the proposition would follow, if we manage to find $2r-3$ distinct real solutions to \eqref{e:solveu_2}.

 Let us search for real solutions satisfying $\uu>\kappa^{-1}$. Due to $\uu\mapsto \kappa^{-2}\uu^{-1}$ symmetry of the equation it would be sufficient to find $r-1$ of such solutions. For that we study the behavior of the function $g_0(\uu)-1$ on the segment $[\kappa^{-1},+\infty)$.

Immediately from its definition, $g_0$ has $r$ zeros at points $\uu=\sfq^{\sfa_i}$ and $r$ poles at points $\uu=\sfq^{\sfb_i}$, $1\leq i \leq r$. There is another pole at $\uu=\sfq^{-\sfT}$. Finally, there are two cases: in case A, $\sfq^{\sfN}$ is another zero in $[\uu,+\infty)$; in case B, $\kappa^{-2}\sfq^{-\sfN}$ is another pole in $[\uu,+\infty)$. (Only one of these two points belongs to $[\kappa^{-1},+\infty)$; this point is inside $(\kappa^{-1}, \sfq^{\sfb_r})$.) Other zeros and poles of $g_0$ are outside $[\kappa^{-1},+\infty)$. In particular, we conclude that $g_0$ is positive on each segment $(\sfq^{\sfb_i}, \sfq^{\sfa_i})$.

 Note that $g_0(\uu)=1$ is a polynomial equation of degree $2r+2$. It has two explicit solutions at $\uu=\pm \kappa^{-1}$. In addition, on  each segment $[\sfq^{\sfb_i}, \sfq^{\sfa_i}]$ the continuous function $g_0(\uu)$ changes from $+\infty$ to $0$; hence, there are points $\sfc_i\in(\sfa_i,\sfb_i)$, $1\leq i \leq r$, such that $g_0(\sfq^{\sfc_i})=1$. Due to symmetry $g_0(\kappa^{-2} \uu^{-1})=(g_0(\uu))^{-1}$, we also have $g_0(\kappa^{-2} \sfq^{-\sfc_i})=1$. Therefore, we have found all $2r+2$ zeros of $g_0(\uu)-1$ and all these zeros are simple.

We can now analyze the function $\frac{1}{1-g_0(\uu)}$ on $[\kappa^{-1},+\infty)$: it has poles at $r+1$ points $\{\kappa^{-1}, \sfq^{\sfc_r}, \sfq^{\sfc_{r-1}},\dots, \sfq^{\sfc_1}\}$ and zeros at $r+1$ points $\{\sfq^{\sfb_r},\sfq^{\sfb_{r-1}},\dots, \sfq^{\sfb_1}, \sfq^{-\sfT}\}$; there is one more zero at $\kappa^{-2} \sfq^{-\sfN}$ in case $B$. This is schematically shown in Figure \ref{f:signs} together with signs of $\frac{1}{1-g_0(\uu)}$, which can be understood from zeros, poles, and observation $\lim_{\uu\to+\infty} g_0(\uu)=\sfq^{-\sfT}>1$. We conclude that on each segment $(\sfq^{\sfc_i}, \sfq^{\sfc_{i-1}})$, $1<i\leq r$, the function $\frac{1}{1-g_0(\uu)}$ changes from $+\infty$ to $-\infty$ and so does the right-hand side of \eqref{e:solveu_2}. Thus, there is a solution to \eqref{e:solveu_2} on each these $r-1$ segments and we found the desired $r-1$ real solutions on $(\kappa^{-1},+\infty)$.
\begin{figure}[t]
\center
\includegraphics[width=0.45\linewidth]{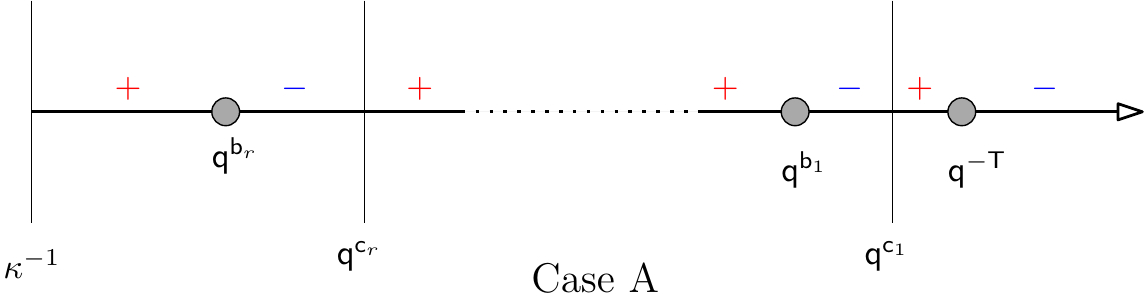}\hfill \includegraphics[width=0.45\linewidth]{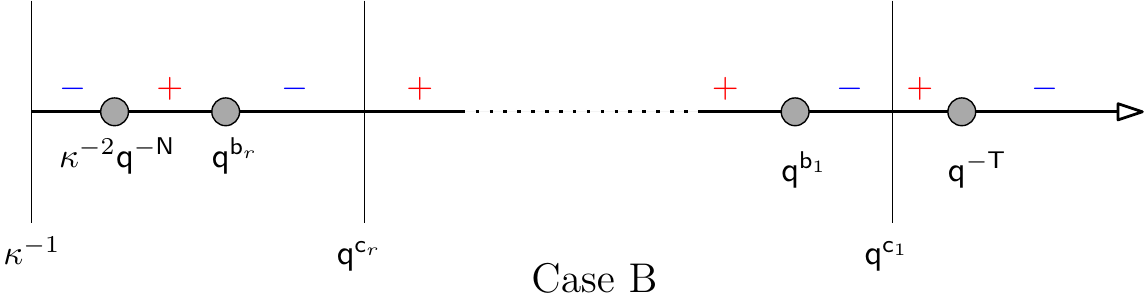}
\caption{Zeros (gray circles), poles (vertical lines), and signs for the function $\frac{1}{1-g_0(\uu)}$ for $u\in(\kappa^{-1},+\infty)$. \label{f:signs}}
\end{figure}
 %$$
 % g_0(\uu)-1=\frac{ \sfq^{-\sfT}(\uu-\sfq^{\sfN})(\uu -\kappa^{-2}\sfq^{\sfT}) \prod_{i=1}^r\left[(\uu - \sfq^{\sfa_i})(\uu - \kappa^{-2} \sfq^{-\sfb_i})\right]-(\uu -\kappa^{-2} \sfq^{-\sfN})(\uu-\sfq^{-\sfT})\prod_{i=1}^r \left[(\uu -\kappa^{-2} \sfq^{-\sfa_i})
%(\uu-\sfq^{\sfb_i})\right]}{
%(\uu -\kappa^{-2} \sfq^{-\sfN})(\uu-\sfq^{-\sfT})\prod_{i=1}^r \left[(\uu -\kappa^{-2} \sfq^{-\sfa_i})
%(\uu-\sfq^{\sfb_i})\right]}.
% $$
\end{proof}

\begin{proposition} \label{Proposition_liquid}A point $(\sft,x)$, $0\leq \sft \leq \sfT$, $\sft-\sfT\leq x\leq \sfN$, belongs to the liquid region $\sfL(\sfP)$ of \eqref{e:defLP} (in the notations of Theorem \ref{Theorem_LLN_abstract} it means that $0<\rho_\sft(x)<1$) if and only if the equation \eqref{e:solveu_2} has a solution such that $\uu^{-1}+\kappa^2 \uu$ is non-real.
\end{proposition}
\begin{proof}
 Suppose that $0<\rho_\sft(x)<1$. Then, as we observed in Step 1 of the proof of Theorem \ref{Theorem_three_descriptions}, $\lim_{\eps\to 0+} f_\sft(x+\ii \eps)$ is non-real. Using Proposition \ref{Proposition_LLN_equation}, this implies through \eqref{e:slope_through_u} that $\sfq^{-u}+\kappa^2 \sfq^u$ is non-real. Since $\sfq^u$ solves \eqref{e:slope_equation}, which coincides at $z=x$ with \eqref{e:solveu_2} in the variable $\uu=\sfq^u$, we conclude that \eqref{e:solveu_2} has a solution with non-real $\uu^{-1}+\kappa^2 \uu$.

 In the opposite direction, suppose that \eqref{e:solveu_2} has a  solution with non-real $\uu^{-1}+\kappa^2 \uu$; among such solutions let $w^0$ be the one with negative imaginary part. As in Step 3 in the proof of Theorem \ref{Theorem_three_descriptions}, we start the dynamics $(\sft,w^\sft, \cS^\sft)_{\sft\geq 0}$ from this particular $w^0$. The dynamics is rewritten as the linear evolution in \eqref{e:x17} and comparing this evolution with \eqref{e:solveu_2} we conclude that $w^\sft$ hits the real axis at the time $\sft$ given in the statement of Proposition \ref{Proposition_liquid} and at point $w^\sft=\sfq^{-x}+\kappa^2 \sfq^{x-\sft}$. On the other hand, the value of $\cS^\sft$ is preserved in the time evolution and equals $\cS^0$, which is non-real according to Step 4 in the proof of Theorem \ref{Theorem_three_descriptions} (applied for $\sft=0$). Looking back to the definition of $\cS^\sft$ from \eqref{e:R_def}, we conclude that $f_\sft(x)$ should be non-real. Hence, as in Step 1 of the proof of Theorem \ref{Theorem_three_descriptions}, we should necessary have $0<\rho_\sft(x)<1$.
\end{proof}

\begin{proof}[Proof of Theorem \ref{t:limitshape}]
 We use Theorem \ref{Theorem_LLN_abstract} and note that the height function of Theorem \ref{t:limitshape} can be written as
 $$
  \eps h(\eps^{-1}\sft,\eps^{-1}\sfx)=\int_{-\infty}^{\sfx} \rho(s,\bmx(\eps^{-1}\sft)) \rd s.
 $$
 Using  \eqref{e:weak_convergence} and noting that $h(\eps^{-1}\sft,\eps^{-1}\sfx)$ is uniformly Lipschitz in $(\sft,\sfx)$, we conclude that the convergence of \eqref{e:hlimit2} holds and the limiting height function is described through
 \begin{equation}
 \label{e:x34}
  \sfh(\sft ,  \sfx )= \int_{-\infty}^\sfx \rho_\sft(s) \rd s.
 \end{equation}
 Proposition \ref{Proposition_liquid} implies the description of the liquid region $\sfL(\sfP)$, as claimed in Theorem \ref{t:limitshape}. We further use the description of $\rho_\sft$ of Proposition \ref{Proposition_LLN_equation}. Taking $(\sft,\sfx)$ in the liquid region and sending $z\to \sfx+\ii 0$ in \eqref{e:slope_through_u}, we arrive at the formula \eqref{e:ftx}. The only difference is that we approached the definition of the complex slope $f_\sft$ in a different way: in Proposition \ref{Proposition_LLN_equation} we used \eqref{e:double_complex_slope} in terms of the measure $\rho_\sft$, while in Theorem \ref{t:limitshape} we used definition \eqref{e:localdd} in terms of the local proportions of lozenges. Hence, we need to show that these two definitions coincide: for that we take $f_\sft$ as given by \eqref{e:double_complex_slope} and show it satisfies \eqref{e:localdd}. We computed $\tfrac{1}{\pi}\arg f_\sft(\sfx+\ii 0)=1-\rho_\sft(\sfx)$ in Step 1 of the proof of Theorem \ref{Theorem_three_descriptions}. Comparing \eqref{e:x34} with \eqref{eq_proportions_through_derivatives}, we conclude that $p_\La(\sft,\sfx)=1-\rho_\sft(x)$ and arrive at the first identity of \eqref{e:localdd}.

 For the second identity of \eqref{e:localdd} we need to use the PDE (version of the complex Burgers equation) of \eqref{e:Burgers_first_appearance}. Choose $M$ to be a very large positive number. We have
 \begin{align*}
 p_{\La}+p_{\Lc}&=1-\del_{\sfx} \sfh(\sft,\sfx)+\del_{\sfx} \sfh(\sfx,\sft)+\del_\sft  \sfh(\sft,\sfx)=1+\del_\sft  \sfh(\sft,\sfx)=1+\del_\sft \int_{-M}^\sfx \rho_\sft(s) \rd s\\
 &= \del_\sft\int_{-M}^{\sfx} \left[(\rho_\sft(s)-1) \mathbf 1_{s\geq \sft -\sfT}\right] \rd s=-\frac{1}{\pi} \del_\sft \lim_{\eps\to 0+}\int_{-M}^\sfx \arg f_\sft(s+\ii\eps )\rd s
 \\&=-\frac{1}{\pi}  \lim_{\eps\to 0+}\int_{-M}^\sfx \del_{\sft}\left[ \Im \ln f_\sft(s+\ii\eps )\right]\rd s
 =-\frac{1}{\pi}  \lim_{\eps\to 0+}\int_{-M}^\sfx  \Im\left[  \del_{\sft} \ln f_\sft(s+\ii\eps )\right]\rd s
 \\
 &=\frac{1}{\pi}  \lim_{\eps\to 0+}\int_{-M}^\sfx  \Im\left[  \del_{s} \ln (1-f_\sft(s+\ii\eps ))-\ln(\sfq)\frac{\kappa^2 \sfq^{s+\ii\eps -\sft}+\sfq^{-s-\ii\eps }}{\kappa^2\sfq^{s+\ii\eps -\sft}-\sfq^{-s-\ii\eps }}\right]\rd s
 \\
 &=\frac{1}{\pi}  \Im \left[\lim_{\eps\to 0+}\int_{-M}^\sfx   \del_{s} \ln (1-f_\sft(s+\ii\eps ))\rd s\right]=\frac{1}{\pi} \arg \left(  \frac{1-f_\sft(\sfx+\ii 0)}{1-f_\sft(-M+\ii 0 )}\right).
 \end{align*}
For $z=-M+\ii 0$ with large positive $M$, the integrand in \eqref{e:double_complex_slope} is close to $\ln\sfq \cdot (1-\rho_\sft(s))$. Hence, the exponent is positive and $f_\sft(-M+\ii 0)$ is a real number greater than $1$. We conclude that
$$
 p_{\La}+p_{\Lc}=\frac{1}{\pi} \arg \left(  \frac{f_\sft(\sfx+\ii 0)-1}{f_\sft(-M+\ii 0 )-1}\right)=\frac{1}{\pi}\arg  \bigl(f_\sft(\sfx+\ii 0 )-1\bigr),
$$
which matches the second identity of \eqref{e:localdd}.
\end{proof}

\begin{proof}[Proof of Theorem \ref{t:arctic}]
 We have shown in Propositions \ref{Proposition_number_of_roots}, \ref{Proposition_liquid} that for $(\sft,\sfx)$ in the liquid region, \eqref{e:solveu_2} has non-real complex conjugate solutions, while outside the liquid region all the solutions are real. Hence, on the arctic curve bounding the liquid region, the two complex solutions glue together and \eqref{e:solveu_2} has a double root, which is the same as \eqref{e:solveu} having a double root. Hence, the $u$--derivative of the right-hand side in \eqref{e:solveu} should vanish and we have a system of two equations:
 \begin{equation}
 \label{e:x35}
  \begin{cases}\sfq^{-\sfx}+\kappa^2 \sfq^{\sfx-\sft}=V(u)-\frac{U(u)}{\sfq^\sft}, \\
    V'(u)=\frac{U'(u)}{\sfq^\sft}.
   \end{cases}
 \end{equation}
We parameterize solutions to \eqref{e:x35} by $u$: given $u$, $\sfq^{\sft}$ is found from the second equation and then $\sfq^\sfx$ is found by solving the quadratic equation in $\sfq^{\sfx}$ of the first equation. This is precisely \eqref{eq_parameterization_intro}.

It remains to understand the possible values for the parameter $u$. It is convenient to argue in terms of $w=\sfq^{-u}+\kappa^2 \sfq^u$. As we know from Steps 3 and 5 in the proof of Theorem \ref{Theorem_three_descriptions}, if we take $\sft\in (0,\sfT)$ and $w$ in the lower half-plane, then $\sfx$ given by the first equation of \eqref{e:x35} is the point where the flow $w^\sft$ (started from $w^0=w$) first leaves the lower half-plane and $(\sft,\sfx)$ belong to the liquid region in this case; because everything is invariant under conjugations, the same argument applies to $w$ in the upper half-plane. We conclude that $w$ has to be real.

Let us now show that each real $w$ (except for those where some parts of \eqref{e:x35} are singular) leads to a point on the arctic curve in this way. For that we recall Lemma \ref{Lemma_bijectivity} and Remark \ref{Remark_extension_to_the_boundary}: at time $\sft$ the part of the boundary $\{w\in \mathbb C\mid \Im(w-(\sfq^{-\sft}-1) \cS_0(w))\leq 0\}$ with negative imaginary part forms a collection of curves $\gamma_\sft$ in the lower half-plane. These curves are in bijection with the section $\mathcal L_\sft$ of the liquid region at time $\sft$ and, therefore, their end-points (which are on the real axis) correspond to the points of the arctic boundary at times $\sft$. As time $\sft$ grows, the set $\{w\in \mathbb C\mid \Im(w-(\sfq^{-\sft}-1) \cS_0(w))\leq 0\}$ becomes larger and larger, until it swallows entire lower halfplane at time $\sfT$, as follows from the representation \eqref{e:U_measure} for $\cS_0(w)$. Hence, the end-points of the curves $\gamma_\sft$ eventually reach all points on the real as $\sft$ changes from $0$ to $\sfT$.
 \end{proof}

\section{Gaussian Free Field}
\label{Section_GFF}

In this section we prove Theorem \ref{t:GFF}.

\subsection{Stochastic evolution of fluctuations}

Our first task is to use the result of Theorem \ref{t:loopstudy}, as recorded in Section \ref{Section_massaging}, to produce a stochastic evolution equation for the Gaussian macroscopic fluctuations of random tilings. We use the notations of Section \ref{Section_assumptions}, in particular, $\sft=\eps t$. We introduce the field of centered fluctuations:
\begin{equation}
\label{e:centered_G_def}
G_t(z)=\int_{\bl(\sft)}^{\br(\sft)} \frac{\sfb'_\sft(z) \rho(s;\bmx(t))}{\sfb_\sft(z)-\sfb_\sft(s)}- \bE \left[\int_{\bl(\sft)}^{\br(\sft)} \frac{\sfb'_\sft(z) \rho(s;\bmx(t))}{\sfb_\sft(z)-\sfb_\sft(s)} \right],  z\in\mathbb C,  \sfb_\sft(z)\notin \sfb_\sft([\bl(\sft),\br(\sft)]).
\end{equation}
\begin{proposition} \label{Proposition_stochastic_evolution}
 In the framework of Theorem \ref{Theorem_LLN_abstract} and using the notations of Sections \ref{Section_tilings_setup}, \ref{Section_assumptions},
  for $\sft=t\eps$ we have as $\eps\to 0$:
 \begin{equation}\label{e:dmg_tilings_3}
\frac{1}{\varepsilon}\Bigl[G_{t+1}(z)-G_t(z)\Bigr]= \del_z\left[G_t(z) \frac{ \bE \tilde f_\sft(z)}{1- \bE \tilde f_\sft(z)}\right] +\Delta\cM_t(z)+\mathcal R,
\end{equation}
where $\Delta\cM_t(z)$ are mean $0$ (both conditionally on $\bmx(t)$, and unconditionally)  random variables such that $\{\varepsilon^{-1/2}\Delta \cM_t(z)\}_{z\in \Lambda_\sft\setminus[\bl(\sft),\br(\sft)]}$ are asymptotically Gaussian with covariance  given by
\begin{align}\label{e:covT_tilings_3}
\frac{\bE[\Delta \cM_t(z_1)\Delta \cM_{t'}(z_2)]}{\varepsilon}= \frac{\delta_{t=t'}}{2\pi \ri}\oint_{\cin} \frac{f_\sft(w)}{ f_\sft(w)-1}\frac{\sfb_\sft'(w)\sfb_\sft'(z_1)}{(\sfb_\sft(w)-\sfb_\sft(z_1))^2} \frac{\sfb_\sft'(w)\sfb_\sft'(z_2)}{(\sfb_\sft(w)-\sfb_\sft(z_2))^2} \rd w+\oo(1),
\end{align}
the contour $\cin\subset \Lambda_\sft$ encloses $[\bl(\sft), \br(\sft)]$, but not $z_1$ or $z_2$;  $f_\sft(w)$ is as in \eqref{e:double_complex_slope}, \eqref{e:double_complex_slope_2}. The asymptotic Gaussianity is in the convergence of moments sense, as in Theorem \ref{t:loopstudy}. The (random) remainder $\mathcal R$ satisfies an upper bound
\begin{equation}
\label{e:x40}
 |\mathcal R|=\oo(\eps)+\oo(1) \left|G_t(z)\right|+\oo(1) \bE \left|G_t(z)\right|,
\end{equation}
where all the implicit constant can be made deterministic and uniform in $\eps$, in $t=0,1,\dots,T-1$, and in $z$ belonging to compact subsets of $\mathbb C$, such that $\sfb_\sft(z)$ is away from $\sfb_\sft([\bl(\sft),\br(\sft)])$.
\end{proposition}
\begin{proof} Subtracting from the result of Proposition \ref{Proposition_tilings_increment} its expectation, we get
 \begin{equation}\label{e:dmg_tilings_4}
\frac{1}{\varepsilon}\Bigl[G_{t+1}(z)-G_t(z)\Bigr]=\bigl(\bE \del_z \ln\tilde\cB_{\sft}(z)-\del_z \ln\tilde\cB_{\sft}(z)\bigr)+\eps(R_t(z)-\bE R_t(z))+\Delta\cM_t(z)+ \OO(\eps^2).
\end{equation}
Using Theorem \ref{Theorem_LLN_abstract}, we conclude that $\eps(R_t(z)-\bE R_t(z))$ can be absorbed into $o(\eps)$ term in \eqref{e:x40}. $O(\eps^2)$ in \eqref{e:dmg_tilings_4} is also absorbed in the same remainder. Further, writing $\tilde f_\sft(z)=\bE \tilde f_\sft(z) + (\tilde f_\sft(z)-\bE\tilde f_\sft(z))$ and  using \eqref{e:x4}, we have
\begin{multline}
\label{e:x36}
\del_z \ln\tilde\cB_{\sft}(z)= \del_z  \ln \tilde \varphi_t^-(z) - \frac{\del_z \tilde f_\sft(z)}{1-\tilde f_\sft(z)}\\=\del_z  \ln \tilde \varphi_t^-(z) - \frac{\del_z  \bE \tilde f_\sft(z)}{1- \bE \tilde f_\sft(z)}
- \frac{\left[\del_z \tilde f_\sft(z)-\del_z  \bE \tilde f_\sft(z)\right]}{1- \bE \tilde f_\sft(z)}-  \frac{\left[\tilde f_\sft(z)-\bE \tilde f_\sft(z)\right] \del_z  \bE \tilde f_\sft(z)}{(1- \bE \tilde f_\sft(z))^2}\\ + \left|\tilde f_\sft(z)-\bE \tilde f_\sft(z)\right| \OO\left( \left|\tilde f_\sft(z)-\bE \tilde f_\sft(z)\right|+\left|\partial_z\tilde f_\sft(z)-\partial_z \bE\tilde f_\sft(z)\right|\right).
\end{multline}
The implicit constant in the remainder admits a deterministic bound $C(z)$, which stays uniformly bounded as long as $\sfb_\sft(z)$ is bounded away from $\sfb_\sft([\bl(\sft),\br(\sft)])$.

Let us introduce a notation
$$
m_t(z)=\bE \left[\int_{\bl(\sft)}^{\br(\sft)} \frac{\sfb'_\sft(z) \rho(s;\bmx(t))}{\sfb_\sft(z)-\sfb_\sft(s)} \right].
$$
Then we have using \eqref{e:x28}:
\begin{equation}
\label{e:x37}
 \tilde f_\sft(z) =\exp(G_t(z))  \left[-\exp(m_t(z)) \frac{\tilde \varphi_\sft^+(z)}{\tilde \varphi_\sft^-(z)} \right]= (1+G_t(z)) \left[-\exp(m_t(z)) \frac{\tilde \varphi_\sft^+(z)}{\tilde \varphi_\sft^-(z)} \right] + \OO\bigl( |G_t(z)|^2\bigr).
\end{equation}
The remainder in the last formula admits a bound $C(z) |G_t(z)|^2$ with a deterministic function $C(z)$, which stays uniformly bounded as long as $\sfb_\sft(z)$ is bounded away from $\sfb_\sft([\bl(\sft),\br(\sft)])$. Thus, taking expectation of \eqref{e:x37},  we get
\begin{equation}
\label{e:x38}
 \bE \tilde f_\sft(z) =\left[-\exp(m_t(z)) \frac{\tilde \varphi_\sft^+(z)}{\tilde \varphi_\sft^-(z)} \right] + \OO\bigl( \bE |G_t(z)|^2\bigr).
\end{equation}
Hence, subtracting \eqref{e:x38} from \eqref{e:x37}, we get
\begin{equation}
\label{e:x39}
 \tilde f_\sft(z)-\bE \tilde f_\sft(z)= G_t(z) \bE \tilde f_\sft(z)
 + \OO\bigl( |G_t(z)|^2\bigr)+ \OO\bigl(\bE |G_t(z)|^2\bigr) + \OO\bigl(|G_t(z)| \bE |G_t(z)|^2\bigr).
\end{equation}
Because $|G_t(z)|$ is bounded and tends to $0$ as $\eps\to 0$ by Theorem \ref{Theorem_LLN_abstract}, we can replace all the remainders in the last formula by $\oo(1) \left|G_t(z)\right|+\oo(1) \bE \left|G_t(z)\right|$. Applying the same procedure to $\partial_z \tilde f_\sft(z)$ (leading to the answer, which is $z$--derivative of \eqref{e:x39}), we also conclude that
\begin{equation}
\label{e:x39_2}
 \partial_z \tilde f_\sft(z)-\partial_z \bE \tilde f_\sft(z)= \partial_z\left( G_t(z) \bE \tilde f_\sft(z) \right)
 + \oo(1) \left|G_t(z)\right|+\oo(1) \bE \left|G_t(z)\right|.
\end{equation}
The approximations we have just developed, imply (using Theorem \ref{Theorem_LLN_abstract}) that the remainder in \eqref{e:x36} can be absorbed into $\oo(1) \left|G_t(z)\right|$ term in \eqref{e:x40}. Thus, subtracting \eqref{e:x36} from its expectation and using \eqref{e:x39}, \eqref{e:x39_2}, we get
\begin{align*}
 \bE \del_z \ln\tilde\cB_{\sft}(z)-\del_z \ln\tilde\cB_{\sft}(z)=& \frac{\del_z[ G_t(z) \bE \tilde f_\sft(z)]}{1- \bE \tilde f_\sft(z)}+  \frac{ G_t(z) \bE \tilde f_\sft(z) \del_z  \bE \tilde f_\sft(z)}{(1- \bE \tilde f_\sft(z))^2}\\
 &+\oo(1) \left|G_t(z)\right|+\oo(1) \bE \left|G_t(z)\right|.
\end{align*}
Plugging the last expression into \eqref{e:dmg_tilings_4}, and noting that the conditional covariance of the martingale part \eqref{e:covT_tilings_2}, by Theorem \ref{Theorem_LLN_abstract}, approximates the unconditional covariance of \eqref{e:covT_tilings_3},  we reach \eqref{e:dmg_tilings_3}.
\end{proof}

\subsection{Fluctuations along the characteristic flow} \label{Section_fluctuations_characteristic}
The evolution equation \eqref{e:dmg_tilings_3} for fluctuations might seem complicated. However, it plays nicely with the characteristic flow \eqref{e:z_DE}, which we used in Section \ref{Section_Limit_shape_proofs}, because the ratio $\frac{ f_\sft(z)}{f_\sft(z)-1}$ appears in both.

We introduce a function of two variables $z(\sft,u)$ defined through the flow $z^\sft$ with initial condition $u$:
\begin{equation}
\label{e:x41}
 z(\sft,u):=z^\sft, \qquad \qquad  \begin{cases}\del_\sft z^\sft=\frac{f_\sft(z^\sft)}{f_\sft(z^\sft)-1},& \sft\geq 0,\\ z^0=u. \end{cases}.
\end{equation}
Here $u$ is complex, and $\sft\geq 0$ is real. The map $u\mapsto z(\sft,u)$ agrees with complex conjugations: $z(\sft,\bar u)=\overline{z(\sft,u)}$.
We  also recall that \eqref{e:ztsolver_1}, or \eqref{e:ztsolver_2}, or \eqref{e:x17} give more explicit forms for $z(\sft,u)$, not involving any differential equations.

We also need a discrete $\eps$--dependent version of the characteristic flow \eqref{e:x41}, which is defined on the grid $\sft=\eps t$, $t=0,1,2,\dots$, through
\begin{equation}
\label{e:x41_2}
 z^{\eps}(\sft,u):=z^{\eps; \sft}, \qquad \qquad  \begin{cases} z^{\eps; \sft+\eps}= z^{\eps; \sft}+\eps \frac{\bE \tilde f_\sft(z^{\eps;\sft})}{\bE \tilde f_\sft(z^{\eps; \sft})-1},& \sft\in \eps \mathbb Z_{\ge 0},\\ z^{\eps;0}=u. \end{cases}.
\end{equation}
Let us emphasize that $z^{\eps}(\sft,u)$ is a deterministic map. As $\eps\to 0$, it approximates $z(\sft,u)$, as we now explain.

\begin{definition}\label{d:admissible} We call a pair $(\hat \sft,u)\in  \mathbb R_{\geq 0}\times \mathbb C $ \emph{admissible}, if the differential equation \eqref{e:x41} has a solution for times $0\leq \sft \leq \hat \sft$ and at all these times $\sfb_\sft(z^\sft)\notin \sfb_\sft([\bl(\sft),\br(\sft)]$ and $f_\sft(z^\sft)$ is finite and not equal to $1$. In particular, the condition implies that $z(\hat \sft,u)$ is well-defined.
\end{definition}
Following the notations of Section \ref{Section_Limit_shape_proofs}, if we assume $u\in \bD_0^\circ$ and $0\leq \sft\leq \sft(u)$, where $\sft(u)$ is the first time when $z^{\sft}$ hits the boundary of $\bD_\sft$, then Lemma \ref{Lemma_bijectivity} shows that $(\sft,u)$ is admissible and $z(\sft,u)$ takes all values in $\bD_\sft^\circ$, as $u$ varies in $\bD_0^\circ$.

\begin{lemma}
\label{Lemma_discrete_flow} Take any admissible pair $(\hat \sft,\hat u)$. There exists $\varsigma>0$, such that
\begin{equation}
\label{eq_discrete_flow_conv}
  \lim_{\eps\to 0} z^\eps(\sft,u)= z(\sft,u), \quad \text{ and }\quad  \lim_{\eps\to 0} \partial_u z^\eps(\sft,u)= \partial_u z(\sft,u)
\end{equation}
for all $(\sft,u)$ with $\sft\ge 0$ and satisfying $|\sft-\hat \sft|<\varsigma$, $|u-\hat u|<\varsigma$. The limits are uniform over such pairs $(\sft,u)$.
\end{lemma}
\begin{proof}
Note that from the definition, if $(\hat \sft, \hat u)$ is admissible pair, then so is $(\hat \sft+\varsigma_1, \hat u)$, as long as $\varsigma_1>0$ is small enough. Consider the trajectory of the flow $z^{\sft}$ \eqref{e:x41} started at times $0$ at $\hat u$ and up to time $\hat\sft+\varsigma_1$. We can choose small $\varsigma_2>0$, such that $f_\sft(z)$ is well-defined and bounded away from $1$ for all $\sft\in[0,\hat\sft+\varsigma_1]$ and $z$ in $\varsigma_2$--neighborhood of $z^{\sft}$.
By Theorem \ref{Theorem_LLN_abstract}, for such pairs $(\sft,z)$, uniformly, we have
\begin{align}\label{e:Efconverge}
 \lim_{\eps\to 0} \frac{\bE \tilde f_\sft(z)}{\bE \tilde f_\sft(z)-1}= \frac{f_\sft(z)}{f_\sft(z)-1}.
\end{align}

To compare the $\eps$--dependent characteristic flow $z^{\eps,\sft}$ with $z^\sft$, we introduce the following  forward Euler discretization of the ordinary differential equation \eqref{e:x41}
%\begin{equation}
%\label{e:wtzflow}
% \wt z(\sft,u):=\wt z^\sft, \qquad \qquad  \begin{cases}\del_\sft \wt z^\sft=\frac{\bE \wt f_\sft(\wt z^\sft)}{\bE \wt f_\sft(\wt z^\sft)-1},& \sft\geq 0,\\ \wt z^0=u. \end{cases}.
%\end{equation}
\begin{equation}
\label{e:wtzflow}
 \widehat z^{\eps}(\sft,u):=\widehat z^{\eps; \sft}, \qquad \qquad  \begin{cases} \widehat z^{\eps; \sft+\eps}= \widehat z^{\eps; \sft}+\eps \frac{ f_\sft(\widehat z^{\eps;\sft})}{ f_\sft(\widehat z^{\eps; \sft})-1},& \sft\in \eps \mathbb Z_{\ge 0},\\ \widehat z^{\eps;0}=u. \end{cases}.
\end{equation}
By standard theorems on the forward Euler method for approximating ordinary differential equations (e.g.\ \cite[Theorem 2.4]{atkinson2011numerical}) the discrete flow $\widehat z^{\eps,\sft}$ of \eqref{e:wtzflow} approximates as $\eps\to 0$ the continuous flow $z^{\sft}$ of \eqref{e:x41} with an error $\OO(\varepsilon)$.

It remains to compare two discrete flows: $\widehat z^{\eps,\sft}$ and $z^{\eps,\sft}$. Taking difference of \eqref{e:x41_2} and \eqref{e:wtzflow}, noticing that $\frac{f_\sft(z)}{f_\sft(z)-1}$ is Lipschitz
in $z$, and using \eqref{e:Efconverge}, we get for $C>0$:
\begin{multline*}
\left|\frac{z^{\varepsilon;\sft+1}-\widehat z^{\varepsilon;\sft+1}}{\varepsilon}\right|- \left|\frac{z^{\varepsilon;\sft}-\widehat z^{\varepsilon;\sft}}{\varepsilon}\right|\le 
\left|\frac{z^{\varepsilon;\sft+1}-\widehat z^{\varepsilon;\sft+1}}{\varepsilon}- \frac{z^{\varepsilon;\sft}-\widehat z^{\varepsilon;\sft}}{\varepsilon}\right|\\ =\left|\frac{ f_\sft(\widehat z^{\eps;\sft})}{ f_\sft(\widehat z^{\eps; \sft})-1}-\frac{\bE \tilde f_\sft(z^{\eps;\sft})}{\bE \tilde f_\sft(z^{\eps; \sft})-1}\right|\leq C|z^{\varepsilon;\sft}-\widehat z^{\varepsilon;\sft}|+\oo_\varepsilon(1),
\end{multline*}
where the term $\oo_\varepsilon(1)$ tends to $0$ as $\eps\to 0$.
Hence, applying the discrete Gr{\"on}wall's inequality (see, e.g., \cite{holte2009discrete}) to $|z^{\varepsilon;\sft}-\widehat z^{\varepsilon;\sft}|$, we conclude that $|z^{\varepsilon;\sft}-\widehat z^{\varepsilon;\sft}|$ is uniformly small and the discrete flow $ z^{\varepsilon;\sft}$ approximates as $\eps\to 0$ the discrete flow $\widehat z^{\varepsilon;\sft}$ for $\sft\in[0,\hat\sft+\varsigma_1]$.

Therefore, if we start the discrete flow $z^{\eps,\sft}$ of \eqref{e:x41} close to $\hat u$ at $\sft=0$, it will approximate as $\eps\to 0$ the continuous flow $ z^{\sft}$ of \eqref{e:x40} for any $|\sft-\hat \sft|\leq \varsigma$; uniformity in $u$ follows from uniform continuity of $\frac{\bE \tilde f_\sft(z)}{\bE \tilde f_\sft(z)-1}$ and $\frac{f_\sft(z)}{f_\sft(z)-1}$ in $\eps$ and in $z$ belonging to the above $\varsigma_2$ neighborhoods. This proves the first limit relation in \eqref{eq_discrete_flow_conv}.

For the second identity note that $z^{\eps}(\sft,u)$ is a holomorphic function of $u$. Hence, expressing the derivative as the Cauchy contour integral, the second limit in \eqref{eq_discrete_flow_conv} follows from the first one.
\end{proof}

Let us now explain why $z^{\eps}(\sft,u)$ is helpful in studying fluctuations.

\begin{corollary} \label{Corollary_CLT_increment} In the framework of Theorem \ref{Theorem_LLN_abstract} and using the notations of Sections \ref{Section_tilings_setup}, \ref{Section_assumptions}, and \ref{Section_Limit_shape_proofs}
  for $\sft=t\eps$, we have as $\eps\to 0$:
  \begin{multline}\label{e:dmg_tilings_increment}
\frac{G_{\eps^{-1}\sft+1}\left(z^\eps(\sft+\eps,u)\right)\partial_u z^\eps(\sft+\eps,u) -G_{\eps^{-1}\sft}\left(z^\eps(\sft,u)\right)\partial_u z^\eps(\sft,u)}{\varepsilon}\\ = \Delta\cM_t\bigl(z^\eps(\sft,u)\bigr)\partial_u z^\eps(\sft+\eps,u)+\mathcal R,
\end{multline}
where random mean $0$ asymptotically Gaussian function $\Delta\cM_t(z)$ is as in \eqref{e:dmg_tilings_3}, while $\mathcal R$ is different, but admits a similar upper bound:
\begin{equation}
\label{e:x40_2}
 |\mathcal R|=\oo(\eps)+\oo(1) \left|G_t(z^\eps(\sft,u))\right|+\oo(1) \bE \left|G_t(z^\eps(\sft,u))\right|.
\end{equation}
\end{corollary}
\begin{proof}
 Taylor-expanding $G_{t+1}(z)$ in $z$ and with the exact meaning of the remainder $\mathcal R$ changing from line to line, \eqref{e:dmg_tilings_3} implies that
 \begin{equation}\label{e:dmg_tilings_5}
\frac{1}{\varepsilon}\left[G_{t+1}\left(z+\eps \frac{ \bE \tilde f_\sft(z)}{\bE \tilde f_\sft(z)-1}\right)-G_t(z)\right]+ G_t(z)\del_z\left[ \frac{\bE \tilde f_\sft(z)}{\bE \tilde f_\sft(z)-1}\right] = \Delta\cM_t(z)+\mathcal R.
\end{equation}
Recalling \eqref{e:x41_2}, note that
\begin{align*}
 \partial_u \frac{ z^\eps(\sft+\eps,u)- z^\eps(\sft,u)}{\eps}&= \partial_u \left[ \frac{\bE \tilde f_\sft(z^\eps(\sft,u))}{\bE \tilde f_\sft(z^\eps(\sft,u))-1}\right]=\partial_u z^\eps(\sft,u) \, \del_z\left[ \frac{\bE \tilde f_\sft(z)}{\bE \tilde f_\sft(z)-1}\right]_{z=z^\eps(\sft,u)}
 \\ &=\bigl(\partial_u z^\eps(\sft+\eps,u) + O(\eps)\bigr) \, \del_z\left[ \frac{\bE \tilde f_\sft(z)}{\bE \tilde f_\sft(z)-1}\right]_{z=z^\eps(\sft,u)}.
\end{align*}
We plug $z=z^\eps(\sft,u)$ into \eqref{e:dmg_tilings_5}, multiply by $\partial_u z^\eps(\sft+\eps,u)$, and use the last identity to get (the remainder $\mathcal R$ is again different, yet it satisfies \eqref{e:x40_2}):
 \begin{multline}\label{e:dmg_tilings_6}
\frac{1}{\varepsilon}\Bigl[G_{t+1}\left(z^\eps(\sft+\eps,u)\right)-G_t(z^\eps(\sft,u))\Bigr] \partial_u z^\eps(\sft+\eps,u)\\+ G_t(z^\eps(\sft,u))\partial_u \frac{ z^\eps(\sft+\eps,u)- z^\eps(\sft,u)}{\eps}
= \Delta\cM_t(z^\eps(\sft,u))\partial_u z^\eps(\sft+\eps,u)+\mathcal R.
\end{multline}
Simplifying the left-hand side of \eqref{e:dmg_tilings_6}, we get \eqref{e:dmg_tilings_increment}.
\end{proof}

An advantage of \eqref{e:dmg_tilings_increment} over \eqref{e:dmg_tilings_3} is that while the latter is complicated, the former stochastic evolution equation is straightforward to solve by summing the right-hand side. Let us state the result of the summation.

\begin{theorem} \label{Theorem_fluctuations_characteristic} Fix $k=1,2,\dots$ and take $k$ admissible pairs $(\sft_i,u_i)$, $1\leq i \leq k$. In the setting of Theorem \ref{Theorem_LLN_abstract}, the random variables $\eps^{-1} G_{\eps^{-1}\sft_i}(z(\sft_i,u_i))$ converge in the sense of moments (uniformly over admissible pairs belonging to compact sets) to a $k$--dimensional Gaussian random vector with zero mean and covariance given by:
\begin{multline}
\label{e:cov_characteristic}
 \lim_{\eps\to 0} \eps^{-2} \bE \bigl[ G_{\eps^{-1}\sft_i}(z( \sft_i, u_i)) G_{\eps^{-1}\sft_j}(z(\sft_j,u_j)) \bigr]\\= \frac{1}{\partial_{u_i} z(\sft_i,u_i) \partial_{u_j} z(\sft_j,u_j)}\partial_{u_i} \partial_{u_j} \ln \left[ \frac{\sfb_{0}(u_i)-\sfb_{0}(u_j) }{\sfb_{\tau}\bigl(z(u_i,\tau)\bigr)-\sfb_{\tau}\bigl(z(u_j,\tau)\bigr) } \right], \qquad \tau=\min(\sft_i,\sft_j).
\end{multline}
\end{theorem}

In the rest of Section \ref{Section_fluctuations_characteristic} we prove Theorem \ref{Theorem_fluctuations_characteristic} through a series of lemmas. Assuming $\eps^{-1}\sft$ to be a positive integer, summing the result of Corollary \ref{Corollary_CLT_increment}, noting that $G_0(z)=0$,
and inserting the dependence of the remainders $\mathcal R$ on $m$ into the notations. we get
\begin{equation}\label{e:x42}
 \frac{1}{\eps} G_{\eps^{-1}\sft}\left(z^\eps(\sft,u)\right)\partial_u z^\eps(\sft,u) = \sum_{m=0}^{\eps^{-1}\sft}\Delta\cM_m\bigl(z^\eps(\eps m,u)\bigr)\partial_u z^\eps(\eps m+\eps,u)+ \sum_{m=0}^{\eps^{-1}\sfs} \mathcal R_m.
\end{equation}

Our first step is to analyze the martingale sum in \eqref{e:x42}.

\begin{lemma} \label{Lemma_use_mclt} In the setting of Theorem \ref{Theorem_fluctuations_characteristic}, the random variables
 \begin{equation}\sum_{m=0}^{\eps^{-1}\sft}\Delta\cM_m\bigl(z^\eps(\eps m,u)\bigr)\partial_u z^\eps(\eps m+\eps,u)\Bigr|_{(\sft,\sfu)=(\sft_i,\sfu_i)}, \qquad 1 \le i \le k, \label{eq_x5}
 \end{equation}
 converge in distribution to $k$--dimensional Gaussian random vector with zero mean and covariance given by \eqref{e:cov_characteristic}. If we metrize the distributional convergence, then the convergence is uniform as in Theorem \ref{Theorem_fluctuations_characteristic}.
\end{lemma}
\begin{proof}
 By Corollary \ref{Corollary_CLT_increment}, each term in \eqref{eq_x5} is a martingale difference. Hence, by the martingale central limit theorem \cite{hall2014martingale}, this sum is asymptotically Gaussian in the sense of convergence in distribution. The asymptotic covariance of the $i$th and $j$th sum is given by
\begin{multline} \label{e:x44}
   \lim_{\eps\to 0} \frac{1}{\partial_{u_i} z^\eps(\sft_i,u_i) \partial_{u_j} z^\eps(\sft_j,u_j)}   \sum_{m=0}^{\eps^{-1}\min(\sft_i,\sft_j)}\Biggl( \partial_{u_i} z^\eps(\eps m+\eps,u_i) \partial_{u_j} z^\eps(\eps m+\eps,u_j) \\ \times \bE \Bigl[\Delta\cM_m\bigl(z^\eps(\eps m,u_i)\bigr)\Delta\cM_m\bigl(z^\eps(\eps m,u_j)\bigr) \Bigr]\Biggr).
\end{multline}
Using \eqref{e:covT_tilings_3} and Lemma \ref{Lemma_discrete_flow}, the sum over $m$ approximates as $\eps\to 0$ the integral:
\begin{align*}
  \int_0^{\min(\sft_i,\sft_j)} &\frac{\partial_{u_i} z(\sft,u_i) \partial_{u_j} z(\sft,u_j)}{2\pi \ri} \\
  &\times
  \left[\oint_{\cin} \frac{f_\sft(w)}{ f_\sft(w)-1}\frac{\sfb_\sft'(w)\sfb_\sft'(z(\sft,u_i) )}{(\sfb_\sft(w)-\sfb_\sft(z(\sft,u_i) ))^2} \frac{\sfb_\sft'(w)\sfb_\sft'(z(\sft,u_j) )}{(\sfb_\sft(w)-\sfb_\sft(z(\sft,u_j)))^2} \rd w\right] \rd \sft,
\end{align*}
where the contour $\cin$ (which might depend on $\sft$) encloses $[\bl(\sft),\br(\sft)]$, but not $z(\sft,u_i)$ or $z(\sft,u_j)$. We further notice that the last integral can be rewritten as
\begin{equation} \label{e:x43}
  \int_0^{\min(\sft_i,\sft_j)} \partial_{u_1} \partial_{u_2} \left[ \frac{1}{2\pi \ri}\oint_{\cin} \frac{f_\sft(w)}{ f_\sft(w)-1}\cdot \frac{\sfb_\sft'(w) }{\sfb_\sft(w)-\sfb_\sft(z(\sft,u_i) )}\cdot \frac{\sfb_\sft'(w)}{\sfb_\sft(w)-\sfb_\sft(z(\sft,u_j))} \rd w\right] \rd \sft.
\end{equation}
Let us analyze the $\rd w$ contour integral. For that we let $\omega'_-$ denote the image of the contour $\cin$ under the map $w\mapsto -w +\sft - 2 \log_\sfq \kappa$. Recalling the symmetries
\begin{align*}\label{e:sym_2}
\sfb_\sft(w)=\sfb_\sft(-w+\sft-2\log_\sfq\kappa), \quad \sfb'_\sft(w)=-\sfb'_\sft(-w+\sft-2\log_\sfq\kappa),
\end{align*}
and $f_\sft(w)=1/ f_\sft(-w+\sft-2\log_\sfq\kappa)$, we conclude that
\begin{multline}
 \frac{1}{2\pi \ri} \oint_{\cin} \frac{f_\sft(w)}{ f_\sft(w)-1}\cdot \frac{\sfb_\sft'(w) }{\sfb_\sft(w)-\sfb_\sft(z(\sft,u_i) )}\cdot \frac{\sfb_\sft'(w)}{\sfb_\sft(w)-\sfb_\sft(z(\sft,u_j))} \rd w\\=  \frac{1}{2\pi \ri} \oint_{\omega'_-} \frac{f_\sft(w)}{ f_\sft(w)-1}\cdot \frac{\sfb_\sft'(w) }{\sfb_\sft(w)-\sfb_\sft(z(\sft,u_i) )}\cdot \frac{\sfb_\sft'(w)}{\sfb_\sft(w)-\sfb_\sft(z(\sft,u_j))} \rd w
\\= \frac{1}{4\pi \ri}\oint_{\cin \cup \omega'_-} \frac{f_\sft(w)}{ f_\sft(w)-1}\cdot \frac{\sfb_\sft'(w) }{\sfb_\sft(w)-\sfb_\sft(z(\sft,u_i) )}\cdot \frac{\sfb_\sft'(w)}{\sfb_\sft(w)-\sfb_\sft(z(\sft,u_j))} \rd w.
\end{multline}
In the last integral we change variables by introducing $W=\sfq^w$, $Z_i=\sfq^{z(\sft,u_i)}$, $Z_j=\sfq^{z(\sft,u_j)}$, and $F_\sft(W)=f_\sft(w)$. We get
\begin{align*}
 \frac{1}{4\pi\ii}\oint_{\sfq^{\cin} \cup \sfq^{\omega'_-}} \frac{F_\sft(W)}{ F_\sft(W)-1} \frac{\kappa^2 \sfq^{-\sft}W -W^{-1} }{\kappa^2 \sfq^{-\sft}W +W^{-1} -\kappa^2 \sfq^{-\sft}Z_i -Z_i^{-1} }\\
 \times \frac{\kappa^2 \sfq^{-\sft}W -W^{-1} }{\kappa^2 \sfq^{-\sft}W +W^{-1} -\kappa^2 \sfq^{-\sft}Z_j -Z_j^{-1} } \cdot \frac{\rd W}{W}.
\end{align*}
We compute this contour integral as (minus) the sum of the residues outside the integration contour:
\begin{itemize}
\item As $W\to \infty$, $F_\sft(W)\to \sfq^{\sft-\sfT}$, as we computed right before \eqref{e:x16}. Hence, the integrand behaves as
$$
\frac{\sfq^{\sft-\sfT}}{\sfq^{\sft-\sfT}-1} \cdot \frac{1}{W} + \OO\left(\frac{1}{W^2}\right),
$$
and the residue is a constant which does not depend on $u_i$ or $u_j$. Therefore, it vanishes in differentiations in \eqref{e:x43}.

\item  By the symmetries \eqref{e:sym_2} we also have a constant residue at $W\to 0$, which vanishes in differentiations.

\item As we have seen in Step 1 of the proof of Proposition \ref{Proposition_off_criticality}, outside our integration contour $F_\sft(W)=1$ is possible only at $W=\pm \sfq^{\sft/2} \kappa^{-1}$. At these points the numerators $\kappa^2 \sfq^{-\sft}W -W^{-1}$ vanish and there are no residues.
\item At $W=Z_i$, the residue is
$$
 \frac{1}{2} \cdot \frac{ F_\sft(Z_i)}{1-F_\sft(Z_i)}\cdot  \frac{\kappa^2 \sfq^{-\sft} Z_i - Z_i^{-1}}{\kappa^2 \sfq^{-\sft}Z_i +Z_i^{-1} -\kappa^2 \sfq^{-\sft}Z_j -Z_j^{-1}}.
$$
\item At $W=\kappa^{-2} Z_i^{-1} \sfq^{\sft}$ the residue is the same.
\item At $W=Z_j$ and at $W=\kappa^{-2} Z_j^{-1} \sfq^{\sft}$ the residues are
$$
\frac{1}{2} \cdot \frac{ F_\sft(Z_j)}{1-F_\sft(Z_j)}\cdot  \frac{\kappa^2 \sfq^{-\sft} Z_j - Z_j^{-1}}{\kappa^2 \sfq^{-\sft}Z_j +Z_j^{-1} -\kappa^2 \sfq^{-\sft}Z_i -Z_i^{-1}}.
$$
\end{itemize}
Summing the four non-trivial residues, and setting $\tau=\min(\sft_i,\sft_j)$, \eqref{e:x43} evaluates to
\begin{multline}
  -\partial_{u_1} \partial_{u_2}  \int_0^{\tau} \frac{ \frac{f_\sft(z(\sft,u_i))}{ f_\sft(z(\sft,u_i))-1} \sfb'_\sft(z(\sft,u_i))-\frac{f_\sft(z(\sft,u_j))}{ f_\sft(z(\sft,u_j))-1} \sfb'_\sft(z(\sft,u_j)) }{\sfb_\sft(z(\sft,u_i))-\sfb_\sft(z(\sft,u_j))}
 \rd \sft\\=  -\partial_{u_1} \partial_{u_2}  \int_0^{\tau} \partial_\sft \ln\bigl[\sfb_\sft(z(\sft,u_i))-\sfb_\sft(z(\sft,u_j))\bigr]
 \rd \sft= \ln\left[ \frac{\sfb_\sft(z(0,u_i))-\sfb_\sft(z(0,u_j))}{\sfb_\sft(z(\tau,u_i))-\sfb_\sft(z(\tau,u_j))}\right].
\end{multline}
Recalling that this was the computation of the asymptotics of the second line in \eqref{e:x44}, we arrive at the covariance of \eqref{e:cov_characteristic}.
\end{proof}

Next, we upgrade from convergence in distribution to convergence of moments.

\begin{lemma} \label{Lemma_use_mclt_moments}The convergence in Lemma \ref{Lemma_use_mclt} is also in the sense of moments.
\end{lemma}
\begin{proof}
Convergence in distribution together with uniform integrability implies the convergence of the moments. Hence, we need to show that the sums in \eqref{eq_x5} have moments bounded uniformly in $\eps$. Thus, we would like to estimate the $2p$--th moment of the sum in \eqref{eq_x5} for each $p=1,2,\dots$. The Burkholder--Rosenthal inequality \cite[Equation (2.9)]{osekowski2012note} (see also \cite{burkholder1973distribution, hitczenko1990best}) gives that for each $p \ge 1$, there exists a constant $\fC_p>0$, such that:
\begin{align}
\label{e:Ros}
		\bE&\left| \sum_{m=0}^{\eps^{-1}\sft}\Delta\cM_m\bigl(z^\eps(\eps m,u)\bigr)\partial_u z^\eps(\eps m+\eps,u)\right|^{2p}\\
		&\leq \fC_p \bE \Bigg[ \bigg(\sum_{m=0}^{\eps^{-1}\sft}\bE \big[ | \Delta\cM_m\bigl(z^\eps(\eps m,u)\bigr)\partial_u z^\eps(\eps m+\eps,u)|^2 \big| \bmx(m) \big] \bigg)^p \Bigg]\notag \\
		&+\fC_p\sum_{m=0}^{\eps^{-1}\sft}\bE \big[ |\Delta\cM_m\bigl(z^\eps(\eps m,u)\bigr)\partial_u z^\eps(\eps m+\eps,u)|^{2p} \big].\notag
\end{align}
For the $m$-th term in the first sum on the righthand side of \eqref{e:Ros}, denoting $\sfs=\varepsilon m$, thanks to Proposition \ref{Proposition_tilings_increment} (by taking $z_1=z^\eps(\sfs,u), z_2=\overline{z}(\sfs,u)$ in \eqref{e:covT_tilings_2}), we have
\begin{multline}\label{e:EDM}
\left|\bE \big[ | \Delta\cM_m\bigl(z^\eps(\eps m,u)\bigr)\partial_u z^\eps(\eps m+\eps,u)|^2 \big| \bmx(m) \big]\right|
\leq \varepsilon |\partial_u z^\eps(\eps m+\eps,u)|^2
\\ \times \left|
\frac{1}{2\pi \ri}\oint_{\cin} \frac{\tilde f_\sfs(w)}{\tilde f_\sfs(w)-1}\frac{\sfb_\sfs'(w)\sfb_\sfs'(z^\eps(\sfs,u))}{(\sfb_\sfs(w)-\sfb_\sfs(z^\eps(\sfs,u)))^2} \frac{\sfb_\sfs'(w)\sfb_\sfs'(\overline{z}(\sfs,u))}{(\sfb_\sfs(w)-\sfb_\sfs(\overline{z}(\sfs,u)))^2} \rd w+\oo(1)
\right|,
\end{multline}
where $z^\eps(\sfs,u)$ is outside the integration contour $\cin$. Note that $o(1)$ in the last formula, which originates from $o(1)$ in \eqref{e:covT}, is uniform in the sense of Remark \ref{Remark_uniformity}: it can be upper-bounded by a deterministic constant tending to $0$ as $\eps\to 0$.

 By our assumption $(\sft,u)$ is an admissible pair (recall from Definition \ref{d:admissible}). Hence, it follows from Lemmas \ref{Lemma_bijectivity} and \ref{Lemma_discrete_flow}, that $|\partial_u z^\eps(\varepsilon m+\varepsilon,u)|$ is uniformly bounded for $0\leq m\leq \varepsilon^{-1}\sft$. Moreover, $\sfb_\sfs(z^\eps(\sfs,u))$ is bounded away from $\sfb_\sfs([\bl(\sfs),\br(\sfs)]$ and we can take the contour such that $\sfb_\sfs(w)$ is bounded away from $\sfb_\sfs([\bl(\sfs),\br(\sfs)]$, thus, the last two factors in the integrands
\begin{align}
\left|\frac{\sfb_\sfs'(w)\sfb_\sfs'(z^\eps(\sfs,u))}{(\sfb_\sfs(w)-\sfb_\sfs(z^\eps(\sfs,u)))^2}\right|,
\left| \frac{\sfb_\sfs'(w)\sfb_\sfs'(\overline{z}(\sfs,u))}{(\sfb_\sfs(w)-\sfb_\sfs(\overline{z}(\sfs,u)))^2}\right|
\end{align}
are uniformly bounded. In \textbf{Step 1} of Proposition \ref{Proposition_off_criticality}, we have studied the behavior of $\widetilde f_\sfs(w)$. It is uniformly bounded and bounded away from $1$, when $\sfb_{\sfs}(w)$ is bounded away from $\sfb_\sfs([\bl(\sfs),\br(\sfs)]$. Thus on $\cin$, $|\tilde f_\sfs(w)/{\tilde f_\sfs(w)-1}|$ is also uniformly bounded. Therefore, there exists a deterministic constant $C$ (independent of $\varepsilon, m$) such that \eqref{e:EDM} is bounded by
\begin{align}\label{e:EDM2}
\left|\bE \big[ | \Delta\cM_m\bigl(z^\eps(\eps m,u)\bigr)\partial_u z^\eps(\eps m+\eps,u)|^2 \big| \bmx(m) \big]\right|\leq C\varepsilon.
\end{align}
Summing the above expression for $0\leq m\leq \varepsilon^{-1}\sft$, we produce a deterministic bound on the random variable under the first expectation on the righthand side of \eqref{e:Ros}:
	\begin{align}\label{e:firstb}
\bigg(\sum_{m=0}^{\eps^{-1}\sft}\bE \big[ | \Delta\cM_m\bigl(z^\eps(\eps m,u)\bigr)\partial_u z^\eps(\eps m+\eps,u)|^2 \, \big|\, \bmx(m) \big] \bigg)^p
			\le C^p (\sft +\varepsilon)^p.
	\end{align}

Switching to the second term in the righthand side of \eqref{e:Ros}, thanks to Proposition \ref{Proposition_stochastic_evolution}, $\varepsilon^{-1/2}\Delta\cM_m\big((z^\eps(\eps m,u)\big)$ is asymptotically Gaussian in the convergence of moments sense, and we have
	\begin{align*}
		\frac{1}{\varepsilon^{p}} & \bE \big[  |\Delta\cM_m(z^\eps(\eps m,u))\big|^{2p} \, \big|\, \bmx(m) \big]\\
		& = \bE \big[  |\varepsilon^{-1/2}\Delta\cM_m(z^\eps(\eps m,u))\big|^{2p}\, \big|\,  \bmx(m) \big]\\
		&\leq 2^{2p}\bE \left[ \Re[\varepsilon^{-1/2}\Delta\cM_m\big]^{2p}+\Im[\varepsilon^{-1/2}\Delta\cM_m \bigr]^{2p}\,\big|\, \bmx(m) \right]\\
		&\leq 2^{2p}(2p-1)!! \left(\bE \big[ \Re[\varepsilon^{-1/2}\Delta\cM_m\big]^{2}\big| \bmx(m) \big]\right)^p+\left(\bE \big[\Im[\varepsilon^{-1/2}\Delta\cM_m \big]^{2}\big| \bmx(m) \big]\right)^p+\oo(1)\\
		&\leq 2^{2p+1}(2p-1)!! \left(\bE \big[ |\varepsilon^{-1/2}\Delta\cM_m(z^\eps(\varepsilon m, u))|^2\,\big|\, \bmx(m) \big]\right)^p+\oo(1),
	\end{align*}
where in the second inequality we used that $ \Re[\varepsilon^{-1/2}\Delta\cM_m]$ and $ \Im[\varepsilon^{-1/2}\Delta\cM_m]$ are asymptotically Gaussian. Using the above estimate together with \eqref{e:EDM2} and a uniform bound on $|\partial_u z^\eps(\varepsilon m+\varepsilon,u)|$, we estimate
\begin{multline*}
\bE \big[ |\Delta\cM_m\bigl(z^\eps(\eps m,u)\bigr)\partial_u z^\eps(\eps m+\eps,u)|^{2p} \big]\\= |\partial_u z^\eps(\eps m+\eps,u)|^{2p}\bE \left[ \bE \big[ |\Delta\cM_m\bigl(z^\eps(\eps m,u)\bigr)|^{2p}\, \big|\, \bmx(m) \big]\right]\le  C' \eps^p,
\end{multline*}
for a deterministic constant $C'>0$. Summing this estimate for $0\leq m\leq \varepsilon^{-1}\sft$, and recalling that $p\ge 1$, we conclude that the second term on the righthand side of \eqref{e:Ros} is uniformly bounded:
\begin{align}\label{e:secondb}
&\phantom{{}={}}\sum_{m=0}^{\eps^{-1}\sft}\bE \big[ |\Delta\cM_m\bigl(z^\eps(\eps m,u)\bigr)\partial_u z^\eps(\eps m+\eps,u)|^{2p} \big]\leq C''.
\end{align}
It follows that
\begin{align}
\label{e:Ros1}
		\bE&\left| \sum_{m=0}^{\eps^{-1}\sft}\Delta\cM_m\bigl(z^\eps(\eps m,u)\bigr)\partial_u z^\eps(\eps m+\eps,u)\right|^{2p}\leq C.
\end{align}
		
%Using the estimate above, the second term in \eqref{e:Ros} can be bounded by the first term (up to a constant depending on $p$)
%\begin{align*}
%&\phantom{{}={}}\sum_{m=0}^{\eps^{-1}\sft}\bE \big[ |\Delta\cM_m\bigl(z^\eps(\eps m,u)\bigr)\partial_u z^\eps(\eps m+\eps,u)|^{2p} \big]\\
%&\leq \sum_{m=0}^{\eps^{-1}\sft} 2^{p-1}(2p-1)!!  \bE \big[ |\Delta\cM_m(z^\eps(\varepsilon m, u)) \partial_u z^\eps(\eps m+\eps,u) |^2\big| \bmx(m) \big]^p+\varepsilon^p|\partial_u z^\eps(\eps m+\eps,u)|\oo(1)\\
%&\leq 2^{p-1}(2p-1)!! \left(\sum_{m=0}^{\eps^{-1}\sft}  \bE \big[ |\Delta\cM_m(z^\eps(\varepsilon m, u)) \partial_u z^\eps(\eps m+\eps,u) |^2\big| \bmx(m) \big]\right)^p\\
%&+\sum_{m=0}^{\eps^{-1}\sft}  \varepsilon^p|\partial_u z^\eps(\eps m+\eps,u)|\oo(1),
%\end{align*}
%which is again bounded.
Plugging \eqref{e:firstb} and \eqref{e:secondb} into \eqref{e:Ros}, we conclude that the sums in \eqref{eq_x5} have moments bounded uniformly in $\eps$.
\end{proof}

We further show that the remainders $\mathcal R_m$ in \eqref{e:x42} do not contribute to the asymptotics.

\begin{lemma} \label{Lemma_G_at_eps_CLT}
 The random variables $\frac{1}{\eps} G_{\eps^{-1}\sft}\left(z^\eps(\sft,u)\right)\partial_u z^\eps(\sft,u)$ have the same $\eps\to 0$ asymptotics (both in distribution and in the sense of moments) as the sums \eqref{eq_x5} in Lemmas \ref{Lemma_use_mclt} and \ref{Lemma_use_mclt_moments}. The $\eps\to 0$ convergence is uniform as in Theorem \ref{Theorem_fluctuations_characteristic}.
\end{lemma}
\begin{proof}

Summing the result of Corollary \ref{Corollary_CLT_increment}, we get for any $1\leq \eps^{-1}\sfs\leq \eps^{-1}\sft$, we have
\begin{equation}\label{e:duz}
 \frac{1}{\eps} G_{\eps^{-1}\sfs}\left(z^\eps(\sfs,u)\right)\partial_u z^\eps(\sfs,u) = \sum_{m=0}^{\eps^{-1}\sfs}\Delta\cM_m\bigl(z^\eps(\eps m,u)\bigr)\partial_u z^\eps(\eps m+\eps,u)+ \sum_{m=0}^{\eps^{-1}\sfs} \mathcal R_m',
\end{equation}
where, as a consequence of \eqref{e:x40_2}, the remainders satisfy for a constant $C_1>0$:
\begin{align}\label{e:mR}
|\mathcal R_{m+1}'|\leq C_1\bigl(\varepsilon+|G_{ m}(z^\eps(\eps m,u))|+\bE|G_{ m}(z^\eps(\eps m,u))|\bigr).
\end{align}
A corollary of \eqref{e:mR} and Hoelder's inequality is
\begin{align}\label{e:mR_p}
\bE|\mathcal R_{m+1}'|^{2p}\leq C_2\bigl(\varepsilon^{2p}+\bE|G_{ m}(z^\eps(\eps m,u))|^{2p}\bigr), \qquad p=1,2,\dots.
\end{align}
Next, dividing by $\partial_u z^\eps(\sfs,u)$ and rising both sides of \eqref{e:duz} to $2p$-th moments, we get
\begin{multline}\label{e:sum00}
\bE\left[|G_{\eps^{-1}\sfs}/\varepsilon|^{2p}\right]
\leq  \frac{2^{2p}}{|\partial_u z^\eps(\sfs,u)|^{2p}}\Biggl(\bE \left|\sum_{m=0}^{\eps^{-1}\sfs}\Delta\cM_m\bigl(z^\eps(\eps m,u)\bigr)\partial_u z^\eps(\eps m+\eps,u)\right|^{2p}\\ + \bE \left(\sum_{m=0}^{\eps^{-1}\sfs} |\mathcal R_m'|\right)^{2p}\Biggr)
\leq C_3\left(1+ (\eps^{-1}\sfs+1)^{2p-1} \sum_{m=0}^{\eps^{-1}\sfs} \bE |\mathcal R_m'|^{2p} \right)
\\ \leq C_4\left(1+ \varepsilon \sum_{m=0}^{\eps^{-1}\sfs}\bE |\mathcal R_m'/\varepsilon|^{2p} \right),
\end{multline}
where in the second line we used \eqref{e:Ros1} to bound the first term. By plugging \eqref{e:mR_p} into \eqref{e:sum00}, we get the following estimate,
\begin{align}\label{e:sum0}
\bE\left[|G_{\eps^{-1}\sfs}/\varepsilon|^{2p}\right]
\leq C_5\left(1+\varepsilon\sum_{m=0}^{\eps^{-1}\sfs-1}\bE\left[ |G_{m}/\eps|^{2p}\right]\right).
\end{align}
We use \eqref{e:sum0} to prove by induction that
\begin{align}\label{e:sumG}
\sum_{m=0}^{\eps^{-1}\sfs-1}\bE\left[ |G_{m}/\eps^{-1}|^p\right]\leq \frac{1}{\varepsilon}\left((1+\eps C_5)^{\eps^{-1}\sfs}-1\right),
\end{align}
where we emphasize that the constants $C_5$ in \eqref{e:sum0} and in \eqref{e:sumG} are the same.
The statement \eqref{e:sumG} holds for $\sfs=0$, because the sum is empty and the right-hand side of the inequality also vanishes. Next we prove it for $\sfs+\eps$ assuming that it holds for $\sfs$. Using the induction hypothesis, \eqref{e:sum0} gives that
\begin{align}\begin{split}\label{e:Gsum1}
\bE\left[|G_{\eps^{-1}\sfs}/\varepsilon|^{2p}\right]\leq C_5\left( 1+\left((1+\eps C_5)^{\eps^{-1}\sfs}-1\right)
 \right)
 = C_5 (1+\eps C_5)^{\eps^{-1}\sfs}.
\end{split}\end{align}
Combining \eqref{e:sumG} and \eqref{e:Gsum1}, we get
\begin{align*}
\sum_{m=0}^{\eps^{-1}\sfs}\bE\left[ |G_{m}/\eps^{-1}|^p\right]
&\leq
\frac{1}{\varepsilon}\left((1+\eps C_5)^{\eps^{-1}\sfs}-1\right)
+ C_5 (1+\eps C_5)^{\eps^{-1}\sfs}\\
&=
\frac{1}{\varepsilon}\left((1+\eps C_5)^{\eps^{-1}(\sfs+1)}-1\right).
\end{align*}
This finishes the induction. In particular, as a consequence of \eqref{e:Gsum1}, we have a uniform bound: \begin{align}\begin{split}\label{e:Gmoment}
\bE\left[|G_{\eps^{-1}\sfs}/\varepsilon|^{2p}\right]
 \leq C_6, \qquad \text{for any }0\leq \sfs\leq \sft.
\end{split}\end{align}

\smallskip

Once we know the uniform bound \eqref{e:Gmoment} we can obtain much better estimates on $\mathcal R_m$. Indeed, we recall that, in fact, according to \eqref{e:x40_2}, the constants in \eqref{e:mR} and \eqref{e:mR_p} tend to $0$ as $\eps\to 0$. In other words, \eqref{e:mR_p} can be replaced by a much stronger bound
\begin{align}\label{e:mR_p_strong}
\bE|\mathcal R_{m+1}'|^{2p}\leq o(1) \cdot \bigl(\varepsilon^{2p}+\bE|G_{ m}(z^\eps(\eps m,u))|^{2p}\bigr), \qquad p=1,2,\dots, \qquad \eps\to 0.
\end{align}
Together with \eqref{e:Gmoment}, this implies
\begin{equation}
 \bE|\mathcal R_{m}'|^{2p}\leq o(\eps^{2p}),
\end{equation}
with implicit constant in $o(1)$ being uniform over $m$. Therefore, for any $p=1,2,\dots$,
\begin{equation}
\bE \left|\sum_{m=0}^{\eps^{-1}\sft} \mathcal R_m'\right|^{2p}\le (\eps^{-1} \sft+1)^{2p} \sup_{0\le m \le \eps^{-1}\sft}  \bE|\mathcal R_{m}'|^{2p}=o(1).
\end{equation}
Hence, the second sum in \eqref{e:x42} does not contribute to $\eps\to 0$ asymptotics. On the other hand, the asymptotics of the first sum has already been analyzed in Lemmas \ref{Lemma_use_mclt} and \ref{Lemma_use_mclt_moments}.
\end{proof}

We can now finish the proof of Theorem \ref{Theorem_fluctuations_characteristic}.

\begin{proof}[Proof of Theorem \ref{Theorem_fluctuations_characteristic}]
We will only prove the $k=1$ statement, as the general case is the same.

Take an admissible pair $(\sft,u)$ and let $\mathfrak z=z(\sft,u)$. Define $u^\eps=[z^\eps(\sft,\cdot)]^{-1}(\mathfrak z)$, where the functional inverse is applied in the second argument of the function $z^{\eps}(\sft,\cdot)$. Since $u\mapsto z(\sft,u)$ is a conformal bijection by Lemma \ref{Lemma_bijectivity} (in particular, its $u$--derivative is non-zero), using Lemma \ref{Lemma_discrete_flow} we conclude that $u^{\eps}$ is well-defined for all small-enough $\eps$ and $\lim_{\eps\to 0} u^\eps=u$.

Lemmas \ref{Lemma_use_mclt}, \ref{Lemma_use_mclt_moments}, \ref{Lemma_G_at_eps_CLT} together with Lemma \ref{Lemma_discrete_flow} imply that $\frac{1}{\eps} G_{\eps^{-1}\sft}\left(z^\eps(\sft,u^\eps)\right)\partial_u z^\eps(\sft,u^\eps)$ is asymptotically Gaussian (in the sense of convergence of the moments), with the variance given by \eqref{e:cov_characteristic}, in which you replace $z(\sft,u)$ by $\mathfrak z$. Because $z^\eps(\sft,u^\eps)=\mathfrak z=z(\sft,u)$ and  $\lim_{\eps\to 0}\partial_u z^\eps(\sft,u^\eps)=\partial_u z(\sft,u)$  by Lemma \ref{Lemma_discrete_flow}, the convergence of $\frac{1}{\eps} G_{\eps^{-1}\sft}\left(z^\eps(\sft,u^\eps)\right)\partial_u z^\eps(\sft,u^\eps)$ is equivalent to the desired convergence of $\frac{1}{\eps} G_{\eps^{-1}\sft}\left(z(\sft,u)\right)\partial_u z(\sft,u)$ and the limits are the same.
\end{proof}

\subsection{Matching the covariance structure with the Gaussian Free Field}
In this section we finish the proof of Theorem \ref{t:GFF}. Theorem \ref{Theorem_fluctuations_characteristic} implies that the fluctuations of the height function of tilings are asymptotically Gaussian and it remains to match the covariance structure with that of the Gaussian Free Field. We only present the details under Assumption \ref{a:para}: $0<\sfq<1$ and large real $\kappa$.

Our first step is to recast the result of Theorem \ref{Theorem_fluctuations_characteristic} in terms of the height functions. In parallel with \eqref{e:x41}, we introduce a function $w(\sft,w_0):=\sfb_\sft(z(\sft,\sfb_0^{-1} (w_0))$, which maps $w_0$ to $w^\sft$ and is given more explicitly by \eqref{e:x17}.

\begin{corollary} \label{Corollary_height_covariance}
 In the setting of Theorem \ref{t:GFF}, fix $k\in\mathbb Z_{>0}$ times $0< \sft_1\leq \sft_2 \leq \dots \leq \sft_k < \sfT$ and $k$ real-valued functions $f_1(x), \dots, f_k(x)$, such that $f_i(x)=\partial_x [F_i(\sfb_\sft(x))]$ and $F_i(x)$ is analytic in a complex neighborhood of the real line $\mathbb R$. Then the random vector
 \begin{equation}
 \label{e:x47}
  \left(\sqrt{\pi}\int_{\bl(\sft_i)}^{\br(\sft_i)} f_i(\sfx) \bigl(h(\varepsilon^{-1}\sft_i,\varepsilon^{-1}\sfx)-\bE[h(\varepsilon^{-1}\sft_i,\varepsilon^{-1}\sfx)]\bigr)\rd \sfx\right)_{i=1}^k,
 \end{equation}
 converges in the sense of moments to a centered Gaussian vector with covariance of $i$--th and $j$--th components given by
 \begin{equation}
 \label{e:x49}
  -\frac{1}{4\pi}\oint_{\cC_i}\oint_{\cC_j}\del_{w_i}\del_{w_j}\bigl[\log \left(w_i-w_j\right)\bigr]\, F_i(w(\sft_i,w_i))F_j(w(\sft_j,w_j))\, \rd w_i\rd w_j,
 \end{equation}
 where the positively oriented contour $\cC_i$ in the complex plane is chosen so that its $w(\sft_i,\cdot)$--image in the $w$--plane is a simple curve enclosing $\sfb_\sft([\bl(\sft_i),\br(\sft_i)])$ and staying in the neighborhood where $F_i$ is analytic; similarly for $\cC_j$. If we assume that $t_i\leq t_j$, then $\cC_i$ should be inside $\cC_j$.
\end{corollary}
\begin{remark}
 It is plausible that a version of Corollary \ref{Corollary_height_covariance} holds for general analytic $f_i(x)$, i.e.\ the more restrictive form $f_i(x)=\sfb'_\sft(x) F_i'(\sfb_\sft(x))$ is not necessary. A technical difficulty in proving such an extension is the singular behavior of the flow $z^\sft$, as in \eqref{e:x41}, at the corners of the domain $\bD_\sft$, i.e.\ at the points $\sft/2-\log_\sfq(\kappa)$ and $\sft/2-\log_\sfq(\kappa)-\ii \frac{\pi}{\ln \sfq}$, where $f_\sft(z)$ equals $1$. We do not have to deal with the flow going through these points (which is an obstacle for applying Theorem \ref{Theorem_fluctuations_characteristic}), if we use the more restrictive choice of $f_i(x)$, as in Corollary \ref{Corollary_height_covariance}.
\end{remark}

\begin{proof} We recall the relation between the height function and the empirical density of the particles used in the definition \eqref{e:centered_G_def} of $G_t(z)$:
$$
  \eps h(\eps^{-1}\sft,\eps^{-1}\sfx)=\int_{-\infty}^{\sfx} \rho(s,\bmx(\eps^{-1}\sft)) \rd s.
$$
Hence, integrating by parts,
using
\eqref{e:linear_statistic_contour} and \eqref{e:centered_G_def}  we have
\begin{multline}
\int_{\bl(\sft_i)}^{\br(\sft_i)} f_i(\sfx) \bigl(h(\varepsilon^{-1}\sft_i,\varepsilon^{-1}\sfx)-\bE[h(\varepsilon^{-1}\sft_i,\varepsilon^{-1}\sfx)]\bigr)\rd \sfx
\\=-\frac{1}{\eps}
\int_{\bl(\sft_i)}^{\br(\sft_i)} F_i(\sfb_{\sft_i}(s)) \bigl(\rho(s,\bmx(\eps^{-1}\sft))-\bE[ \rho(s,\bmx(\eps^{-1}\sft))]\bigr)\rd s
=\frac{\eps^{-1}}{2\pi\ii} \oint_{\Upsilon_i} F_i(\sfb_{\sft_i}(z)) G_{\eps^{-1}\sft_i}(z) \rd z,
\label{e:x45}
\end{multline}
where $\Upsilon_i$ is the contour shown in Figure \ref{f:contours_CLT}. The main feature of this contour is that it encloses the $[\bl(\sft_i),\br(\sft_i)]$ interval and no other singularities of $G_{\eps^{-1}\sft_i}(z)$; in addition its $\sfb_{\sft_i}(\cdot)$--image stays close to the real line, guaranteeing that $F_i(\sfb_{\sft_i}(z))$ is holomorphic everywhere inside the contour. The contour depends on a large negative real parameter $M_i$ and consists of two parts: the first part (shown in solid in Figure \ref{f:contours_CLT}) starts at $M_i-\ii \pi/\ln \sfq$ on the top border of $\bD_\sft$, stays inside $\bD_\sft$ and closely follows its boundary until reaching $M_i$, and then proceeds in the same way from $M_i$ to $M_i+\ii \pi/\ln \sfq$ staying inside $\overline \bD_\sft$. The second part (shown dashed in Figure \ref{f:contours_CLT}) returns from $M_i+\ii \pi/\ln \sfq$ to $M_i-\ii \pi/\ln \sfq$ by following the three straight lines of the boundary of $\bD_\sft\cup \overline \bD_\sft$.

\begin{figure}[t]
\center
\includegraphics[width=0.8\linewidth]{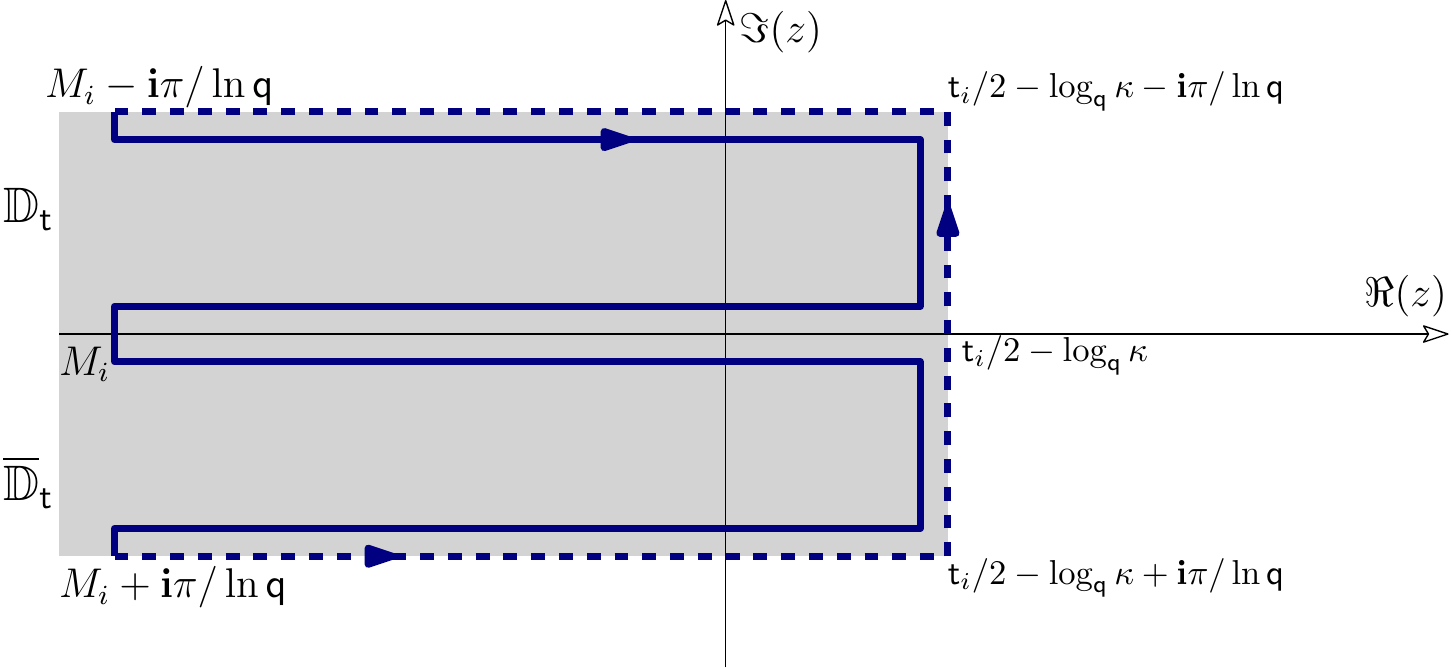}

\bigskip
\bigskip

\includegraphics[width=0.8\linewidth]{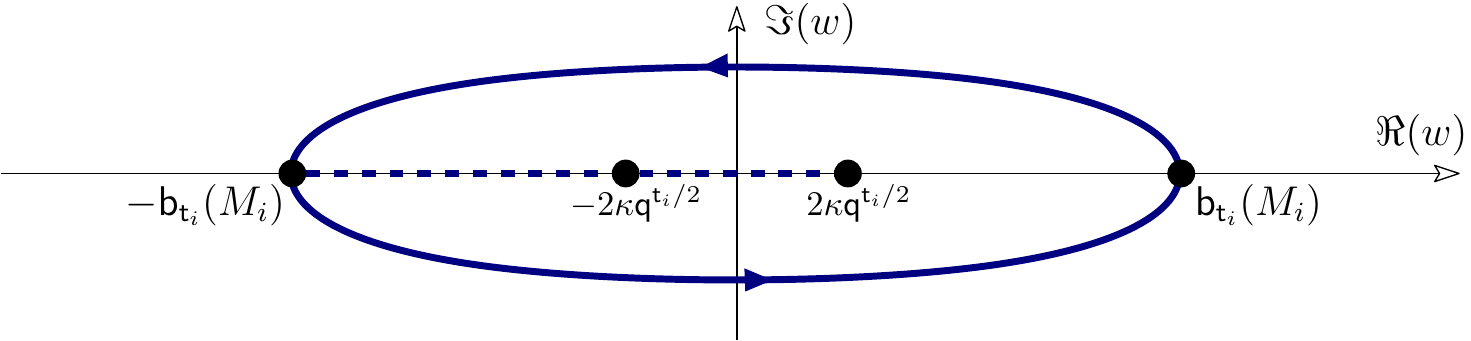}
\caption{Top panel: contour $\Upsilon_i$; the integral over the dashed part vanishes. Bottom panel: contour $\sfb_{\sft_i}(\Upsilon_i)$ has two parts; the solid part is $w(\sft_i,\cdot)$--image of $\cC_i$.}
\label{f:contours_CLT}
\end{figure}

Observe that the part of the integral \eqref{e:x45} over the dashed contour  in Figure \ref{f:contours_CLT} vanishes: indeed the integral over its part in the upper halfplane is minus the integral over its part in the lower half-plane, as follows from the identities:
$$
 F_i(\sfb_{\sft_i}(z+2 \ii \pi/\ln\sfq  )) G_{\eps^{-1}\sft_i}(z+2 \ii \pi/\ln\sfq )=F_i(\sfb_{\sft_i}(z)) G_{\eps^{-1}\sft_i}(z)
$$
and for $r\in \mathbb R$,
\begin{align*}
 F_i(\sfb_{\sft_i}(\sft_i/2-\log_\sfq \kappa + \ii r )) G_{\eps^{-1}\sft_i}(\sft_i/2-\log_\sfq \kappa + \ii r )\\
 =-F_i(\sfb_{\sft_i}(\sft_i/2-\log_\sfq \kappa - \ii r )) G_{\eps^{-1}\sft_i}(\sft_i/2-\log_\sfq \kappa - \ii r ).
\end{align*}
Hence, it is sufficient to only integrate in \eqref{e:x45} over the solid part of the $\Upsilon_i$--contour, which we denote $\Upsilon_i^{s}$; the $\sfb_{\sft_i}$--image of $\Upsilon_i^s$ is a simple closed loop, as in the bottom panel of Figure \ref{f:contours_CLT}.

Lemma \ref{Lemma_bijectivity}, Remark \ref{Remark_extension_to_the_boundary}, and the identity $z(\sft,\bar u)=\overline{ z(\sft,u)}$
 guarantee that we can find a contour $\Upsilon^0_i$ in the $\bD_0\cup \bar \bD_0$ inside the $u$--plane, such that $\Upsilon_i^c$ is its image under $u\mapsto z(\sft_i,w_0)$. We can then change variables $z=z(\sft_i,u_i)$ in \eqref{e:x45} converting it into
\begin{equation}
\label{e:x46}
 \frac{\eps^{-1}}{2\pi\ii} \oint_{\Upsilon_i^0} F_i\bigl(\sfb_{\sft_i}(z(\sft_i,u_i))\bigr) G_{\eps^{-1}\sft_i}(z(\sft_i,u_i)) \partial_{u_i} z(\sft_i,u_i) \rd u_i.
\end{equation}
The joint moments of the integrals \eqref{e:x46} are readily reduced to the integrals of the joint moments of random variables $G_{\eps^{-1}\sft_i}(z(\sft_i,u_i)) \partial_{u_i} z(\sft_i,u_i)$, which were computed in Theorem \ref{Theorem_fluctuations_characteristic}. Hence, applying this theorem (all the points $(\sft_i, u_i)$, $u_i\in \Upsilon^0_i$, are admissible as long as $M_i$ is large enough), we conclude that the integrals \eqref{e:x46} are asymptotically Gaussian (jointly over $i=1,\dots,k$) with covariance given by
\begin{multline}\label{e:x48}
   -\frac{\eps^{-2}}{4\pi^2} \bE\Biggl[ \oint_{\Upsilon_i^0} F_i\bigl(\sfb_{\sft_i}(z(\sft_i,u_i))\bigr) G_{\eps^{-1}\sft_i}(z(\sft_i,u_i)) \partial_{u_i} z(\sft_i,u_i) \rd u_i \\ \times \oint_{\Upsilon_j^0} F_j\bigl(\sfb_{\sft_j}(z(\sft_j,u_j))\bigr) G_{\eps^{-1}\sft_j}(z(\sft_j,u_j)) \partial_{u_i} z(\sft_j,u_j) \rd u_j\Biggr]\to\\ -\frac{1}{4\pi^2} \oint_{\Upsilon_i^0}  \oint_{\Upsilon_j^0}  \partial_{u_i} \partial_{u_j} \ln \left[ \frac{\sfb_{0}(u_i)-\sfb_{0}(u_j) }{\sfb_{\tau}\bigl(z(\tau,u_i)\bigr)-\sfb_{\tau}\bigl(z(\tau,u_j)\bigr) } \right]  F_i\bigl(\sfb_{\sft_i}(z(\sft_i,u_i))\bigr) F_j\bigl(\sfb_{\sft_j}( z(\sft_j,u_j))\bigr)  \rd u_j \rd u_i,
\end{multline}
where $\tau=\min(\sft_i,\sft_j)$. We split the logarithm of the ratio in \eqref{e:x48} into the sum of two logarithms. The term involving $\sfb_0$ matches \eqref{e:x49}, if we change the variables $w_i=\sfb_0(u_i)$, denote $\cC_i=\sfb_0(\Upsilon_i^0)$ and take into account $\sqrt{\pi}$ prefactor in \eqref{e:x47}. It remains to show that the term involving $\sfb_\tau$ vanishes.

 Up to now we have not used the nesting condition for $\cC_i$ and $\cC_j$ contours, and it becomes important at this point. We now assume $t_i\leq t_j$, so that $\tau=t_i$ and we also assume that $\cC_i$ is inside $\cC_j$, which implies $M_i> M_j$. The part of \eqref{e:x48} involving $\sfb_\tau$ is transformed into
\begin{align}\begin{split}
\label{e:x51}
 \frac{1}{4\pi^2} \oint_{\Upsilon_i^0}  \oint_{\Upsilon_j^0}  &\frac{\partial_{u_i} [z(u_i,\tau)]\sfb'_{\tau}\bigl(z(\tau,u_i)\bigr) \cdot \partial_{u_j} [z(\tau,u_j)]\, \sfb'_{\tau}\bigl(z(\tau,u_j)\bigr) }{\left[\sfb_{\tau}\bigl(z(\tau,u_i)\bigr)-\sfb_{\tau}\bigl(z(\tau,u_j)\bigr)\right]^2 }\\
& \times F_i\bigl(\sfb_{\sft_i}(z(\tau,u_i))\bigr) F_j\bigl(\sfb_{\sft_j}(z(\sft_j,u_j))\bigr)  \rd u_j \rd u_i.
\end{split}\end{align}
We isolate in the last integral the $u_i$--dependent factors and change the variable $v=\sfb_\tau(z(\tau,u_i))$ in this part, getting
\begin{equation}
\label{e:x52}
\oint_{w(\sft,\cC_i)} \frac{ F_i(v)  }{\left[v-\sfb_{\tau}\bigl(z(\tau,u_j)\bigr)\right]^2 } \rd v,
\end{equation}
where the integration contour ${w(\sft,\cC_i)}$ is precisely the solid contour in the bottom panel of Figure \ref{f:contours_CLT}. Because $\sfb_{\sft_j}(u_j)$ belong to $\cC_j$, it is outside $\cC_i$; hence, $\sfb_{\tau}\bigl(z(\tau,u_j))$ is outside the contour $w(\sft,\cC_i)$ (the map $w_0\mapsto w(\sft, w_0)$ preserves nesting, because \eqref{e:x18} moves all points up). Therefore, the integrand in \eqref{e:x52} has no singularities inside the integration contour and the integrals in \eqref{e:x52} and in \eqref{e:x51} vanish.
\end{proof}

Next, we deal with the map $\Omega$ and the associated pullback of the Gaussian Free Field.

We recall that $\bH^\pm$ are the (open) upper and lower halfplanes. We also recall that the liquid region $\sfL(\sfP)$ consists of those $(\sft,\sfx)$, for which the equation \eqref{e:solveu} has non-real roots $u$, see Propositions \ref{Proposition_number_of_roots} and \ref{Proposition_liquid}. There are $4$ such roots, leading to two values of $\sfq^{-u}+\kappa^2 \sfq^u$. We let $u=u(\sft,\sfx)$ be the root in $\bD_0^\circ$ and set $\Omega(\sft,\sfx)=\sfq^{-u}+\kappa^2 \sfq^u$ corresponding to this root. Then $\Omega$ is in the lower half-plane $\bH^-$ and its conjugate $\bar \Omega$ is in $\bH^+$. Let us repeat and prove Lemma \ref{Lemma_bijection_intro}:

\begin{lemma}
\label{Lemma_bijection_main}
 The map $(\sft,\sfx)\mapsto \Omega(\sft,\sfx)$ is a bijection between the liquid region $\sfL(\sfP)$ and $\bH^-$.
\end{lemma}
\begin{proof} We take $\omega\in \bH^-$ and show that there exists $(\hat \sft,\sfx)\in \sfL(\sfP)$, such that $\omega=\Omega(\hat \sft,\sfx)$. As in Section \ref{Section_Limit_shape_proofs}, we consider the dynamics $w^\sft$ of \eqref{e:x17} started at $w^0=\omega$: equivalently, this is $w^\sft=\sfb_\sft(z^\sft)$ with $z^\sft$ solving \eqref{e:DE_combined}. We recall that by \eqref{e:Nevanlinna}, the imaginary part of $w^\sft$ is a continuous monotonously increasing function of $\sft$. By our choice of $\omega$,  $\Im(w^0)<0$. On the other hand, by \eqref{e:x50}, $\Im(w^\sfT)>0$. Hence, there is a unique choice of $0<\hat \sft<\sfT$, for which $w^{\hat \sft}$ is real. The equation $\sfq^{-\sfx}+\kappa^2 \sfq^{\sfx-\hat \sft}=w^{\hat \sft}$, as an equation on $\sfx$, has two real roots and we choose the smaller one, which matches $z^{\tilde \sft}$. Then $(\hat \sft,\sfx)\in\sfL(\sfP)$, as we explained right above Lemma \ref{Lemma_bijectivity} and in the proof of Proposition \ref{Proposition_liquid}. On the other hand, if we express $w^{\hat \sft}$ through $u$, using \eqref{e:x17} and $w^0=\sfq^{-u}+\kappa^2 \sfq^u$, then $\sfq^{-\sfx}+\kappa^2 \sfq^{\sfx-\hat \sft}=w^{\hat \sft}(u)$ becomes \eqref{e:solveu} as an equation for $u$. Therefore, $\Omega(\hat \sft,\sfx)=\sfq^{-u}+\kappa^2 \sfq^u=w^0=\omega$, as desired.

We also need to show that $(\hat \sft,\sfx)$ satisfying $\omega=\Omega(\hat \sft,\sfx)$ is unique. Indeed, given $(\hat \sft, \sfx)$, we construct $\omega$ using \eqref{e:solveu}. But then comparison of \eqref{e:x17} with \eqref{e:solveu} (they are the same) shows that the flow $w^\sft$, $\sft\geq 0$, started at $w^0=\omega$ hits the real axis precisely at $w^{\hat \sft}=\sfq^{-\sfx}+\kappa^2 \sfq^{\sfx-\hat \sft}$; the hitting time uniquely determines $\hat \sft$ and the hitting point uniquely determines $\sfx$.

Alternatively, the same uniqueness can be proven by looking directly at \eqref{e:solveu} and noticing that once we fix $u$, the value of $\sft$ is uniquely reconstructed by taking the imaginary part of \eqref{e:solveu} and noticing that the left-hand side is real. Once we know both $u$ and $\sft$, the value of $\sfx$ is also readily reconstructed from \eqref{e:solveu}.
\end{proof}

\begin{definition} \label{Definition_gamma_curve}
 We let $(\sfx(\omega),\sft(\omega))$ be the inverse map to $\Omega(\sft,\sfx)$.
 For each $0<\hat \sft<\sfT$ we let
 $$\gamma_{\hat \sft}=\{ \Omega(\hat \sft,\sfx)\mid (\hat \sft,\sfx)\in\sfL(\sfP)\}=\{\omega\in \bH^-\mid \sft(\omega)=\hat \sft\}.$$
\end{definition}
Because the map $\Omega$ is defined in terms of algebraic equations, $\gamma_\sft$ is a union of finitely many simple curves. These curves can not be closed, as that would contradict $\Omega$ being bijective; they also can not end at non-real points, since by implicit function theorem, we are always able to slightly extend each curve around each of its non-real points. We conclude that $\gamma_\sft$ is a union of finitely many curves, which start and end on the real axis.

\bigskip

As the final ingredient, we define the Gaussian Free Field.

\begin{definition} \label{Definition_GFF} The Gaussian Free Field in $\bH^-$ (or in $\bH^+$) with Dirichlet boundary conditions is a generalized centered Gaussian field $\GFF(z)$, $z\in \bH^-$ (or $z\in\bH^+$) with covariance
\begin{equation}
 \mathrm{Cov} \bigl(\GFF(z), \GFF(w)\bigr)=K(z,w)= - \frac{1}{2\pi}\ln\left|\frac{z-w}{z-\bar w} \right|.
\end{equation}
Formally, it is defined in terms of its pairings with test-measures: for a signed (real) measure $\mu$ on $\bH^-$ (or $\bH^+$), such that
$$
 \iint K(z,w) \mu(\rd z) \mu(\rd w) <\infty,
$$
we define a mean $0$ Gaussian random variable $\langle \GFF, \mu\rangle$, so that for several such $\mu$'s the joint distribution of pairings is a Gaussian vector with covariance
\begin{equation}
\label{e:GFF_covariance}
 \mathrm{Cov}  \bigl(\langle \GFF,\mu \rangle , \langle \GFF, \tilde \mu\rangle \bigr)=
 \iint K(z,w) \mu(\rd z) \tilde \mu(\rd w).
\end{equation}
\end{definition}
We refer to \cite{Sheffield_GFF,Werner_Powell_GFF}, and \cite[Lecture 11]{VG2020} for the more detailed definition and discussion of the GFF. We remark that GFF can be identified with a random element of a certain Sobolev space (this is in parallel with the Brownian motion, which is first defined in terms of its covariance, and the recast as a random continuous function).

Let us also specialize the definition of the Gaussian Free Field to two types of signed measures $\mu$. First, if $\mu$ is absolutely continuous and has (signed) density $\mu(x+\ii y)$ (and similarly for $\tilde \mu$), then \eqref{e:GFF_covariance} becomes
$$
 \int \int \int \int K(x+\ii y, \tilde x+ \ii \tilde y) \mu(x+\ii y) \tilde \mu(\tilde x+ \ii \tilde y)\, \rd x\, \rd y\, \rd \tilde x\, \rd \tilde y.
$$
Second, if $\mu$ is supported on a curve $\gamma$ and has density $\mu(z)$ along this curve (and similarly for $\tilde \mu$), then \eqref{e:GFF_covariance} becomes
$$
 \int_\gamma \int_{\tilde \gamma} K(z,w) \mu(z) \tilde\mu(\tilde z)\, |\rd z| |\rd w|.
$$

\begin{proposition} \label{Proposition_covariance_transformation}
 The asymptotic covariance \eqref{e:x49} in Corollary \ref{Corollary_height_covariance} can be rewritten using Definitions \ref{Definition_gamma_curve} and \ref{Definition_GFF} as
 \begin{multline}
 \label{e:x55}
  \int_{\sfx_i\in \mathcal L_{\sft_i} } \int_{\sfx_j\in \mathcal L_{\sft_j}} K\bigl(\Omega(\sft_i,\sfx_i),\Omega(\sft_j,\sfx_j)\bigr)  f_j(\sfx_j)  f_i(\sfx_i)  \rd  \sfx_j \rd \sfx_i=\\  \int_{\gamma_{\sft_i}} \int_{\gamma_{\sft_j}} K(w_i,w_j) \cdot \bigl[f_j(\sfx(w_j)) \sfx'(w_j) \rd w_j\bigr] \cdot \bigl[f_i(\sfx(w_i))  \sfx'(w_i) \rd w_i\bigr],
 \end{multline}
 where   $\mathcal L_{\sft}$ in the second formula is the liquid region at time $\sft$, as in \eqref{e:Liquid_def_2}; $\sfx'(w_i) \rd w_i$ in the second formula should be understood as the directional derivative of $\sfx(w_i)$ along the $\gamma_{\sft_i}$ curve multiplied by $|\rd w_i|$; the equivalence of two forms in \eqref{e:x55} is by change of variables.
\end{proposition}
\begin{proof} We assume $\sft_i\leq \sft_j$ throughout the proof and recall that \eqref{e:x49} reads
 \begin{equation}
 \label{e:x53}
  -\frac{1}{4\pi}\oint_{\cC_i}\oint_{\cC_j}\del_{w_i}\del_{w_j}\bigl[\log \left(w_i-w_j\right)\bigr]\, F_i(w(\sft_i,w_i))F_j(w(\sft_j,w_j))\, \rd w_i\rd w_j,
 \end{equation}
 We  deform the contour $\cC_i$ into $\gamma_{\sft_i} \cup \bar \gamma_{\sft_i}$ and deform the contour $\cC_j$ into $\gamma_{\sft_j}\cup \bar \gamma_{\sft_j}$ in \eqref{e:x53}; the involved contours are schematically shown in Figure \ref{Fig_gamma_curve}. Let us explain that this deformation is possible and does not change the value on the integral. Because as $\sft$ grows the points $w^{\sft}$ move up (recall \eqref{e:x17}) until they reach the real axis and the curves $\gamma_{\sft_i}$, $\gamma_{\sft_j}$ are disjoint by Lemma \ref{Lemma_bijection_main}, $\gamma_{\sft_j}$ is situated below $\gamma_{\sft_j}$. In other words, the union of the simple closed contours $\gamma_{\sft_i} \cup \bar \gamma_{\sft_i}$ is inside the union of the simple contours of $\gamma_{\sft_j} \cup \bar \gamma_{\sft_j}$, matching the same configuration for $\cC_i$ and $\cC_j$. Hence, the factor $\del_{w_i}\del_{w_j} \log(w_i-w_j)=\frac{1}{(w_i-w_j)^2}$ remains non-singular in the deformation and does not create any residues.

\begin{figure}[t]
\begin{center}
 \includegraphics[width=0.99\linewidth]{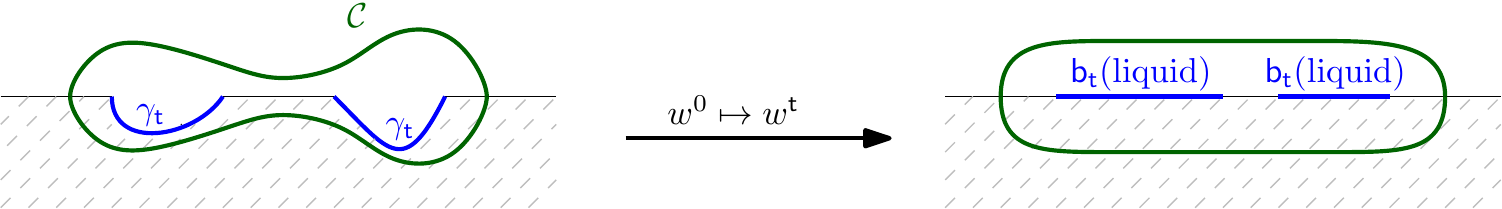}
 \caption{The action of $w^0\mapsto w^{\sft}$ map on the contours of integration in the situation when the liquid region has two segments.}
 \label{Fig_gamma_curve}
 \end{center}
 \end{figure}

For the remaining two factors $F_i(\cdot)$ and $F_j(\cdot)$, we use the fact that $w(\sft,w_0)$ is a holomorphic function of $w_0$ outside $\gamma_\sft \cup \bar \gamma_{\sft}$. This fact follows from Lemma \ref{Lemma_bijectivity}: the Lemma gives holomorphicity in the lower half-plane; conjugating, we also get holomorphicity in the upper half-plane, and applying the reflection principle we get holomorphicity on the real axis outside $\gamma_\sft \cup \bar \gamma_{\sft}$, cf.\ Figure \ref{Fig_gamma_curve}.

We conclude that no poles are encountered in the deformation and transform \eqref{e:x53} into
 \begin{equation}
 \label{e:x530}
  -\frac{1}{4\pi}\oint_{\gamma_{\sft_i} \cup \bar \gamma_{\sft_i}}\oint_{\gamma_{\sft_j} \cup \bar \gamma_{\sft_j}}\del_{w_i}\del_{w_j}\bigl[\ln (w_i-w_j))\bigr]\, F_i(w(\sft_i,w_i))F_j(w(\sft_j,w_j))\, \rd w_i\rd w_j.
 \end{equation}
 Note that $w(\sft_i,w_i)=w(\sft_i,\bar w_i)$ is real on the integration contour and so is $F_i(w(\sft_i,w_i))$ (because $f_i$ and $F_i$ were required to be real on the real line). Similarly, $F_j(w(\sft_j,w_j))$ is real. Hence, $F_i$ and $F_j$ factors do not change when we conjugate the variables. Therefore, differentiating the logarithm explicitly, splitting \eqref{e:x530} into four integrals, corresponding to upper and lower half-plane parts of $w_i$ and $w_j$ contours and recombining, we get an alternative expression:
 \begin{equation}
 \label{e:x54}
  -\frac{1}{4\pi}\int_{\gamma_{\sft_i} }\int_{\gamma_{\sft_j}}\left(\frac{\rd w_i\rd w_j}{ (w_i-w_j)^2}-\frac{\rd \bar w_i\rd w_j}{ (\bar w_i-w_j)^2}-\frac{\rd w_i\rd \bar w_j}{ (w_i-\bar w_j)^2}+\frac{\rd \bar w_i\rd \bar w_j}{ (\bar w_i-\bar w_j)^2} \right) F_i\bigl(\sfb_{\sft_i}(\sfx(w_i))\bigr) F_j\bigl(\sfb_{\sft_i}(\sfx(w_j))\bigr),
 \end{equation}
where we also identified $\sfb_{\sft_i}(\sfx(w_i))=w(\sft_i,w_i)$ and $\sfb_{\sft_j}(\sfx(w_j))=w(\sft_j,w_j)$. We change the variables $\sfx_i=\sfx(w_i)$, $\sfx_j=\sfx(w_j)$, notice that $w_i=\Omega(\sft_i,\sfx_i)$ and $w_j=\Omega(\sft_j,\sfx_j)$, and then integrate by parts in $\sfx_i$ and in $\sfx_j$ to get
 \begin{equation}
 \label{e:x56}
  -\frac{1}{4\pi}\int_{\mathcal L_{\sft_i} }\int_{\mathcal L_{\sft_j}} \ln\left[ \frac{\Omega(\sft_i,\sfx_i)-\Omega(\sft_j,\sfx_j)}{\bar \Omega(\sft_i,\sfx_i)-\Omega(\sft_j,\sfx_j)} \cdot \frac{\bar \Omega(\sft_i,\sfx_i)-\bar \Omega(\sft_j,\sfx_j)}{\Omega(\sft_i,\sfx_i)-\bar \Omega(\sft_j,\sfx_j)} \right] f_i(\sfx_i)f_j(\sfx_j) \rd \sfx_j \rd \sfx_i.
 \end{equation}
 Note that there are no additional terms in the integration by parts, because $w_i$ is real at the end-points of connected components of $\gamma_i$ and $w_j$ is real at the end-points of connected components of $\gamma_j$. \eqref{e:x56} matches \eqref{e:x55}.
\end{proof}

We have now collected all ingredients of Theorem \ref{t:GFF} and can state its precise form.

\begin{theorem} \label{theorem_GFF_main}
  Consider the limit regime \eqref{eq_parameter_scaling} and assume that we are in the real or imaginary case. Fix $k\in\mathbb Z_{>0}$ times $0< \sft_1\leq \sft_2 \leq \dots \leq \sft_k < \sfT$ and $k$ real functions $f_1(x), \dots, f_k(x)$, such that $f_i(x)=\partial_x \bigl[F(\sfb_{\sft_i}(x))\bigr]$ for a real-analytic function $F$. Then the random vector
 \begin{equation*}
  \left(\sqrt{\pi}\int_{\bl(\sft_i)}^{\br(\sft_i)} f_i(\sfx) \bigl(h(\varepsilon^{-1}\sft_i,\varepsilon^{-1}\sfx)-\bE[h(\varepsilon^{-1}\sft_i,\varepsilon^{-1}\sfx)]\bigr)\rd \sfx\right)_{i=1}^k,
 \end{equation*}
 converges in the sense of moments to a centered Gaussian vector with covariance of $i$--th and $j$--th components given by
 \begin{equation}
 \label{e:x560}
  \int_{\sfx_i\in \mathcal L_{\sft_i} } \int_{\sfx_j\in \mathcal L_{\sft_j}} \left[-\frac{1}{2\pi} \ln\left|\frac{\Omega(\sft_i,\sfx_i)-\Omega(\sft_j,\sfx_j)}{\bar \Omega(\sft_i,\sfx_i)-\Omega(\sft_j,\sfx_j)}\right|\right]  f_j(\sfx_j)  f_i(\sfx_i)  \rd  \sfx_j \rd \sfx_i,
 \end{equation}
 where   $\mathcal L_{\sft}=\{\sfx\mid (\sft,\sfx)\in \sfL(\sfP)\}$ is the section of the liquid region at time $\sft$, as in \eqref{e:Liquid_def_2}, and $\Omega$ is the bijection of Lemmas \ref{Lemma_bijection_intro} and \ref{Lemma_bijection_main}.
\end{theorem}
\begin{remark}
 Since $-\frac{1}{2\pi} \ln\left|\frac{\Omega(\sft_i,\sfx_i)-\Omega(\sft_j,\sfx_j)}{\bar \Omega(\sft_i,\sfx_i)-\Omega(\sft_j,\sfx_j)}\right|$ is the composition of the covariance of GFF in $\bH^+$ given by \eqref{eq_GFF_cov} or \eqref{e:GFF_covariance} with the map $\bar \Omega (\sft,\sfx)$, \eqref{e:x560} is the covariance of the pairings with $\Omega$--pullback of the Gaussian Free Field, as claimed in Theorem \ref{t:GFF}. By integrating the result of Theorem \ref{theorem_GFF_main} in $\sft$--direction, we can also extend the convergence statement to pairings of GFF with real test-functions $f(\sft,\sfx)$ of the form $f(\sft,\sfx)=\partial_x\bigl[ F(\sft, \sfb_\sft(\sfx))\bigr]$, where $F$ is analytic in $\sfx$--variable and continuous in $\sft$--variable.
\end{remark}
\begin{remark}
 The functions of the form $f(x)=\partial_x \bigl[F(\sfb_{\sft_i}(x))\bigr]$ with real-analytic $F(\cdot)$ are dense in the space of all continuous function on $[\bl(\sft),\br(\sft)]$ equipped with the uniform convergence topology; hence, these test--functions are sufficient to uniquely determine the limiting Gaussian field.
\end{remark}

\begin{proof}[Proof of Theorem \ref{theorem_GFF_main}] This is combination of Corollary \ref{Corollary_height_covariance} with Proposition \ref{Proposition_covariance_transformation}.
\end{proof}

%%%%%%%%%%%%%%%%%%%%%%%%%%%%%%%%%%%%%%%%%%%%%%
%% Example with single Appendix:            %%
%%%%%%%%%%%%%%%%%%%%%%%%%%%%%%%%%%%%%%%%%%%%%%
\begin{appendix}
\section{Proofs of claims in Section \ref{s:examples}}
\label{Section_Appendix_examples}
\begin{proof}[Proof of Claim \ref{c:Jackdescending}]
We recall from  \cite{Stanley_Jack},
\cite[Chapter VI: (7.14’),  (6.11), and Section 10]{macdonald1998symmetric}, and \cite[(2.3)]{OO_Jack} that
\begin{align}\begin{split}\label{e:Jackexp}
&J_{\bmla/\bmmu}(1;\theta)=
\prod_{1\leq i\leq j\leq n}\frac{f(y_i-y_j)f(x_i-x_{j+1}-\theta)}{f(y_i-x_{j+1}-\theta)f(x_i-y_j)},\qquad f(x)=\frac{\Gamma(x+1)}{\Gamma(x+\theta)};\\
&J_\bmmu(\bm1;\theta)=\prod_{1\leq i<j\leq n}\left[\frac{\Gamma(y_i-y_j+\theta)}{\Gamma(y_i-y_{j})} \cdot \frac{\Gamma(\theta(j-i))}{\Gamma(\theta(j-i+1))}\right].
\end{split}\end{align}
Plugging \eqref{e:Jackexp} into \eqref{e:transi4.5} and ignoring the factors which do not depend on $\bmy$, we get
\begin{align}\label{e:transi5.5copy}
\bP(\bmy|\bmx)\propto\prod_{i<j}(y_i-y_j) \prod_{1\leq i\leq j\leq n}\frac{\Gamma(y_i-x_{j+1})\Gamma(x_i-y_j+\theta)}{\Gamma(y_i-x_{j+1}-\theta+1)\Gamma(x_i-y_j+1)}.
\end{align}
This finishes the proof of \eqref{e:transi5.5}. In the following we show that \eqref{e:transi5.5copy} is a special case of \eqref{e:m2}.
We recall the general form of the transition probability \eqref{e:m2}:
\begin{equation}\label{e:m2copy1}
\bP(\bmy|\bmx)
=\frac{1}{Z(\bmx)}  \prod_{1\leq i<j\leq n}\bigl(b(y_i)-b(y_j)\bigr)\, \prod_{i=1}^n \left[
w(y_i) \prod\limits_{\begin{smallmatrix}\ell \in \bL(\bmx) \\ \ell>y_i\end{smallmatrix}} \frac{1}{b(\ell)-b(y_i)}\prod\limits_{\begin{smallmatrix}\ell\in \bL(\bmx)\\ \ell<y_i\end{smallmatrix}} \frac{1}{b(y_i)-b(\ell)}\right].
%&\propto
%\prod_{i<j}\frac{b(x_i)-b(x_j)}{b(y_i)-b(y_j)}\prod_{i}\frac{1}{\prod_{\bL\setminus\{x_i\}\ni\ell>x_i} b(\ell)-b(x_i)\prod_{\bL\setminus\{x_i\}\ni\ell<x_i} b(x_i)-b(\ell)}
%\prod_{i=1}^n  w(x_i),
\end{equation}
Let us take  $b(z)=z$ in \eqref{e:m2copy1}. The second factor becomes $\prod_{i<j}(y_i-y_j)$.
For the product over $\ell>y_i$ in the third factor in \eqref{e:m2copy1} we have
\begin{align}\begin{split}\label{e:fp100}
&\phantom{{}={}}\prod_{i=1}^n \prod\limits_{\begin{smallmatrix}\ell \in \bL(\bmx) \\ \ell>y_i\end{smallmatrix}} (\ell-y_i)
=\prod_{i=1}^n \left[ \prod_{\ell=y_i+1}^{x_i}(\ell-y_i)\prod_{j=1}^{i-1}\prod_{\ell=x_{j+1}+\theta}^{x_j}(\ell-y_i)\right]\\
&=\prod_{i=1}^n\left[ \Gamma(x_i-y_i+1)\prod_{j=1}^{i-1}\frac{\Gamma(x_j-y_i+1)}{\Gamma(x_{j+1}-y_i+\theta)}\right]
=\prod_{j=1}^n\Gamma(x_1-y_j+\theta)\prod_{1\leq i\leq j\leq n}\frac{\Gamma(x_i-y_j+1)}{\Gamma(x_{i}-y_j+\theta)},
\end{split}\end{align}
where we exchanged the roles of $i$ and $j$ in the last equality.
Similarly, for the $\ell<y_i$ factor we have
\begin{align}\begin{split}\label{e:fp200}
&\phantom{{}={}}\prod_{i=1}^n \prod\limits_{\begin{smallmatrix}\ell \in \bL(\bmx) \\ \ell<y_i\end{smallmatrix}}(y_i-\ell)
=\prod_{i=1}^n \left[\prod_{\ell=x_{i+1}+\theta}^{y_i-1} (y_i-\ell)\prod_{j=n}^{i+1}\prod_{\ell=x_{j+1}+\theta}^{x_j} (y_i-\ell)\right]\\
&=\prod_{i=1}^n\left[\Gamma(y_i-x_{i+1}-\theta+1)\prod_{j=i+1}^{n}\frac{\Gamma(y_i-x_{j+1}-\theta+1)}{\Gamma(y_i-x_{j})}\right]\\
&=\prod_{i=1}^n \Gamma(y_i-x_{n+1})\prod_{1\leq i\leq j\leq n}\frac{\Gamma(y_i-x_{j+1}-\theta+1)}{\Gamma(y_i-x_{j+1})}.
\end{split}\end{align}
Plugging  \eqref{e:fp100} and \eqref{e:fp200} into \eqref{e:m2copy1}, we get $\bP(\bmy|\bmx)$ is proportional to
\begin{align*}
\prod_{1\leq i<j\leq n}(y_i-y_j)\prod_{1\leq i\leq j\leq n}\frac{\Gamma(x_i-y_j+\theta)\Gamma(y_i-x_{j+1})}{\Gamma(x_{i}-y_j+1)\Gamma(y_i-x_{j+1}-\theta+1)}
\prod_{i=1}^n \frac{w(y_i)}{\Gamma(x_1-y_i+\theta)\Gamma(y_i-x_{n+1})},
\end{align*}
which simplifies to \eqref{e:transi5.5copy} if
$w(z)={\Gamma(x_1-z+\theta)}{\Gamma(z-x_{n+1})}.$
\end{proof}

\begin{proof}[Proof of Claim \ref{c:mdensity2}]
The evaluation of the Macdonald polynomial under the principle specialization $\rho=(1,t,\dots, t^{n-1})$ is explicit, see \cite[Chapter VI, (6.11)]{macdonald1998symmetric}:
\begin{align}\label{e:MDP}
P_\bmla(1,t,\dots, t^{n-1};q,t)=t^{\sum_{i=1}^{n}(i-1)\la_i}\prod_{i<j\leq n}\frac{(q^{\la_i-\la_j}t^{j-i};q)_\infty}{(q^{\la_i-\la_j}t^{j-i+1};q)_\infty}\frac{(t^{j-i+1};q)_\infty}{(t^{j-i};q)_\infty},
\end{align}
for $\bmla=\lambda_1\geq \lambda_2\dots\geq\lambda_n\geq 0$; $P_\bmla(\rho;q,t)=0$  if $\lambda_{n+1}>0$. We further use the $w_{u,v}$ automorphism of the algebra of symmetric functions $\Lambda$ defined on the power sums by:
$$
 w_{u,v}(p_k)=(-1)^{k-1} \frac{1-u^k}{1-v^k} p_k.
$$
As shown in \cite[Chapter VI, Section 7]{macdonald1998symmetric},
$$
 w_{t,q}\bigl(P_{\bmla'/\bmmu'}(\cdot; t,q)\bigr)= Q_{\bmla/\bmmu}(\cdot ; q,t),
$$
where $\bmla'$ and $\bmmu'$ are transposed Young diagrams $\bmla$ and $\bmmu$, respectively. Hence, with $\rho'$ being the specialization $\beta_1=a$, as in \eqref{eq_principal_ascending}, we have
$$
 Q_{\bmla/\bmmu}(\beta_1=a;q,t)= P_{\bmla'/\bmmu'}(\alpha_1=a; t,q).
$$
The latter specialization of the skew $P$--function in one variable is \cite[Chapter VI, (7.14')]{macdonald1998symmetric}; in the notations there,  it is $a^{|\bmla|-|\bmmu|}\psi_{\bmla'/\bmmu'}(t,q)=a^{|\bmla|-|\bmmu|}\psi'_{\bmla/\bmmu}(q,t)$ with an additional requirement that $\bmmu'\prec\bmla'$. The latter is the same as requiring
\begin{align}\label{e:vs}
\la_i=\mu_i+e_i, \quad e_i\in\{0,1\},\quad 1\leq i\leq n.
\end{align}
The final formula is
\begin{align*}
Q_{\bmla/\bmmu}(\beta_1=a;q,t)
&=a^{|\bmla|-|\bmmu|} \prod_{i< j\leq n}
\bm1(e_i=0, e_j=1)\frac{(1-q^{\mu_i-\mu_j}t^{j-i-1})(1-q^{\mu_i-\mu_j-1}t^{j-i+1})}{(1-q^{\mu_i-\mu_j}t^{j-i})(1-q^{\mu_i-\mu_j-1}t^{j-i})}.
\end{align*}
Let us also rewrite \eqref{e:MDP} in terms of $\bmmu$ and $\bme$ using \eqref{e:vs}:
\begin{align*}\begin{split}
&\phantom{{}={}}P_\bmla(1,t,\dots, t^{n-1};q,t)
=t^{\sum_{i=1}^{N}(i-1)\mu_i}t^{\sum_{i=1}^{N}(i-1)e_i}\prod_{i<j\leq n}\frac{(q^{\mu_i-\mu_j+e_i-e_j}t^{j-i};q)_\infty}{(q^{\mu_i-\mu_j+e_i-e_j}t^{j-i+1};q)_\infty}\frac{(t^{j-i+1};q)_\infty}{(t^{j-i};q)_\infty}\\
&=t^{\sum_{i=1}^{n}(i-1)\mu_i}\prod_{i<j\leq n}\Biggl[\frac{(q^{\mu_i-\mu_j}t^{j-i};q)_\infty}{(q^{\mu_i-\mu_j}t^{j-i+1};q)_\infty}\frac{(t^{j-i+1};q)_\infty}{(t^{j-i};q)_\infty}t^{\sum_{i=1}^{N}(i-1)e_i}\\
&\left(\bm1(e_i=0, e_j=1)\frac{1-q^{\mu_i-\mu_j-1}t^{j-i}}{1-q^{\mu_i-\mu_j-1}t^{j-i+1}}
+\bm1(e_i=1, e_j=0)\frac{1-q^{\mu_i-\mu_j}t^{j-i+1}}{1-q^{\mu_i-\mu_j}t^{j-i}}+\bm1(e_i=e_j)\right)\Biggr].
\end{split}\end{align*}
Combining the last two formulas together we convert \eqref{eq_ascending_transition} into
\begin{align}\begin{split}\label{e:MPT2}
&\bP(\bmla\mid \bmmu)
\propto a^{|\bmla|-|\bmmu|}t^{\sum_{i=1}^{N}(i-1)e_i}\times\\
&\prod_{i<j}\left[\bm1(e_i=0, e_j=1)\frac{1-q^{\mu_i-\mu_j}t^{j-i-1}}{1-q^{\mu_i-\mu_j}t^{j-i}}
+\bm1(e_i=1, e_j=0)\frac{1-q^{\mu_i-\mu_j}t^{j-i+1}}{1-q^{\mu_i-\mu_j}t^{j-i}}+\bm1(e_i=e_j)\right],
\end{split}\end{align}
with additional requirement \eqref{e:vs}.
%We denote $t=q^\theta$ for some $\theta>0$, and encode the young diagrams $ \bmmu,\bmla$ as particle systems $\bmx, \bmx+\bme$, with
%\begin{align}
%&\bmx=(x_1,x_2,\dots, x_{n})\in \bW_\theta^{n},\quad x_i=\mu_i-\theta(i-1),\quad x_i+e_i=\la_i-\theta(i-1), \quad1\leq i\leq n.
%\end{align}
Setting $t=q^{\theta}$ and encoding the Young diagrams $ \bmmu$ and $\bmla$ as particle systems $\bmx\in  \bW_\theta^{n}$ and $\bmx+\bme\in  \bW_\theta^{n}$, we rewrite the transition probability \eqref{e:MPT2}
\begin{align*}
\bP(\bmx+\bme\mid \bmx)\propto a^{\sum_{i=1}^n e_i}t^{\sum_{i=1}^{N}(i-1)e_i}
\prod_{i<j}\frac{1-q^{x_i-x_j+\theta(e_i-e_j)}}{1-q^{x_i-x_j}}\propto  a^{\sum_{i=1}^n e_i}
\prod_{i<j}\frac{q^{x_i+\theta e_i}-q^{x_j+\theta e_j}}{q^{x_i}-q^{x_j}},
\end{align*}
which is a special case of \eqref{e:m1}, with $b(z)=q^z$, $\phi^-(z)=1$, and $\phi^+(z)=a$.
\end{proof}

\begin{proof}[Proof of Claim \ref{c:Mdescending}]
Using \cite[Chapter VI, (7.14')]{macdonald1998symmetric} and notations \eqref{e:bracketf}, we have
\begin{align}\label{e:Pmula2}
P_{\bmla/\bmmu}(t^{n};q,t)
=t^{n(|\bmla|-|\bmmu|)} f(1)^{n}\prod_{i<j\leq n}f(q^{\mu_i-\mu_j}t^{j-i})\prod_{i\leq j\leq n}
\frac{f(q^{\la_i-\la_{j+1}}t^{j-i})}{f(q^{\mu_i-\la_{j+1}}t^{j-i})f(q^{\la_i-\mu_j}t^{j-i})}.
\end{align}
Plugging \eqref{e:Pmula2} and \eqref{e:MDP} and setting $t=q^{\theta}$, we rewrite the transition probability in Definition \ref{Definition_Macdonald_descending} as
\begin{align}\label{e:Mdescend2copy}
\bP(\bmy|\bmx)
\propto q^{-\theta \sum_{i=1}^{n}(n-i+1)y_i} \prod_{i<j\leq n}
(1-q^{y_i-y_j})
\prod_{i\leq j\leq n}\frac{\Gamma_q(y_i-x_{j+1})}{\Gamma_q(y_i-x_{j+1}+1-\theta)}
\frac{\Gamma_q(x_i-y_j+\theta)}{\Gamma_q(x_i-y_j+1)},
\end{align}
which matches \eqref{e:Mdescend}.
In the following we show that \eqref{e:Mdescend2copy} is a special case of \eqref{e:m2}, with $b(z)=q^{-z}$. We recall the transition probability \eqref{e:m2}
\begin{equation}\label{e:m2copy_2}
\bP(\bmy|\bmx)
=\frac{1}{Z(\bmx)}  \prod_{1\leq i<j\leq n}\bigl(b(y_i)-b(y_j)\bigr)\, \prod_{i=1}^n \left[
w(y_i) \prod\limits_{\begin{smallmatrix}\ell \in \bL(\bmx) \\ \ell>y_i\end{smallmatrix}} \frac{1}{b(\ell)-b(y_i)}\prod\limits_{\begin{smallmatrix}\ell\in \bL(\bmx)\\ \ell<y_i\end{smallmatrix}} \frac{1}{b(y_i)-b(\ell)}\right].
\end{equation}
We take  $b(z)=q^{-z}$ in \eqref{e:m2copy_2}. The second factor becomes,
\begin{align}\label{e:fp0}
\prod_{i<j}(b(y_i)-b(y_j))
=\prod_{i<j}(q^{-y_i}-q^{-y_j})
=\prod_{i<j}(1-q^{y_i-y_j})q^{-y_i}
=\prod_{i<j}(1-q^{y_i-y_j})\prod_i q^{-(n-i)y_i}.
\end{align}
For the last product in \eqref{e:m2copy_2}, we have
\begin{multline}\label{e:fp}
 \prod\limits_{\begin{smallmatrix}\ell \in \bL(\bmx) \\ \ell>y_i\end{smallmatrix}} \bigl(b(\ell)-b(y_i)\bigr) \prod\limits_{\begin{smallmatrix}\ell \in \bL(\bmx) \\ \ell<y_i\end{smallmatrix}} \bigl(b(y_i)-b(\ell)\bigr)\\
\propto
\left( \prod\limits_{\begin{smallmatrix}\ell \in \bL(\bmx) \\ \ell>y_i\end{smallmatrix}} \frac{1-q^{\ell-y_i}}{1-q}\right)\left( \prod\limits_{\begin{smallmatrix}\ell \in \bL(\bmx) \\ \ell<y_i\end{smallmatrix}} \frac{1-q^{y_i-\ell}}{1-q}\right)\left( \prod\limits_{\begin{smallmatrix}\ell \in \bL(\bmx) \\ \ell>y_i\end{smallmatrix}}q^{-\ell} \prod\limits_{\begin{smallmatrix}\ell \in \bL(\bmx) \\ \ell<y_i\end{smallmatrix}}q^{-y_i}\right).
\end{multline}
We transform the first factor in the righthand side of \eqref{e:fp} as
\begin{multline}\label{e:fp1}
\prod_{i=1}^n \left(\prod_{\ell=y_i+1}^{x_i}\frac{1-q^{\ell-y_i}}{1-q}\cdot \prod_{j=1}^{i-1}\left[\prod_{\ell=x_{j+1}+\theta}^{x_j}\frac{1-q^{\ell-y_i}}{1-q}\right]\right)\\
=\prod_{i=1}^n\left(\Gamma_q(x_i-y_i+1)\prod_{j=1}^{i-1}\frac{\Gamma_q(x_j-y_i+1)}{\Gamma_q(x_{j+1}+\theta-y_i)}\right)
=\prod_{i=1}^n \Gamma_q(x_1-y_i+\theta)\prod_{1\leq i\leq j\leq n}\frac{\Gamma_q(x_i-y_j+1)}{\Gamma_q(x_{i}-y_j+\theta)},
\end{multline}
where in the last identity we swapped $i$ and $j$. We also transform the second factor in \eqref{e:fp}:
\begin{multline}\label{e:fp2}
\prod_{i=1}^n \left(\prod_{\ell=x_{i+1}+\theta}^{y_i-1} \frac{1-q^{y_i-\ell}}{1-q}\prod_{j=n}^{i+1}\left[\prod_{\ell=x_{j+1}+\theta}^{x_j} \frac{(1-q^{y_i-\ell})}{1-q}\right]\right)\\
=\prod_{i=1}^n\left( \Gamma_q(y_i-x_{i+1}-\theta+1)\prod_{j=i+1}^{n}\frac{\Gamma_q(y_i-x_{j+1}-\theta+1)}{\Gamma_q(y_i-x_{j})}\right)\\
=\prod_{i=1}^n \Gamma_q(y_i-x_{n+1}) \prod_{1\leq i\leq j\leq n}\frac{\Gamma_q(y_i-x_{j+1}-\theta+1)}{\Gamma_q(y_i-x_{j+1})}.
\end{multline}
Finally, we transform the  last factor in the righthand side of \eqref{e:fp}:
\begin{multline}\label{e:fp3}
\prod_{i=1}^n\left( \prod\limits_{\begin{smallmatrix}\ell \in \bL(\bmx) \\ \ell>y_i\end{smallmatrix}}q^{-\ell} \prod\limits_{\begin{smallmatrix}\ell \in \bL(\bmx) \\ \ell<y_i\end{smallmatrix}}q^{-y_i}\right)
=q^{-\sum\limits_{i=1}^n  \left(\sum\limits_{\ell=y_i+1}^{x_i }\ell+\sum\limits_{j=1}^{i-1} \sum\limits_{\ell=x_{j+1}+\theta}^{x_j} \ell\right)-\sum\limits_{i=1}^n\left( \sum\limits_{\ell=x_{i+1}+\theta}^{y_i-1}y_i+\sum\limits_{j=i+1}^n \sum\limits_{\ell=x_{j+1}+\theta}^{x_j}y_i\right)}\\
\propto q^{\sum\limits_{i=1}^n (y_i^2+y_i)/2 - \sum\limits_{i=1}^n y_i(y_i-x_{n+1}-(n-i+1)\theta +n-i)}
=q^{-\sum\limits_{i=1}^n \left(y_i^2/2 +y_i(-x_{n+1}-(n-i+1)\theta +n-i-1/2)\right)}.
\end{multline}
Combining \eqref{e:fp0}, \eqref{e:fp}, \eqref{e:fp1} \eqref{e:fp2} and \eqref{e:fp3}, we rewrite \eqref{e:m2copy_2} as
$$
 \prod_{i<j}(1-q^{y_i-y_j})\prod_{i\leq j}\frac{\Gamma_q(x_i-y_j+\theta)}{\Gamma_q(x_{i}-y_j+1)}\frac{\Gamma_q(y_i-x_{j+1})}{\Gamma_q(y_i-x_{j+1}-\theta+1)}
\prod_{i=1}^n \frac{q^{-\theta (n-i+1) y_i}q^{y_i^2/2-y_i (x_{n+1}+1/2)}w(y_i)}{\Gamma_q(x_1-y_i+\theta)\Gamma_q(y_i-x_{n+1})},
$$
which matches \eqref{e:Mdescend} with $w(z)$ as in \eqref{eq_weight_Macdonald}.
\end{proof}

\begin{proof}[Proof of Claim \ref{c:MKdescending}]
We recall the parameters $s^2=q^v$, and the coefficients \eqref{e:phii}
\begin{align}\begin{split}\label{e:phiicopy}
\psi_{\bmla\setminus\bmmu}^{(i)}(u; q,t,s)
&=(u/t)^{|\bmla|-|\bmmu|}t^{n(\bmmu)-n(\bmla)}
\frac{C_\bmla^0(s^2qt/u;q,t)C_\bmmu^0(u/t;q,t)}{C_\bmla^0(u;q,t)C_\bmmu^0(s^2qt/u;q,t)}\\
&\prod_{(i,j)\in \bmla\atop \lambda_j'=\mu_j'}\frac{1-q^{\la_i+j-1}t^{-\la_j'-i+3}s^2}{1-q^{\mu_i-j+1}t^{\mu_j'-i}}
\prod_{(i,j)\in \bmla\atop \lambda_j'\neq\mu_j'}\frac{1-q^{\la_i-j}t^{\la_j'-i+1}}{1-q^{\mu_i+j}t^{-\mu_j'-i+1}s^2}\\
&\prod_{(i,j)\in \bmmu\atop \lambda_j'=\mu_j'}\frac{1-q^{\la_i-j+1}t^{\la_j'-i}}{1-q^{\mu_i+j-1}t^{2-\mu_j'-i}s^2}
\prod_{(i,j)\in \bmmu\atop \lambda_j'\neq\mu_j'}\frac{1-q^{\la_i+j}t^{2-\la_j'-i}s^2}{1-q^{\mu_i-j}t^{\mu_j'-i+1}},
\end{split}\\
\notag C_\bmmu^0(x;q,t)&=\prod_{(i,j)\in \bmmu}(1-q^{j-1}t^{1-i}x)=\prod_{i\geq 1}(t^{1-i}x;q)_{\mu_i}
\end{align}
We take $u=t^{n+1}$ and $s=t^{n}\sqrt{t_0t_1t_2t_3/qt}$, and recall $t^{n}\sqrt{t_0t_1t_2t_3/qt}=q^{v/2}$. We have
\begin{align}\label{e:ut1}
(u/t)^{|\bmla|-|\bmmu|}t^{n(\bmmu)-n(\bmla)}
=t^{n(|\bmla|-|\bmmu|)+n(\bmmu)-n(\bmla)}=q^{-\theta (\sum_i (n-(i-1))y_i-\sum_i (n-(i-1))x_i)}.
\end{align}
and
\begin{align}\begin{split}\label{e:ct1}
&\phantom{{}={}}\frac{C_\bmmu^0(u/t;q,t)}{C_\bmmu^0(s^2qt/u;q,t)}
=\frac{C_\bmmu^0(t^n;q,t)}{C_\bmmu^0(q^{v+1}t^{-n};q,t)}
=\frac{C_\bmmu^0(q^{n\theta};q,t)}{C_\bmmu^0(q^{v+1-n\theta};q,t)}\\
&=\prod_{i=1}^n\frac{\Gamma_q(\mu_i+(n-i+1)\theta)}{\Gamma_q((n-i+1)\theta)}\frac{\Gamma_q(v+1-(n+i-1)\theta)}{\Gamma_q(v+1-(n+i-1)\theta+\mu_i)}\\
&= \prod_{i=1}^n\frac{\Gamma_q(y_i+n\theta)\Gamma_q(v+1-(n+i-1)\theta)}{\Gamma_q((n-i+1)\theta)\Gamma_q(y_i+v+1-n\theta)}.
\end{split}\end{align}
In the following, we simplify the products in $\psi_{\bmla/\bmmu}^{(i)}(u;q,t,s)$. For the first product term in \eqref{e:phiicopy}
\begin{align}\begin{split}\label{e:pt1}
\prod_{(i,j)\in \bmla\atop \lambda_j'=\mu_j'}\frac{1-q^{\la_i+j-1}t^{-\la_j'-i+3}s^2}{1-q^{\mu_i-j+1}t^{\mu_j'-i}}
&=\prod_k\prod_{(i,j)\in \bmla\atop \lambda_j'=\mu_j'=k}\frac{1-q^{\la_i+j-1}t^{-k-i+3}s^2}{1-q^{\mu_i-j+1}t^{k-i}}\\
&=\prod_k\prod_{1\leq i\leq k\atop \la_{k+1}+1\leq j\leq \mu_k}\frac{1-q^{\la_i+j-1}t^{-k-i+3}s^2}{1-q^{\mu_i-j+1}t^{k-i}}\\
&=\prod_{i\leq k}\frac{\Gamma_q(x_i+y_k+v+\theta)\Gamma_q(y_i-y_k+1)}{\Gamma_q(x_i+x_{k+1}+v+2\theta)\Gamma_q(y_i-x_{k+1}-\theta+1)}.
\end{split}\end{align}
For the second product term in \eqref{e:phiicopy}
\begin{align}\begin{split}\label{e:pt2}
\prod_{(i,j)\in \bmla\atop \lambda_j'\neq\mu_j'}\frac{1-q^{\la_i-j}t^{\la_j'-i+1}}{1-q^{\mu_i+j}t^{-\mu_j'-i+1}s^2}
&=\prod_k \prod_{(i,j)\in \bmla\atop \lambda_j'=k, \mu_j'=k-1}\frac{1-q^{\la_i-j}t^{k-i+1}}{1-q^{\mu_i+j}t^{-k-i+2}s^2}\\
&=\prod_k \prod_{1\leq i\leq k\atop \mu_{k}+1\leq j\leq \la_{k}}\frac{1-q^{\la_i-j}t^{k-i+1}}{1-q^{\mu_i+j}t^{-k-i+2}s^2}\\
&=\prod_{i\leq k}\frac{\Gamma_{q}(x_i-y_k+\theta)\Gamma_q(y_i+y_k+1+v)}{\Gamma_q(x_i-x_k+\theta)\Gamma_q(y_i+x_k+1+v)}.
\end{split}\end{align}
For the third product term in \eqref{e:phiicopy}, we have
\begin{align}\begin{split}\label{e:pt3}
\prod_{(i,j)\in \bmmu\atop \lambda_j'=\mu_j'}\frac{1-q^{\la_i-j+1}t^{\la_j'-i}}{1-q^{\mu_i+j-1}t^{2-\mu_j'-i}s^2}
&=\prod_k \prod_{(i,j)\in \bmmu\atop \lambda_j'=\mu_j'=k}\frac{1-q^{\la_i-j+1}t^{k-i}}{1-q^{\mu_i+j-1}t^{2-k-i}s^2}\\
&=\prod_k \prod_{1\leq i\leq k,\atop \la_{k+1}+1\leq j\leq \mu_k}\frac{1-q^{\la_i-j+1}t^{k-i}}{1-q^{\mu_i+j-1}t^{2-k-i}s^2}\\
&=\prod_{i\leq k} \frac{\Gamma_q(x_i-x_{k+1}-\theta+1)\Gamma_{q}(y_i+x_{k+1}+v+\theta)}{\Gamma_q(x_i-y_k+1)\Gamma_q(y_i+y_k+v)}.
\end{split}\end{align}
Finally for the last product term in \eqref{e:phiicopy}, we have
\begin{align}\begin{split}\label{e:pt4}
\prod_{(i,j)\in \bmmu\atop \lambda_j'\neq\mu_j'}\frac{1-q^{\la_i+j}t^{2-\la_j'-i}s^2}{1-q^{\mu_i-j}t^{\mu_j'-i+1}}
&=\prod_k\prod_{(i,j)\in \bmmu\atop \lambda_j'=k+1,\mu_j'=k}\frac{1-q^{\la_i+j}t^{1-k-i}s^2}{1-q^{\mu_i-j}t^{k-i+1}}\\
&=\prod_k\prod_{1\leq i\leq k\atop\mu_{k+1}+1\leq j\leq \lambda_{k+1}}\frac{1-q^{\la_i+j}t^{1-k-i}s^2}{1-q^{\mu_i-j}t^{k-i+1}}\\
&=\prod_{1\leq i\leq k}\frac{\Gamma_q(y_i-x_{k+1})\Gamma_q(x_i+x_{k+1}+1+v)}{\Gamma_q(y_i-y_{k+1})\Gamma_q(x_i+y_{k+1}+1+v)}.
\end{split}\end{align}
Putting \eqref{e:ut1}, \eqref{e:ct1}, \eqref{e:pt1}, \eqref{e:pt2}, \eqref{e:pt3}, and \eqref{e:pt4} all together, we get the following expression for the transition probability \eqref{e:MKdensity}:
\begin{align}\begin{split}\label{e:KMdensityhicopy}
&\bP(\bmy|\bmx)=
\prod_{i=1}^n \left[q^{-\theta (\sum_i (n-(i-1))y_i-\sum_i (n-(i-1))x_i)}\frac{\Gamma_q(y_i+n\theta)\Gamma_q(v+1-(n+i-1)\theta)}{\Gamma_q((n-i+1)\theta)\Gamma_q(y_i+v+1-n\theta)}\right]\\
&\prod_{i\leq k\leq n}\frac{\Gamma_q(x_i+y_k+v+\theta)\Gamma_q(y_i-y_k+1)}{\Gamma_q(x_i+x_{k+1}+2\theta+v)\Gamma_q(y_i-x_{k+1}-\theta+1)}
\prod_{i\leq k\leq n+1}\frac{\Gamma_{q}(x_i-y_k+\theta)\Gamma_q(y_i+y_k+1+v)}{\Gamma_q(x_i-x_k+\theta)\Gamma_q(y_i+x_k+1+v)}\\
&\prod_{i\leq k\leq n}\frac{\Gamma_q(x_i-x_{k+1}-\theta+1)\Gamma_{q}(y_i+x_{k+1}+v+\theta)}{\Gamma_q(x_i-y_k+1)\Gamma_q(y_i+y_k+v)}
\prod_{i\leq k\leq n}\frac{\Gamma_q(y_i-x_{k+1})\Gamma_q(x_i+x_{k+1}+1+v)}{\Gamma_q(y_i-y_{k+1})\Gamma_q(x_i+y_{k+1}+1+v)}\\
%&\propto\prod_{i<k}(1-q^{y_i-y_k})\prod_{i\leq k}\frac{\Gamma_q(y_i-x_{k+1})\Gamma_{q}(x_i-y_k+\theta)}{\Gamma_q(x_i-y_k+1)\Gamma_q(y_i-x_{k+1}-\theta+1)}
%\frac{\Gamma_q(x_i-x_{k+1}-\theta+1)}{\Gamma_q(x_i-x_{k+1}+\theta)}\\
%&\prod_{i\leq k}(1-q^{y_i+y_k+v})\prod_{i\leq n,k\leq n+1}\frac{\Gamma_q(y_i+x_k+v+\theta)}{\Gamma_q(y_i+x_k+1+v)}\prod_{i\leq k}\frac{\Gamma_q(x_i+x_{k+1}+1+v)}{\Gamma_q(x_i+x_{k+1}+2\theta+v)}\\
%&\prod_i q^{-\theta (\sum_i (n-(i-1))y_i-\sum_i (n-(i-1))x_i)}\frac{\Gamma_q(y_i+n\theta)}{\Gamma_q(y_i+v+1-n\theta)}\frac{\Gamma_q(x_i-y_{n+1}+\theta)\Gamma_q(y_i+y_{n+1}+1+v)}{\Gamma_q(y_i-y_{n+1})}
%\\
&\propto\prod_{1\leq i<k\leq n}(1-q^{y_i-y_k})\prod_{1\leq i\leq k\leq n}\frac{\Gamma_q(y_i-x_{k+1})\Gamma_{q}(x_i-y_k+\theta)}{\Gamma_q(x_i-y_k+1)\Gamma_q(y_i-x_{k+1}-\theta+1)}\\
&\times \prod_{1\leq i\leq k\leq n}(1-q^{y_i+y_k+v})\left[\prod_{i=1}^n \prod_{k=1}^{ n+1}\frac{\Gamma_q(y_i+x_k+v+\theta)}{\Gamma_q(y_i+x_k+1+v)}\right]\prod_{i=1}^n q^{-\theta \sum_{i=1}^n (n-(i-1))y_i}.
\end{split}\end{align}
This finishes the proof of \eqref{e:KMdensityhi}. In the following we check that \eqref{e:KMdensityhi} is a  special case of \eqref{e:m2}, with $b(z)=q^{-z}+q^{z+v}$. We notice that the difference factorizes:
\begin{align*}
 b(z_1)-b(z_2)=q^{-z_1}+ q^{z_1+v}-q^{-z_2}-q^{z_2+v}
 =(q^{-z_1}-q^{-z_2})(1-q^{z_1+z_2+v}).
\end{align*}
We recall the transition probability \eqref{e:m2}
\begin{equation}\label{e:m2copy_3}
\bP(\bmy|\bmx)
=\frac{1}{Z(\bmx)}  \prod_{1\leq i<j\leq n}\bigl(b(y_i)-b(y_j)\bigr)\, \prod_{i=1}^n \left[
w(y_i) \prod\limits_{\begin{smallmatrix}\ell \in \bL(\bmx) \\ \ell>y_i\end{smallmatrix}} \frac{1}{b(\ell)-b(y_i)}\prod\limits_{\begin{smallmatrix}\ell\in \bL(\bmx)\\ \ell<y_i\end{smallmatrix}} \frac{1}{b(y_i)-b(\ell)}\right].
\end{equation}
With  $b(z)=q^{-z}+q^{z+v}$, the first product becomes
\begin{align}\label{e:1fp0}
\prod_{i<j}(b(y_i)-b(y_j))
=\prod_{i<j}(q^{-y_i}-q^{-y_j})(1-q^{y_i+y_j+v}).
\end{align}
For the denominators in \eqref{e:m2copy_3} we have
\begin{align}\begin{split}\label{e:1fp}
&\phantom{{}={}} \prod\limits_{\begin{smallmatrix}\ell\in \bL(\bmx)\\ \ell>y_i\end{smallmatrix}} \bigl(b(\ell)-b(y_i)\bigr) \prod\limits_{\begin{smallmatrix}\ell\in \bL(\bmx)\\ \ell<y_i\end{smallmatrix}}\bigl(b(y_i)-b(\ell)\bigr)\\
&=\prod\limits_{\begin{smallmatrix}\ell\in \bL(\bmx)\\ \ell>y_i\end{smallmatrix}}  (q^{-\ell}-q^{-y_i})(1-q^{y_i+\ell+v})\prod\limits_{\begin{smallmatrix}\ell\in \bL(\bmx)\\ \ell<y_i\end{smallmatrix}} (q^{-y_i}-q^{-\ell})(1-q^{y_i+\ell+v})\\
&\propto
\prod\limits_{\begin{smallmatrix}\ell\in \bL(\bmx)\\ \ell>y_i\end{smallmatrix}}  (q^{-\ell}-q^{-y_i})\prod\limits_{\begin{smallmatrix}\ell\in \bL(\bmx)\\ \ell<y_i\end{smallmatrix}} (q^{-y_i}-q^{-\ell})
\prod_{\ell\in \bL\setminus\{y_i\}}(1-q^{\ell+y_i+v}).
\end{split}\end{align}
The first two products on the righthand side of \eqref{e:1fp} are computed in \eqref{e:fp}, \eqref{e:fp1}, \eqref{e:fp2} and \eqref{e:fp3}.
For the last product term on the righthand side of \eqref{e:1fp}, we have
\begin{align}\begin{split}\label{e:2fp}
\prod_{i=1}^n \, \prod_{\ell\in \bL\setminus\{y_i\}}(1-q^{\ell+y_i+v})
&=\prod_{i=1}^n \frac{\prod_{\ell\in \bL}(1-q^{\ell+y_i+v})}{1-q^{2y_i+v}}
=\prod_{i=1}^n \frac{\prod_{j=1}^n \prod_{\ell=x_{j+1}+\theta}^{x_j}(1-q^{\ell+y_i+v})}{ 1-q^{2y_i+v}}\\
&=\prod_{i=1}^n \left[\frac{1}{1-q^{2y_i+v}}\prod_{j=1}^n \frac{\Gamma_q(y_i+x_j+v+1)}{\Gamma_q(y_i+x_{j+1}+v+\theta)}\right]\\
&=\prod_{i=1}^n \frac{1}{1-q^{2y_i+v}} \cdot \prod_{i=1}^n \prod_{j=1}^{n+1} \frac{\Gamma_q(y_i+x_j+v+1)}{\Gamma_q(y_i+x_{j}+v+\theta)}
\cdot \prod_{i=1}^n \frac{\Gamma_q(y_i+x_1+v+\theta)}{\Gamma_q(y_i+x_{n+1}+v+1)}.
\end{split}\end{align}
Combining  \eqref{e:2fp} with transformations of \eqref{e:1fp} from the previous proof, we rewrite \eqref{e:m2copy_3} as
\begin{align*}
&\prod_{1\leq i<j \leq n}(1-q^{y_i-y_j})\prod_{1\leq i\leq j\leq n}(1-q^{y_i+y_j+v})\prod_{1\leq i\leq j\leq n}\frac{\Gamma_q(x_i-y_j+\theta)}{\Gamma_q(x_{i}-y_j+1)}\frac{\Gamma_q(y_i-x_{j+1})}{\Gamma_q(y_i-x_{j+1}-\theta+1)}
\\
&\times \prod_{i=1}^{n} \prod_{j=1}^{n+1} \frac{\Gamma_q(y_i+x_{j}+v+\theta)}{\Gamma_q(y_i+x_j+v+1)} \prod_{i=1}^n \left[\frac{\Gamma_q(y_i+x_{n+1}+v+1)}{\Gamma_q(y_i+x_1+v+\theta)}\frac{q^{-\sum_i(n-i+1)\theta y_i}q^{y_i^2/2-y_i (x_{n+1}+1/2)}w(y_i)}{\Gamma_q(x_1-y_i+\theta)\Gamma_q(y_i-x_{n+1})}\right],
\end{align*}
which matches \eqref{e:KMdensityhicopy} if
\begin{align*}
w(z)=q^{-z^2/2+z (x_{n+1}+1/2)}{\Gamma_q(x_1-z+\theta)}{\Gamma_q(z-x_{n+1})}\frac{\Gamma_q(z+x_1+v+\theta)}{\Gamma_q(z+x_{n+1}+v+1)}.
\end{align*}
This finishes the proof of Claim \ref{c:MKdescending}.
\end{proof}

\begin{proof}[Proof of Theorem \ref{Theorem_qRacah_partition}]
Let us do some simplifications in $\psi^{(i)}_{\bm\lambda/\bm\mu}$ of Theorem \ref{Theorem_Koornwinder_branching} in the $q=t$ case. We additionally set  $t^{2n} t_0 t_1 t_2 t_3=p$.
 Note that if $\lambda'_j=\mu'_j$, then $(i,j)\in{\bmla}$ is equivalent to $(i,j)\in{\bmmu}$. On the other hand, if $(i,j)\in {\bmla}/{\bmmu}$, then necessarily $\lambda'_j\ne \mu'_j$. Hence,
 {\footnotesize
\begin{multline*}
 \prod_{\substack{(i,j)\in {\bmla}\\\lambda'_j=\mu'_j}}
\frac{1-q^{\lambda_i+j} q^{-\lambda'_j-i} p }
     {1-q^{\mu_i-j+1} q^{\mu'_j-i}}
\prod_{\substack{(i,j)\in {\bmla}\\\lambda'_j\ne \mu'_j}}
\frac{1-q^{\lambda_i-j} q^{\lambda'_j-i+1}}
     {1-q^{\mu_i+j-1} q^{-\mu'_j-i} p }
\prod_{\substack{(i,j)\in {\bmmu}\\\lambda'_j=\mu'_j}}
\frac{1-q^{\lambda_i-j+1}q^{\lambda'_j-i}}
     {1-q^{\mu_i+j-1}q^{-\mu'_j-i} p}
\prod_{\substack{(i,j)\in {\bmmu}\\\lambda'_j\ne \mu'_j}}
\frac{1-q^{\lambda_i+j} q^{-\lambda'_j-i} p}
     {1-q^{\mu_i-j} q^{\mu'_j-i+1}}
\\=\prod_{(i,j)\in{\bmla}}(1-q^{\lambda_i-j} q^{\lambda'_j-i+1})(1-q^{\lambda_i+j} q^{-\lambda'_j-i} p) \prod_{(i,j)\in{\bmmu}} \frac{1}{(1-q^{\mu_i-j+1} q^{\mu'_j-i})(1-q^{\mu_i+j-1} q^{-\mu'_j-i} p )}
\\
\times
\prod_{(i,j)\in {\bmla}/{\bmmu}}
\frac{1}
     {1-q^{\mu_i+j-1} q^{-\mu'_j-i} p }
\prod_{(i,j)\in {\bmla}/{\bmmu}}
\frac{1}
     {1-q^{\lambda_i+j} q^{-\lambda'_j-i} p}.
\end{multline*}}%
We now simplify the last line. Note that since $\bm\mu\prec\bm\lambda$ and $(i,j)\in\bm\lambda/\bm\mu$, we have $\lambda'_j=i$ and $\mu'_j=i-1$. Hence, we get
\begin{align*}
\prod_{i} \left[\prod_{j=\mu_i+1}^{\lambda_i}
\frac{1}
     {1-q^{\mu_i+j} q^{-2i} p }\right]\cdot\left[
 \prod_{j=\mu_i+1}^{\lambda_i}
\frac{1}
     {1-q^{\lambda_i+j} q^{-2i} p} \right]
  &=  \prod_{i} \left[
 \prod_{j=1}^{2(\lambda_i-\mu_i)}
\frac{1}
     {1-q^{2\mu_i+j} q^{-2i} p} \right]
  \\&=  \prod_i \frac{(q^{2(\lambda_i-i)+1}p;q)_\infty}{(q^{2(\mu_i-i)+1} p;q)_\infty} .
\end{align*}

Let us now rewrite the branching rule for Koornwinder polynomials at $q=t$:
\begin{lemma} \label{Lemma_qt_branching}
 Consider Koornwinder polynomials with $t=q$, $t_0=r\cdot t^{-n-1}$. Define the quasi-branching coefficient through:
\begin{multline*}
 K^{(n+1)}_{\bm \lambda}(x_1,\dots,x_n, r\cdot q^{-n-1};\, q,q;\, r\cdot q^{-n-1} ,t_1,t_2,t_3) \\ =\sum_{\bm\mu\prec \bm\lambda} \tilde \psi_{\bmla/\bmmu} \cdot K^{(n)}_\bmmu(x_1,\dots,x_n;\, q,q;\, r\cdot q^{-n} ,t_1,t_2,t_3).
\end{multline*}
Then we have:
 $$
   \tilde \psi_{\bmla/\bmmu}= \frac{g_{n+1}(\bmla)}{g_n(\bmmu)} \prod_{i=1}^{n+1}\left[ q^{-\lambda_i}\frac{1-q^{2(\lambda_i-i)+n}r t_1 t_2 t_3}{1-q^{-2i+n}r t_1 t_2 t_3} \right].
 $$
 where $  g_n({\bmmu})$ is
 \begin{equation}
 \label{eq_g_def}
 (r)^{-|{\bmmu}|} C^0_{\bmmu}(rt_1 t_2 t_3/q^2)  C^0_{\bmmu}(rt_1/q) C^0_{\bmmu}(rt_2/q)C^0_{\bmmu}(rt_3/q)\cdot  q^{(n+1)|{\bmmu}|} \prod_{i\geq 1}  \frac{(q^{2(\mu_i-i)+n}r t_1 t_2 t_3;q)_\infty}{(q^{-2i+n}r t_1 t_2 t_3;q)_\infty}.
 \end{equation}
\end{lemma}
\begin{proof}
 We denote $q^{2n} t_0 t_1 t_2 t_3=p$. Comparing with Theorem \ref{Theorem_Koornwinder_branching}, we have
 $$
    \tilde \psi_{\bmla/\bmmu}= \psi^{(i)}_{\bmla/\bmmu}\left(q^{n+1};\,q,q, q^{-1}\sqrt{p}\right) \cdot \frac{k^0_{\bmla}(q,q,q^{n+1};\, r\cdot q^{-n-1}:t_1, t_2,t_3)}{k^0_{\bmmu}(q,q,q^{n};\, r\cdot q^{-n}:t_1, t_2,t_3)}.
 $$
 Using the above simplifications for $\psi^{(i)}_{\bmla/\bmmu}$, the last expression is
 {\footnotesize
 \begin{multline*}
   \frac{q^{n|{\bmla}|-n({\bmla})}}{q^{n|{\bmmu}|-n({\bmmu})}} \frac{C^0_{\bmla}(p q^{-1-n}) C^0_{\bmmu}(q^n)}{C^0_{\bmmu}(p q^{-1-n}) C^0_{\bmla}(q^{n+1})}  \frac{\prod\limits_{(i,j)\in{\bmla}}(1-q^{\lambda_i-j} q^{\lambda'_j-i+1})(1-q^{\lambda_i+j} q^{-\lambda'_j-i} p)}{\prod\limits_{(i,j)\in{\bmmu}} (1-q^{\mu_i-j+1} q^{\mu'_j-i})(1-q^{\mu_i+j-1} q^{-\mu'_j-i} p )} \prod_i \frac{(q^{2(\lambda_i-i)+1}p;q)_\infty}{(q^{2(\mu_i-i)+1} p;q)_\infty}
   \\ \times \frac{(r/q)^{-|{\bmla}|} q^{n({\bmla})}} {(r/q)^{-|{\bmmu}|} q^{n({\bmmu})}} \frac{C^0_{\bmla}(q^{n+1})C^0_{\bmla}(rt_1/q) C^0_{\bmla}(rt_2/q)C^0_{\bmla}(rt_3/q)}{C^0_{\bmmu}(q^{n})C^0_{\bmmu}(rt_1/q) C^0_{\bmmu}(rt_2/q)C^0_{\bmmu}(rt_3/q)}  \frac{\prod\limits_{(i,j)\in{\bmmu}} (1-q^{\mu_i-j+1} q^{\mu'_j-i})(1-q^{\mu_i+j-1} q^{-\mu'_j-i} p )}{\prod\limits_{(i,j)\in{\bmla}}(1-q^{\lambda_i-j} q^{\lambda'_j-i+1})(1-q^{\lambda_i+j} q^{-\lambda'_j-i} p)}
   \\
   =    \frac{q^{(n+1)|{\bmla}|}}{q^{(n+1)|{\bmmu}|}} \frac{C^0_{\bmla}(p q^{-1-n}) }{C^0_{\bmmu}(p q^{-1-n}) }
    \frac{(r)^{-|{\bmla}|} } {(r)^{-|{\bmmu}|} } \frac{C^0_{\bmla}(rt_1/q) C^0_{\bmla}(rt_2/q)C^0_{\bmla}(rt_3/q)}{C^0_{\bmmu}(rt_1/q) C^0_{\bmmu}(rt_2/q)C^0_{\bmmu}(rt_3/q)} \prod_i \frac{(q^{2(\lambda_i-i)+1}p;q)_\infty}{(q^{2(\mu_i-i)+1} p;q)_\infty},
 \end{multline*}
 }
 with all $C^0$ being evaluated at $(\cdot;\, q,q)$. It remains to note that
 $$
  p=q^{n-1} r t_1 t_2 t_3. \qedhere
 $$
\end{proof}

Here is a restatement of Lemma \ref{Lemma_qt_branching}: using the notation  \eqref{eq_g_def} we define a symmetric function $\tilde K_{\bmla}(x_1,x_2,\dots,x_n)$ through
 \begin{equation*}
  \tilde K_{\bmmu}(x_1,x_2,\dots,x_n):= \frac{K^{(n)}_{\bmmu}(x_1,\dots,x_n;\, q,q;\, r\cdot q^{-n} ,t_1,t_2,t_3)}{ g_n({\bmmu})}.
 \end{equation*}
 (It still depends on $q,r, t_1,t_2,t_3$, but we are hiding this dependence.) Then we have
 \begin{equation*}
   \tilde K_{\bmla}(x_1,x_2,\dots,x_n, r q^{-n-1})=\sum_{{\bmmu}\prec{\bmla}} \left[\prod_{i=1}^{n+1}\left[ q^{-\lambda_i}\frac{1-q^{2(\lambda_i-i)+n}r t_1 t_2 t_3}{1-q^{-2i+n}r t_1 t_2 t_3} \right]  \right]   \tilde K_{\bmmu}(x_1,x_2,\dots,x_n).
 \end{equation*}

Let us now match to the notations of Theorem \ref{Theorem_qRacah_partition}. For that we transform the weight arising in branching of $n$ (rather than $n+1$) variable polynomial
\begin{equation}
\label{eq_transformation}
 q^{-\lambda_i}\frac{1-q^{2(\lambda_i-i)+n-1}r t_1 t_2 t_3}{1-q^{-2i+n-1}r t_1 t_2 t_3} = \frac{ \sqrt{r t_1 t_2 t_3}\, q^{\lambda_i+\frac{n-1}{2}-i} - \frac{1}{\sqrt{r t_1 t_2 t_3}\, q^{\lambda_i+(n-1)/2-i}}}{ \sqrt{r t_1 t_2 t_3}\, q^{(n-1)/2-i} - \frac{1}{\sqrt{r t_1 t_2 t_3}\, q^{(n-1)/2-i}}}.
\end{equation}
Denoting $x_i=\lambda_i+\frac{n+1}{2}-i$ and $\sigma=q^{-1} \sqrt{r t_1 t_2 t_3}$ and removing the $\lambda_i$--independent normalization, we get the weight
$$
 \sigma q^{x_i} -\frac{1}{\sigma  q^{x_i}}.
$$
We are now ready to finish the proof of Theorem \ref{Theorem_qRacah_partition}. We compute
 \begin{equation}
 \label{eq_x0}
 \sum_{\emptyset={\bmla}^{(0)}\prec{\bmla}^{(1)}\prec\dots\prec {\bmla}^{(n)}={\bmla}} \prod_{k=1}^n \prod_{i=1}^k \left( \sigma q^{\lambda^{(k)}_i+\frac{k+1}{2}-i}-\frac{1}{\sigma  q^{\lambda^{(k)}_i+\frac{k+1}{2}-i}}\right),
 \end{equation}
 which differs from the left-hand side of \eqref{eq_qRacah_sum} by allowing $k=n$.
 Using Lemma \ref{Lemma_qt_branching} recurrently for $n-1$, $n-2$,\dots,  adopting  the notation $\sigma =q^{-1}\sqrt{r t_1 t_2 t_3}$, and transforming the weights as in \eqref{eq_transformation}, the sum \eqref{eq_x0} evaluates to
 \begin{equation}
 \label{eq_x4}
  \left[\prod_{k=1}^n \prod_{i=1}^k \left( \sigma  q^{\frac{k+1}{2}-i}-\frac{1}{\sigma q^{\frac{k+1}2-i}}\right)\right] \frac{K^{(n)}_{\bmmu}(r q^{-1} , r q^{-2},\dots, r q^{-n};\, q,q;\, r\cdot q^{-n} ,t_1,t_2,t_3)}{g_n({\bmla})}.
 \end{equation}
 For the principal specialization of Koornwinder polynomials we use another formula from \cite{rains2005bc}:
 $$
 K^{(n+1)}_{\bmla}(t_0,t_0 t,\dots,t_0 t^{n};\, q,t;\, t_0,t_1,t_2,t_3) =(t_0 t^{n})^{-|{\bmla}|} t^{n({\bmla})} \frac{C^0_{\bmla}(t^{n+1}, t^{n} t_0 t_1, t^{n} t_0 t_2, t^{n} t_0 t_3;\, q,t)}{C_{\bmla}^-(t;\, q,t) C_{\bmla}^+(t^{2n} t_0 t_1 t_2 t_3/q;\, q,t)},
 $$
 where $C^0_{\bmla}$ is as in Theorem \ref{Theorem_Koornwinder_branching} and
 {\allowdisplaybreaks
\begin{align*}
C^+_{\bmla}(x;q,t)&:=\prod_{(i,j)\in {\bmla}} (1-q^{\lambda_i+j-1}
t^{2-\lambda'_j-i} x)
=
\prod_{1\leq i\leq l}
\frac{(q^{\lambda_i} t^{2-l-i} x;q)}
{(q^{2\lambda_i} t^{2-2i} x;q)}
\prod_{1\leq i<j\leq l}
\frac{(q^{\lambda_i+\lambda_j} t^{3-i-j} x;q)}
{(q^{\lambda_i+\lambda_j} t^{2-i-j} x;q)},\\
C^-_{\bmla}(x;q,t)&:=\prod_{(i,j)\in {\bmla}} (1-q^{\lambda_i-j}
t^{\lambda'_j-i} x)
=
\prod_{1\leq i\leq l}
\frac{(x;q)}
{(q^{\lambda_i} t^{l-i} x;q)}
\prod_{1\leq i<j\leq l}
\frac{(q^{\lambda_i-\lambda_j} t^{j-i} x;q)}
{(q^{\lambda_i-\lambda_j} t^{j-i-1} x;q)}.
\end{align*}
}
 Hence, \eqref{eq_x4} becomes
  \begin{multline*}
    \left[\prod_{k=1}^n \prod_{i=1}^k \left( \sigma  q^{\frac{k+1}{2}-i}-\frac{1}{\sigma q^{\frac{k+1}2-i}}\right)\right] (r/q)^{-|{\bmla}|} q^{n({\bmla})} \frac{C^0_{\bmla}(q^{n})C^0_{\bmla}(r t_1/q)C^0_{\bmla}(r t_2/q)C^0_{\bmla}(r t_3/q)}{C_{\bmla}^-(q;\, q,q) C_{\bmla}^+( \sigma  q^{n-1};\, q,q)}
\\ \times  (r)^{|{\bmla}|} \frac{1}{C^0_{\bmla}( \sigma^2)  C^0_{\bmla}(rt_1/q) C^0_{\bmla}(rt_2/q)C^0_{\bmla}(rt_3/q)\cdot  q^{(n+1)|{\bmla}|}} \prod_{i\geq 1}  \frac{(q^{-2i+n+2} \sigma^2;q)_\infty}{(q^{2(\lambda_i-i)+n+2} \sigma^2;q)_\infty}
\\ =   \left[\prod_{k=1}^n \prod_{i=1}^k \left( \sigma  q^{\frac{k+1}{2}-i}-\frac{1}{\sigma q^{\frac{k+1}2-i}}\right)\right] q^{-n|{\bmla}|} q^{n({\bmla})} \frac{C^0_{\bmla}(q^{n})}{C_{\bmla}^-(q;\, q,q) C_{\bmla}^+( \sigma^2 q^{n-1};\, q,q)}
\\ \times  \frac{1}{C^0_{\bmla}( \sigma^2)} \cdot  \prod_{i\geq 1}  \frac{(q^{-2i+n+1} \sigma^2;q)_\infty}{(q^{2(\lambda_i-i)+n+1} \sigma^2;q)_\infty}
\\=  \left[\prod_{k=1}^n \prod_{i=1}^k \left( \sigma  q^{\frac{k+1}{2}-i}-\frac{1}{\sigma q^{\frac{k+1}2-i}}\right)\right]  q^{-n|{\bmla}|+n({\bmla})}
\prod_{1\leq i<j\leq n}
\bigl(1-q^{\lambda_i-\lambda_j} q^{j-i}\bigr) \bigl(1- \sigma^2 q^{\lambda_i+\lambda_j+n-i-j+1}\bigr)
\\ \times
 \prod_{i=1}^n \left[ \frac{(q^{\lambda_i+n-i+1};q)_\infty}{(q;q)_\infty} \frac{(\sigma^2 q^{-2i+n+2} ;q)_\infty}{( \sigma^2 q^{2(\lambda_i-i)+n+2} ;q)_\infty}
\frac{( \sigma^2 q^{2(\lambda_i-i)+n+1} ;q)_\infty}{( \sigma^2 q^{\lambda_i-i+1} ;q)_\infty}  \cdot \frac{(q^{1-i} q^n;q)_{\lambda_i}}{ (\sigma^2 q^{1-i}  ;q)_{\lambda_i}} \right].
 \end{multline*}
It remains to transform the last formula.
First, let us simplify the factor in $\prod_{i=1}^n$:
\begin{multline}
\label{eq_x10}
 \frac{(q^{\lambda_i+n-i+1};q)_\infty}{(q;q)_\infty} \frac{(\sigma^2 q^{-2i+n+2} ;q)_\infty}{(\sigma^2 q^{2(\lambda_i-i)+n+2} ;q)_\infty}
\frac{( \sigma^2 q^{2(\lambda_i-i)+n+1} ;q)_\infty}{( \sigma^2 q^{\lambda_i-i+1} ;q)_\infty}  \cdot \frac{(q^{1-i} q^n;q)_{\lambda_i}}{ (\sigma^2 q^{1-i}  ;q)_{\lambda_i}}\\=
\frac{(q^{n-i+1};q)_\infty}{(q;q)_\infty} \frac{(\sigma^2 q^{-2i+n+2};q)_\infty}{( \sigma^2  q^{1-i} ;q)_\infty}
\bigl(1- \sigma^2 q^{2(\lambda_i-i)+n+1} \bigr)=
\frac{1- \sigma^2 q^{2(\lambda_i-i)+n+1}}{(q;q)_{n-i}\, ( \sigma^2  q^{1-i} ;q)_{n-i+1}}
\end{multline}
Next, let us modify the $\prod_{k=1}^n \prod_{i=1}^k$ factor:
\begin{multline}
\label{eq_x2}
  \left[\prod_{k=1}^n \prod_{i=1}^k \left( \sigma  q^{\frac{k+1}{2}-i}-\frac{1}{\sigma q^{\frac{k+1}2-i}}\right)\right]= (-\sigma)^{-n(n+1)/2} \prod_{1\leq i \leq k \leq n}\left(1-\sigma^2 q^{k+1-2i}\right)\\= (-\sigma)^{-n(n+1)/2}  \prod_{i=1}^n (\sigma^2 q^{1-i};q)_{n-i+1}.
\end{multline}
We observe a nice cancellation between $q$-Pochammer symbols involving $\sigma$ in the last two formulas. Finally, we rewrite the remaining factors:
\begin{multline}
\label{eq_x3}
 q^{-n|{\bmla}|+n({\bmla})}
\prod_{1\leq i<j\leq n}
\bigl(1-q^{\lambda_i-\lambda_j} q^{j-i}\bigr) \bigl(1- \sigma^2 q^{\lambda_i+\lambda_j+n-i-j+1}\bigr)\\=
 q^{\sum_{i=1}^n (i-1-n)\lambda_i}
\prod_{1\leq i<j\leq n}
\bigl(1-q^{\lambda_i-\lambda_j+j-i}- \sigma^2 q^{\lambda_i+\lambda_j+n-i-j+1}+ \sigma^2 q^{2 \lambda_i+n-2i+1}  \bigr)
\\=
 q^{\sum_{i=1}^n (i-1-n)\lambda_i}
\prod_{1\leq i<j\leq n}q^{\lambda_i-i+\frac{n+1}{2}}
\left(q^{i-\lambda_i-\frac{n+1}{2}}+ \sigma^2 q^{\lambda_i-i+\frac{n+1}{2}}  -q^{j-\lambda_j-\frac{n+1}{2}}- \sigma^2 q^{\lambda_j-j+\frac{n+1}{2}} \right)
\\=
 \left[\prod_{i=1}^n q^{(\frac{n+1}{2}-i)(n-i)}\right] q^{-|{\bmla}|}
\prod_{1\leq i<j\leq n}
\left(q^{i-\lambda_i-\frac{n+1}{2}}+ \sigma^2 q^{\lambda_i-i+\frac{n+1}{2}}  -q^{j-\lambda_j-\frac{n+1}{2}}- \sigma^2 q^{\lambda_j-j+\frac{n+1}{2}} \right)
\end{multline}
It remains to multiply \eqref{eq_x10}, \eqref{eq_x2}, \eqref{eq_x3} and divide by $ \prod\limits_{i=1}^k \left( \sigma q^{\lambda_i+\frac{k+1}{2}-i}-\frac{1}{\sigma  q^{\lambda_i+\frac{k+1}{2}-i}}\right)$ to get the result.
\end{proof}

\end{appendix}

\begin{acks}[Acknowledgments]
The authors thank A.~Aggarwal, E.~Dimitrov, and N.~Nekrasov for helpful discussions. We are grateful to the Galileo Galilei Institute for Theoretical Physics, where some of these discussions took place.We thank anonymous referees for proofreading the manuscript and useful suggestions.
\end{acks}

%%%%%%%%%%%%%%%%%%%%%%%%%%%%%%%%%%%%%%%%%%%%%%
%% Funding information, if any,             %%
%% should be provided in the                %%
%% funding section.                         %%
%%%%%%%%%%%%%%%%%%%%%%%%%%%%%%%%%%%%%%%%%%%%%%
\begin{funding}
The work of V.G.\ was partially supported by NSF Grants DMS-1664619, DMS-1949820, DMS-2152588,  and by the Office of the Vice Chancellor for Research and Graduate Education at the University of Wisconsin--Madison with funding from the Wisconsin Alumni Research Foundation.
The research of J.H. is supported by the Simons Foundation as a Junior Fellow at
the Simons Society of Fellows, and NSF Grants DMS-2054835 and DMS-2331096.
\end{funding}

\bibliographystyle{imsart-number} % Style BST file (imsart-number.bst or imsart-nameyear.bst)
\bibliography{References.bib}       % Bibliography file (usually '*.bib')

%%% or include bibliography directly:
%\begin{thebibliography}{4}
%%%
%\bibitem{r1}
%\textsc{Billingsley, P.} (1999). \textit{Convergence of
%Probability Measures}, 2nd ed.
%Wiley, New York.
%
%\bibitem{r2}
%\textsc{Bourbaki, N.}  (1966). \textit{General Topology}  \textbf{1}.
%Addison--Wesley, Reading, MA.
%
%\bibitem{r3}
%\textsc{Ethier, S. N.} and \textsc{Kurtz, T. G.} (1985).
%\textit{Markov Processes: Characterization and Convergence}.
%Wiley, New York.
%
%\bibitem{r4}
%\textsc{Prokhorov, Yu.} (1956).
%Convergence of random processes and limit theorems in probability
%theory. \textit{Theory  Probab.  Appl.}
%\textbf{1} 157--214.
%\end{thebibliography}

\end{document}